\RequirePackage{fix-cm}
\documentclass[smallcondensed]{svjour3}     
\smartqed  
\usepackage{booktabs} 
\usepackage{graphicx}
\usepackage{amssymb}
\usepackage{float}
\usepackage{color}
\usepackage[most]{tcolorbox}
\usepackage{cite}
\usepackage[title,titletoc]{appendix}
\usepackage[ruled,vlined]{algorithm2e}
\UseRawInputEncoding
\usepackage{graphics,graphicx,epstopdf}
\usepackage{physics}
\usepackage{hyperref}
\usepackage{amsmath}
\usepackage{bm}
\usepackage{bbm}
\usepackage{subcaption}
\allowdisplaybreaks
\usepackage[export]{adjustbox}

\makeatother
\SetKwInput{KwInput}{Input}
\SetKwInput{KwConstants}{Constants}
\SetKwInput{KwOutput}{Output}
\SetKwFor{AgentFor}{each agent}{in parallel:}{}
\SetKwFor{SnodeFor}{each supernode}{in parallel:}{}
\SetKwFor{SuperEdgeFor}{each superedge}{in parallel:}{}
\SetKwFor{EdgeFor}{each edge}{in parallel:}{}
\SetKwFor{SubgraphFor}{each subgraph}{in parallel:}{}

\definecolor{blue}{rgb}{0,0,0.9}
\definecolor{red}{rgb}{0.9,0,0}
\definecolor{green}{rgb}{0,0.9,0}
\def\red#1{\textcolor{red}{#1}}
\def\blue#1{\textcolor{blue}{#1}}

\newcommand{\ben}{\begin{enumerate}}
\newcommand{\een}{\end{enumerate}}
\newcommand{\bfr}{\begin{frame}}
\newcommand{\efr}{\end{frame}}
\newcommand{\bit}{\begin{itemize}}
\newcommand{\eit}{\end{itemize}}
\newcommand{\argmin}{\mathop{\rm argmin}}

\newcommand{\R}{\mathbb R}

\newcommand{\cA}{\mathcal A}
\newcommand{\cB}{\mathcal B}
\newcommand{\cC}{\mathcal C}

\newcommand{\cE}{\mathcal E}
\newcommand{\cF}{\mathcal F}
\newcommand{\cG}{\mathcal G}

\newcommand{\cJ}{\mathcal J}

\newcommand{\cN}{\mathcal N}
\newcommand{\cO}{\mathcal O}

\newcommand{\cS}{\mathcal S}
\newcommand{\cT}{\mathcal T}

\newcommand{\cV}{\mathcal V}

\newcommand{\bb}{{\mathbf{b}}}
\newcommand{\bd}{{\mathbf{d}}}

\newcommand{\bx}{{\mathbf{x}}}
\newcommand{\by}{{\mathbf{y}}}
\newcommand{\bz}{{\mathbf{z}}}

\newcommand{\bB}{{\mathbf{B}}}
\newcommand{\bD}{{\mathbf{D}}}
\newcommand{\bE}{{\mathbf{E}}}
\newcommand{\bI}{{\mathbf{I}}}
\newcommand{\bM}{{\mathbf{M}}}
\newcommand{\bT}{{\mathbf{T}}}
\newcommand{\bW}{{\mathbf{W}}}

\def\norm#1{\left\|#1\right\|}
\def\inner#1{\left\langle#1\right\rangle}

\def\kron{\mbox{{\scriptsize$\,\otimes\,$}}}

\def\eps{\epsilon}

\newtheorem{assumption}{Assumption}

\def\dist{{\rm dist}}

\newcommand{\cNout}{\mathcal{N}^{\mathrm{out}}}
\newcommand{\cNin}{\mathcal{N}^{\mathrm{in}}}
\newcommand{\tildecNin}{\widetilde{\mathcal{N}}^{\mathrm{in}}}
\newcommand{\tildecNout}{\widetilde{\mathcal{N}}^{\mathrm{out}}}

\newcommand{\tx}{{\texttt{x}}}

\newcommand{\tm}{{\texttt{m}}}

\newcommand{\bdelta}{{\bm\delta}}

\usepackage{array}
\begin{document}
 
\title{  
A New Decomposition Paradigm for Graph-structured Nonlinear Programs via Message Passing 
\vspace{-.3cm}}


\author{Kuangyu Ding         \and
     Marie Maros  \and   {Gesualdo Scutari} \vspace{-2cm}}


\institute{K. Ding \at Edwardson School of Industrial Engineering, Purdue University, West Lafayette, IN. 
              \email{ding433@purdue.edu}.           
        \and
        M. Maros \at Wm Michael Barnes '64 Department of Industrial and Systems Engineering, Texas A\&M University, College Station, TX. 
              \email{mmaros@tamu.edu}.
              \and
        G. Scutari \at Edwardson School of Industrial Engineering \& Elmore Family School of Electrical and Computer Engineering, Purdue University, West Lafayette, IN. \email{gscutari@purdue.edu}.  \smallskip \newline The work of Ding and Scutari has been supported by the ONR Grant \#  N00014-24-1-2751.       
}

\date{ \vspace{-2.5cm} }

\maketitle

\begin{abstract}
We study finite-sum nonlinear programs with localized variable coupling   encoded by a (hyper)graph. We introduce   a graph-compliant decomposition framework that brings message passing into continuous optimization in a rigorous, implementable, and provable  way.  The (hyper)graph is partitioned into tree clusters (hypertree factor graphs). At each iteration, agents update in parallel by solving local subproblems whose objective splits into  an {\it intra}-cluster term summarized by cost-to-go messages from one min-sum sweep on the cluster tree, and   an {\it inter}-cluster coupling   term handled Jacobi-style using the latest out-of-cluster  variables. 
To reduce computation/communication, the method supports graph-compliant surrogates that replace exact messages/local solves with compact low-dimensional parametrizations; in hypergraphs,
  the same principle enables surrogate hyperedge splitting, to tame heavy hyperedge overlaps while retaining finite-time intra-cluster message updates and efficient computation/communication.  
We establish convergence for (strongly) convex and nonconvex objectives, with topology- and partition-explicit rates that quantify curvature/coupling effects and guide clustering and scalability. To our knowledge, this is the first convergent message-passing method on   loopy graphs.

\end{abstract}
\vspace{-.3cm} 

\section{Introduction}\vspace{-.2cm} 
\label{sec:intro}
We study optimization problems with graphical structure, modeling localized interactions among decision variables.   These interactions are encoded by a (hyper)graph. Let $\cG=(\cV,\cE)$ be a hypergraph with node set  $\cV$ of cardinality $|\cV|=m$; nodes will henceforth be referred to as \emph{agents}. The hyperedge set is   $\cE\subset 2^{\cV}$. Agent $i$ controls a vector $x_i\in \mathbb{R}^d$;   we stack
 $ \mathbf{x} = [x_1^\top, \ldots, x_m^\top]^\top \in \mathbb{R}^{md}$.
For any hyperedge \(\omega\in\cE\), denote by \(x_\omega=(x_i)_{i\in\omega}\in\R^{d|\omega|}\)
the subvector indexed by \(\omega\).
Each agent $i$ and hyperedge $\omega\in \mathcal E$ is associated with the smooth functions $\phi_i:\mathbb{R}^d\to \mathbb{R}$ and $\psi_{\omega}:\mathbb{R}^{|\omega|d}$, respectively. 
The  optimization problem reads\vspace{-0.1cm}
\begin{equation}
    \label{prob_hyper}   \min_{\bx\in\R^{md}}\;\Phi(\bx):=\sum_{i=1}^m\phi_{i}(x_i)+\sum_{\omega\in\cE}\psi_{\omega}(x_{\omega}).
   \tag{P} 
\end{equation}
Throughout the paper,  we assume  
\eqref{prob_hyper}   admits a solution.

A common special case is {\it pairwise} coupling, where         \(\cE\subseteq\{\{i,j\}:i\neq j\}\) (every hyperedge has size two). The graphical model reduces to   an undirected graph with pairwise interactions, and   (\ref{prob_hyper}) specializes to \vspace{-0.2cm}\begin{equation}
    \min_{\mathbf{x} \in \mathbb{R}^{md}} \Phi(\mathbf{x}) = \sum_{i = 1}^m \phi_i(x_i) + \sum_{(i,j) \in \mathcal{E}} \psi_{ij}(x_i, x_j), 
    \label{P} 
  \tag{P$^\prime$}  
\vspace{-0.2cm}\end{equation}
where \(\psi_{ij}:\R^{2d}\to \mathbb{R}\) is symmetric in its arguments,  i.e., $\psi_{ij}(x_i, x_j)= \psi_{ji}(x_j, x_i).$ 

While many problems can be cast in the form~\eqref{prob_hyper} (e.g., by using a single factor over all variables), our focus is on \emph{sparse} graphical structure, meaning each coupling term $\psi_{\omega}$ depends only on a small subset of agents’ variables. Such sparsity is central in applications ranging from Markov random fields in statistical physics to graph-based signal processing, communication and power networks, and large-scale decentralized optimization; see Sec.~\ref{sec:app} for motivating examples.

Leveraging graphical sparsity, our goal is to design decentralized algorithms that decompose~\eqref{prob_hyper} into weakly coupled subproblems solved via \textit{localized computation} and \textit{single-hop} neighbor communication. Key desiderata are: \textbf{(i)} communication efficiency (costs scale with local degree, not network size), \textbf{(ii)} scalability to large graphs, and \textbf{(iii)} minimal coordination overhead (no central coordinator). As we detail next, existing approaches generally fall short. 


 \textbf{(i) Graph-agnostic decentralizations.} Naively decentralizing gradient descent or classical block methods (Jacobi/Gauss--Seidel)~\cite{BerTsi89} is communication inefficient and ignores the graphical structure. In gradient descent,  agent \(i\) needs\vspace{-0.1cm}
$$\nabla_{i}\Phi(\bx)
=\nabla\psi_i(x_i)+
\sum_{\omega\in \mathcal{E}:\, i\in \omega}\;
\nabla_{i}\psi_{\omega}\!\big(x_{\omega}\big),\vspace{-0.1cm}$$
which depends on all variables in each factor scope $\omega$  containing $i$; thus agents must collect and broadcast information beyond single-hop neighborhoods, incurring heavy coordination overhead. Likewise, vanilla Jacobi requires dense, iteration-by-iteration neighbor synchronization, while Gauss--Seidel enforces sequential updates and global scheduling. In both cases, per-iteration communication scales with network size rather than local degree.

\textbf{(ii) Consensus-based methods.} Problem~\eqref{prob_hyper}      can be cast as min-sum  program  over networks--i.e., $\min_x \sum_{i=1}^m f_i(x)$--for which a vast literature of gossip/consensus schemes exists, including  decentralized  gradient methods~\cite{AA09,AAP10,KQW16,Amir18}, gradient tracking-based  algorithms~\cite{di2016next,sun2022distributed,nedic2017achieving,scutari2019distributed} and      primal-dual variants~\cite{shi2015extra,yuan2018exact,WQKGW14};  we refer to the tutorial~\cite{Nedic_Olshevsky_Rabbat2018} and the monograph~\cite{sayed2014adaptation} for   further references. 
Despite their differences, these methods share a common \emph{lifting}: each agent maintains a local copy $x_i$ of the \emph{full} decision vector $\mathbf{x}$ and enforces agreement via gossip averaging or dual reformulations. This is ill-suited for \emph{sparsely coupled} problems like~\eqref{prob_hyper}:
\textbf{(a)} \emph{memory/communication inflation}---agents store and exchange $md$-dimensional vectors regardless of  their degree or size of local factor scopes; and
\textbf{(b)} \emph{consensus-limited rates}---convergence is governed by the mixing spectral properties (e.g., spectral gap), which can be slow on poorly connected graphs and scale poorly with $m$ on some topologies~\cite{scaman2017optimal,Xu_Tian_Sun_Scutari_2020,Cao_Chen_Scutari_2025}.
This motivates methods that operate on \emph{single-hop, scope-sized} messages without replicating the entire decision vector.

 \textbf{(iii) Methods for   overlapping neighborhood objectives.} The line of work  \cite{Cai_Keyes_02,Shin_et_al_20,Cannelli2021,Chezhegov_et_al_22,Latafat02112022} develops   decentralized algorithms  tailored to
\emph{neighborhood}-{\it scoped} cost functions--a special instance of~\eqref{prob_hyper}--of the form \vspace{-0.2cm}\begin{equation}\label{eq:neigh_P}  \min_{\mathbf{x} \in \mathbb{R}^{md}} \,\,\Phi(\bx)=\sum_{i=1}^m  \phi_{i}(x_{\mathcal{N}_i}),\vspace{-0.2cm}\tag{P$^{\prime\prime}$}\end{equation} where $x_{\cN_i}$ is the vector containing the variables of agents $i$ and its immediate neighbors. 
The difficulty here is that a local update at agent $i$ (or on $x_{\cN_i}$) requires \emph{current} information on
blocks  {\it controlled by other agents}. The following are existing strategies  to cope with this multi-hop coupling (hyperedge).  


\emph{(a) Overlapping domain decomposition/Schwarz-type patch updates.}These methods form \emph{overlapping} subproblems by enlarging each \emph{block of variables} (patch) by creating \emph{local copies} of the variables in the overlap (one copy per patch), thereby decoupling the local minimizations.  The algorithm then alternates block updates with the exchange of overlap values (or boundary
interfaces)  across neighboring patches to reconciliate overlap-inconsistency; see~\cite{Shin_et_al_20} and Schwarz/domain-decompositions~\cite{Toselli_Widlund_05, Cai_Keyes_02}.  
In graph settings as \eqref{eq:neigh_P}, this yields decentralized (asynchronous)    Jacobi- or Gauss-Seidel-like updates on overlapping neighborhoods. Their convergence is established most notably for strongly convex {\it quadratic} objectives \cite{Shin_et_al_20}--typically under implicit con- tractivity/diagonal-dominance-type conditions  
that are difficult to verify a priori. Moreover, when convergence rates are provided, they are rarely explicit: the contraction factor is usually buried in graph- and overlap-dependent constants (e.g., spectral radii of certain operators), making it hard to extract clean scaling laws with respect $d$, $m$, or the degree of overlapping. 

  \emph{(b) Asynchrony, delays, and event-triggered exchanges.}
A complementary line of work reduces the communication burden   by
limiting \emph{when} and \emph{which} information (primal blocks, gradients, or local models) is exchanged,
via asynchronous updates, randomized activations, or event-triggered communication; see, e.g.,~\cite{Cannelli2021,Chezhegov_et_al_22,Latafat02112022} and references therein.
These works typically prove convergence   making the effect of asynchrony explicit through delay/activation parameters.
However, the \emph{structural} dependence of the rate on overlap-induced coupling--e.g., hyperedge size $|\cN_i|$, overlap multiplicity, or topology features such as diameter/expansion--is rarely exposed.
Instead, scalability is absorbed into global constants (Lipschitz/monotonicity moduli, operator norms, error-bound parameters, etc.) that conflate graph and overlap effects and are hard to interpret or certify a priori.
Consequently, these schemes provide limited guidance on
how communication savings trade off against convergence as $m$ grows or overlap increases.\vspace{-.4cm}

 \subsection{A message-passing viewpoint for graph-structured optimization} \label{sec:message-passing-primer}\vspace{-.2cm}
This paper takes a fundamentally different route than existing graph-compliant decompositions: we develop (to our knowledge) the first framework for nonlinear programs on graphs and hypergraphs that is \emph{rigorously built around message passing}, with an explicit focus on computational and communication efficiency.

Message passing originates in statistics and information theory for inference on graphical models~\cite{WainwrightJordan2008}. Its direct transplant to continuous nonlinear programs typically leads to methods that may fail to converge on loopy graphs and incur prohibitive communication overhead. 
These bottlenecks help explain why message passing has remained peripheral in optimization, despite its appealing locality.

To make this perspective precise,  next we provide a brief primer on min-sum/message passing and then review existing min-sum-type methods for nonlinear programs and their limited convergence guarantees.

 \noindent\textbf{Message passing on acyclic graphs: a primer.} Consider the pairwise instance of   (\ref{P}) over a path graph $\mathcal G=(\mathcal V, \mathcal E)$, $\mathcal V=\{1,\ldots, m\}$ and $\mathcal E=\{\{i,i+1\}\,:\, i=1,\ldots m-1\}$:\vspace{-.1cm}
 \begin{equation}
   \underset{x_1,\ldots, x_m \in \mathbb{R}^{d}}{\text{minimize}}  \,\, \sum_{i = 1}^m \phi_i(x_i) + \sum_{j=1}^{m-1} \psi_{j, j+1}(x_j, x_{j+1}). 
    \label{P_path}
\end{equation}
Here, each node $i$ controls the variable $x_i$,  has access to $\phi_i$,  $\psi_{i-1, i}$, $\psi_{i, i+1}$, and communicates with its immediate neighbors (nodes $i-1$ and $i+1$). 

When the   graph is \emph{cycle-free} as above, min-sum/sum-product message passing reduces to dynamic programming, leveraging  the following decomposition of (\ref{P_path}):   
   \begin{equation}
  x_i^\star\in  \underset{x_i \in \mathbb{R}^{d}}{\text{argmin}}\,\,    \phi_i(x_i) +  \mu_{i-1\to i}^\star(x_i)+  \mu_{i+1\to i}^\star(x_i),  \quad \forall i=1,\ldots, m,
    \label{P_i_pat}\vspace{-0.2cm}
\end{equation}
where \vspace{-0.3cm}
\begin{subequations}\label{eq:both-message}
    \begin{align}
    \mu_{i-1\to i}^\star(x_i)&:= \underset{x_1,\ldots, x_{i-1} \in \mathbb{R}^{d}}{\min}  \sum_{j=1}^{i-1}\phi_j(x_j) + \sum_{j=1}^{i-1} \psi_{j, j+1}(x_j, x_{j+1}),\label{eq:left-message} \\
     \mu_{i+1\to i}^\star(x_i)&:= \underset{x_{i+1},\ldots, x_{m} \in \mathbb{R}^{d}}{\min}  \sum_{j=i+1}^{m}\phi_j(x_j) + \sum_{j=i+1}^{m-1} \psi_{j, j+1}(x_j, x_{j+1}), \label{eq:right-message}
\end{align}\end{subequations}
with boundary conditions \(\mu_{0\to 1}^\star(\cdot)\equiv 0\) and \(\mu_{m+1\to m}^\star(\cdot)\equiv 0\). 

The functions $\mu_{i-1\to i}^\star$ and  $\mu_{i+1\to i}^\star$   are  the optimal costs of the left and right subchains of node $i$,   conditioned on $x_i$ through the boundary factors  $\psi_{i-1,i}$ and $\psi_{i,i+1}$, respectively. 
 It is assumed that $\mu_{i-1\to i}^\star(\cdot)$ and  $\mu_{i+1\to i}^\star(\cdot)$  are sent to node $i$ from its neighbors $i-1$ and $i+1$, respectively; hence the name  ``messages''.

  Given the pairwise model in (\ref{eq:left-message})-(\ref{eq:right-message}), the messages    
  factorizes as:   
\begin{subequations}
    \begin{align}
    \mu_{i-1\to i}^\star(x_i)&= \underset{x_{i-1} \in \mathbb{R}^{d}}{\min}   \phi_{i-1}(x_{i-1}) + \psi_{i-1, i}(x_{i-1}, x_{i})+  \mu_{i-2\to i-1}^\star(x_{i-1}),\label{eq:left-message_2} \\
    \mu_{i+1\to i}^\star(x_i)&= \underset{x_{i+1} \in \mathbb{R}^{d}}{\min}   \phi_{i+1}(x_{i+1}) + \psi_{i, i+1}(x_{i}, x_{i+1})+  \mu_{i+2\to i+1}^\star(x_{i+1}). \label{eq:right-message_2}
\end{align}\end{subequations}
This recursive  decomposition suggests a simple iterative procedure for the computation of these functions over the path graph, using only neighboring   exchanges:   given the iteration index $\nu=0,1,2,\ldots$ and arbitrary initialization of   all messages (with $\mu_{0\to 1}^{0}\!\equiv\!0$ and $\mu_{m+1\to m}^{0}\!\equiv\!0$),   
at iteration $\nu +1$, perform a {\it forward} pass 
\vspace{-0.1cm}
\[
\mu_{i\to i+1}^{\nu+1}(x_{i+1})
=\min_{x_i}\Big\{\phi_i(x_i)+\psi_{i,i+1}(x_i,x_{i+1})+\mu_{i-1\to i}^{\nu}(x_i)\Big\},\quad i=1,\ldots,m-1,
\]
and a {\it backward} pass: for $i=m-1,\ldots,1$, 
\[
\mu_{i+1\to i}^{\nu+1}(x_i)
=\min_{x_{i+1}}\Big\{\phi_{i+1}(x_{i+1})+\psi_{i,i+1}(x_i,x_{i+1})+\mu_{i+2\to i+1}^{\nu}(x_{i+1})\Big\}. 
\]
The algorithm above reaches the fixed point   \eqref{eq:left-message_2}-\eqref{eq:right-message_2}  in {at most $m-1$} iterations (each composed of one forward pass plus one backward pass), each communicating $2(m-1)$ messages total.  Finally each node $i$ compute  $x_i^\star$ via (\ref{P_i_pat}).

 The above idea and convergence guarantees extend to graph-structured problems~\eqref{P} whose $\mathcal G$ is an \emph{acyclic} factor graph; 
 see, e.g.,~\cite{KschischangFreyLoeliger2001}.  
This is the clean regime in which message-passing-based methods  are both exact and well understood.\smallskip 

\noindent {\bf Message passing on loopy graphs.}  On graphs with cycles, convergence guarantees exist only in
{\it restricted regimes}, mostly for \emph{pairwise} models (not directly applicable to (\ref{prob_hyper})).  
{\bf (i)}  {\it Scaled diagonal dominance:} 
The  most general convergence results for min-sum message passing methods are for   \emph{pairwise-separable convex}
objectives (as in \eqref{P}), satisfying  the so-called \emph{scaled diagonal dominance} condition, which enforces that local curvature
dominates pairwise couplings; under this assumption, (a)synchronous min-sum converges,
and explicit convergence  rate bounds are derived~\cite{moallemi2010convergence,zhang2021convergence}.
However,  the dominance requirement significantly
restricts the class of problems  that can be handled. {\bf (ii)}{\it  Quadratic objectives (Gaussian).}
For strictly convex quadratics, \emph{correctness} and \emph{convergence} decouple: if loopy Gaussian Belief Propagation/min-sum converges, it yields the exact minimizer on graphs of arbitrary topology~\cite{weiss2001correctness}.
However,  convergence itself requires additional structure such as    \emph{walk-summability}/\emph{convex decomposability} (and compatible initializations), and not every strongly convex quadratic  function satisfies these conditions~\cite{malioutov2006walksums,Moallemi-VanRoy2009-quadratic}.
Graph-cover analyses explain failure modes and motivate reweighted/parameterized variants; yet,  convergence   for  \emph{arbitrary} strongly convex quadratic  is not established beyond empirical evidence  
\cite{ruozzi2013message}.
{\bf (iii)}{\it  Higher-order (hypergraph) couplings.}
Min-sum   extends to higher-order factor graphs (as in (\ref{prob_hyper})) and remains exact on acyclic graphs~\cite{KschischangFreyLoeliger2001}, but for
\emph{loopy}   nonlinear programs, convergence  theory remains far less developed--beyond dominance-type assumptions for convex factors~\cite{moallemi2010convergence}, topology-/arity-explicit rate characterizations are essentially missing.
  \smallskip 

\noindent \textbf{Pervasive bottlenecks.}
Across (i)–-(iii), a first obstacle is \emph{practical implementability}: messages are \emph{functional} objects in continuous domains, and even for vector-valued quadratics parametrizations can require exchanging \emph{dense} second-order information (matrices)~\cite{moallemi2010convergence,Moallemi-VanRoy2009-quadratic,ruozzi2013message}. 
 This issue is exacerbated  over hypergraphs: factor-to-variable updates entail multi-variable partial minimizations over  hyperedges, so both computation and communication  grow rapidly with the hyperedge size~\cite{moallemi2010convergence}. 
Second, {\it rate interpretability}:   when   bounds exist, they are typically stated through global dominance/spectral constants that entangle topology and problem parameters, so clean scalability laws in graph features (degree, overlap, diameter, separators, etc.) remain unknown in full generality~\cite{Moallemi-VanRoy2009-quadratic,zhang2021convergence}. 

 
\vspace{-0.3cm}

\subsection{Main contributions}
 \vspace{-0.1cm}

The main contributions of the paper are as follows. \smallskip 

\noindent $\bullet$ \textbf{MP-Jacobi: A message-passing-based decomposition.}  We introduce a new decomposition principle for graph-structured nonlinear programs in the form \eqref{P}, built \emph{rigorously} on the machinery of message passing.
Starting from a new graph-compliant fixed-point characterization induced by a partition of the  graph into \textit{tree clusters}, we decouple interactions into:
{\bf (i)}{\it  intra-cluster} couplings, captured by one min--sum sweep on the selected intra-cluster trees, and
{\bf(ii)}{\it  inter-cluster} couplings, handled \textit{Jacobi-style} using out-of-cluster neighbors' most recent variables.
The resulting \textit{Message-Passing Jacobi} (MP-Jacobi) method is a \textit{single-loop}, fully decentralized scheme that, per iteration, performs one (damped) local Jacobi minimization together with one step of intra-cluster message updates.
This yields (to our knowledge) the first {\it provably convergent} message-passing-based decomposition for continuous optimization on \emph{loopy graphs} that is graph-compliant and implementable with purely local, single-hop communication. From the Jacobi viewpoint, it also resolves a long-standing obstacle: realizing a genuine block-Jacobi decomposition without multi-hop replication or consensus-style global coordination. \smallskip

     \noindent$\bullet$ \textbf{Novel convergence analysis  and topology-explicit rates. }
By casting MP-Jacobi as a delayed block-Jacobi method on a partition-induced operator, we establish \textit{linear} (resp. {\it sublinear}) convergence    for strongly convex (resp. merely convex or nonconvex) objectives. Our theory provides (to our knowledge) the \textit{first} convergence guarantees for a \textit{message-passing-based} algorithm on \emph{loopy graphs}, which   do \emph{not} rely on  structural properties   of the loss function such as  diagonal dominance or walk-summability, but only on stepsize tuning. 
Such unrestricted convergence is ensured by the   Jacobi correction  that stabilizes cross-cluster interactions while allowing message passing to run on tree components only.
Moreover, unlike existing message-passing analyses where   rates are buried in global spectral/dominance constants, our   contraction factor is    explicit in terms of (i) local and coupling regularity/curvature parameters and (ii) \textit{topological features of the chosen partition} (e.g., intra-cluster trees and boundary interfaces), thereby offering concrete guidance for partition design and scalability with the network size.\smallskip 

\noindent $\bullet$ \textbf{Surrogation.}
    To overcome the main bottleneck of classical message passing in optimization--high per-iteration computation/communication due to functional messages--
    we develop \textit{surrogate MP-Jacobi}, a variant of MP-Jacobi, based on surrogation.
   It  preserves the same decentralized architecture and convergence guarantees while enabling lightweight message parametrizations (at the level of vectors/ low-dimensional summaries) and offering explicit cost--rate trade-offs.
    We provide practical surrogate constructions that cover common smooth/convex structures and substantially reduce per-iteration complexity, yet preserving fast convergence.\smallskip 
    
\noindent$\bullet$ \textbf{Hypergraphs.}
    We extend the framework to hypergraph models~\eqref{prob_hyper}, by deriving a decomposition method employing hypergraph message-passing {\it within} hypergraph {\it factor graph tree} clusters and Jacobi-like updates over {\it inter}-cluster hyperedge factors.  
    For  cycle-rich  hypergraphs with heavy overlaps, we propose a \textit{surrogate hyperedge-splitting} strategy that removes intra-cluster loops so that tree-based message updates remain applicable, while retaining 
    fast convergence provably.\smallskip 

  \noindent$\bullet$   \textbf{Empirical validation and practical gains.}
    Extensive experiments on graph and hypergraph instances corroborate the theory, quantify the impact of clustering and surrogate design, and demonstrate substantial gains over   gradient-type baselines and decentralized consensus methods, highlighting MP-Jacobi as a scalable primitive for distributed optimization with localized couplings.

 \vspace{-0.4cm}

\subsection{Motivating applications}\label{sec:app}
 \vspace{-0.2cm}
   
\noindent \textbf{(i) Laplacian Linear systems.} Graph-Laplacian systems arise across scientific computing and data science
(e.g., electrical flows \cite{nash1959random}, random-walk quantities such as hitting times, and Markov-chain stationarity \cite{nk2012lx}). 
A standard formulation is  
\vspace{-0.2cm}
$$   \min_{\mathbf{x} \in \mathbb{R}^m} \quad  -\sum_{i=1}^m b_i x_i + \frac{1}{2}\sum_{i=1}^m \sum_{j=1}^m L_{ij}x_i x_j, \vspace{-0.2cm}
$$  
where $L=[L_{ij}]_{i,j=1}^m$ is the graph Laplacian. This is an instance of \eqref{eq:neigh_P}.  

\noindent \textbf{(ii) Signal processing on graphs} \cite{wang2016trend, romero2016kernel} extends traditional signal processing techniques to graph-structured data. A signal on a graph is a function that assigns a value to each node on a graph, representing data. The graph  structure   encodes relations among the data that are represented through edges. A key task is  recovering a graph signal $z$ from a noisy observation $y = z + n,$ where $n \sim \mathcal{N}(0,\Sigma),$ under the smoothness prior that neighboring nodes should have similar values. A canonical estimator solves the following neighborhood-scoped problem 
\cite{isufi2024graph}:\vspace{-0.1cm}
 $$    \min_{\mathbf{x} \in \mathbb{R}^{md}} \quad \sum_{i=1}^m d_i \|y_i - x_i\|^2 + \frac{1}{2\mu}\sum_{(i,j) \in \mathcal{E}}w_{ij}\|x_i - x_j\|^2.
 $$

\noindent \textbf{(iii) Personalized and multi-task learning:} each agent owns a local private dataset and aims to learn a personalized model according to their own learning objective. Neighboring agents have similar objectives and therefore, agents benefit from interacting with their neighbors. This is captured by the problem \cite{dinh2022new, wan2025multitask}:\vspace{-0.1cm}
$$ \min_{\mathbf{x} \in \mathbb{R}^{md}} \quad  \sum_{i=1}^m d_{i}\mathcal{L}_i(x_i) + \frac{1}{2\mu}\sum_{(i,j) \in \mathcal{E}} w_{ij}\|x_i-x_j\|^2,\vspace{-0.1cm}
$$  where $\mathcal{L}_i:\mathbb{R}^{d} \to \mathbb{R}$ denotes agent $i's$ private loss, and $\mu > 0$ is a trade-off parameter. The above formulation also arises in multi-task learning with soft coupling among tasks \cite{ruder2017overview}, where adjacency in the graph now represents tasks that share similarities and thus should inform one another during training. 

\noindent \textbf{(iv) Distributed optimization.} 
 In the classical decentralized   optimization setting,  agents  solve 
$ \min_{x \in \mathbb{R}^{d}} \quad \sum_{i=1}^m f_i(x)$,
 with each  agent $i$   having access   to $f_i$ only and   communicating with its immediate neighbors. 
 Two notorious algorithms  are  
 the DGD-CTA (Decentralized Gradient Descent - Combine Then Adapt) \cite{yuan2016convergence} and the DGD-ATC (DGD-Adapt Then Combine) \cite{chen2012diffusion};  they   can be interpreted  as  gradient descent applied to lifted, graph-regularized problems 
 \begin{equation}
 \min_{\mathbf{x} \in \mathbb{R}^{md}} \quad \sum_{i=1}^m f_i(x_i) + \frac{1}{2\gamma}\|\mathbf{x}\|^2_{\mathbf{I}_{md}-\mathbf{W}\kron\mathbf{I}_d},
 \label{eq:CTA}
 \end{equation}
 and\vspace{-.2cm}
 \begin{equation}
     \min_{\mathbf{x} \in \mathbb{R}^{md}} \quad \sum_{i=1}^m f_i\left( \sum_{j \in \mathcal{N}_i} w_{ij} x_j \right) + \frac{1}{2\gamma}\|\mathbf{x}\|^2_{\mathbf{I}_{md}-\mathbf{W}^2\kron\mathbf{I}_d}.
     \label{eq:ATC}
 \end{equation}

\noindent \textbf{(v) MAP (maximum a posteriori) on  graphical models:} Let $\bx$ be latent variables and $\by$ measurements, with posterior
\[
p(\bx\mid\by)\ \propto\ \prod_{\omega\in\cE}\pi_\omega(x_\omega)\prod_{\omega\in\cE}p_\omega(y_\omega\mid x_\omega),
\]
where $\omega$ indexes a factor scope,   $x_{\omega}$ is the subset of latent variables associated with   $\omega$, and $y_{\omega}$ are the measurements associated with   $x_{\omega}.$   The MAP   estimator is obtained by minimizing the negative log-posterior
$\min_{\bx}\,-\log p(\bx\mid\by)$, which fits~\eqref{prob_hyper} and underpins many applications, 
  such as robotics, speech processing, bio-informatics, finance and computer vision. 
  
  \textbf{(a)} {\it Robot localization and motion planning} (e.g., ~\cite{cadena2017past}).
The variables $x_1,\ldots,x_m\in\R^d$ represent a discrete-time trajectory of robot poses (and possibly landmark states).
The robot acquires a set of $|\cE|$ measurements, each involving a subset of states $\mathbf{x}$:   $y_{\omega} = h_{\omega}(x_{\omega}) + \varepsilon_{\omega}$, $\varepsilon_{\omega} \sim \mathcal{N}(0,\Sigma_{\omega})$, with known measurement model   $h_{\omega}$. Assuming conditional independence across factors given $x_\omega$, the MAP is obtained  by setting 
$
p_\omega(y_\omega\mid x_\omega)\ \propto\ \exp\!\big(-\tfrac12\|h_\omega(x_\omega)-y_\omega\|^2_{\Sigma_\omega^{-1}}\big),
$
while $\pi_\omega(x_\omega)$ encodes priors/regularizers (e.g., motion priors, landmark priors), and can be omitted to obtain
a maximum-likelihood (ML) estimator. The resulting MAP/ML problem is a core computational primitive in  simultaneous localization and mapping. 
  
  \textbf{(b) { \it Hidden Markov Models (HMMs)}} 
  Consider a HMM $(X,Y)$ with hidden states $\{X_i\}$ and observations $\{Y_i\}$, where   $X_{i+1}$ depends only
on $X_i$ and the observation model links $Y_i$ to $X_i$. Given observations $(y_1,\ldots,y_m)$, the goal is to recover the most
likely hidden sequence $(x_1,\ldots,x_m)$. The induced dependence structure is a line graph, with factors
 $\pi_1(x_1)=\mu(x_1)$, $\pi_{(i,i+1)}(x_i,x_{i+1})=q_{(i+1,i)}(x_{i+1}\mid x_i),
$
where $\mu$ is the initial-state density and $q_{(i+1,i)}$ the transition density. The likelihood factors as
$p_\omega(y_\omega\mid x_\omega)=\prod_{i\in\omega}q_i'(y_i\mid x_i)$, where $q_i'$ is the emission density. MAP then reduces to
a structured instance of~\eqref{prob_hyper} (pairwise on a chain).

  
  \textbf{(c) {\it Image denoising/inpainting via Fields of Experts (FoE)}}~\cite{roth2009fields,wang2014efficient}. Pixels (or patches) in an image $\bx$  are treated as nodes in a graph, and a neighborhood system induces high-order cliques: each neighborhood/patch
$\omega$ defines a maximal clique $x_\omega$ and hence a hyperedge. FoE specifies a    high-order Markov random field prior through the product of $L$ expert factors $\pi_{\omega}(x_{\omega}) = \prod_{i=1}^L \phi(J_i^{\top}x_{\omega},\alpha_i)$ where $\{J_\ell\}$ is a learned filter bank and $\{\alpha_\ell\}$ are learned parameters; e.g.,~\cite{wang2014efficient} uses
$\phi(J_\ell^\top x_\omega,\alpha_\ell)=(1+\tfrac12(J_\ell^\top x_\omega)^2)^{-\alpha_\ell}$. 
Given a noisy image $\by$,
 the goal is to find a faithful  denoised image $\mathbf{x}$. For this, set $\prod_{\omega \in \varepsilon}p_{\omega}(y_{\omega}|x_{\omega}) \propto \exp\left(-\frac{\lambda}{\sigma^2}\|\mathbf{x} - \mathbf{y}\|^2 \right),$ so the MAP estimate balances fidelity to $\by$ and the FoE prior, yielding a hypergraph-structured instance of~\eqref{prob_hyper}.

\vspace{-0.7cm}

\subsection{Notation} \vspace{-0.2cm}
The power set of a set $\mathcal{S}$  is denoted by $2^{\mathcal{S}}$. Given an undirected hypergraph \(\cG = (\cV,\cE)\) (assumed throughout the paper having no selfloops), with node-set \(\cV = \{1,2,\ldots,m\}\) and hyperedge-set  \(\cE \subseteq 2^\cV\), the  {hyperedge} neighborhood of agent \(i\) is defined as $\cN_i {:=} \{\, \omega \in \cE \mid i \in \omega \,\}.$  For any  $\mathcal S \subseteq \cV$, we denote its complement by ${\overline{\mathcal{S}}} := \cV \setminus \mathcal S$.  In particular, when $|\omega|=2$, for any $\omega\in\cE$, $\cG$  reduces to the pairwise graph--the neighborhood of agent \(i\) reads  $\cN_i {=} \{\, j \in \cV \mid (i,j) \in \cE \,\}. $   The diameter of  {$\mathcal G$}   is denoted by  {$\mathrm{d}_{\cG}$}. For any integer $s>0$, we write $[s]:=\{1,\ldots, s\}$.

Given \(\bx\in\R^{md}\), \(\cC\subseteq\cV\), 
we write \(\bx_{\cC}\in\R^{|\cC|d}\) for the subvector formed by the blocks of \(\bx\) indexed by \(\cC\); 
for given $\bx=[x_1,\ldots ]$ and  $[\texttt{d}_1, \ldots ]$ of suitable dimensions (clear from the context), with $x_i\in \mathbb{R}^d$, $\texttt{d}_i\in \mathbb{N}$,   and  given $\nu\in \mathbb{N}$, we use the shorthand     $x_{{\cC}}^{\,\nu-\bd}:=(x_i^{\,\nu-\texttt{d}_i})_{i\in\cC}$. {For any set of indices $\mathcal C\subseteq [m]$, define a selection matrix   $U_{\cC}$ that extracts the $d$-dimensional blocks in $\bx$ indexed   by $\cC$: $U_{\cC}^\top \bx=\bx_{\cC}$. 
The orthogonal projector onto the coordinates indexed by $\cC$ is then $P_{\cC}:=U_{\cC}U_{\cC}^\top$.} For $\Phi:\R^{md}\to\R$ and   $\cC\subseteq\cV$, we denote by $\nabla_{\cC}\Phi(\bx)\in\R^{|\cC|d}$ the gradient of $\Phi$ with respect to $\bx_{\cC}$. For index sets $\cA,\cB\subseteq\cV$,  
$
\nabla^{2}_{\cA,\cB}\Phi(\bx)\in\R^{|\cA|d\times|\cB|d}
$
is  the corresponding block Hessian, whose $(i,j)$-th ($d\times d$) block is 
$\nabla_{j}(\nabla_{i}\Phi(\bx))^\top$,    $i\in\cA$ and $j\in\cB$. We say $\Phi\in C^k$ if $\Phi$ is $k$-times differentiable. If, in addition, its $k$-th derivative is (globally) Lipschitz, we write $\Phi\in LC^k$.
Given two functions \(f(x,y)\) and \(g(x,y)\), we say that \(f\) and \(g\) are equivalent with respect to \(x\), denoted $f(x,y) \overset{x}{\sim} g(x,y),$ if the difference \(f(x,y) - g(x,y)\) is independent of \(x\). When the variable is clear from context, we simply write \(f\sim g\). The set of $d\times d$ positive semidefinite (resp. definite) matrices is denoted by $\mathbb{S}_+^d$ (resp.  $\mathbb{S}_{++}^d$).

\vspace{-0.4cm}
\section{Decomposition via Message Passing: Pairwise Problem~\eqref{P}}
\label{sec:alg}\vspace{-0.2cm}
This section focuses on the design of decomposition algorithms for the pairwise formulation \eqref{P}.
Our approach exploits a key property of message-passing algorithms (see Sec.~\ref{sec:message-passing-primer}): on loopless graphs, they converge in a finite number of steps. This motivates the following graph partition and assumption.

 \begin{definition}[condensed graph] 
 \label{def:condensed_graph}
 Given  $\cG = (\cV,\cE)$ and $p\in [m]$, let   \(\cC_1,\ldots, \cC_p\) be a    partition of  \(\cV\),  with associated intra-cluster edge-sets $\cE_1,\ldots, \cE_p$, where  \begin{equation}\label{eq:nonoverlap}
     \cE_r:=\{(i,j)\in \mathcal E\,|\, i,j\in \mathcal C_r\}, \quad  r\in [p];
 \end{equation}  this results in the subgraphs $\mathcal G_r=(\cC_r,\cE_r)$, $r\in [p]$.

The condensed graph relative to    \(\{\cG_r\}_{r=1}^{{p}}\) is  defined as   \(\cG_\cC := (\cV_\cC,\cE_\cC)\), where\vspace{-.1cm}
\begin{itemize} 
    \item    \(\cV_\cC := \{\texttt{c}_1,\ldots,\texttt{c}_p\}\) is the set of supernodes, with  $\texttt{c}_r$   associated with    $\cC_r$; and 
    \item $\cE_\cC\subseteq \cV_\cC\times \cV_\cC$ is the set of superedges:     \((\texttt{c}_r,\texttt{c}_s) \in \cE_\cC\) iff   \((i,j) \in \cE\), for some   \(i\in \cC_r\) and \(j\in \cC_s\),  with $i\neq j$.\vspace{-.1cm}
\end{itemize}
 For each $i\in \cC_r$, define 
   $$\cNin_i:= \cN_i\cap \cC_r,\quad \cNout_i:=\cN_i\setminus\cC_r,\quad \text{and}\quad\cN_{\cC_r}:=\cup_{i\in\cC_r}\cNout_i;$$ $\cNin_i$  (resp. $\cNout_i$) is   the neighborhood of $i\in\cC_r$  within  (resp. outside) $\cG_r$       and    \(\cN_{\cC_r}\) is  the external neighborhood of $\cC_r$. Finally,   
    $D_r:={\rm diam}(\cG_r)$ is   the diameter of $\cG_r$,  with $D:=\max_{r\in[p]}D_r$;   and  
  $\cB_r:=\{i\in\cC_r\,:\,  |\cNin_i|=1\}$ is   set  of leaves of $\cG_r$.
\end{definition}

\begin{assumption}\label{asm:on_the_partion} Each subgraph \(\mathcal G_r\) in \(\cG_\cC\)  
is a tree.  Singleton clusters ($\cE_r=\emptyset$) are   regarded as trees.  
\end{assumption}
Notice that we posit tree-partitions that are {\it edge-maximal}: each cluster contains  all edges  between its internal nodes.

Equipped with the tree partition  $\{\mathcal{G}_r\}_{r=1}^p$  (Assumption~\ref{asm:on_the_partion}), we decompose any   solution \(\mathbf{x}^\star\) of \eqref{P}  over the subgraphs  according to the   fixed-point inclusion:  
\begin{equation}
    \label{eq:general_fp}
    x_i^\star\in\;\argmin_{x_i}
\phi_i(x_i)
    +  \mu^{\star}_{\cNin_i\to i}(x_i) + \mu^{\star}_{\cNout_i\to i}(x_i),\quad i\in \mathcal{C}_r,\,\,r\in[p],\vspace{-0.3cm}
\end{equation}
where 
\begin{align}\label{eq:message_neig_agent}
&\mu^{\star}_{\cNin_i\to i}(x_i)
:= \nonumber \\ &\min_{x_{\cC_r\setminus\{i\}}}
\Biggl(
\sum_{j\in\cC_r\setminus\{i\}}\phi_j(x_j)
+ \sum_{(j,k)\in\cE_r}\psi_{jk}(x_j,x_k)
+ \sum_{j\in\cC_r\setminus\{i\}}\sum_{k\in\cNout_j}
\psi_{jk}(x_j,x_k^{\star})
\Biggr)
\end{align}
and \vspace{-0.2cm}
\begin{equation}  \mu^{\star}_{\cNout_i\to i}(x_i):=\sum_{k\in\cNout_i}\psi_{ik}(x_i,x_k^\star) 
\end{equation}  
represent,  respectively, the optimal cost contributions of the intra-cluster and inter-cluster neighbors of node $i\in\cC_r$--see Fig.~\ref{fig:fp_thm}(a). They constitute the only ``summaries'' that node $i$ needs in order to solve Problem~\eqref{P} locally via \eqref{eq:general_fp}--we term them as {\it messages}, because it is information to be routed to agent $i$. 
Exploiting the pairwise  structure in~\eqref{eq:message_neig_agent}, we show next that    the  intra-cluster message $\mu^{\star}_{\cNin_i\to i}$ can be obtained recursively from pairwise messages along the tree. 
 
For every two nodes  $i$, $j$ in $\mathcal G_r$ connected by an edge,   $(i,j)\in\cE_r$, define the   pairwise message from $j$ to $i$ as 
\begin{equation}
\label{eq:ideal_message_new}
    \mu_{j\rightarrow i}^\star(x_i)
    = \min_{x_j} \Biggl\{
    \psi_{ij}(x_i, x_j) + \phi_j(x_j)
    + \sum_{\substack{k\in \cNin_j \setminus \{i\}}}
      \mu_{k\rightarrow j}^\star(x_j)
    + \sum_{\substack{k \in \cNout_j}}
      \psi_{jk}(x_j, x_k^{\star})
    \Biggr\}.
\end{equation} 
Note that the recursion is well-defined because $\cNin_j\setminus \{i\}=\emptyset$  for every  leaf  node $j$   of $\mathcal G_r$ and $i\in \mathcal C_r$.  
In fact, for such nodes,  \eqref{eq:ideal_message_new} reduces to  \begin{equation}
\label{eq:message_leaves}
    \mu_{j\rightarrow i}^\star(x_i)
    := \min_{x_j} \Biggl\{
    \psi_{ij}(x_i, x_j) + \phi_j(x_j)
    + \sum_{\substack{k \in \cNout_j}}
    \psi_{jk}(x_j, x_k^{\star})
    \Biggr\},\quad j\in\cB_r,\,\,i\in \mathcal{C}_r. 
\end{equation}  
Therefore,   every pairwise  messages along the tree $\cG_r$ can be expressed recursively in terms of the incoming messages routed from the leaves. Fig.~\ref{fig:fp_thm}(b) illustrates the message-passing construction on a tree cluster.  In particular,   the  intra-cluster message $\mu^{\star}_{\cNin_i\to i}$ can be then represented as the aggregate of the pairwise messages from the neighbors of node $i$ within $\mathcal G_r$, i.e.,   \begin{equation}\label{eq:aggregate_message_final}
\mu_{\cNin_i\to i}^\star(x_i)
= \sum_{\substack{j\in\cNin_i}}\mu_{j\to i}^\star(x_i).
\end{equation}

Substituting (\ref{eq:aggregate_message_final}) in (\ref{eq:general_fp}), we   obtain the final decomposition of a solution of  \eqref{P} compliant with the graph structure:   
\begin{equation}
\label{eq:fixpoint2_new}
x_i^{\star}\in\argmin_{x_i}\Biggl\{
\phi_i(x_i)
+ \sum_{\substack{j\in\cNin_i}} \mu_{j\rightarrow i}^\star(x_i)
+ \sum_{\substack{j \in \cNout_i}}
  \psi_{ij}(x_i, x_j^{\star})\Biggr\},\quad i\in \mathcal{C}_r,\,\,r\in[p].
\end{equation}

\begin{figure*}[t]
  \centering
  \setlength{\tabcolsep}{0pt}
  \newcommand{\twosubfigsep}{1.5em}  
  \resizebox{\textwidth}{!}{%
    \begin{tabular}{@{}c@{\hspace{\twosubfigsep}}c@{}}
      \includegraphics[width=0.549\textwidth]{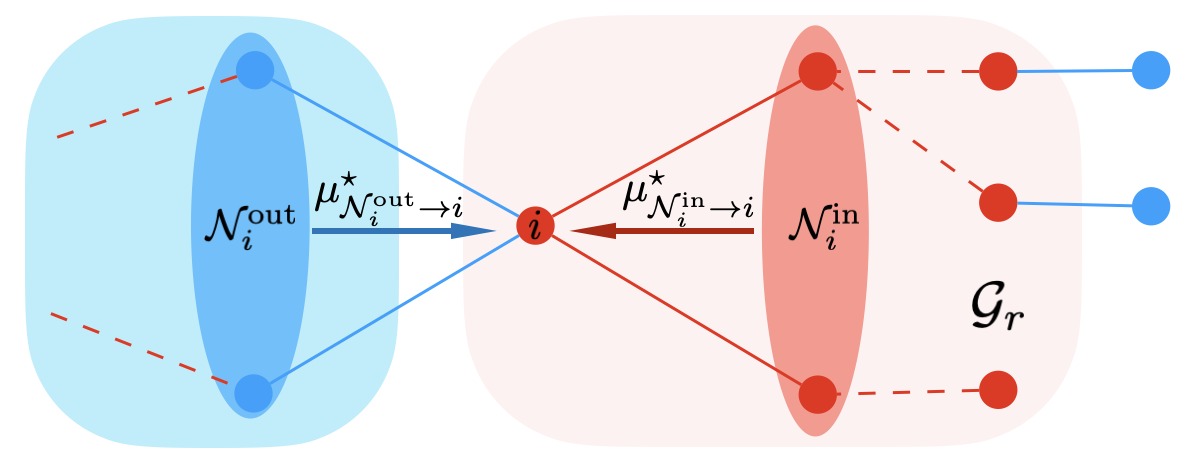} &
      \includegraphics[width=0.451\textwidth]{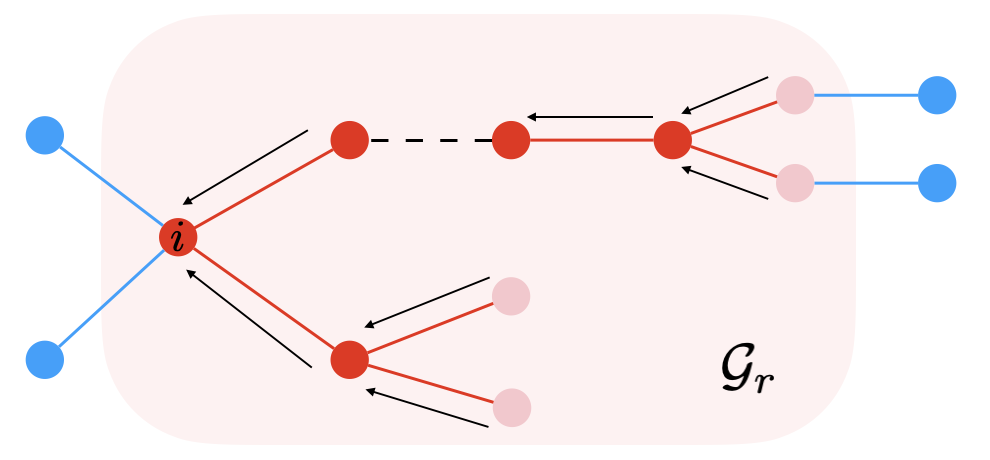} \\
      \footnotesize (a) intra- and inter-cluster messages &
      \footnotesize (b) message propagation from leaves
    \end{tabular}%
  }

  \caption{{\bf (a)} Intra-cluster and inter-cluster messages at node $i$:
  $\mu^{\star}_{\cNin_i\to i}$ collects contributions from neighbors within $\cG_r$
  (red nodes) while $\mu^{\star}_{\cNout_i\to i}$ collects those from neighbors outside
  the cluster (blue nodes). {\bf (b)} Message propagation within $\cG_r$:
  messages are initialized at the leaves (light red nodes) and propagated along the tree
  according to the    recursion～\eqref{eq:ideal_message_new}.}
  \label{fig:fp_thm}
  \vspace{-0.1cm}
\end{figure*}

The proposed decentralized algorithm is obtained by viewing \eqref{eq:fixpoint2_new} as a fixed-point system in the messages ${\mu_{j\to i}^\star}$ and variables ${x_i^\star}$, and iterating on these relations; the resulting scheme is summarized in Algorithm~\ref{alg:main}. To control convergence, we further incorporate the over-relaxation step~\eqref{gossip_update:b} in the update of the agents’ block-variables. The method combines intra-cluster message passing with inter-cluster block-Jacobi updates, in a fully decentralized fashion.\medskip 

\begin{algorithm}[H]
{
 \scriptsize
\caption{\underline{M}essage \underline{P}assing-\underline{J}acobi ({MP-Jacobi})
}
\label{alg:main}

{\bf Initialization:}  $x_i^{0}\in \mathbb{R}^d$, for all $i\in\cV$; 
initial message $\mu_{i\to j}^{0}(\cdot)$ arbitrarily chosen  (e.g., $\mu_{i\to j}^{0}\equiv 0$), for all $(i,j)\in\cE_{r}$ and $r\in [p]$.\\

\For{$\nu=0,1,2,\ldots$ }{
\AgentFor{$i\in\cV$}{\label{line:agent-loop}
\medskip 

\texttt{Performs the following updates:}
\begin{subequations}\label{gossip_update}
\begin{align}
\hat x_i^{\nu+1}&\in \argmin_{x_i}\Big\{\phi_i(x_i)
+ \sum_{j\in\cNin_i}\mu_{j\to i}^{\nu}(x_i)
+ \sum_{k\in\cNout_i}\psi_{ik}(x_i,x_k^{\nu})\Big\},
\label{gossip_update:a}\\
x_i^{\nu+1}&=x_i^{\nu}+\tau_r^\nu\big(\hat x_i^{\nu+1}-x_i^{\nu}\big),
\label{gossip_update:b}  \\\nonumber \\
\hspace{-.4cm}\mu_{i\to j}^{\nu+1}(x_j)&=
\min_{x_i}\Big\{
\phi_i(x_i)+\psi_{j i}(x_j,x_i)
+\!\!\sum_{k\in\cNin_i\setminus\{j\}}\!\!\mu_{k\to i}^{\nu}(x_i)
+\!\!\sum_{k\in\cNout_i}\!\!\psi_{ik}(x_i,x_k^{\nu})
\Big\}, \forall j\in   \cNin_i;
\label{message_update}
\end{align}
\end{subequations}
 
\texttt{Sends out}  $x_i^{\nu+1}$ to all $j\in\cNout_i$  and $\mu_{i\to j}^{\nu+1}(x_j)$  to all $j\in \cNin_i$.
}}}

\end{algorithm}
\medskip 
 
 Fig.~\ref{fig:clustered-trees} summarizes the key principle. The original graph is partitioned into tree subgraphs (clusters) $\mathcal{G}_{r}=(\cC_r,\cE_r)$, inducing the condensed graph $\cG_{\cC}$ with one supernode (possibly a singleton) per cluster.
   The coupling among agents \emph{within} each (non-singleton) cluster $\cC_r$ (i.e., inside each non-singleton supernode of $\cG_{\mathcal C}$) is handled via a min-sum-type message-passing scheme over the edges $\cE_r$; this corresponds to the ``intra-cluster'' interaction term $\sum_{j\in\cNin_i}\mu_{j\to i}^{\nu}(x_i)$ in agent $i$'s subproblem~\eqref{gossip_update:a}. In contrast, the coupling \emph{across} clusters is handled in a Jacobi-like fashion: in the local minimization step~\eqref{gossip_update:a}, each agent $i$ accounts for the cross-terms $\psi_{ij}$ from its inter-cluster neighbors $j\in\cNout_i$ (the singleton supernodes in $\mathcal G_{\cC}$ connected to $i$) via the contribution $\sum_{k\in\cNout_i}\psi_{ik}(x_i,x_k^{\nu})$.

In a nutshell, message passing accounts for all agents' interactions within each tree cluster, while block-Jacobi updates account for the interactions between clusters through the boundary nodes only.   The rationale for this design is that, on a tree, min-sum message passing is known to converge in a finite number of iterations  (e.g., \cite{wainwright2005map,even2015analysis})--on the order of the tree diameter--provided all other quantities are kept fixed. In our algorithm, however, we do {\it not} run an inner message-passing routine to convergence at each outer iteration: only a single forward-backward sweep of messages is performed per round, and these sweeps are interleaved with the block-Jacobi updates. Hence, the overall procedure is a single-loop decentralized algorithm relying solely on single-hop communications. 
 
Notice that when all clusters are singletons ($\cE_r = \emptyset$), no messages are exchanged and Algorithm~\ref{alg:main} becomes  a damped Jacobi method. At the opposite extreme, when the entire graph is  a   tree   (i.e., $\cE_r = \cE$), the scheme (with $\tau_r^\nu\equiv1$) reduces to standard min-sum message passing on a tree.  \vspace{-0.3cm}

\begin{figure}[t]
    \centering
    \includegraphics[width=1\linewidth]{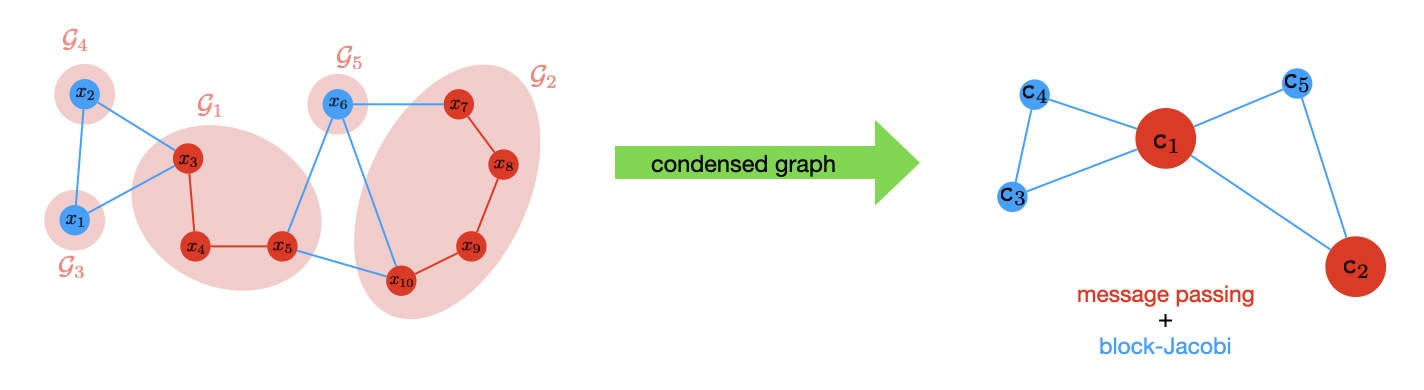}\vspace{-.3cm}
    \caption{Bird’s-eye view  of  Algorithm~\ref{alg:main}. The graph is decomposed into tree clusters   $(\cC_i)_{i=1}^5$. Intra-cluster interactions are handled via  message passing  within the  nonsingleton supernodes $\texttt{c}_1$  and $\texttt{c}_2$)    whereas inter-cluster interactions are handled by Jacobi-type updates at boundary  (singleton) supernodes  $\texttt{c}_3$  and $\texttt{c}_4$, and $\texttt{c}_5$. At the the condensed graph level, the iterations approximate a block-Jacobi.}
    \label{fig:clustered-trees} \vspace{-.1cm}
\end{figure}
\vspace{-0.1cm}
\section{Convergence analysis of Algorithm~\ref{alg:main}}
\label{subsec:global_convergence_pairwise}\vspace{-0.2cm}

In this section, we establish convergence of Algorithm~\ref{alg:main},    under the following   assumptions on the graph $\mathcal G$ and Problem~\eqref{P}.

 \begin{assumption}
 \label{asm:graph}
 The graph \(\cG = (\cV,\cE)\) is undirected and connected. Nodes can only communicate with their immediate neighbors $\cN_i$'s.  
 \end{assumption}
 
\begin{assumption}
\label{asm:nonverlap} Given the partition  $(\cG_r)_{r=1}^p$, each $\mathcal{C}_r$ has no external neighbors that connect to more than one node in $\mathcal{C}_r$: \vspace{-0.2cm}
\begin{equation}
|\cN_k\cap\cC_r|\leq1, \qquad\forall\;k\in\overline{\cC}_r, \;\forall\;r\in[p].
\label{eq:nonoverlap2}
\end{equation}
\end{assumption}
\vspace{-.2cm}
Condition~\eqref{eq:nonoverlap2}  enforces single-node contacts between clusters:   any external node can connect to a cluster through at most one gateway node.  While not necessary,   this assumption is posited  to simplify the notation in some  derivations. 


For each $r\in[p]$, let
$P_r:=U_{\cC_r}U_{\cC_r}^\top$ be the orthogonal projector onto the coordinates in $\cC_r$.
Moreover, define
$P_{\partial r}:=U_{\cN_{\cC_r}}U_{\cN_{\cC_r}}^\top$ as the projector onto the external neighborhood
coordinates $\cN_{\cC_r}$. 
\begin{assumption}
\label{assumption:scvx_smoothness}
Given the partition  $(\cG_r)_{r=1}^p$, the following hold:
    \begin{enumerate}
\item[(i)] $\Phi\in LC^1$ is lower bounded and $\mu$-strongly convex (with $\mu\ge 0$); 
for each $r\in[p]$,  $x_{\cC_r}\mapsto \Phi(x_{\cC_r},z_{\overline{\cC}_r})$ is $\mu_r$-strongly convex
uniformly w.r.t. $z_{\overline{\cC}_r}$,  with  $\mu_r\geq 0$.   
 \item[(ii)] For each $r\in[p]$, there exist constants $L_r>0$ and $L_{\partial r}>0$ such that:
\begin{equation}\label{eq:proj_smoothness}
\|P_r\nabla\Phi(\bx)-P_r\nabla\Phi(\by)\|
\;\le\;
L_r\,\|P_r(\bx-\by)\|,
\quad\forall \bx,\by\ \text{s.t.}\ (I-P_r)\bx=(I-P_r)\by;
\end{equation}
\begin{equation}\label{eq:proj_cross_lip}
  \|P_r\nabla\Phi(\bx)-P_r\nabla\Phi(\by)\|
\;\le\;
L_{{\partial r}}\,\|P_{\partial r}(\bx-\by)\|,
\quad\forall \bx,\by\ \text{s.t.}\ P_r \bx=P_r \by.
\end{equation}
\end{enumerate}When $\mu>0$, we define the global condition number  {$\kappa:=(\max_{r\in [p]} L_r)/\mu$}. When $\mu=0$, the problem admits a solution.
\end{assumption}

Notice that the existence of  $L_r$ follows directly from the  block-Lipschitz continuity of $\nabla\Phi$, whereas the definition of   ${L_{\partial_r}}$ additionally  exploits the graph-induced sparsity. More specifically, for  fixed $P_r \bx$,  $P_r \nabla\Phi(\bx)$ depends only on the blocks of $(I-P_r)\bx$ (i.e., $x_{\overline{\cC}_r}$)  indexed by $\cN_{\cC_r}$ (i.e.,  $P_{\partial r}\bx$). Then, for any $\bx,\by$ with $P_r\bx=P_r\by$, 
\[
P_r\qty(\nabla\Phi(\bx)
-\nabla\Phi(\by))
=P_r\qty(\nabla\Phi\bigl(P_r\bx+P_{\partial r}\bx\bigr)
-\nabla\Phi\bigl(P_r\by+P_{\partial r}\by\bigr)),
\]
which, by the  smoothness of $\Phi$, proves the existence of $L_{\partial r}$. 

Next we study convergence of Algorithm~\ref{alg:main}; we organize the    analysis   in two steps: (i) first, we establish the equivalence between 
the algorithm   and   a (damped) block-Jacobi method with delayed inter-cluster information; then, (ii) we proceed proving  convergence of such a Jacobi method.    \vspace{-0.3cm}

\subsection{Step 1: Algorithm~\ref{alg:main}  as a   block-Jacobi method with delays}
\label{subsec:Step1}

This reformulation builds on      eliminating in~\eqref{gossip_update:a} the messages and making explicit the underlying intra-cluster minimization.  
Specifically,  substitute  \eqref{message_update}, i.e., \begin{equation}
\label{eq:msg-reindexed}
\mu_{j\to i}^{\nu}(x_i)
=\min_{x_j}\Big\{\phi_j(x_j)+\psi_{ij}(x_i,x_j)
+\!\!\sum_{k\in\cNin_j\setminus\{i\}}\!\!\mu_{k\to j}^{\,\nu-1}(x_j)
+\!\!\sum_{k\in\cNout_j}\!\!\psi_{jk}(x_j,x_k^{\,\nu-1})\Big\}.
\end{equation}   into~\eqref{gossip_update:a}. For each $j\in\cNin_i$, this introduces the new variable $x_j$ together with the inner messages $\mu_{k\to j}^{\nu-1}$, ${k\in\cNin_j\setminus{\{i\}}}$.

 We  reapply \eqref{eq:msg-reindexed} to these new messages and repeat the substitution recursively along  the tree $\cG_r$. Each application of \eqref{eq:msg-reindexed} moves one hop farther from $i$ inside $\cG_r$: it adds one local function $\phi_j$, one pairwise factor $\psi_{jk}$, and replaces the expanded message by messages incoming from the next layer of neighbors plus the  boundary terms. When the recursion hits a boundary edge $(j,k)$,   $j\in\cC_r$ and $(j,k)\notin\cE_r$, the expansion stops along that branch and produces the    factor  $\psi_{jk}\!\big(x_j,x_k^{\,\nu-\texttt{d}}\big)$, where $\texttt{d}$ is the   depth of the recursion. Since $\cG_r$ is a tree, every intra-cluster node and edge is reached at most once, so the recursion terminates after finitely many substitutions and aggregates all intra-cluster factors exactly once.   
  Therefore,   one can write
\begin{align*} 
&\sum_{j\in\cNin_i}\mu_{j\to i}^{\nu}(x_i)\nonumber\\
&=\min_{x_{\cC_r\setminus\{i\}}}
\Bigg\{
\sum_{j\in\cC_r}\phi_j(x_j)
+\sum_{(j,k)\in\cE_r}\psi_{jk}(x_j,x_k)
+\sum_{\substack{j\in\cC_r,\,k\in\cNout_j}}\psi_{jk}\big(x_j,x_k^{\,\nu-d(i,j)}\big)
\Bigg\},
\end{align*}
where $d(i,j)\in \mathbb{N}$ is the length of the unique path between $i$ and $j$ in $\cG_r$.  

Using the above expression 
in ~\eqref{gossip_update:a}  we obtain  the following. 
\begin{proposition}
\label{prop:Jacobi_delay_refo}
Under Assumptions~\ref{asm:on_the_partion} and~\ref{asm:graph}, Algorithm~\ref{alg:main} can be rewritten in the equivalent form: for any $i\in\cC_r$ and $r\in [p]$,   
\begin{subequations}
    \begin{align}
 &  x_i^{\nu+1} =x_i^{\nu}+\tau_r^\nu \big(\hat x_i^{\nu+1}-x_i^{\nu}\big),\label{eq:delay_reformulation-cvx-comb}\\
   & \hat x_i^{\,\nu+1}  \!\!
\in \argmin_{x_i}\min_{x_{\cC_r\setminus\{i\}}}
\Bigg\{
\sum_{j\in\cC_r}\ \!\!\!\phi_j(x_j) 
+\!\!\sum_{(j,k)\in\cE_r}\!\!\!\!\psi_{jk}(x_j,x_k) 
+\!\!\!\!\sum_{\substack{j\in\cC_r, k\in\cNout_j}}\!\!\!\!\!\psi_{jk}\big(x_j,\,x_k^{\,\nu-d(i,j)}\big)\label{eq:delay_reformulation}
\Bigg\}.
\end{align}\end{subequations} 
If, in addition, Assumption~\ref{asm:nonverlap} holds,  \eqref{eq:delay_reformulation}  reduces to the following  block-Jacobi update with delays:
\begin{equation}
\label{eq:coordinate_min_delay}
\hat x_i^{\nu+1}
\in\argmin_{x_i}\ \min_{x_{\cC_r\setminus\{i\}}}\,
\Phi\bigl(x_{\cC_r},\,x_{\overline{\cC}_r}^{\,\nu-\bd_i}\bigr),
\end{equation}
where $\bd_i\in\mathbb{N}^{|\overline{\cC}_r|}$ is defined as\vspace{-0.2cm}
\[
(\bd_i)_k :=
\begin{cases}
d(i,j_k), & k\in \cN_{\cC_r},\\
0, & k\in \overline{\cC}_r\setminus\cN_{\cC_r},
\end{cases}
\]
and, for $k\in\cN_{\cC_r}$, $j_k\in\cB_r$ is the unique node such that $(j_k,k)\in\cE$.
\end{proposition}

\begin{proof} See Appendix~\ref{app:proof_prop_Jacobi-equivalence}.  \qed
\end{proof}

   The  distance $d(i,j)$ of  $j\in\cC_r$  from the updating index $i$ along the tree $\cG_r$ satisfies 
$d(i,j)\le\mathrm{diam}(\cG_r)\le|\cC_r|-1$. Clearly, if $|\cC_r|=1$ then  $d(i,j)=0$. 


  The equivalence of the algorithm updates  \eqref{eq:delay_reformulation} with (\ref{eq:coordinate_min_delay}) (under  Assumption~\ref{asm:nonverlap}) along with the postulated  lower boundedness of $\Phi$ (Assumption~\ref{assumption:scvx_smoothness}(i)) readily implies the existence of a minimizer for the subproblems (\ref{eq:coordinate_min_delay}).  


\begin{proposition}\label{prop:well_defined}
Suppose Assumptions~\ref{asm:on_the_partion},~\ref{asm:nonverlap}, and~\ref{assumption:scvx_smoothness} hold. 
Then Algorithm~\ref{alg:main} is well defined: at every iteration $\nu$, all message-update subproblems~\eqref{message_update} and variable-update subproblems~\eqref{gossip_update} admit minimizers. 
If, in addition, $\Phi$ is strongly convex, these minimizers are unique.
\end{proposition}
 
 We emphasize that Assumption \ref{asm:nonverlap} is used only to enable the compact notation in \eqref{eq:coordinate_min_delay} and to streamline the convergence analysis. When it fails, the surrogation technique in Sec.~\ref{sec:surrogate_pairwise}   preserves well-posedness and convergence.\vspace{-0.1cm}

\subsection{Step 2: Convergence analysis of block-Jacobi method with bounded delays}
\label{subsec:convergence_exact}\vspace{-0.2cm}

Now, we establish   convergence for the delayed block-Jacobi scheme in~\eqref{eq:coordinate_min_delay}. 

 Given $\hat x_i^{\nu+1}$, the minimizer produced by agent $i$ at iteration $\nu$  satisfying   \eqref{eq:coordinate_min_delay},   
 define the assembled vector $\hat \bx^{\nu+1}:=[ (\hat x_1^{\nu+1})^{\top},\ldots,  (\hat x_m^{\nu+1})^\top]^\top$. 
For each cluster, set $\hat x_{\mathcal C_r}^{\nu+1}:=(\hat x_i^{\nu+1})_{i\in\mathcal C_r}$.   
 Notice that $\hat x_{\mathcal C_r}^{\nu+1}$ is an assembled vector: its blocks may correspond
to different delay patterns $\bd_i$. 
  Using the projector  operator $P_r:=U_{\mathcal C_r} U_{\mathcal C_r}^\top$, the update of the algorithm can be written in vector form as:  
\begin{equation}
\bx^{\nu+1}
= \bx^\nu + \sum_{r\in [p]} \tau_r^\nu\, P_r(\hat \bx^{\nu+1}-\bx^\nu). 
\label{eq:step_decom}
\end{equation}
We also introduce  the virtual   \emph{non-delayed} block updates: for each $r\in [p]$,  let 
\begin{equation}
\label{eq:non-delay-BJac}
\bar{x}_{\cC_r}^{\nu+1}\in\argmin_{x_{\cC_r}}\Phi(x_{\cC_r},x_{\overline{\cC}_r}^\nu), \quad \text{and}\quad \bar x_i^{\nu+1}=\big[\bar{x}_{\cC_r}^{\nu+1}\big]_i,\,\,\forall i\in \cC_r;
\end{equation}
and the assembled vector $\bar \bx^{\nu+1}:=[ (\bar x_1^{\nu+1})^{\top},\ldots,  (\bar x_m^{\nu+1})^\top]^\top$. Notice that since the clusters form a partition, each $x_i^{\nu+1}$ above is uniquely defined.

The first result is  the descent of $\Phi$ along the iterates $\{\bx^\nu\}$, summarized next.

\begin{lemma}
\label{le:scvx_descent_contraction}
Under Assumption~\ref{assumption:scvx_smoothness} and any stepsize choice satisfying $\tau_r^\nu\geq0$ and  $\sum_{r=1}^p\tau_r^\nu\leq 1$, the following holds:  
\begin{equation}
\label{eq:scvx_descent_contraction}
\begin{aligned}
\Phi(\bx^{\nu+1})\leq\Phi(\bx^\nu)+\sum_{r\in[p]}\tau_r^\nu\qty[-\frac{\norm{P_r  \nabla\Phi(\bx^\nu)}^2}{2L_r}+\frac{L_r}{2}\big\|P_r(\hat \bx^{\nu+1}-\bar \bx^{\nu+1})\big\|^2].
\end{aligned}
\end{equation}  
\end{lemma}

\begin{proof} 
For each $r\in [p]$, $x_{\cC_r}\mapsto \Phi(x_{\cC_r},x_{\bar \cC_r}^\nu)$ is $L_r$-smooth (Assumption~\ref{assumption:scvx_smoothness}(ii)). Therefore
\begin{equation}
\label{eq:basic}
\begin{aligned}
 &\Phi\big(\bx^\nu+P_r (\hat \bx^{\nu+1}-\bx^\nu)\big) 
 \leq  \Phi\big(\bx^\nu+P_r (\bar \bx^{\nu+1}-\bx^\nu)\big)\\
&\quad +\underbrace{\inner{\nabla\Phi\big(\bx^\nu+P_r(\bar \bx^{\nu+1}-\bx^\nu)\big),P_r (\hat \bx^{\nu+1}-\bar \bx^{\nu+1})}}_{=0}+\frac{L_r}{2}\big\|P_r(\hat \bx^{\nu+1}-\bar \bx^{\nu+1})\big\|^2,\\
&\leq  \Phi\big(\bx^\nu\big)-\frac{\norm{P_r \nabla\Phi(\bx^\nu)}^2}{2L_r}+\frac{L_r}{2}\big\|P_r(\hat \bx^{\nu+1}-\bar \bx^{\nu+1})\big\|^2,
\end{aligned}
\end{equation}
where the zero inner-product comes from  $ P_r \nabla\Phi\big(\bx^\nu+P_r(\bar \bx^{\nu+1}-\bx^\nu)\big)=0$,  
due to  
\begin{equation}\label{eq:block-optimality}
\Phi\big(\bx^\nu+P_r\big(\bar \bx^{\nu+1}-\bx^\nu\big)\big)=\min_{x_{\cC_r}}\Phi(x_{\cC_r},x_{\overline{\cC}_r}^\nu).
\end{equation}

 Next, rewrite the update \eqref{eq:step_decom} as\begin{equation*}
\bx^{\nu+1}
=\Big(1-\sum_{r=1}^p\tau_r^\nu\Big)\bx^\nu
+\sum_{r=1}^p\tau_r^\nu\Big(\bx^\nu+P_r(\hat\bx^{\nu+1}-\bx^\nu)\Big).
\end{equation*}
Invoking the Jensens' inequality (and subtracting $\Phi(\bx^\nu)$ from both sides), yields 
\begin{equation}\label{eq:descent_sum_noz}
\Phi(\bx^{\nu+1})-\Phi(\bx^\nu)
\le
\sum_{r=1}^p\tau_r^\nu\Big(\Phi\big(\bx^\nu+P_r(\hat\bx^{\nu+1}-\bx^\nu)\big)-\Phi(\bx^\nu)\Big).
\end{equation}
Finally, applying   (\ref{eq:basic}) to each term of the sum on the RHS of   \eqref{eq:descent_sum_noz}  gives (\ref{eq:scvx_descent_contraction}). \hfill $\square$\end{proof}

 Lemma~\ref{le:scvx_descent_contraction} provides a descent estimate for $\Phi$ driven by the projected gradient norm,
up to the delay-induced discrepancy $\|P_r(\bar \bx^{\nu+1}-\hat \bx^{\nu+1})\|^2$.
For the subsequent contraction analysis, we also need a complementary
 {sufficient decrease} inequality that produces a negative quadratic term in the
actual block displacement $\|P_r (\bx^{\nu+1}-\ \bx^\nu)\|^2$.
This is obtained next exploting strong convexity of $\Phi$.  
\begin{lemma}
\label{le:scvx_descent}
Under Assumption~\ref{assumption:scvx_smoothness} and any stepsize choice satisfying $\tau_r^\nu\geq0$ and  $\sum_{r=1}^p\tau_r^\nu\leq 1$, the following holds:  
\begin{equation}
\label{eq:scvx_descent}
\begin{aligned}
\Phi(\bx^{\nu+1})\leq\Phi(\bx^\nu)+\sum_{r\in[p]}\tau_r^\nu\qty[-\frac{\mu_r}{4}\norm{P_r\big(\hat \bx^{\nu+1}-\bx^\nu\big)}^2+\frac{L_r+\mu_r}{2}\norm{P_r\big(\bar \bx^{\nu+1}-\hat \bx^{\nu+1}\big)}^2].
\end{aligned}
\end{equation}
\end{lemma}
\begin{proof} 
Invoking \eqref{eq:block-optimality} and   strong convexity of $\Phi$, yields
\[
\Phi\qty(\bx^\nu+P_r\qty(\bar\bx^{\nu+1}-\bx^\nu))\leq\Phi(\bx^\nu)-\frac{\mu_r}{2}\norm{P_r(\bar\bx^{\nu+1}-\bx^\nu)}^2.
\]
Using \eqref{eq:basic}, we have
\begin{equation}
\label{eq:basic3}
\begin{aligned}
&\Phi\qty(\bx^\nu+P_r\qty(\hat\bx^{\nu+1}-\bx^\nu))\leq\Phi(\bx^\nu)-\frac{\mu_r}{2}\norm{P_r\qty(\bar\bx^{\nu+1}-\bx^\nu)}^2+\frac{L_r}{2}\norm{P_r(\bar\bx^{\nu+1}-\hat\bx^{\nu+1})}^2.
\end{aligned}
\end{equation}
Note that 
\begin{equation*}
-\norm{P_r(\bar\bx^{\nu+1}-\bx^\nu)}^2\leq-\frac{1}{2}\norm{P_r(\hat\bx^{\nu+1}-\bx^\nu)}^2+\norm{P_r(\bar\bx^{\nu+1}-\hat\bx^{\nu+1})}^2.
\end{equation*}
Grouping together, we have 
\[
\begin{aligned}
&\Phi\qty(\bx^\nu+P_r\qty(\hat\bx^{\nu+1}-\bx^\nu))\leq\Phi(\bx^\nu)-\frac{\mu_r}{4}\norm{P_r(\hat\bx^{\nu+1}-\bx^\nu)}^2+\frac{L_r+\mu_r}{2}\norm{P_r(\bar\bx^{\nu+1}-\hat\bx^{\nu+1})}^2.
\end{aligned}
\]
Applying above inequality to  each term of the sum on the RHS of   \eqref{eq:descent_sum_noz}  gives \eqref{eq:scvx_descent}. 
\qed
\end{proof}
Note that both Lemma~\ref{le:scvx_descent_contraction} and Lemma \ref{le:scvx_descent} contains the delay-induced discrepancy term $\norm{\bar x_{\cC_r}^{\nu+1}-\hat x_{\cC_r}^{\nu+1}}^2$, which is bounded  by the accumulated boundary variations over a window of length at most $D_r$. We have the following lemma.
\begin{lemma}\label{le:iterates_gap} In the setting of Lemma~\ref{le:scvx_descent_contraction} and Lemma \ref{le:scvx_descent}, assume    {$\mu_r>0$ for all $r\in [p]$}. Then,  the following holds:
 for any  $r\in [p]$,
\begin{equation}\label{eq:out-cluster-bound}
\big\|P_r(\hat \bx^{\nu+1}-\bar \bx^{\nu+1})\big\|^2
\le \frac{L_{\partial r}^2\,|\cC_r|D_r}{\mu_r^2}\,
\sum_{\ell=\nu-D_r}^{\nu-1}\big\|P_{\partial r}(\bx^{\ell+1}-\bx^{\ell})\big\|^2.
\end{equation} 
\end{lemma}
\begin{proof} 
Define the auxiliary variable  
\[
\hat x_{\cC_r,i}^{\nu+1}:=\argmin_{x_i}\min_{x_{\cC_r\setminus i}}\Phi\qty(x_{\cC_r},x_{-\cC_r}^{\nu-\bd_i}),\text{ with $\hat x_i^{\nu+1}=\qty[ \hat x_{\cC_r,i}^{\nu+1}]_i$.}
\]
By optimality the condition, 
\[
\begin{aligned}
0
&=\nabla_{\cC_r}\Phi\big(\hat x_{\cC_r,i}^{\nu+1},x_{\overline{\cC}_r}^{\nu-\bd_i}\big)
 -\nabla_{\cC_r}\Phi\big(\bar x_{\cC_r}^{\nu+1},x_{\overline{\cC}_r}^{\nu}\big)\\
&=\underbrace{\nabla_{\cC_r}\Phi\big(\hat x_{\cC_r,i}^{\nu+1},x_{\overline{\cC}_r}^{\nu-\bd_i}\big)
 -\nabla_{\cC_r}\Phi\big(\bar x_{\cC_r}^{\nu+1},x_{\overline{\cC}_r}^{\nu-\bd_i}\big)}_{\text{in-block change}}
 +\underbrace{\nabla_{\cC_r}\Phi\big(\bar x_{\cC_r}^{\nu+1},x_{\overline{\cC}_r}^{\nu-\bd_i}\big)
 -\nabla_{\cC_r}\Phi\big(\bar x_{\cC_r}^{\nu+1},x_{\overline{\cC}_r}^{\nu}\big)}_{\text{cross-block change}}.
\end{aligned}
\]
Taking inner product with $\hat x_{\cC_r,i}^{\nu+1}-\bar x_{\cC_r}^{\nu+1}$, using $\mu_r$-strong convexity in $x_{\cC_r}$ and the cross-Lipschitz bound of
Assumption~\ref{assumption:scvx_smoothness}(ii), we obtain
\[
\begin{aligned}
\mu_r\norm{\hat x_{\cC_r,i}^{\nu+1}-\bar x_{\cC_r}^{\nu+1}}^2\leq&-\inner{\nabla_{\cC_r}\Phi(\bar x_{\cC_r}^{\nu+1},x_{\overline{\cC}_r}^{\nu-\bd_i})-\nabla_{\cC_r}\Phi(\bar x_{\cC_r}^{\nu+1},x_{\overline{\cC}_r}^{\nu}),\hat x_{\cC_r,i}^{\nu+1}-\bar x_{\cC_r}^{\nu+1}}\\
\leq&\norm{\nabla_{\cC_r}\Phi(\bar x_{\cC_r}^{\nu+1},x_{\overline{\cC}_r}^{\nu-\bd_i})-\nabla_{\cC_r}\Phi(\bar x_{\cC_r}^{\nu+1},x_{\overline{\cC}_r}^{\nu})}\norm{\hat x_{\cC_r,i}^{\nu+1}-\bar x_{\cC_r}^{\nu+1}}\\
\leq& L_{\partial r}\norm{x_{\cN_{\cC_r}}^{\nu-\bd_i}-x_{\cN_{\cC_r}}^{\nu}}\norm{\hat x_{\cC_r,i}^{\nu+1}-\bar x_{\cC_r}^{\nu+1}},
\end{aligned}
\]
which proves $\|\hat x_{\cC_r,i}^{\nu+1}-\bar x_{\cC_r,i}^{\nu+1}\|\le \frac{L_{\partial r}}{\mu_r}\,\big\|x_{\cN_{\cC_r}}^{\,\nu-\bd_i}-x_{\cN_{\cC_r}}^{\,\nu}\big\|$. Then, we have 
\[
\begin{aligned}
   &\norm{\hat x_{\cC_r}^{\nu+1}-\bar x_{\cC_r}^{\nu+1}}^2=\sum_{i\in\cC_r}\norm{\hat x_i^{\nu+1}-\bar x_i^{\nu+1}}^2\leq\sum_{i\in\cC_r}\norm{\hat x_{\cC_r,i}^{\nu+1}-\bar x_{\cC_r}^{\nu+1}}^2\\
   \leq&\sum_{i\in\cC_r}\qty(\frac{L_{\partial r}}{\mu_r})^2\norm{x_{\cN_{\cC_r}}^{\nu-\bd_i}-x_{\cN_{\cC_r}}^{\nu}}^2\leq\frac{L_{\partial r}^2|\cC_r|D_r}{\mu_r^2}\sum_{\ell=\nu-D_r}^{\nu-1}\norm{x_{\cN_{\cC_r}}^{\ell+1}-x_{\cN_{\cC_r}}^{\ell}}^2.
\end{aligned}
\]
This completes the proof. \qed
\end{proof}

Notice that when $|\cC_r|=1$,   $D_r=0$, and the right-hand side of (\ref{eq:out-cluster-bound}) equals $0$.

  Lemma~\ref{le:iterates_gap} shows that, for each cluster, the delay discrepancy is controlled only by variations of the iterates on the {\it external neighborhood}  $\cN_{\cC_r}$  of that cluster (equivalently, through the projector $P_{\partial r}$ onto $\cN_{\cC_r}$). 
Hence, delay accumulation acts only on the coordinates in
$\bigcup_{r\in[p]:\,|\cC_r|>1} \cN_{\cC_r}$.
To make this explicit, we introduce a minimal cover of this set by clusters.
Let $\cJ\subseteq [p]$ be the  minimal index set of clusters whose union covers all external neighbors, i.e.,        
\begin{equation}\label{eq:def_J}
\bigcup_{r\in [p]:|\cC_r|>1} \cN_{\cC_r}\ \subseteq\ \bigcup_{r\in \cJ} \cC_r .
\end{equation}
We call $\{\cC_r\}_{r\in\cJ}$ the {\it external-neighborhood clusters}.
In sparse graphs, this set is typically much smaller than the full partition  ($\lvert\cJ\rvert\ll p$ and $\sum_{r\in\cJ}|\cC_r|\ll m$), so the delay accumulation affects only a small subset of coordinates.  

Using the above   results, we can now    state the first convergence result. 
\begin{theorem}
\label{thm:convergence_scvx}
Suppose Assumption  \ref{asm:on_the_partion}, \ref{asm:graph}, and \ref{assumption:scvx_smoothness} hold,  {with $p>1$ and $\mu_r>0,$ for all $r\in [p]$}. Let $\{\bx^\nu\}$ be the iterates generated by the method \eqref{eq:delay_reformulation-cvx-comb} and  \eqref{eq:coordinate_min_delay}. Then, 
\[
\Phi(\bx^{\nu})-\Phi^{\star}\leq c\rho^\nu,\text{ for any $\nu\in\mathbb{N}$,}
\]
where $c\in (0,\infty)$ is a universal constant, and $\rho\in (0,1)$ is defined as follows. 
\begin{enumerate}
    \item[\textbf{(i)}] {\bf (heterogeneous stepsizes):} Under $\sum_{r\in[p]}\tau_r^\nu=1$, $\tau_r^\nu>0$;  and 
    \begin{equation}   \label{eq:delay_lemma_condition_1}
    2D+1\ \le\
    \min\left\{
    \frac{2\max\limits_{r\in[p],\nu\ge 1}\frac{L_r}{\tau_r^\nu}}{\mu},
\frac{\min_{r\in\cJ,\nu\geq1}\frac{\mu_r}{8\tau_r^\nu}}
    {\max_{\nu\geq1}{A^\nu_{\mathcal J}}}
    \right\},
    \end{equation}
   the rate is given by  \vspace{-0.1cm}\begin{equation}\label{eq:rate-varying-tau}
       \rho=1-\frac{\mu}{2}\min_{r\in[p],\nu\geq1}\frac{\tau_r^\nu}{L_r}.\end{equation}  Here $$ A^\nu_{\mathcal J}  : =  \; \max_{\,i\in \cup_{j\in\cJ}\cC_j}\ \sum_{r:\,|\cC_r|>1,\ i\in\cN_{\cC_r}}(\tau_r^\nu A_r),\,\, \text{with }\,\, A_r\ :=\;\frac{(2L_r+\mu_r)L_{\partial r}^2\,|\cC_r|\,D_r}{4\mu_r^2}.$$
    \item[\textbf{(ii)}]  {\bf (uniform stepsizes) :} if  $\tau_r^\nu\equiv\tau$ and
     \begin{equation}   \label{eq:delay_lemma_condition_2}
    \!\!\!\!\tau\ \le\
    \min\left\{
    \frac{1}{p}, \frac{2\kappa}{2D+1},
    \sqrt{\frac{\min_{r\in\cJ}\mu_r}{ 8(2D+1) A_{\mathcal J}}}  \right\},
    \end{equation}
    then 
    \begin{equation}\label{eq:rate-constant-step}
    \rho=1-\frac{\tau}{2 \kappa}.
     \end{equation}
    Here
    \begin{equation}\label{eq:def_A} A_{\cJ} :=\max_{\,i\in \cup_{j\in\cJ}\cC_j}\ \sum_{r:\,|\cC_r|>1,\ i\in\cN_{\cC_r}} A_r. \end{equation}
   
\end{enumerate}
\end{theorem}

\begin{proof}
Fix $\nu\in\mathbb{N}$. For any index set $S\subseteq[m]$, write
$\Delta_S^\nu:=\sum_{\ell=\nu-D}^{\nu-1}\|x_S^{\ell+1}-x_S^\ell\|^2$. By \eqref{eq:scvx_descent_contraction}, \eqref{eq:scvx_descent}, and Lemma~\ref{le:iterates_gap},
\begin{equation*}
\begin{aligned}
\Phi(\bx^{\nu+1})
\;\le\;&
\Phi(\bx^\nu)
-\sum_{r\in[p]}\frac{\tau_r^\nu}{4L_r}\,\|\nabla_{\cC_r}\Phi(\bx^\nu)\|^2
-\sum_{r\in[p]}\frac{\mu_r}{8\tau_r^\nu}\,\|x_{\cC_r}^{\nu+1}-x_{\cC_r}^{\nu}\|^2+\sum_{r\in[p]}\tau_r^\nu A_r
\Delta_{\cN_{\cC_r}}^\nu.
\end{aligned}
\end{equation*}
Discarding the vanishing terms with $|\cC_r|=1$ and using the disjointness of $\{\cC_j\}_{j\in\cJ}$, 
\[
\sum_{r:\,|\cC_r|>1}(\tau_r^\nu A_r)\,\Delta_{\cN_{\cC_r}}^\nu
=
\sum_{\ell=\nu-D}^{\nu-1}\ \sum_{i\in \cup_{j\in\cJ}\cC_j}
\Big(\sum_{r:\,|\cC_r|>1,\ i\in\cN_{\cC_r}}(\tau_r^\nu A_r)\Big)\,
\|x_i^{\ell+1}-x_i^\ell\|^2.
\]
By  definition of $ A_{\mathcal J}^\nu$, we have
\begin{equation*}
\sum_{r:\,|\cC_r|>1}(\tau_r^\nu A_r)\,\Delta_{\cN_{\cC_r}}^\nu
\ \le\ 
 A_{\mathcal J}^\nu\sum_{\ell=\nu-D}^{\nu-1}\ \sum_{i\in \cup_{j\in\cJ}\cC_j}\|x_i^{\ell+1}-x_i^\ell\|^2
\ =\ 
 A_{\mathcal J}^\nu\sum_{j\in\cJ}\Delta_{\cC_j}^\nu.
\end{equation*}
Finally, by strong convexity,
$\|\nabla\Phi(\bx^\nu)\|^2\ge 2\mu\big(\Phi(\bx^\nu)-\Phi^\star\big)$, and thus
\begin{equation*}
\begin{aligned}
\Phi(\bx^{\nu+1})-\Phi^\star
\;\le\;&
\Big(1-\frac{\mu}{2}\min_{r\in[p],\,\nu\ge 1}\frac{\tau_r^\nu}{L_r}\Big)\big(\Phi(\bx^\nu)-\Phi^\star\big)
-\sum_{r\in\cJ}\frac{\mu_r}{8\tau_r^\nu}\,\|x_{\cC_r}^{\nu+1}-x_{\cC_r}^{\nu}\|^2\\
&\quad+ A_{\mathcal J}^\nu\sum_{r\in\cJ}\Delta_{\cC_r}^\nu.
\end{aligned}
\end{equation*}
Invoking the standard delay-window inequality (e.g. \cite[Lemma 5]{feyzmahdavian2023asynchronous}) yields linear convergence, subject to  
\eqref{eq:delay_lemma_condition_1}. Using  heterogeneous  (resp. uniform) stepsize values along with \eqref{eq:delay_lemma_condition_1} yields the linear convergence rates as given in \eqref{eq:rate-varying-tau} (resp. \eqref{eq:delay_lemma_condition_2}). \qed 

\end{proof}

\begin{remark}
    When $\cG$ is a tree, $p=1$ (i.e., message passing on the entire $\cG$), and $\tau_r^\nu\equiv1$, the iterates terminate in finitely many (diameter of $\cG$) rounds, thus recovering the classical finite-time property of min-sum/message passing on trees.  
\end{remark} 
\begin{remark}\label{rem:scvx_without_max_nonoverlap}
The linear rate established in the theorem does not hinge on edge-maximality or  Assumption~\ref{asm:nonverlap_hyper} per se. 
It is suffices that  each cluster ``aggregate surrogate''--obtained  collecting all terms involving $x_{\cC_r}$--is strongly convex in $x_{\cC_r}$   
(see Assumption~\ref{assumption:surrogation} in  Sec.~\ref{sec:surrogate_pairwise} for the formal statement). 
That said, dropping maximality/ non-overlap may jeopardize well-posedness of the algorithmic subproblems (existence of minimizers). 
Both issues are resolved by the surrogate construction in Sec.~\ref{sec:surrogate_pairwise}, which enforces strong convexity and well-posed local updates.
\vspace{-0.4cm}
\end{remark}
  
\subsection{Discussion}
\label{subsec:discussion_rate}\vspace{-0.2cm}
 Theorem~\ref{thm:convergence_scvx} establishes global \emph{linear} convergence of  
\(\{\Phi(\bx^{\nu})\}\),
with a contraction factor \(\rho\in(0,1)\), under heterogeneous stepsizes (subject to (\ref{eq:delay_lemma_condition_1})) or homogeneous ones (satisfying \eqref{eq:delay_lemma_condition_2}). For the sake of conciseness, next we comment only the latter case; for the former, we only recall that, under   $\tau_r^\nu=1/p$, the contraction factor \eqref{eq:delay_lemma_condition_1} becomes \(\rho=1-(1/(2\kappa p))\), matching that  of the \textit{centralized} $p$-block Jacobi. 

  
The rate expression (\ref{eq:rate-constant-step}), under 
\eqref{eq:delay_lemma_condition_2},   reads\vspace{-.2cm}
\begin{equation}
\label{eq:rate_expression}
\rho\;=\;1-\frac{1}{2\kappa}\cdot
\min\!\left\{
\underbrace{\frac{1}{p}}_{({\rm I})},\;
\underbrace{\frac{2\kappa}{2D+1}}_{({\rm II})},\;
\underbrace{\sqrt{\frac{\min_{r\in\cJ}\mu_r}{\,8(2D+1)\,A_{\cJ}\,}}}_{({\texttt{III}})}
\right\}.
\end{equation}
Here, \texttt{(I)} captures the cluster-level effect--typical of Jacobi-type schemes with \(p\) playing
the role of the number of blocks;
 \texttt{(II)} reflects the {\it intra}-cluster geometry via the delay/diameter
parameter \(D\), along with the loss landscape (uniformly across all partitions) through the condition number
\(\kappa\);   and \texttt{(III)} captures the {\it inter}-cluster coupling  through \(A_{\cJ}\), which aggregates the
external interactions \emph{only} over the block-coordinates belonging to the external-neighborhood clusters
\(\{\cC_r\}_{r\in\cJ}\) (i.e., the subset where delay accumulation can occur).  

Eq.~\eqref{eq:rate_expression} shows  a  trade-off: Term \texttt{(I)} increases as \(p\) decreases, whereas terms \texttt{(II)}--\texttt{(III)} raise as
\(D\) decreases. 
For a fixed graph, these quantities are coupled by the   partition $\{\cG_r\}_{r\in [p]}$: coarser clustering (smaller \(p\)) typically
increase the diameters \(D\) and may also increase the   coefficient \(A_{\cJ}\), while finer
partitions (larger \(p\)) reduce \(D\) but can enlarge \(p\). 
 Table~\ref{tab:three_regimes}   suggests the following partition-design principles.
If \texttt{(I)} is active, the bound is \emph{block-count limited} and one should reduce the number of clusters $p$
(coarser clustering), while controlling the induced growth of $D$ and $A_{\cJ}$.
If \texttt{(II)} is active, the bound is \emph{delay/diameter limited} and the priority is to keep clusters shallow
(small $D$); in this regime the contraction is essentially $\kappa$-insensitive.
If \texttt{(III)} is active, the bound is \emph{external-coupling limited} and one should reduce $A_{\cJ}$ by placing
cuts on weak inter-cluster couplings (small $L_{\partial r}$). 

\begin{table}[t!]
\centering
\setlength{\tabcolsep}{6pt}
\renewcommand{\arraystretch}{1.2}
\begin{tabular}{@{}p{0.46\linewidth}p{0.44\linewidth}@{}}
\toprule
{\bf Active term (interpretation)} &  {\bf Margin $1-\rho$} \\
\midrule
\texttt{(I)} (block-count limited) 
& $\displaystyle 1-\rho=\frac{1}{4\kappa p}$ \vspace{0.2cm}\\ 
\texttt{(II)} (delay/diameter limited)
& $\displaystyle 1-\rho=\frac{1}{2D+1}$ \vspace{0.2cm}\\ 
\texttt{(III)} (external-coupling limited)
& $\displaystyle
1-\rho=\frac{1}{2\kappa}\sqrt{\frac{\min_{r\in\cJ}\mu_r}{8(2D+1)A_{\cJ}}}$ \\
\bottomrule 
\end{tabular}
\vspace{-0.15cm}
\caption{Operating regimes of~\eqref{eq:rate_expression}. Each tuning knob improves the bound only while its associated term is active; otherwise the rate saturates.}
\label{tab:three_regimes}\vspace{-0.1cm}
\end{table}
\noindent Figures~\ref{fig:rate_plot_p}--\ref{fig:rate_plot_kappa} empirically illustrate these mechanisms: 
Fig.~\ref{fig:rate_plot_p} shows that, when \texttt{(I)} is active, reducing $p$ yields faster convergence (until
saturation); Fig.~\ref{fig:rate_plot_Ar} isolates the effect of $A_{\cJ}$ and confirms the monotone improvement
predicted by the \texttt{(III)}-regime; and Fig.~\ref{fig:rate_plot_kappa} highlights that, once \texttt{(II)} is
active, the observed iteration count becomes nearly insensitive to $\kappa$.

\begin{figure*}[h]
   \vspace{-0.6cm}
    \begin{center}
        \begin{tabular}{@{}cc@{}}
                \raisebox{-\height}{\includegraphics[width=0.5\linewidth]{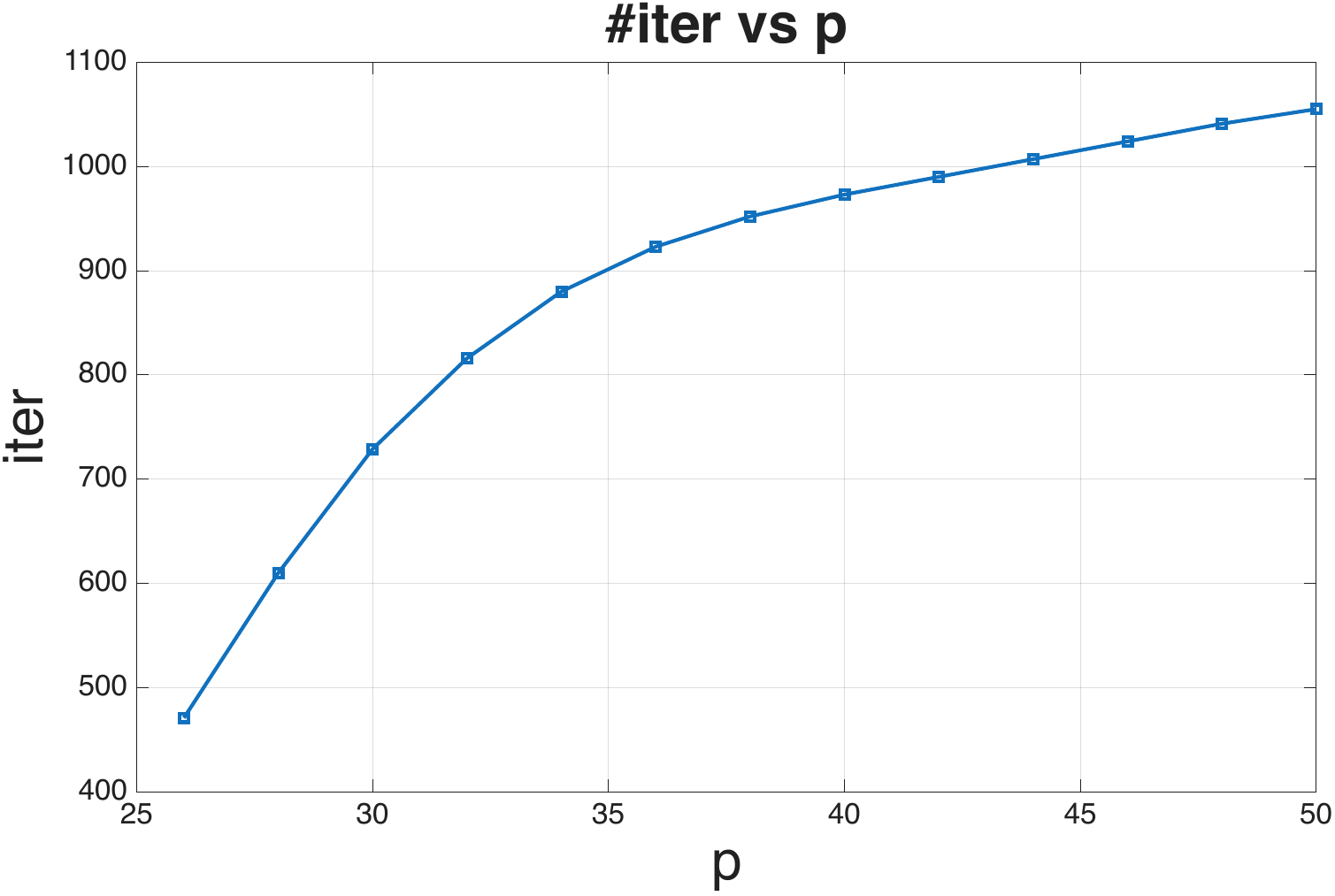}} &
                \raisebox{-\height}{\includegraphics[width=0.45\linewidth]{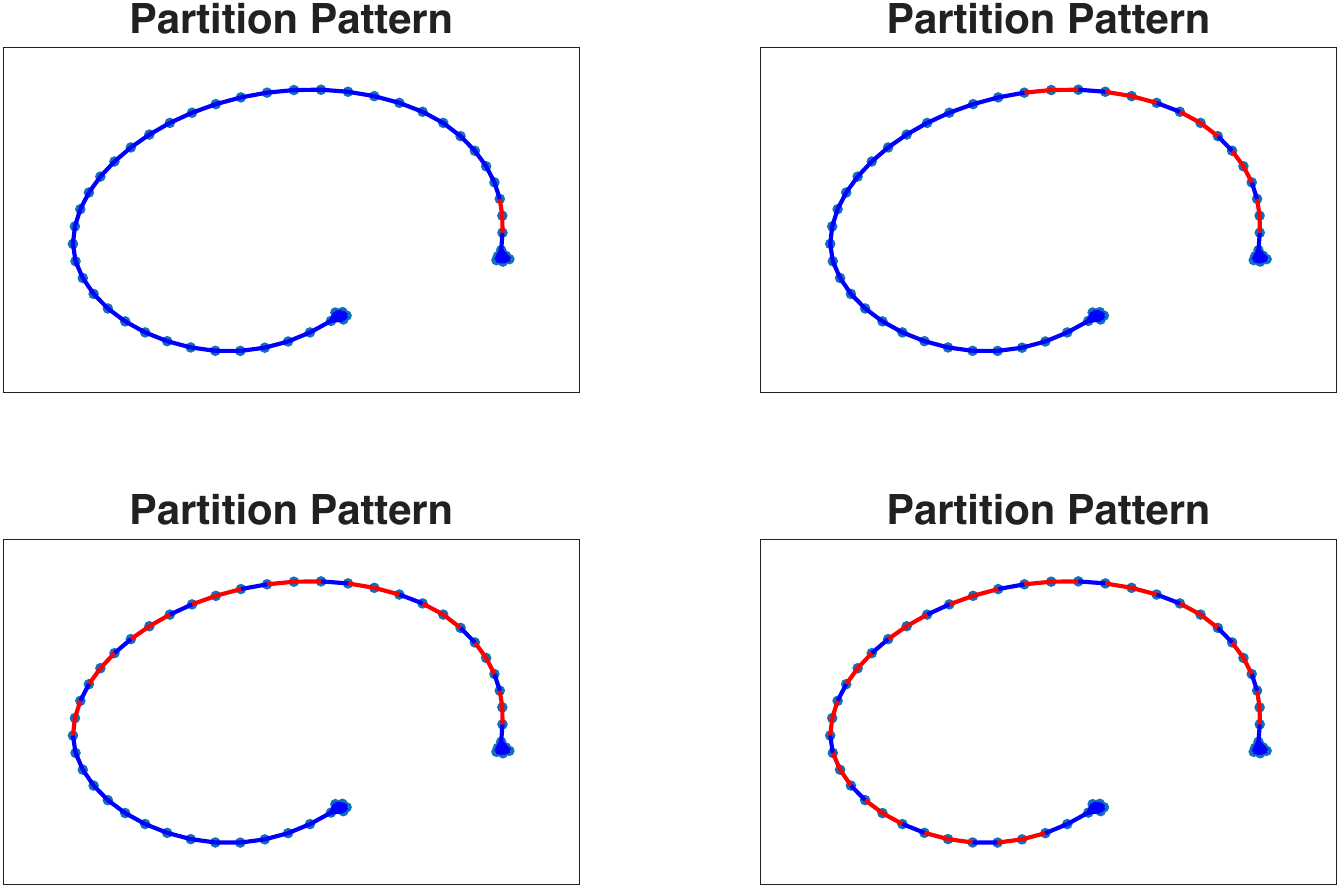}}\\[0.2em]
                \footnotesize{(a) \#iter to $\varepsilon=10^{-6}$ solution vs $p$} &  \footnotesize{(b) Partitions:  different $p$, all with  $D=2$.}       
        \end{tabular}
        \caption{Graph with two cliques (degree 6) connected by a long path (length 42). All partitions (different $p$) share the same $D$, $A_{\cJ}$, $(\mu_r)_{r\in\cJ}$ and $\kappa$.}
        \label{fig:rate_plot_p}
    \end{center}\vspace{-0.4cm}
\end{figure*}
\begin{figure*}[h]
  \vspace{-0.3cm}  \begin{center}
        \setlength{\tabcolsep}{0.0pt}  
        \scalebox{1}{\begin{tabular}{ccc}
                \includegraphics[width=0.33\linewidth]{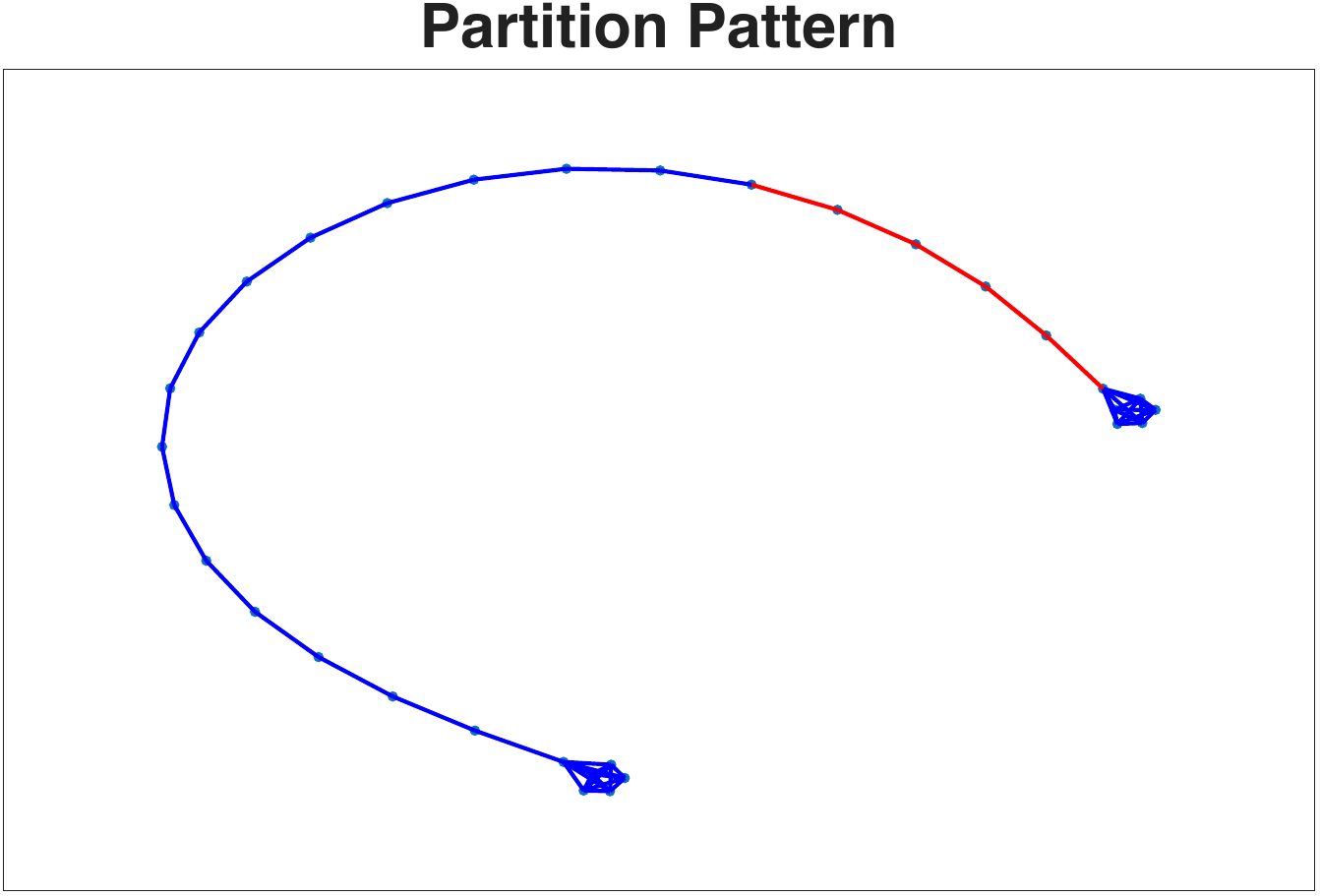}&
                \includegraphics[width=0.33\linewidth]{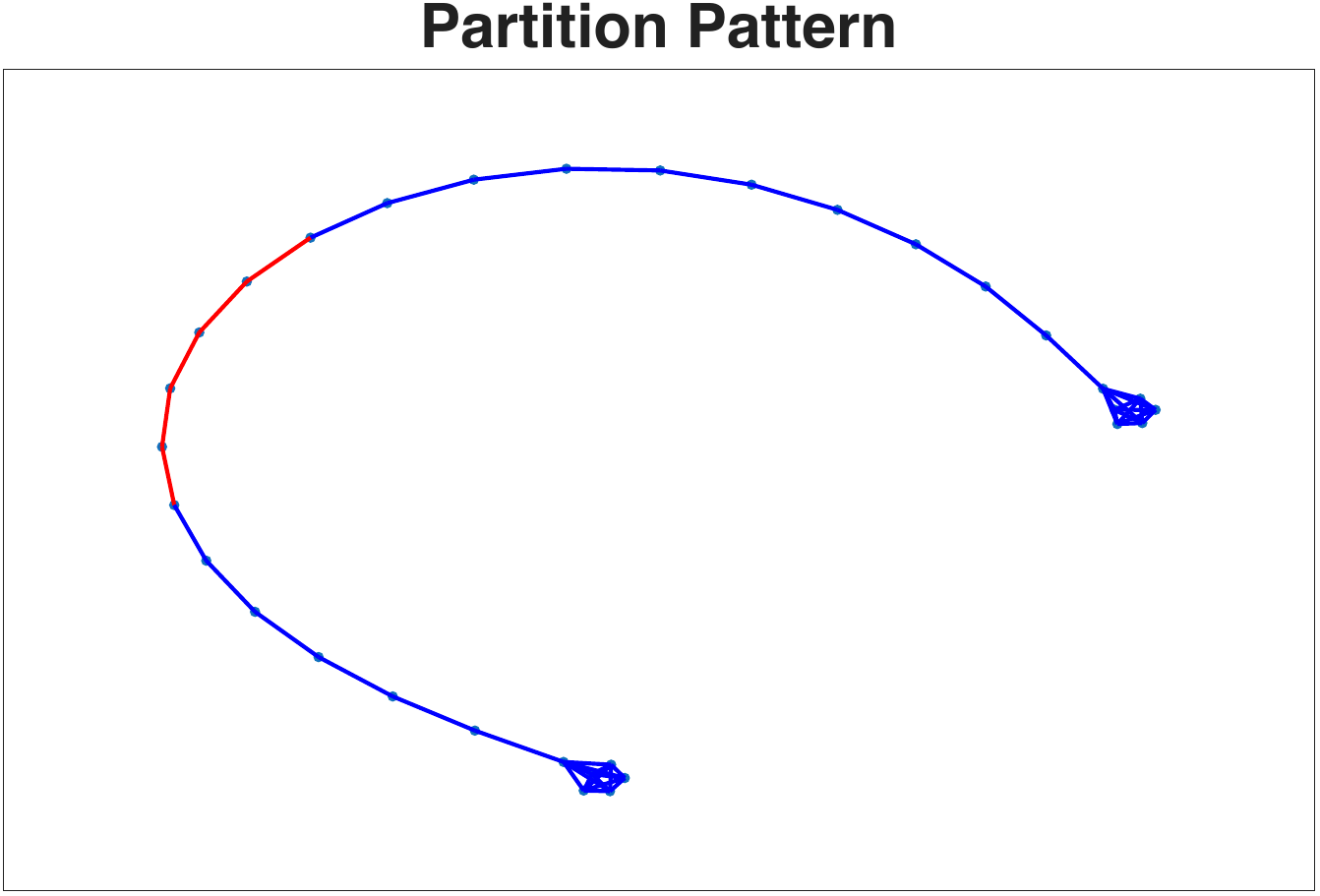}&
                \includegraphics[width=0.33\linewidth]{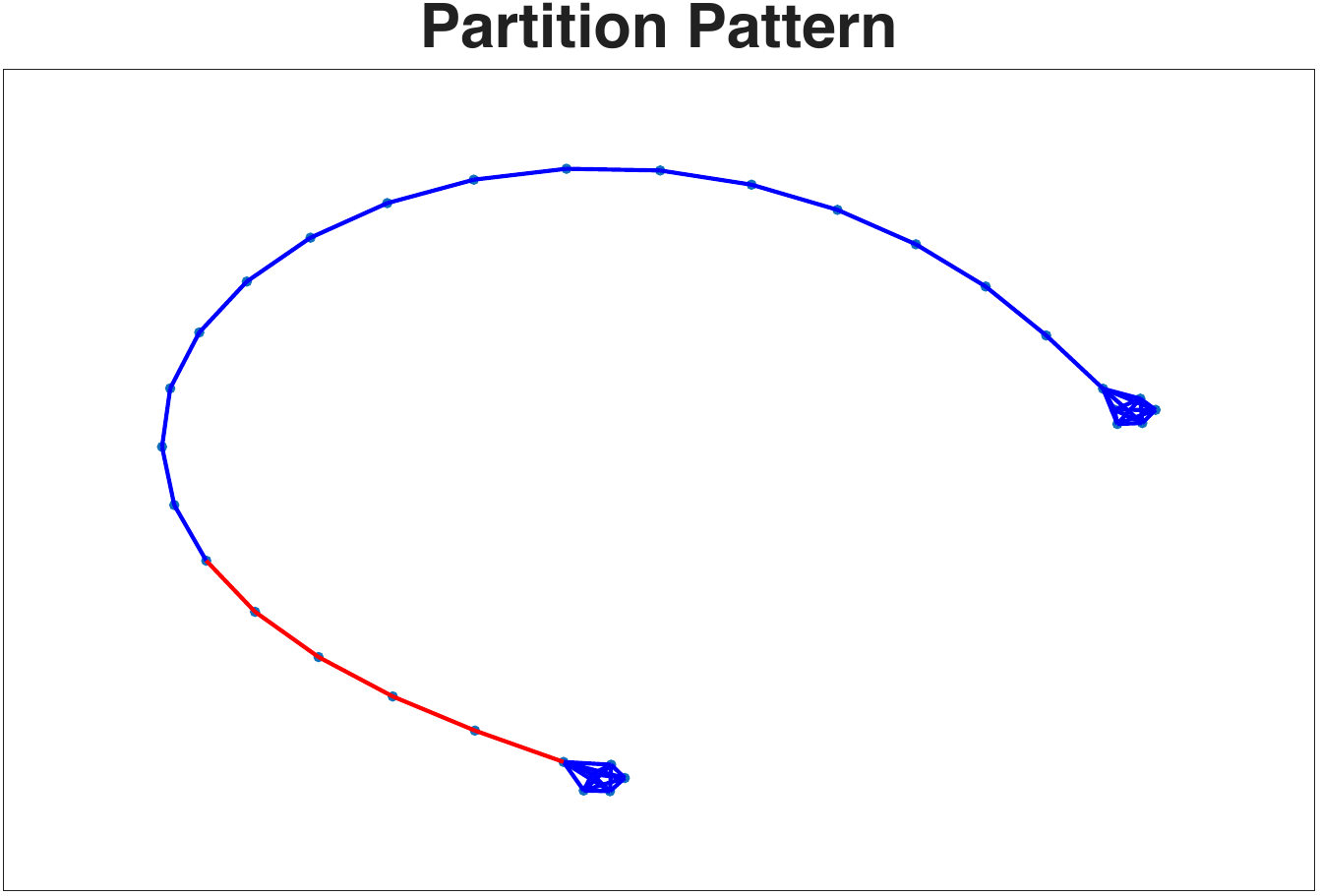}\\
                \multicolumn{1}{c}{\footnotesize{(a) Large $A_{\cJ}$. \#iter=705.}} &  \multicolumn{1}{c}{\footnotesize{(b) Small $A_{\cJ}$. \#iter=539.}}&
                \multicolumn{1}{c}{\footnotesize{(c) Small $A_{\cJ}$. \#iter=555.}}       
        \end{tabular}}
        \caption{Graph with   two cliques connected by a long path ({length =20});   $A_{\cJ}$ varies, with constant $p$, $D$, $(\mu_r)_{r\in\cJ}$ and $\kappa$.}
        \label{fig:rate_plot_Ar}
    \end{center}
     \vspace{-0.1cm}
\end{figure*}
\begin{figure*}[t]
\centering

\begin{subfigure}[t]{0.33\textwidth}
  \centering
  \includegraphics[width=\linewidth,height=4.2cm,keepaspectratio]{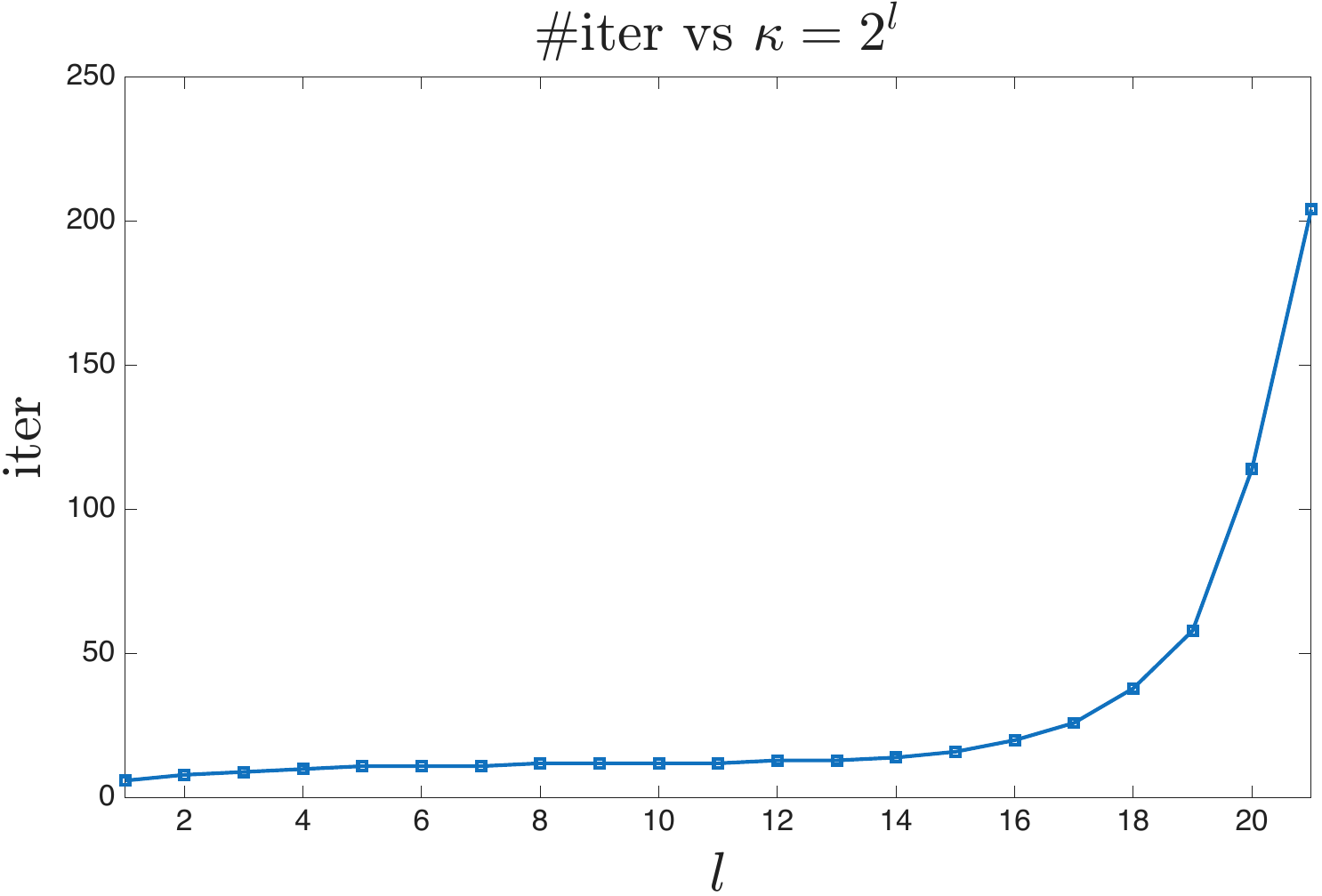}
  \caption{\#iter vs $\kappa$}
\end{subfigure}\hfill
\begin{subfigure}[t]{0.33\textwidth}
  \centering
  \includegraphics[width=\linewidth,height=4.2cm,keepaspectratio]{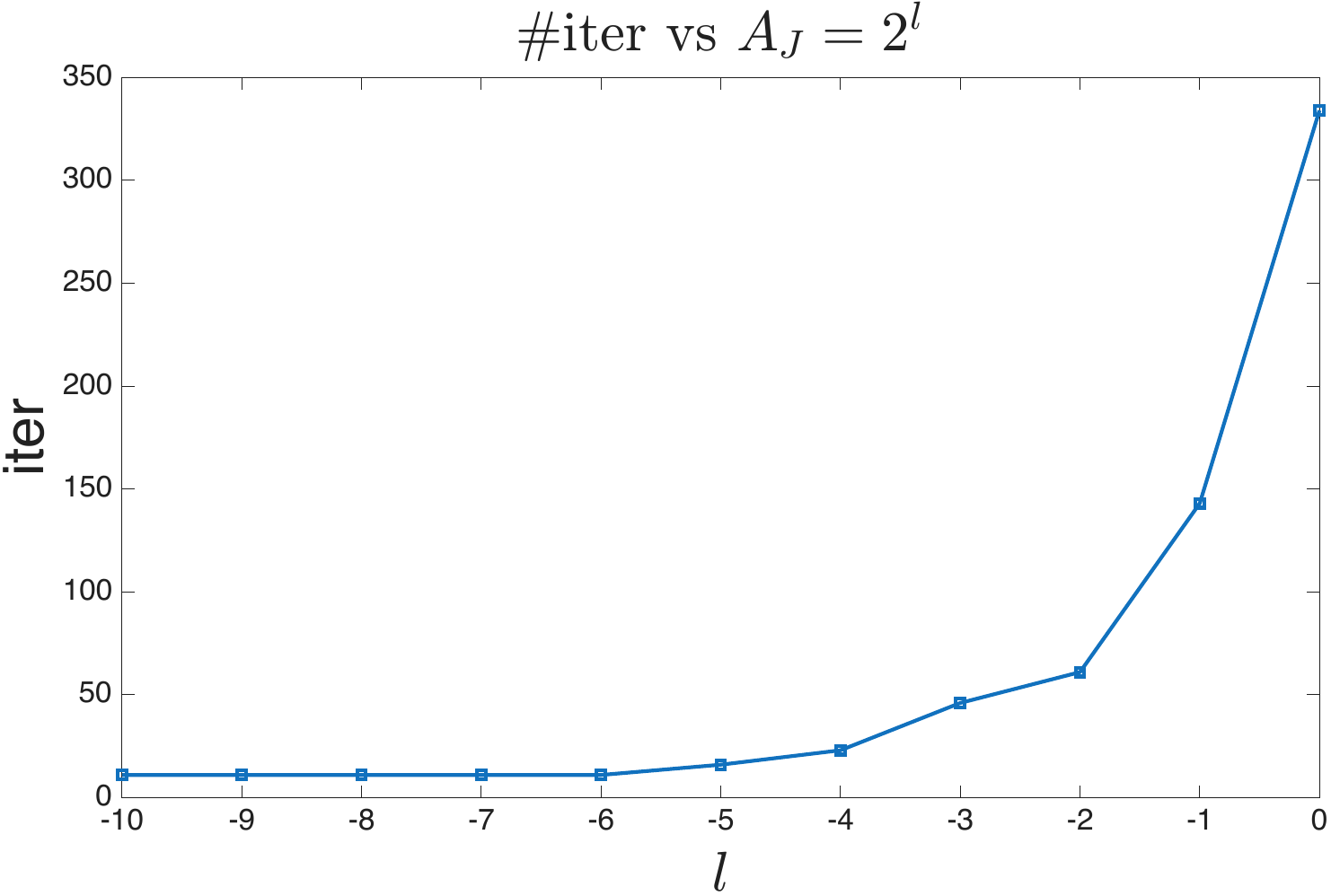}
  \caption{\#iter vs $A_{\cJ}$}
\end{subfigure}\hfill
\begin{subfigure}[t]{0.30\textwidth}
  \centering
  \raisebox{0.2cm}{%
    \includegraphics[width=\linewidth,height=4.2cm,keepaspectratio]{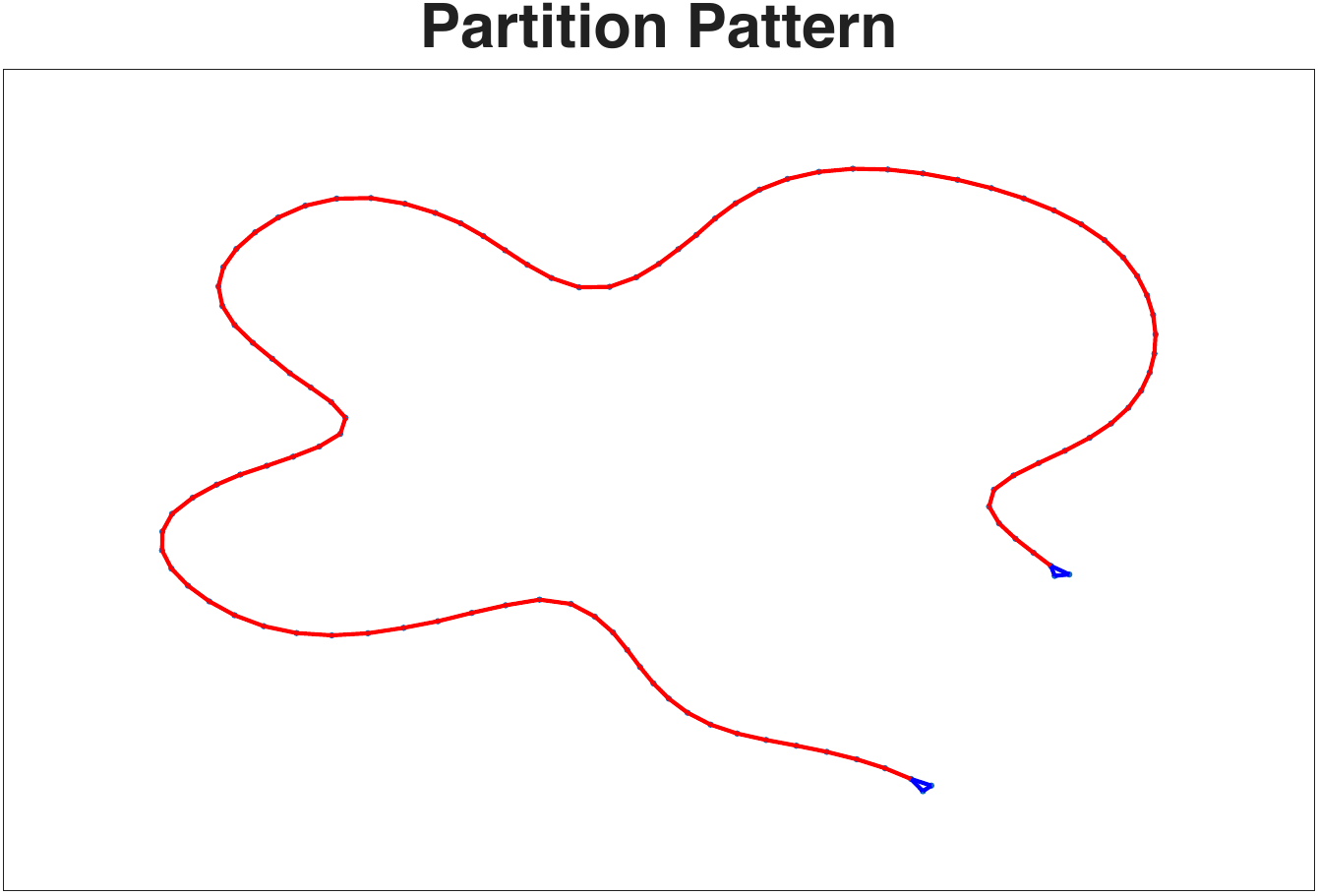}%
  }
  \caption{Graph partition}
\end{subfigure}
\caption{Graph as in panel (c). In (a), $\kappa$ varies while keeping all the other parameters fixed, whereas in (b) only $A_{\cJ}$ changes.}
\label{fig:rate_plot_kappa}\vspace{-0.3cm}
\end{figure*}

\vspace{-0.3cm}

\subsection{Optimal rate scaling for ring and 2D graphs}\vspace{-0.2cm}
\label{subsubsec:partition_guidance}
 We specialize here the general rate expression~\eqref{eq:rate_expression} to two representative graph families
(rings and 2D grids) and to simple parametric partition patterns. Our goal is not to characterize \emph{all}
possible partitions, but to show how ~\eqref{eq:rate_expression}  guides the design and leads to explicit \emph{rate scalings}
in terms of the network size. \vspace{-0.2cm}
\subsubsection{Ring graphs}\vspace{-0.2cm}
 Let $\cG$ be a ring. We study two partition strategies, namely:   $(\mathcal P_1)$  \textit{one path plus singletons}, and $(\mathcal P_2)$ \textit{equal length paths}. \smallskip

\noindent \textbf{ $(\mathcal P_1)$    one path plus singletons:} Fix $D\in\{1,\ldots,m-2\}$ and define
one non-singleton cluster $\cC_1=\{s,s+1,\ldots,s+D\}$ (mod $m$) and all remaining nodes as singletons.  
For   $D\le m-2$, one has $p=m-D$, $|\cC_1|=D+1$,
$
\cN_{\cC_1}=\{s-1,\ s+D+1\}$,     and 
$\cJ=\{s-1,\ s+D+1\}$; hence  $A_\cJ$ in \eqref{eq:def_A} reads
$$A_{\cJ}=A_1
=\frac{(2L_1+\mu_1)L_{\partial 1}^2\,(D+1)D}{4\mu_1^2}.
$$
Substituting the expression of $A_{\cJ}$ in term \texttt{(III)} in \eqref{eq:rate_expression}, we obtain \[
\texttt{(III)}(D)=\sqrt{\frac{ \min_{j\in \cJ}\mu_{j}}{8(2D+1)A_{\cJ}(D)}}
=\frac{c_1}{\sqrt{(2D+1)D(D+1)}}=\Theta(D^{-3/2}),
\]
where $c_1:=\sqrt{(\min_{j\in \cJ}\mu_{j})/2}\cdot\frac{\mu_1}{L_{\partial 1}\sqrt{2L_1+\mu_1}}$ is independent of $D$.

We proceed to minimize \eqref{eq:rate_expression} with respect to $D\in \{1,m-2\}$, for large $m$. To choose $D$, note that $\texttt{(I)}(D)=1/(m-D)$ increases with $D$, while $\texttt{(II)}(D)=2\kappa/(2D+1)$
and $\texttt{(III)}(D)$ decrease with $D$. Hence the maximizer of $\min\{\texttt{(I)},\texttt{(II)},$ $\texttt{(III)}\}$
is obtained by equating $\texttt{(I)}$ with the smallest decreasing term.
For large $m$, taking $D=\Theta(m)$ makes $\texttt{(III)}(D)=\Theta(m^{-3/2})\ll \texttt{(I)}(D)=\Theta(m^{-1})$,
so $\texttt{(III)}$ becomes the bottleneck; thus the relevant balance is $\texttt{ (I)}\asymp\texttt{(III)}$. Solving $1/(m-D)\asymp D^{-3/2}$ yields  the asymptotic optimizer
\[
D^\star=\Theta(m^{2/3}),\qquad p^\star=m-D^\star=m-\Theta(m^{2/3}),
\] and the resulting rate scaling
\[
1-\rho(D^\star)=\Theta\!\left(\frac{1}{\kappa m}\right),
\qquad
(1-\rho(D^\star))^{-1}=\Theta(\kappa m).
\]
Notice that the resulting contraction scaling exhibits the same Jacobi-type dependence on the number of
blocks, since $p=m-D\approx m$. Indeed, ${\mathcal  P}_1$ partitions the ring into one
main cluster and many singletons, so the condensed graph remains nearly as large as the original one,
and the rate is therefore bottlenecked by term~\texttt{(I)}$=1/p$. Despite $|\cJ|=2$ localizes delay
accumulation, this affects only term~\texttt{(III)} and cannot overcome the global limitation imposed by
$p\approx m$. 
Motivated by this observation, we next analyze a balanced partition ${\mathcal  P}_2$, which leverages the ring
symmetry to enforce $p\ll m$ while keeping $D$ controlled.  \smallskip 

\noindent \textbf{$(\mathcal P_2)$   equal length paths: } Assume $(D+1)\mid m$ and set $p=m/(D+1)$. Partition the ring into $p$ disjoint contiguous paths (clusters): for each $r\in [p]$,
\[
\cC_r=\{(r-1)(D+1)+1,\ (r-1)(D+1)+2,\ \ldots,\ r(D+1)\}\quad(\text{indices mod }m),
\]
 so that $|\cC_r|=D+1$ and $D_r=D$ for all $r\in[p]$. 

Fix $r\in[p]$. Since $\cC_r$ is a contiguous path segment on the ring, the only edges leaving $\cC_r$
are the two boundary edges adjacent to its endpoints. 
Therefore,  the  
 neighbor outside $\cC_r$ is
\[
\cN_{\cC_r}=\{(r-1)(D+1),\ r(D+1)+1\}\quad(\text{mod }m).
\]
Observe that $(r-1)(D+1)$ is the last node of cluster $\cC_{r-1}$, and $r(D+1)+1$ is the first node of cluster
$\cC_{r+1}$ (with indices modulo $p$). Therefore, taking the union over $r\in [p]$ yields exactly the set
of all cluster endpoints, namely the nodes of the form $r(D+1)$ and $(r-1)(D+1)+1$: \vspace{-0.1cm}
\[
\bigcup_{r:\,|\cC_r|>1}\cN_{\cC_r}
=
\{\,r(D+1)\,:\ r\in [p]\,\}\ \cup\ \{\, (r-1)(D+1)+1\,:\ r\in [p]\,\}. \vspace{-0.1cm} \]
 Finally, each endpoint belongs to exactly one cluster, so any cover of this
set by clusters must include every $\cC_r$, and the minimal such cover is $\cJ=[p]$.

Let us   derive the explicit  expression of  $A_\cJ$. 
Notice that $i\in\cN_{\cC_r}$ holds iff
$i$ equals one of these two boundary-neighbor nodes (mod $m$). Since  $\{\cN_{\cC_r}\}_{r=1}^p$
are disjoint in the ring equal-length construction,
for every fixed $i$, the   sum  in \eqref{eq:def_A} contains at most one nonzero term, and equals either $A_r$ for a unique $r$
or $0$. Maximizing over $i$   and using $|\cC_r|=D+1$ and $D_r=D$ in $A_r$ yield\vspace{-0.1cm}  \[
A_{\cJ}=\max_{r\in[p]} A_r
=(D+1)D\cdot \max_{r\in[p]}\frac{(2L_r+\mu_r)L_{\partial r}^2}{4\mu_r^2}.\vspace{-0.1cm}
\]
Using the above expression of $A_{\cJ}$, term \texttt{(III)} becomes \[
\texttt{(III)}(D)
=\sqrt{\frac{ \min_{r\in[p]}\mu_r}{8(2D+1)A_{\cJ}(D)}}
=\frac{c_2}{\sqrt{(2D+1)D(D+1)}}=\Theta(D^{-3/2}),
\]
for some constant $c_2>0$ independent of $D$. Moreover  $p=m/(D+1)$ gives ${\rm (I)}(D)=(D+1)/m=\Theta(D/m)$. Balancing $\texttt{(I)}\asymp\texttt{(III)}$ yields
\vspace{-0.1cm}\[
\frac{D}{m}\asymp D^{-3/2}
\quad\Longrightarrow\quad
D^\star=\Theta(m^{2/5}),
\qquad
p^\star=\frac{m}{D^\star+1}=\Theta(m^{3/5}).
\]
At this choice, for fixed $\kappa$,  $\texttt{(II)}(D^\star)=\Theta(\kappa m^{-2/5})\gg \Theta(m^{-3/5})=\texttt{(I)}(D^\star)\asymp\texttt{(III)}(D^\star)$,
so $\texttt{(II)}$ is inactive and $\min\{\texttt{(I)},\texttt{(II)},\texttt{(III)}\}=\Theta(m^{-3/5})$.
The resulting rate reads  
\[
1-\rho(D^\star)=\Theta\!\left(\frac{1}{\kappa}\,m^{-3/5}\right),
\qquad
(1-\rho(D^\star))^{-1}=\Theta(\kappa\,m^{3/5}).
\]
This rate scaling strictly improves over the one-path-plus-singletons partition ${\mathcal P}_1$.  This  comes from exploiting the translational symmetry of the ring
to reduce the number of blocks from $p\approx m$ (under ${\mathcal P}_1$) to $p=\Theta(m^{3/5})$ (under ${\mathcal P}_2$),
without letting the intra-cluster diameter $D$ grow too aggressively; the optimal choice
$D^\star=\Theta(m^{2/5})$ balances the block  limitation \texttt{(I)} against the coupling-driven term \texttt{(III)}.  

Overall, on rings (and more generally on homogeneous circulant topologies), balanced partitions that distribute
cluster diameters evenly are preferable: they leverage symmetry to simultaneously reduce $p$ and keep $D$
controlled, yielding a provably better asymptotic rate. Unbalanced partitions such as ${\mathcal  P}_1$ are mainly
useful when structural or modeling considerations single out a localized strongly-coupled region to be kept
within one cluster, but they do not improve the Jacobi-type scaling inherent to having $p\approx m$. 
 \vspace{-0.2cm}
 
 \subsubsection{2D-grid  graphs}
  Let $\cG$ be a $\sqrt m\times \sqrt m$ grid (assume $m$ is a perfect square) with row-wise indexing.
Fix $D\in\{1,\ldots,\sqrt{m}-1\}$. We consider {\it horizontal path} clusters of length $D$: 
\[
\cC_r=\{(r-1)\sqrt m+1,\ (r-1)\sqrt m+2,\ \ldots,\ (r-1)\sqrt m+(D+1)\},\quad r\in [\sqrt{m}];
\]
  and let all remaining nodes be singletons. The number of clusters is
\[
p=\sqrt{m}+\bigl(m-(D+1)\sqrt{m}\bigr)=m-D\sqrt{m}.
\]
Moreover, $|\cC_r|=D+1$ and $D_r=D$ for all $r\in[n]$ (while singletons have diameter $0$), hence
$D=\max_{r\in[p]}D_r$.

We now compute $\cJ$ and $A_{\cJ}$ for this partition and minimize    \eqref{eq:rate_expression}  over $D$.

Fix a row-cluster $\cC_r$. Since $\cC_r$ lies in columns $1,\ldots,D+1$, its external neighborhood
$\cN_{\cC_r}$ consists of:
(i) the right neighbor of its row-endpoint, that is  $(r-1)\sqrt m+(D+2)$ (when $D\le \sqrt m-2$);
(ii) the vertical neighbors (rows $r\pm1$) of all nodes in $\cC_r$ (when those rows exist). Thus, for $D\le \sqrt m-2$,
\begin{align*}
\cN_{\cC_r}
&=\{(r-1)\sqrt m+(D+2)\}\ \cup\
\{(r-2)\sqrt m+c:\ c=1,\ldots,D+1: r\ge2\}\ \\&\quad \cup\
\{r\sqrt m+c:\ c=1,\ldots,D+1: r\le \sqrt m-1\}.
\end{align*}
Therefore, \vspace{-0.2cm}\begin{equation}\label{eq:grid_union_extnbr_clean_m}
\bigcup_{r:\,|\cC_r|>1}\cN_{\cC_r}
=
 {\bigcup_{r=1}^{\sqrt m} \cC_r} 
\ \cup\
 {\{(r-1)\sqrt m+(D+2):\ r=1,\ldots,\sqrt m\}}.
\end{equation}
(The second set is absent when $D=\sqrt m-1$.)
Any cover of \eqref{eq:grid_union_extnbr_clean_m} by clusters
must include all row-clusters $\cC_1,\ldots,\cC_{\sqrt m}$; and when $D\le \sqrt m-2$, it must also include the $\sqrt m$
singleton clusters at column $D+2$. Hence, for $D\le \sqrt m-2$, 
\begin{equation}\label{eq:grid_J_clean_m}
\cJ=\{1,\ldots,\sqrt m\}\ \cup\ \{\text{singleton clusters } \{(r-1)\sqrt m+(D+2)\}:\ r\in[\sqrt m]\},
\end{equation}
and  $|\cJ|=2\sqrt m\ll p$; whereas for $D=\sqrt m-1$ one has $\cJ=[\sqrt m]$.

The only clusters with $|\cC_r|>1$ are the $\sqrt m$ row-clusters, and for each such $r$,
\[
A_r(D)=\frac{(2L_r+\mu_r)L_{\partial r}^2\,|\cC_r|\,D_r}{4\mu_r^2}
=\frac{(2L_r+\mu_r)L_{\partial r}^2\,(D+1)D}{4\mu_r^2}.
\]
We  derive the expression of  $A_\cJ$ as function of the above quantities.  Fix $i\in\cup_{j\in\cJ}\cC_j$. If $i$ is a node in column $D+2$, then $i\in\cN_{\cC_r}$ for a \emph{unique} row $r$,
and the corresponding sum in    \eqref{eq:def_A} equals $A_r(D)$. If instead $i$ lies in columns $1,\ldots,D+1$, then $i$ can be adjacent
(vertically) to at most two row-clusters (the ones in the rows immediately above and below), so the sum contains
at most two terms. Therefore,
\begin{equation}\label{eq:grid_AJ_bounds_clean_m}
\max_{1\le r\le \sqrt m}A_r(D)\ \le\ A_{\cJ} \ \le\ 2\max_{1\le r\le \sqrt m}A_r(D),
\end{equation}
and in particular $A_{\cJ}=\Theta(D(D+1))$. Consequently,
\[
\texttt{(III)}(D)
=\sqrt{\frac{\min_{r\in\cJ}\mu_r}{8(2D+1)A_{\cJ}}}
=\frac{c_{\rm grid}}{\sqrt{(2D+1)D(D+1)}}=\Theta(D^{-3/2}),
\]
for some constant $c_{\rm grid}>0$ independent of $D$. 

We now maximize $\min\{\texttt{(I)}(D),\texttt{(II)}(D),\texttt{(III)}(D)\}$ over $D\in[\sqrt m-1]$ for large $m$. Recall, $\texttt{(I)}(D)=1/(m-D\sqrt m)$  and $\texttt{(II)}(D)=2\kappa/(2D+1)$.  For fixed $\kappa$ and large $\sqrt m$, the balance \texttt{(I)}$\asymp$\texttt{(II)} would give
$D=\Theta(\sqrt m)$ and thus $\texttt{(II)}=\Theta(\kappa/\sqrt m)$, whereas $\texttt{(III)}=\Theta(m^{-3/4})\ll\kappa/\sqrt m$,
so \texttt{(III)} becomes the bottleneck. Therefore the relevant balance is \texttt{(I)}$\asymp$\texttt{(III)}.

To solve it, write\vspace{-0.2cm}
\[
D=\sqrt m-\Delta,\qquad \Delta\in\{1,\ldots,\sqrt m-1\}.
\]
Then $p= \Delta\sqrt m$, so  
$\texttt{(I)}(D)={1}/({\Delta\sqrt m}).$ 
Moreover, for $D=\sqrt m-\Delta$ with $\Delta=o(\sqrt m)$ one has $\texttt{(III)}(D)=\Theta(D^{-3/2})=\Theta(m^{-3/4})$.
Equating $\texttt{(I)}\asymp\texttt{(III)}$ gives
\[
({\Delta\sqrt m})^{-1}\ \asymp\ m^{-3/4}
\quad\Longrightarrow\quad
\Delta\ \asymp\ m^{1/4}.
\]
Therefore,
\[
D^\star=\sqrt m-\Theta(m^{1/4}),
\qquad
p^\star=m-D^\star\sqrt m
=\Theta(m^{3/4}).
\]
At this choice,
\[
\texttt{(II)}(D^\star)=\Theta\!\left( {\kappa}{  m^{-1/2}}\right)\gg \Theta(m^{-3/4})
=\texttt{(I)}(D^\star)\asymp\texttt{(III)}(D^\star),
\]
so $\texttt{(II)}$ is inactive and $\min\{\texttt{(I)},\texttt{(II)},\texttt{(III)}\}=\Theta(m^{-3/4})$.
Substituting into~\eqref{eq:rate_expression} yields\vspace{-0.2cm}
\[
1-\rho(D^\star)=\Theta\!\left( {\kappa}^{-1}\,m^{-3/4}\right),
\qquad
(1-\rho(D^\star))^{-1}=\Theta(\kappa\,m^{3/4}).
\]

 This results shows that,  for this  {one-directional}
partition family, the exponent $3/4$ in the rate scaling is the unavoidable outcome of the tradeoff between reducing $p$ and the
simultaneous deterioration of \texttt{(III)} driven by the growth of $A_{\cJ}$.
In fact,  the only way to decrease the Jacobi-type term \texttt{(I)}$=1/p$
is to make the row paths as long as possible, since $p=m-D\sqrt m$ shrinks only when $D$ approaches $\sqrt m$.
However, longer paths also enlarge the delay--coupling penalty:  
$A_r\propto |\cC_r|D_r=(D+1)D$, so $A_{\cJ}$ grows like $D(D+1)$ and term \texttt{(III)} decays as
$\Theta(D^{-3/2})$. Balancing these two opposing effects forces $D$ to be ``almost'' $\sqrt m$ but not equal to it,
specifically $D^\star=\sqrt m-\Theta(m^{1/4})$, which leaves a residual block count $p^\star=\Theta(m^{3/4})$
and yields the overall rate scaling $(1-\rho)^{-1}=\Theta(\kappa m^{3/4})$.  \vspace{-0.5cm}

\section{Enhancing computation and communication via surrogation}\vspace{-0.2cm}
\label{sec:surrogate_pairwise}

We introduce a surrogate-based variant of Algorithm~\ref{alg:main} to  reduce both per-iteration computation and communication as well as  ensure well-posed local updates  when Assumption~\ref{asm:nonverlap} does not hold.  Updates in the form~\eqref{gossip_update:a} and~\eqref{message_update} require solving local minimization problems \emph{exactly}, each intra-cluster message $\mu_{j\to i}^\nu(\cdot)$ is a function of the receiver's variable, which can be expensive to compute and transmit. This is particularly acute beyond   quadratic function-messages, which are generally infinite-dimensional objects. Even for quadratic objectives, exact messages are quadratic functions parameterized by matrices that may be dense, hence communication demanding for large block-dimensions. Our approach is to replace the local costs by suitable surrogates: this  ensures the existence of local minimizers,    makes the subproblems cheap to solve, and induces messages with lightweight parameterizations (e.g., affine or structured-quadratic). 

More formally, at iteration $\nu$, each agent $i$ uses reference values
$y_i:=x_i^\nu$, $i\in \cV$   and $y_{\cN_i}:=\{x_j^\nu\}_{j\in\cN_i}
$, and employs   model surrogates    $\tilde\phi_i(\cdot;y_i)$ and $\tilde\psi_{ij}(\cdot,\cdot;\,y_i,y_j)$ around the reference values, replacing $\phi_i(\cdot)$ and $\psi_{ij}(\cdot,\cdot)$, respectively. 
The resulting scheme is summarized in Algorithm~\ref{alg:main1_surrogate}.  \vspace{-0.4cm}

\begin{algorithm}[th]
{
 \scriptsize
\caption{\underline{M}essage \underline{P}assing-\underline{J}acobi ({MP-Jacobi}) with Surrogate}
\label{alg:main1_surrogate}

{\bf Initialization:}  $x_i^{0}\in \mathbb{R}^d$, for all $i\in\cV$; 
initial message $\widetilde \mu_{i\to j}^{0}(\cdot)$ arbitrarily chosen  (e.g., $ {\widetilde \mu_{i\to j}^{0}}\equiv 0$), for all $(i,j)\in\cE_{r}$ and $r\in [p]$.\\

\For{$\nu=0,1,2,\ldots$ }{
\AgentFor{$i\in\cV$}{\label{line:agent-loop_hyper}
\medskip 

\texttt{Performs the following updates:}
\begin{subequations}\label{gossip_update_surrogate}
\begin{align}
\hat x_i^{\nu+1}&\in \argmin_{x_i}\Big\{\widetilde\phi_i(x_i;x_i^\nu)
+ \sum_{j\in\cNin_i}{\widetilde\mu_{j\to i}^{\nu}(x_i)}
+ \sum_{k\in\cNout_i}\widetilde\psi_{ik}(x_i,x_k^{\nu};x_i^\nu,x_k^\nu)\Big\},
\label{gossip_update_surrogate:a}\\
x_i^{\nu+1}&=x_i^{\nu}+\tau_r^\nu\big(\hat x_i^{\nu+1}-x_i^{\nu}\big),
\label{gossip_update_surrogate:b}  \\\nonumber \\
\hspace{-.4cm} {{\widetilde \mu_{i\to j}^{\nu+1}(x_j)}}&=
\min_{x_i}\qty{
\begin{aligned}
&\widetilde\phi_i(x_i;x_i^\nu)+\widetilde\psi_{j i}(x_j,x_i;x_j^\nu,x_i^\nu)
\\
&+\sum_{k\in\cNin_i\setminus\{j\}}\widetilde\mu_{k\to i}^{\nu}(x_i)
+\sum_{k\in\cNout_i}\widetilde\psi_{ik}(x_i,x_k^{\nu};x_i^\nu,x_k^\nu)
\end{aligned}
}, \forall j\in   \cNin_i;
\label{message_update_surrogate}
\end{align}
\end{subequations}

\texttt{Sends out}  $x_i^{\nu+1}$ to all $j\in\cNout_i$  and  {$\widetilde \mu_{i\to j}^{\nu+1}(x_j)$}  to all $j\in \cNin_i$.
}}}
\end{algorithm}
 \vspace{-0.4cm}
\begin{remark}
    While Algorithm~\ref{alg:main1_surrogate} is presented using the same surrogate models in both the variable and message updates, surrogation can be applied selectively. For instance, one may keep the variable update \eqref{gossip_update_surrogate:a} exact while surrogating the message update \eqref{message_update_surrogate} to reduce communication; conversely, surrogating \eqref{gossip_update_surrogate:a} primarily targets local computational savings. The same convergence guarantees can be established   for these mixed exact/surrogate variants, but we omit the details for brevity--some numerical results involving these instances are presented in Sec.~\ref{sec:experiments}.\end{remark}
 \vspace{-0.4cm}

 \subsection{Surrogate regularity conditions}\label{sec:surrogate-regularity}\vspace{-0.1cm}

We request mild  regularity conditions on the surrogates functions  to guarantee convergence of Algorithm~\ref{alg:main1_surrogate}.  
We state these conditions 
directly on the \emph{cluster surrogate objective} induced by the local  models, as formalized next.

For any  cluster $r\in[p]$ and $\bx\in \mathbb{R}^{md}$,   recall the partition $\bx=(x_{\cC_r},x_{\overline{\cC}_r})$, where $x_{\cC_r}$ are the cluster variables and  
$x_{\overline{\cC}_r}$ are the variables outside the cluster.  
Define the cluster-relevant portion of the global objective:
\begin{equation}
\label{eq:Phi_r_def_simplified} 
\Phi_r(\bx)
:=\sum_{i\in\cC_r}\phi_i(x_i)
+\sum_{(i,j)\in\cE_r}\psi_{ij}(x_i,x_j)
+\sum_{i\in\cC_r}\ \sum_{k\in\cNout_i}\psi_{ik}(x_i,x_k).
\end{equation}
Given surrogates $\tilde\phi_i(\cdot;\cdot)$ and $\tilde\psi_{ij}(\cdot,\cdot;\cdot,\cdot)$,
we build an aggregated surrogate for $\Phi_r$ around the \emph{reference points} collected into    \[
\zeta_r:=\big(y_{\cC_r},\,y_{\overline{\cC}_r},\,y_{\cE_r}\big),
\]
where 
$y_{\cC_r}:=\{y_i\}_{i\in\cC_r}$ (cluster nodes),    $y_{\overline{\cC}_r}:=\{y_k\}_{k\in\overline{\cC}_r}$ (outside nodes), and  $y_{\cE_r}$ is the  intra-cluster edge reference stack
\[y_{\cE_r}:=\big(( y_e, y_e^\prime)\big)_{e\in\cE_r}.
\]
 Fixing an arbitrary ordering of $\cE_r$, we  identify each edge-stack
$y_{\cE_r}=((y_e,y_e'))_{e\in\cE_r}$ with its stacked vector in $\R^{2|\cE_r|d}$
(endowed with the Euclidean norm); with a slight abuse of notation, we use the same symbol $y_{\cE_r}$ for these vectors.

The aggregated surrogate around $\zeta_r$ is defined as  
\begin{equation}
\label{eq:agg_surrogate_compact_fix}
\begin{aligned}
\tilde\Phi_r\big(\bx;\zeta_r\big)
:=\;&\sum_{i\in\cC_r}\tilde\phi_i(x_i;y_i)
+\sum_{(i,j)\in\cE_r}\tilde\psi_{ij}\big(x_i,x_j;y_{(i,j)},y_{(i,j)}^\prime\big)\\
&+\sum_{i\in\cC_r}\ \sum_{k\in\cNout_i}\tilde\psi_{ik}\big(x_i,x_k;y_i,y_k\big).
\end{aligned}
\end{equation}
When the ``exact surrogates'' are used ($\tilde\phi_i\equiv\phi_i$ and $\tilde\psi_{ij}\equiv\psi_{ij}$),
one has $\tilde\Phi_r\equiv\Phi_r$.

\noindent \textit{Consistency:} We say that $\zeta_r$ is \emph{consistent with} $\bx$ if all references coincide with the associated
blocks of $\bx$, i.e.,\[y_{\cC_r}=x_{\cC_r},\quad y_{\overline{\cC}_r}=x_{\overline{\cC}_r},\quad \text{and}\quad y_{\cE_r}=x_{\cE_r}:=((x_i,x_j))_{(i,j)\in \cE_r}.\]

We request the following conditions on the aggregate surrogates $\tilde\Phi_r$, $r\in [p]$.
\begin{assumption}[cluster surrogate regularity]
\label{assumption:surrogation}
Each surrogate $\tilde\Phi_r$ in~\eqref{eq:agg_surrogate_compact_fix} satisfies the following conditions: 
\begin{enumerate}
\item[(i)] \texttt{(gradient consistency)} For all $\bx=(x_{\cC_r},x_{\overline{\cC}_r})$ and  all $\zeta_r$ consistent with $\bx$,
\begin{equation}
\label{eq:cluster_touching}
\nabla_{\cC_r}\tilde\Phi_r\big(\bx;\zeta_r\big)
=\nabla_{\cC_r}\Phi_r(\bx).
\end{equation}
\item[(ii)]  For all $\bx=(x_{\cC_r},x_{\overline{\cC}_r})$ and all  reference tuples $\zeta_r$,
\begin{equation}\label{eq:sur_majorization_clusterlevel}
\Phi_r(\bx)
\ \le\
\tilde\Phi_r\big(\bx;\,\zeta_r\big),
\end{equation}
with equality whenever $\zeta_r$ is consistent with $\bx$.
 \item[(iii)] \texttt{(uniform strong convexity and smoothness in $x_{\cC_r}$)} 
There exist    constants $\tilde\mu_r,\tilde L_r\in (0,\infty)$ such that, for all  $x_{\overline{\cC}_r}$ and all  tuple $\zeta_r$,
the mapping
\[
u\ \mapsto\ \tilde\Phi_r\big((u,x_{\overline{\cC}_r});\,\zeta_r\big),\quad u\in \mathbb{R}^{d|\cC_r|},
\]
is $\tilde\mu_r$-strongly convex and has $\tilde L_r$-Lipschitz continuous gradient.  Define the surrogate condition number $\tilde{\kappa}$ as $\tilde{\kappa}:= (\max_{r\in[p]}\tilde L_r)/\mu$.  

\item[(iv)] \texttt{(sensitivity to intra-cluster edge references)}
There exists $\tilde\ell_r\in(0,\infty)$ such that, for all $\bx=(x_{\cC_r},x_{\overline{\cC}_r})$
and all $y_{\cC_r}$, the mapping
\[
v\ \mapsto\ \nabla_{\cC_r}\tilde\Phi_r\big(\bx;\,(y_{\cC_r},x_{\overline{\cC}_r},v)\big),
\qquad v\in\R^{2|\cE_r|d},
\]
is $\tilde\ell_r$-Lipschitz, i.e., for all $v,v^\prime\in\R^{2|\cE_r|d}$,
\[
\Big\|
\nabla_{\cC_r}\tilde\Phi_r\big(\bx;\,(y_{\cC_r},x_{\overline{\cC}_r},v)\big)
-
\nabla_{\cC_r}\tilde\Phi_r\big(\bx;\,(y_{\cC_r},x_{\overline{\cC}_r},v^\prime)\big)
\Big\|
\le \tilde\ell_r\,\|v-v^\prime\|.
\]

\item[(v)] \texttt{(locality through boundary variables)} 
There exists a constant $\tilde L_{\partial r}\in (0,\infty)$ such that,  for all  $x_{\cC_r}$,  $y_{\cC_r}$, and   
$y_{\cE_r}$, the mapping
\[
w\ \mapsto\ \nabla_{\cC_r}\tilde\Phi_r\big((x_{\cC_r},w);\,(y_{\cC_r},w,y_{\cE_r})\big),\quad w\in\R^{|\overline{\cC}_r|d},
\]
is $\tilde L_{\partial r}$-Lipschitz along $P_{\partial r}$,  i.e., 
\begin{align} 
& 
\Big\|
\nabla_{\cC_r}\tilde\Phi_r\big((x_{\cC_r},w);\,(y_{\cC_r},w,y_{\cE_r})\big)
-
\nabla_{\cC_r}\tilde\Phi_r\big((x_{\cC_r},w^\prime);\,(y_{\cC_r},w^\prime,y_{\cE_r})\big)
\Big\|\nonumber \\
&\quad  \le\
\tilde L_{\partial r}\,\big\|P_{\partial r}(w-w^\prime)\big\|\quad \forall  w,w^\prime\in\R^{|\overline{\cC}_r|d}.
\end{align} 
\end{enumerate}
\end{assumption}

 Condition  {\bf (i)} is the cluster-level analogue of first-order consistency: when the reference tuple is
consistent with the current point, the surrogate matches the \emph{cluster gradient} of the original objective.
This is the key property ensuring that any fixed point of the algorithm is a solution for the original problem. Condition~{\bf (ii)}  is a standard majorization requirement: it makes each cluster update a descent step for $\Phi$
(up to the delay terms accounted for in the analysis), and is the main mechanism behind monotonicity/contractivity.   
    Condition  {\bf (iii)} ensures strong-convexity and smoothens  of the cluster subproblems. This guarantees that each cluster subproblem has a unique minimizer, and it provides the curvature/smoothness constants
$(\tilde\mu_r,\tilde L_r)$ affecting the convergence rate.    Conditions \textbf{(iv)}
  controls throughout the constants $\tilde\ell_r$'s how much the \emph{cluster gradients} change when the intra-cluster
\emph{edge-reference stack} is perturbed; it  governs the error induced by using delayed/stale edge references in the surrogate
messages. One can  make $\tilde\ell_r$ small by employing edge surrogates whose gradients depend weakly (Lipschitzly)
on their reference arguments--for instance, quadratic/linearized models where the reference enters only through
a Hessian approximation that is stable along the iterates. Condition~{\bf (v) }quantifies the dependence of the cluster gradient on the \emph{outside block} (and its boundary reference),
measured only through the boundary projector $P_{\partial r}$.
This matches the graph structure locality: in \eqref{eq:agg_surrogate_compact_fix}, the outside block enters $\tilde\Phi_r$
only through cross--cluster terms $\{\tilde\psi_{ik}: i\in\cC_r,\ k\in\cNout_i\}$. Consequently, changing components of $x_{\overline{\cC}_r}$ (and of the
corresponding outside reference $y_{\overline{\cC}_r}$) \emph{outside} $\cN_{\cC_r}$ does not affect
$\tilde\Phi_r$ and hence does not affect $\nabla_{\cC_r}\tilde\Phi_r$. 

\textit{Sufficient local conditions:} A simple   recipe to enforce Assumption~\ref{assumption:surrogation} is   on the {\it local} surrogates. For instance,      \emph{touching/gradient-consistency} and \emph{majorization} properties on 
$\tilde\phi_i$ and $\tilde\psi_{ij}$  lift to~{\bf (i)}--{\bf(ii)}.
Strong convexity and smoothness in~{\bf (iii)} can be ensured by choosing all the cluster surrogates
to include enough curvature in the optimized variables (e.g., via quadratic/proximal regularization in
the node models and/or in the edge surrogates) and smoothness. 
 The sensitivity bounds~{\bf (iv)}--{\bf (v)} are typically obtained by requiring that the
relevant \emph{gradients} of the local edge surrogates depend Lipschitz-continuously on their
\emph{reference arguments}.

 However, enforcing majorization or strong convexity \emph{termwise} can be unnecessarily restrictive, especially when the couplings $\psi_{ij}$ are nonconvex. Assumption~\ref{assumption:surrogation} requires these properties only for the {\it aggregated} cluster surrogates $\tilde\Phi_r$, so that curvature and upper-bounding can be achieved {\it collectively} at the cluster level. This is a key departure from the surrogation conditions   adopted in decomposition-based methods, which are usually imposed   on the individual local models  \cite{facchinei_vi-constrained_2014,facchinei_feasible_2017,scutari_parallel_2017}. We next give examples of local surrogates   that may fail   Assumption~\ref{assumption:surrogation} term-wise, yet   inducing aggregated cluster surrogates that do satisfy Assumption~\ref{assumption:surrogation}.\vspace{-0.2cm}

\vspace{-0.3cm}
\subsection{Some surrogate examples }
\label{subsec:surrogate_example}
\vspace{-0.2cm}
 
The following are practical designs satisfying Assumption~\ref{assumption:surrogation}  and useful to reduce computation and/or  communication costs.         
Throughout, at iteration $\nu$ we employ reference points $y_i=x_i^\nu$ and  $y_{\cE_r}=\big(( x_i^\nu, x_j^\nu)\big)_{(i,j)\in\cE_r}$, for all $i\in \cV$ and $r\in [p]$.

 \smallskip 
 \noindent {\bf (i)}  \textbf{first-order   surrogates:}  Choose, for some $\alpha>0$ and any $\nu\geq 0$,
\begin{subequations}
    \begin{align} 
\tilde\phi_i(x_i;x_i^\nu)
&:= \phi_i(x_i^\nu)+\langle\nabla\phi_i(x_i^\nu),x_i-x_i^\nu\rangle+\frac{1}{2\alpha}\|x_i-x_i^\nu\|^2,
\label{eq:surrogate_FO_phi}\\
\tilde\psi_{ij}(x_i,x_j;x_i^\nu,x_j^\nu)
&:= \psi_{ij}(x_i^\nu,x_j^\nu)
+\langle\nabla_i\psi_{ij}(x_i^\nu,x_j^\nu),x_i-x_i^\nu\rangle\nonumber\\
&\quad +\langle\nabla_j\psi_{ij}(x_i^\nu,x_j^\nu),x_j-x_j^\nu\rangle,
\label{eq:surrogate_FO_psi}
\end{align} 
 for all $i\in \cV$ and $(i,j)\in \cE$.  Using  \eqref{eq:surrogate_FO_phi}--\eqref{eq:surrogate_FO_psi} in
 \eqref{message_update_surrogate} yields:  \begin{equation}
    \label{eq:affine-message}
\tilde\mu_{i\to j}^{\nu+1}(x_j)
\;\sim\;
\Big\langle \nabla_j\psi_{ji}(x_j^\nu,x_i^\nu),\,x_j-x_j^\nu\Big\rangle,
 \end{equation}\end{subequations} 
 Hence each directed message is encoded (up to constants) by a single $d$-vector. 
 
Using (\ref{eq:affine-message}) and \eqref{eq:surrogate_FO_phi}--\eqref{eq:surrogate_FO_psi} in \eqref{gossip_update_surrogate:a}, the   variable subproblem
\eqref{gossip_update_surrogate:a} yields the gradient-delayed closed form:
\[
\hat x_i^{\nu+1}
= x_i^\nu-\alpha\!\left(
\nabla\phi_i(x_i^\nu)
+ \!\!\sum_{j\in\cNin_i}\!\!\nabla_i\psi_{ij}(x_i^{\nu-1},x_j^{\nu-1})
+ \!\!\sum_{k\in\cNout_i}\!\!\nabla_i\psi_{ik}(x_i^\nu,x_k^\nu)
\right).
\] 
To implement this surrogate instance of the algorithm, agents only exchange two $d$-dimensional vectors with their neighbors  $\cNin_i$ and $\cNout_i$ per iterations. 

Note that the above  surrogates satisfy Assumption \ref{assumption:surrogation}; in particular the majorization condition is ensured by choosing (a sufficiently small) $\alpha$.     
 \smallskip 
 
 \noindent {\bf (ii)} \textbf{Quadratic messages and Schur recursion:} Let $\tilde\phi_i(\cdot;x_i^\nu)$   be any strongly convex quadratic function  (e.g.,  as in
\eqref{eq:surrogate_FO_phi} or a second-order model) with {Hessian matrix $Q_i\in \mathbb S_{++}^d$}, $i\in \cV$; and  let $\tilde\psi_{ij}$ be chosen  as the quadratic surrogate 
\begin{subequations}
\begin{equation*}
\begin{aligned}
\widetilde\psi_{ij}\bigl(x_i,x_j;x_i^\nu,x_j^\nu\bigr)
&= \psi_{ij}(x_i^\nu,x_j^\nu)
  + \langle \nabla_i\psi_{ij}(x_i^\nu,x_j^\nu),\,x_i-x_i^\nu\rangle
  + \langle \nabla_j\psi_{ij}(x_i^\nu,x_j^\nu),\,x_j-x_j^\nu\rangle \\
&\quad + \langle M_{ij}(x_i-x_i^\nu),\,x_j-x_j^\nu\rangle
  + \tfrac{1}{2}\|x_i-x_i^\nu\|_{M_i}^2
  + \tfrac{1}{2}\|x_j-x_j^\nu\|_{M_j}^2,
\end{aligned} 
\end{equation*}
with each $M_i\in \mathbb S_+^d$  
and   (possibly structured)  cross matrices  $M_{ij}\in \mathbb S^d$.  
Assume     the messages are initialized as quadratic functions: 
\begin{equation}\label{eq:message_quadratic_init}
\tilde\mu_{i\to j}^{0}(x_j)\;\sim\;\tfrac12\|x_j\|_{H_{i\to j}^{0}}^{2}+\langle h_{i\to j}^{0},x_j\rangle,
\end{equation}
for some $H_{i\to j}^{0}\in \mathbb{S}_+^d$ and $h_{i\to j}^{0}\in \mathbb{R}^d$. Then, for every $\nu\ge 0$,  the updated messages   produced by
\eqref{message_update_surrogate} remain quadratic:   
\begin{equation}\label{eq:message_quadratic_recursion}
\tilde\mu_{i\to j}^{\nu+1}(x_j)\;\sim\;\tfrac12\|x_j\|_{H_{i\to j}^{\nu+1}}^{2}+\langle h_{i\to j}^{\nu+1},x_j\rangle.  
\end{equation}
Moreover the curvature matrices satisfy the Schur-complement recursion   
\[
H^{\nu+1}_{i\to j}
= M_j - M_{ji}\!\left({Q_i}+M_i+\!\!\sum_{k\in\cNin_i\setminus \{j\}}\! H^\nu_{k\to i}\right)^{-1}\! M_{ji}^\top.
\]
(The corresponding recursion for $h_{i\to j}^{\nu+1}$ is explicit as well, but omitted   for brevity). 

Since each
$\tilde\mu_{i\to j}^{\nu+1}(x_j)$ is   specified (up to constants) by the pair
$\big(H_{i\to j}^{\nu+1},\,h_{i\to j}^{\nu+1}\big)$,  communication is efficient whenever  
$H_{i\to j}^{\nu+1}$ admits a compact parametrization. In particular, by selecting  
$\{M_i\}$ and $\{M_{ij}\}$ and initializing $\{H_{i\to j}^{0}\}$ so that the recursion preserves a
\emph{structured} family (e.g., (block-)diagonal, banded, sparse, low-rank$+$diagonal), each message can be
transmitted using $\mathcal{O}(d)$ (or a small multiple thereof) parameters rather than $\mathcal{O}(d^2)$.
\end{subequations} Also, 
 choosing properly   $Q_i$, $M_i$, and  $M_j$ (typical large) ensures    Assumption~\ref{assumption:surrogation} holds.
\smallskip 

\noindent {\bf (iii)} \textbf{Consensus optimization via CTA-formulation:} Consider the consensus optimization problem \eqref{eq:CTA}; it is an instance of \eqref{P}  with\vspace{-0.2cm}$$\phi_i(x_i)=f_i(x_i)+\frac{1-w_{ii}}{2\gamma}\norm{x_i}^2,\,\,   \psi_{ij}(x_i,x_j)=-\frac{1}{\gamma}w_{ij}\inner{x_i,x_j},\,\, \text{and } w_{ij}=[W]_{ij}.$$   Choose a partial linearization surrogate for $\phi_i$ and the exact surrogate for $\psi_{ij}$, i.e., \vspace{-0.2cm}
\begin{equation}
\label{example:partial_linear}
\begin{aligned}
&\tilde\phi_i(x_i;x_i^\nu)
:= f_i(x_i^\nu)+\langle\nabla f_i(x_i^\nu),x_i-x_i^\nu\rangle
+\tfrac12\|x_i-x_i^\nu\|_{Q_i}^2
+\frac{1-w_{ii}}{2\gamma}\|x_i\|^2 ,\\
&\widetilde\psi_{ij}(x_i,x_j;x_i^\nu,x_j^\nu)
 = \psi_{ij}(x_i,x_j)=-\tfrac{1}{\gamma}w_{ij}\langle x_i,x_j\rangle,
\end{aligned}
\end{equation}
 {with $Q_i\in \mathbb{S}_{+}$}. If all the initial messages $\tilde\mu_{i\to j}^{0}(x_j)$ are quadratic as in   \eqref{eq:message_quadratic_init},  all 
$\tilde\mu_{i\to j}^{\nu+1}(x_j)$ will be  quadratic as in \eqref{eq:message_quadratic_recursion}, with curvature recursion $H_{i\to j}^{\nu+1}$ 
given by the following recursion  
\vspace{-0.2cm}
\begin{equation}
H_{i\to j}^{\nu+1}
= -\frac{w_{ij}^2}{\gamma^2}
\left({Q_i+\frac{1-w_{ii}}{\gamma}I_d}+\sum_{k\in\cNin_i\setminus \{j\}} H_{k\to i}^\nu\right)^{-1}.
\label{eq:eg_Hessian_message_update}
\end{equation}
Therefore, if one chooses $Q_i$  to be (block-)diagonal and initializes $\{H_{i\to j}^0\}$ as
(block-)diagonal, then \eqref{eq:eg_Hessian_message_update} preserves (block-)diagonality for all $\nu$,
so each message can be transmitted using $\cO(d)$ parameters (diagonal of $H_{i\to j}^\nu$ plus $h_{i\to j}^\nu$),
rather than $\cO(d^2)$. In the isotropic choice $Q_i\propto  I_d$,   (and scalar initializations $\{H_{i\to j}^0\propto  I_d\}$),
even the curvature reduces to a single scalar per message.

 When $\gamma$ is small (the typical regime), the consensus regularizer $\frac{1-W_{ii}}{2\gamma}\|x_i\|^2$ dominates the curvature of $f_i$, so the partial linearization surrogate $\tilde \phi_i$ in  (\ref{example:partial_linear}) is expected to be  an accurate model of $f_i$. Consequently, the surrogate scheme can retain near-exact convergence behavior
  than the algorithm using no surrogation while substantially reducing communication: with quadratic parametrization, each message is encoded by two vectors (diagonal entries and a linear term) rather than a dense matrix plus a vector (or, for non quadratic $f_i$'s, a full function). Finally, Assumption~\ref{assumption:surrogation} is ensured choosing each $Q_i-q_i I\in \mathbb{S}_+^d$, for sufficiency large   $q_i>0.$   
  \vspace{-0.6cm}

\subsection{Convergence analysis: strongly convex objective }  
\vspace{-0.2cm}


We are ready to present the convergence results of Algorithm \ref{alg:main1_surrogate}; we consider for simplicity only the case of homogeneous stepsize and refer to the  Appendix~\ref{proof_th_surrogate} for the case of heterogeneous stepsize values.

\begin{theorem}[uniform stepsize]
\label{thm:convergence_scvx_surrogate_uniform}
Suppose Assumptions~\ref{asm:on_the_partion}, \ref{asm:graph}, \ref{asm:nonverlap}, \ref{assumption:scvx_smoothness}, and~\ref{assumption:surrogation} hold,
with $\mu>0$ and $\tilde\mu_r>0$, for all $r\in[p]$. Let $\{\bx^\nu\}$ be generated by Algorithm~\ref{alg:main1_surrogate}.   
Choose  $\tau_r^\nu\equiv\tau$ such that \vspace{-0.1cm}
\begin{equation}\label{eq:delay_lemma_condition_2_surrogate_clean4}
\tau\ \le\
\min\left\{
\frac{1}{p},\ 
\frac{2\tilde\kappa}{2D+1},\
\sqrt{\frac{\min_{r\in\cJ\ \cup\ \{s\in[p]:\,|\cC_s|>1\}}\tilde\mu_r}{8(2D+1)\, \qty(A_{\cJ}+\max_{r:\,|\cC_r|>1}\tilde A_r) }}
\right\},
\end{equation}
where   $A_{\cJ}$ is defined as in (\ref{eq:def_A}), with $A_r$ therein now given by 
\[
A_r:=\frac{(2\tilde L_r+\tilde\mu_r)\,\tilde L_{\partial r}^2\,|\cC_r|\,D_r}{4\,\tilde\mu_r^2};\] and 
$$\tilde A_r:=\frac{(2\tilde L_r+\tilde\mu_r)\,\tilde\ell_r^2\,|\cC_r|\,D_r}{4\,\tilde\mu_r^2}\,
\Big(\max_{i\in\cC_r}\deg_{\cG_r}(i)\Big).$$
Then,
\[
\Phi(\bx^{\nu})-\Phi^{\star}\leq c\,\rho^\nu,\qquad \forall\,\nu\in\mathbb{N},
\]
for some universal constant $c\in(0,\infty)$, where
\begin{equation}\label{eq:rate-constant-step-surrogate_clean4}
\rho=1-\frac{\tau}{2\tilde\kappa}.
\end{equation}
\end{theorem}
\begin{proof}
    See Appendix~\ref{proof_th_surrogate}.\hfill $\square$
\end{proof}

 The following remarks are in order.  Theorem~\ref{thm:convergence_scvx_surrogate_uniform} parallels Theorem~\ref{thm:convergence_scvx}: linear convergence is guaranteed, with a  contraction factor $\rho$ now   given by (\ref{eq:rate-constant-step-surrogate_clean4}),    depending on the surrogate condition number  $\tilde\kappa$. Relative to the exact case (Theorem~\ref{thm:convergence_scvx}), the stepsize restrictions in
\eqref{eq:delay_lemma_condition_2_surrogate_clean4} retain the same three-way structure as in
\eqref{eq:rate_expression}:
\texttt{(I)} the block-count limitation $1/p$ (Jacobi-type effect),
\texttt{(II)} the delay/diameter limitation $(2D+1)^{-1}$,
and \texttt{(III)} an external-coupling limitation. 

The main differences are as follows. 
\textbf{(i)} {\it larger condition number.}  The constants $(L_r,\mu_r,L_{\partial r})$ are replaced by their surrogate counterparts
$(\tilde L_r,\tilde\mu_r,\tilde L_{\partial r})$ in $A_r$ and in the admissible range of $\tau$.
Since majorization typically inflates curvature ($\tilde L\ge L$), one may have
$\tilde\kappa\ge\kappa$, hence a slower rate for a given stepsize--the expected price for cheaper
local computation/communication. \textbf{(ii)}{\it A new delay channel   through edge-reference sensitivity.} The denominator in the \texttt{(III)}-type bound now involves
$\max\{A_{\cJ},\,\max_{r:\,|\cC_r|>1}\tilde A_r\}$.
The term $A_{\cJ}$ is the same external-neighborhood coupling aggregate as in the exact case,  
whereas $\tilde A_r$ is \emph{new} and quantifies the impact of stale \emph{intra-cluster edge references}.
This effect is governed by the reference-sensitivity constant $\tilde\ell_r$
in Assumption~\ref{assumption:surrogation}.(iv): tighter edge surrogates (smaller $\tilde\ell_r$) reduce
$\tilde A_r$ and enlarge the admissible stepsize region.
 \textbf{(iii)} \emph{Partition design must control two coupling measures.}
As in the exact case, good partitions balance a small number of clusters $p$ against moderate diameters $D_r$.
In the surrogate setting, one must additionally keep both $A_{\cJ}$ and $\max_{r:\,|\cC_r|>1}\tilde A_r$
moderate; the latter favors surrogate designs with limited edge-reference sensitivity and/or small
non-singleton diameters $D_r$. An examples is briefly discussed next. \vspace{-0.2cm}

 \paragraph{Example (interpreting $\tilde\ell_r$).} For quadratic edge surrogates (Sec.~\ref{subsec:surrogate_example}(ii)) and $C^2$ couplings $\psi_{jk}$,
$\tilde\ell_r$ is controlled by   the mismatch between the surrogate cross-curvature $M_{jk}$ and the true
cross-Hessian of $\psi_{jk}$ along the reference points: indeed,  
\[
\nabla_{(u,u^\prime)}\,\nabla_{(x_j,x_k)}\,\widetilde\psi_{jk}\big((x_j,x_k);(u,u^\prime)\big)
=
\nabla^2\psi_{jk}(u,u^\prime)-M_{jk},\quad   (j,k)\in\cE_r.
\]
Hence choosing $M_{jk}\approx \nabla^2\psi_{jk}(x_j^\nu,x_k^\nu)$ (in operator norm, along the iterates)
yields a smaller $\tilde\ell_r$, hence reduces $\tilde A_r$. This leads to   a milder additional delay penalty in the stepsize condition.
\vspace{-0.3cm}

\subsection{Convergence analysis: nonstrongly convex objective }  \vspace{-0.1cm}
  When $\Phi$ is not strongly convex (and possibly nonconvex), we employ surrogation to ensure that every local subproblem is well-defined. For merely convex objectives,  
to derive a bound in terms of the \emph{function value gap} $\Phi(\bx^\nu)-\Phi(\bx^\star)$, we need to relate this gap to a squared-distance term involving $\bx^\nu$ and $\bx^\star\in\argmin_{\bx}\Phi(\bx)$.
This requires a global smoothness property for a   \emph{full aggregated} surrogate objective, so that $\Phi(\bx^\nu)$ can be related to $\Phi(\bx^\star)$ through  a standard descent inequality.

To this end, we first collect the terms in $\Phi$ that are \emph{independent} of the block variable $x_{\cC_r}$. For a given cluster $\cC_r$ and any $\bx\in\R^{md}$, define the remainder
\begin{equation}   
R_r(x_{\overline{\cC}_r}):=\sum_{i\in\overline{\cC}_r}\phi_i(x_i)+\sum_{i,j\in\overline{\cC}_r,(i,j)\in\cE}\psi_{ij}(x_i,x_j).
\end{equation}
Then, for a given $\bx$ and $\zeta_r=(y_{\cC_r},y_{\overline{\cC}_r},y_{\cE_r})$, we introduce the \emph{full aggregated surrogate} associated with $\cC_r$:  
\begin{equation}   \label{def:entire_aggregate_surrogate}
\widetilde\Phi_r^{\rm all}\qty(\bx;\zeta_r):=\widetilde\Phi_r(\bx;\zeta_r)+R_r(x_{\overline{\cC}_r}).
\end{equation}
By construction, $\widetilde\Phi_r^{\rm all}$ consists of   all surrogate terms associated with the  subgraph $\cG_r$ and   all other original objective terms that are independent of  $x_{\cC_r}$.
\begin{assumption}\label{assumption:entire_surroagte}
For each $r\in[p]$, the function $\widetilde\Phi_r^{\rm all}$ is  $\bar L_r$-smooth  with respect to its full argument $\bz=(\bx,\zeta_r)$, i.e., for all $\bz$ and $\bz^\prime$,
\[
\big\|\nabla\widetilde\Phi_r^{\rm all}(\bz)-\nabla\widetilde\Phi_r^{\rm all}(\bz')\big\|
\le \bar L_r\,\|\bz-\bz'\|.
\]
\end{assumption}


\begin{theorem}\label{thm:convergence_cvx}
Suppose Assumptions~\ref{asm:on_the_partion}, \ref{asm:graph}, \ref{asm:nonverlap},
\ref{assumption:scvx_smoothness}, \ref{assumption:surrogation}, and
\ref{assumption:entire_surroagte} hold, with $\mu=0$ and $\tilde\mu>0$.
Let $\{\bx^\nu\}$ be generated by Algorithm~\ref{alg:main1_surrogate} with uniform stepsize
$\tau_r^\nu\equiv\tau$ such that 
\[
\tau \ \le\ \min\left\{
\frac{1}{p},\ 
\sqrt{\frac{\widetilde\mu}{16(D+1)\big(A_\cJ+\max_{r:\,|\cC_r|>1}\tilde A_r\big)}},\ 
\frac{\tilde L_{\min}}{\big(L+(\sigma+1)(D+1)\big)\max_{r\in[p]}|\cC_r|}
\right\},
\]
where $\widetilde\mu=\min_{r\in[p]}\widetilde\mu_r$, $\tilde L_{\min}=\min_{r\in[p]}\tilde L_r$,
$L=\max_{r\in[p]}L_r$, $\sigma=\max_{r\in[p]}\sigma_r$, and $\sigma_r:=\max_{i\in\cC_r}\deg_{\cG_r}(i)$, $A_\cJ$, and $\tilde A_r$ are defined as in
Theorem~\ref{thm:convergence_scvx_surrogate_uniform}.
For any 
\[
\nu\ \ge\ \frac{8\big(\tilde L-\tilde\mu/\max_{r\in[p]}|\cC_r|\big)}{\tilde\mu}
\;+\;
\frac{pK}{A_\cJ+\max_{r:\,|\cC_r|>1}\tilde A_r},
\]
with 
\[
K:=\max_{r\in[p]}\left\{
\frac{\bar L_r(\sigma_r+1)(2D_r+1)}{2}
+\frac{|\cC_r|^2D_r\big(\widetilde L_{\partial r}^2+\sigma_r\tilde\ell_r^2\big)}{2\widetilde\mu_r^2}
\right\},
\]
 it holds that
\[
\Phi(\bx^\nu)-\Phi^\star
\ \le\frac{1}{\nu}\left(
 {\Phi(\bx^0)-\Phi^\star
+\frac{\widetilde L_{\min}}{2\tau\max_{r\in[p]}|\cC_r|}\,\|\bx^0-\bx^\star\|^2}\right).
\]   
\end{theorem}

\begin{proof}
See Appendix \ref{proof_convergence_cvx}. \qed
\end{proof}

\begin{theorem}\label{thm:convergence_ncvx}
Suppose Assumptions~\ref{asm:on_the_partion}, \ref{asm:graph}, \ref{asm:nonverlap},
\ref{assumption:scvx_smoothness}.(ii), and \ref{assumption:surrogation} hold.
Let $\{\bx^\nu\}$ be generated by Algorithm~\ref{alg:main1_surrogate} with uniform stepsize
$\tau_r^\nu\equiv\tau$ satisfying
\[
\tau\leq\min\left\{\frac{1}{p},\ \sqrt{\frac{\widetilde\mu}{8D\big(A_\cJ+\max_{r:\,|\cC_r|>1}\tilde A_r\big)}}\right\}.
\]
Then
\begin{equation}\label{eq:sublinear_rate_cvx_ncvx}
\min_{\ell\in\{D,\dots,\nu+D-1\}}\norm{\nabla\Phi(\bx^\ell)}^2
\ \le\
\frac{4\tilde L\big(\Phi(\bx^0)-\Phi^\star\big)}{\tau\,\nu},
\end{equation}
where $\widetilde\mu$, $\tilde L$, $A_\cJ$, and $\tilde A_r$ are defined as in
Theorem~\ref{thm:convergence_cvx}.
\end{theorem}

\begin{proof}
See Appendix \ref{proof_convergence_ncvx}. \qed
\end{proof}

 Unlike classical block-Jacobi and min-sum message passing, which typically rely on strong convexity for convergence, Theorems~\ref{thm:convergence_cvx} and~\ref{thm:convergence_ncvx} show that Algorithm~\ref{alg:main1_surrogate} is well posed and achieves sublinear convergence in both convex and nonconvex settings. This robustness stems from the use of \emph{strongly convex aggregated surrogate} subproblems. In either regime, the stepsize conditions exhibit the same structural drivers: the block-count factor $1/p$, the delay/diameter constraint, and a bound on the strength of the inter-cluster coupling. As a consequence, the trade-offs and partition-guidance principles in Sec.~\ref{subsec:discussion_rate} extend verbatim beyond strong convexity.

\vspace{-0.4cm}

\section{Decomposition via Message Passing:  Hypergraph Problem~\eqref{prob_hyper}}\vspace{-0.2cm}
In this section we extend the proposed framework to the hypergraph formulation~\eqref{prob_hyper}. We begin introducing   basic definitions  used throughout the section.

Given the hypergraph  $\cG=(\cV,\cE)$   underlying     ~\eqref{prob_hyper}, we view each hyperedge $\omega\in\cE$ as a \emph{factor} representing the coupling term $\psi_\omega(x_\omega)$. This naturally leads to the following bipartite representation, based on the (condensed)  factor graph.
\begin{definition}[factor graph]\label{def-factor-graph}
The (hypergraph) factor graph associated with $\cG=(\cV,\cE)$ is the bipartite graph
$\cF(\cG):=(\cV\cup\cE,\ \cA)$ with variable nodes $\cV$, factor nodes $\cE$, and incidence edges
\(
\cA:=\{(i,\omega)\in\cV\times\cE:\ i\in\omega\}.
\) 
For a sub-hypergraph $\cG_r=(\cC_r,\cE_r)$, we denote by $\cF_r:=\cF(\cG_r)$ the induced factor graph, i.e.,
$\cF_r=(\cC_r\cup\cE_r,\ \cA_r)$ with $\cA_r:=\{(i,\omega)\in  \cC_r\times \cE_r\,:\,  i\in\omega\}$.
\end{definition}
We next introduce the hypergraph analogue of the condensed graph (Def.~\ref{def:condensed_graph}).
As in the pairwise setting, we adopt the edge-maximality convention    at the level of clusters:
each cluster contains \emph{all} factors supported entirely on its variables.

\begin{definition}[condensed factor graph]\label{def:condensed_factor_graph}
Given $\cG=(\cV,\cE)$ and $p\in[m]$, let $\cC_1,\ldots,\cC_p$ be a partition of $\cV$, with associated intra-cluster   hyperedge sets
\begin{equation}\label{eq:nonoverlap_hyper}
\cE_r:=\{\omega\in\cE:\ \omega\subseteq \cC_r\},\qquad r\in[p]; 
\end{equation}
this results in the sub-hypergraphs $\cG_r:=(\cC_r,\cE_r)$, with induced factor graphs
$\cF_r:=\cF(\cG_r)$.  
Let  
$\cE_{\rm in}:=\bigcup_{r\in[p]}\cE_r$ and $\cE_{\rm out}:=\cE\setminus \cE_{\rm in}$  denote the \emph{intra-cluster}  and \emph{inter-cluster} factors, respectively.

The \emph{condensed factor graph} relative to $\{\cF_r\}_{r=1}^p$ is the bipartite graph
$\cF_{\cC}:=(\cV_{\cC}\cup \cE_{\rm out},\ \cA_{\cC})$ where
\begin{itemize}
\item $\cV_{\cC}:=\{\texttt{c}_1,\ldots,\texttt{c}_p\}$ is the set of \emph{supernodes}, with $\texttt{c}_r$
associated with $\cC_r$; and
\item $
\cA_{\cC}:=\{(\texttt{c}_r,\omega)\in \cV_{\cC}\times \cE_{\rm out}:\ \omega\cap \cC_r\neq\emptyset\}$ is the set of
incidence edges between supernodes and cross-cluster factors.
\end{itemize}
\end{definition}

\begin{figure*}[t]
    \centering
    \setlength{\tabcolsep}{6pt}
    \renewcommand{\arraystretch}{1.0}
    \begin{tabular}{cc}
        \begin{subfigure}{0.5\linewidth}
            \centering
            \includegraphics[width=\linewidth]{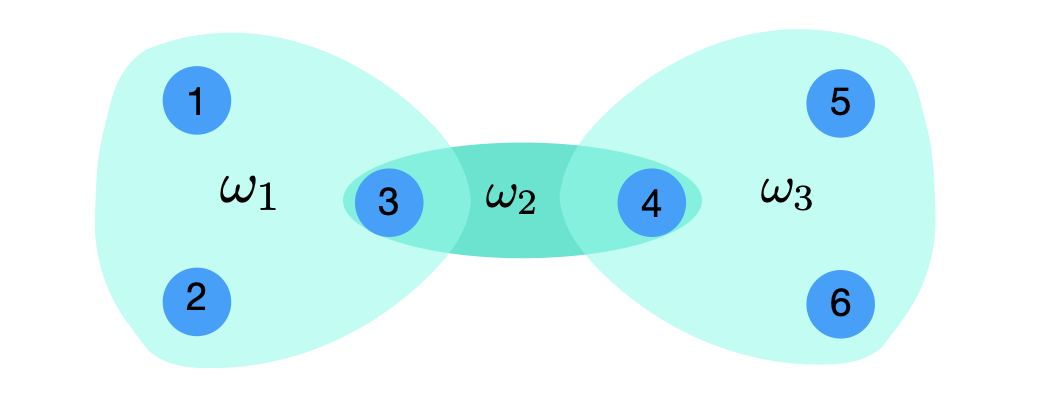}\vspace{-0.3cm}
            \caption{\,Hypergraph $\cG$.}
            \label{fig:sub-a}
        \end{subfigure}
        &
        \begin{subfigure}{0.5\linewidth}
            \centering
            \includegraphics[width=\linewidth]{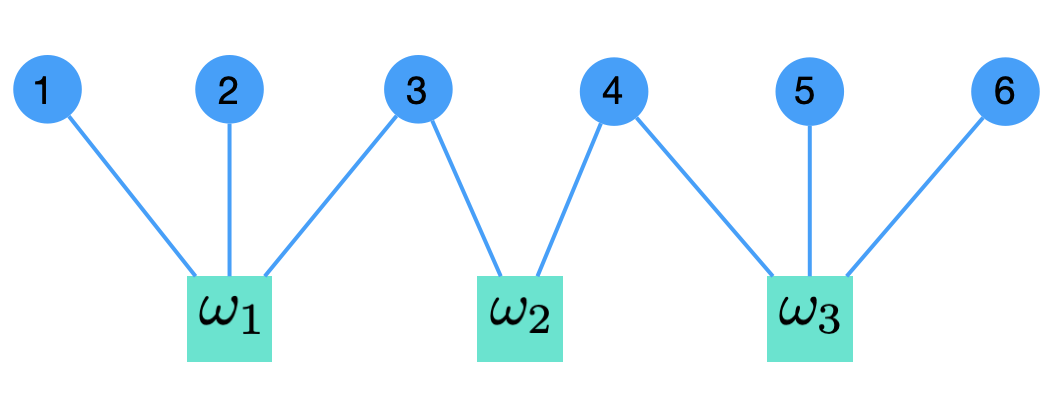}
            \caption{\,Factor graph $\cF(\cG)$.}
            \label{fig:sub-b}
        \end{subfigure}
        \\
        \begin{subfigure}{0.5\linewidth}
            \centering
            \includegraphics[width=\linewidth]{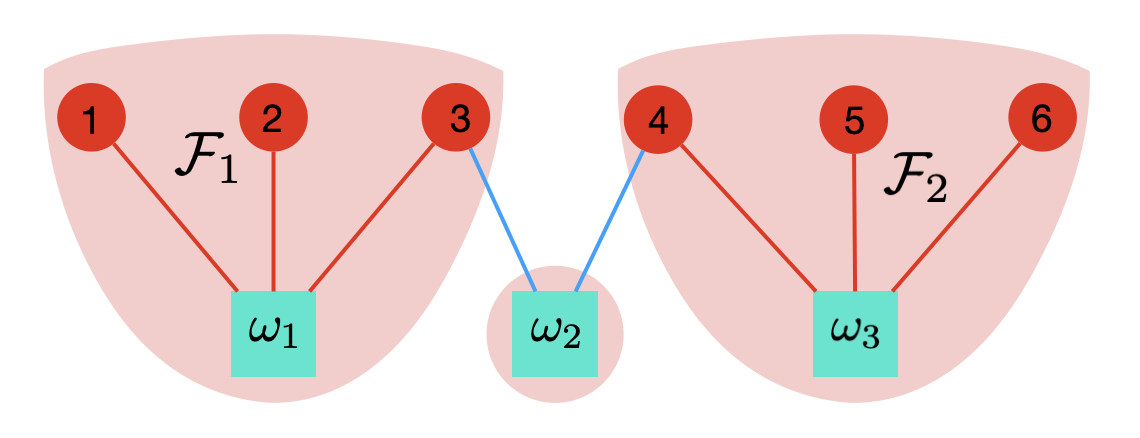}
            \caption{\,Partition of $\cF(\cG)$.}
            \label{fig:sub-c}
        \end{subfigure}
        &
        \begin{subfigure}{0.5\linewidth}
            \centering
            \includegraphics[width=\linewidth]{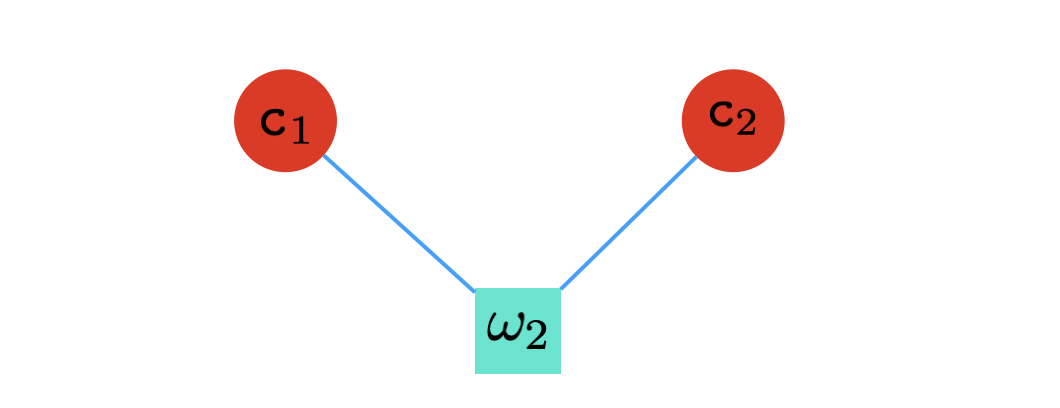}
            \caption{\,Condensed factor graph $\cF_{\cC}$.}
            \label{fig:sub-d}
        \end{subfigure}
    \end{tabular}
    \caption{Hypergraph-to-factor-graph representation and clustering:
   (a)  hypergraph $\cG$;   (b)    bipartite factor graph $\cF(\cG)$.
    (c) example of  partition of $\cF(\cG)$ into   factor trees $\cF_r$;   (d)    condensed factor graph  corresponding to the partition in (c).}
    \label{fig:hyper_factor_graph}\vspace{-0.2cm}
\end{figure*}
\noindent Fig.~\ref{fig:hyper_factor_graph} illustrates the passage from a hypergraph to its factor graph and to a condensed factor graph under some  cluster partitions.

\textit{Neighborhoods \& factor-graph distance:}  Neighbors   are   \emph{incident factors}.   For $i\in\cC_r$, define (to simplify notation, we keep the same symbols as in the pairwise case) 
\[
\cN_i:=\{\omega\in\cE:\ i\in\omega\},\qquad
\cNin_i:=\cN_i\cap\cE_r,\qquad
\cNout_i:=\cN_i\setminus\cNin_i;
\]
and let $\cN_{\cC_r}:=\bigcup_{i\in\cC_r}\cNout_i$. Note that $\cNin_i$ corresponds to intra-cluster factor nodes adjacent to $i$ in  $\cF_r$,
whereas $\cNout_i$ collects the inter-cluster factors incident to $i$ (i.e., factors adjacent to the
supernode $\texttt{c}_r$ in $\cF_{\cC}$ through $i$).  
  When $\cF_r$ is connected, we denote by $D_r:={\rm diam}(\cF_r)$ its diameter, and let $D:=\max_{r\in[p]}D_r$.  A \emph{factor-graph path} from $u$ to $v$ in $\cF(\cG)$ ($u,v$ may be variable or factor nodes) is a finite sequence of
\emph{distinct} nodes $(z_0,\ldots,z_\ell)\subseteq \cV\cup\cE$ such that $z_0=u$, $z_\ell=v$, and $(z_{t-1},z_t)\in\cA$, for all $t=1,\ldots,\ell$.
The \emph{factor-graph distance} is
\[
\dist_{\cF}(u,v):=
\begin{cases}
0, & u=v,\\[1mm]
\min\{\ell\ge 1:\ \exists\ \text{a factor-graph path of length }\ell\text{ from }u\text{ to }v\}, & u\neq v.
\end{cases}
\]
For variable nodes $i,j\in\cV$, define the induced variable-to-variable distance
$
d(i,j):=(1/2)\,\dist_{\cF}(i,j),
$
and for $i\in\cV$ and $\omega\in\cE$, define the variable-to-factor distance
$
d(i,\omega):=(1/2)\big(\dist_{\cF}(i,\omega)-1\big).
$

As for  the pairwise setting, we focus on   hypertree partition clusters.

\begin{definition}[Hypertree]\label{def:hypertree_factor}
A hypergraph $\cG=(\cV,\cE)$ is said to be a \emph{(hyper)tree} if its factor graph $\cF(\cG)$ is connected and acyclic
(equivalently, $\cF(\cG)$ is a tree).
\end{definition}

Notice that the above  notion of a hypertree   coincides with
\emph{Berge-acyclicity} \cite{bretto2013hypergraph}. 
 
\begin{assumption}[tree-clusters]\label{asm:on_the_partion_hyper}
Let $\{\cG_r\}_{r=1}^p$ be a partition satisfying Def. \ref{def:condensed_factor_graph};  every  induced factor graph $\cF_r$ is a tree (Singleton clusters $\cE_r=\emptyset$ allowed).  
\end{assumption}
Assumption~\ref{asm:on_the_partion_hyper} requires that, within each subgraph $\cG_r$, the variable--factor incidence
structure is acyclic--see Fig.~\ref{fig:hyper_factor_graph}.(b). In Sec.~\ref{sec:acyclic factor} we leverage this tree structure  to construct   min-sum message passing within  $\cF_r$ yielding   finite-round  cost-to-go
messages for the cluster subproblem. In Sec.~\ref{subsec:hyperedge_splitting} we will remove
Assumption~\ref{asm:on_the_partion_hyper} and introduce a   hyperedge-splitting
construction to handle cycle-rich hypergraphs where such   tree partitions
may not exist  or yield too large clusters/hyperedges.

 \vspace{-0.3cm}

\subsection{Algorithm design for acyclic factor graph-partitioning}\label{sec:acyclic factor} \vspace{-0.1cm}
Fix a hypertree partition  $\{\mathcal{G}_r=(\cC_r,\cE_r)\}_{r=1}^p$. In analogy with the pairwise case, we separate
inter- and intra-cluster interactions, but now in the  factor-graph  sense.
Specifically, 
we confine 
min-sum message passing to the intra-cluster
factor trees $\cF_r$, so that   couplings  {\it within}   clusters are aggregated into local cost-to-go
messages entering the agents' subproblems.  Couplings \emph{across} clusters are mediated through the {\it condensed factor graph} $\cF_{\cC}$: each inter-cluster factor $\omega\in\cE_{\rm out}$  is  incident to  the supernodes $\texttt{c}_r$ it intersects through the 
incidences $(\texttt{c}_r,\omega)\in\cA_{\cC}$, and its contribution is handled in a Jacobi fashion using the
most recently  out-of-cluster variables.    

More formally,  for any intra-cluster incidence $(i,\omega)\in\cA_r$ (i.e., $\omega\in\cNin_i$), define variable-to-factor and
factor-to-variable messages by
\begin{subequations}\label{eq:minsum_hyper_messages_fp}
\begin{align}
\mu_{i\to \omega}^\star(x_i)
&:=\phi_i(x_i)
+\sum_{\omega'\in\cNin_i\setminus\{\omega\}}\mu_{\omega'\to i}^\star(x_i)
+\sum_{\omega'\in\cNout_i}\psi_{\omega'}\big(x_i,\,x_{\omega'\setminus \{i\}}^{\star}\big),
\label{eq:message_i_to_omega_fp}\\
\mu_{\omega\to i}^\star(x_i)
&:=\min_{x_{\omega\setminus i}}
\Big\{\psi_{\omega}\big(x_i,x_{\omega\setminus \{i\}}\big)
+\sum_{j\in\omega\setminus\{i\}}\mu_{j\to\omega}^\star(x_j)\Big\},
\qquad \omega\in\cNin_i.
\label{eq:message_omega_to_i_fp}
\end{align}
\end{subequations}
Then, 
any  solution \(\mathbf{x}^\star\) of \eqref{P}  satisfies the cluster-wise optimality decomposition: 
\begin{equation}
\label{eq:fixpoint2_hypergraph}
x_i^{\star}\in\argmin_{x_i} \phi_i(x_i) + \underbrace{\sum_{\omega\in \cNin_i} \mu_{\omega\rightarrow i}^\star(x_i)}_{:= \mu^{\star}_{\cNin_i\to i}(x_i)}+\underbrace{\sum_{\omega\in\cNout_i}\psi_{\omega}\big(x_i,x_{\omega\setminus \{i\}}^{\star}\big)}_{:= \mu^{\star}_{\cNout_i\to i}(x_i)},\quad i\in \mathcal{C}_r,\,\,r\in[p].
\end{equation}
Here, $\mu^{\star}_{\cNin_i\to i}$ is the   (factor-to-variable) \emph{cost-to-go} term induced by min-sum
on the tree factor graph $\cF_r$,    summarizing all interactions {\it within} $\cC_r$ as seen
from node $i$. 
In contrast, $\mu^{\star}_{\cNout_i\to i}$ is the \emph{external coupling contribution} collecting all
cross-cluster factors $\psi_{\omega}(x_i,x_{\omega\setminus \{i\}}^{\star})$ incident to $i$.

We now eliminate the explicit variable$\to$factor messages by substituting
\eqref{eq:message_i_to_omega_fp} (with $i$ replaced by $j$) into \eqref{eq:message_omega_to_i_fp}. This gives,
for every $\omega\in\cE_r$ and $i\in\omega$,
\begin{align}\label{eq:fac2var_fp_abs}
 &\mu_{\omega\to i}^{\star}(x_i)
=\min_{x_{\omega\setminus\{i\}}}
\Bigg\{
\psi_{\omega}\big(x_i,\,x_{\omega\setminus\{i\}}\big)
\nonumber \\&\quad + \sum_{j\in\omega\setminus\{i\}}
\underbrace{\Big(
\phi_j(x_j)
+\!\!\sum_{\omega'\in\cNin_j\setminus\{\omega\}}\!\!\mu_{\omega'\to j}^{\star}(x_j)
+\!\!\sum_{\omega'\in\cNout_j}\!\!\psi_{\omega'}\big(x_j,\,x_{\omega'\setminus\{j\}}^{\star}\big)
\Big)}_{\text{aggregate variable-side cost}}
\Bigg\}.
\end{align}
Equation~\eqref{eq:fac2var_fp_abs} is nothing but the factor-to-variable min-sum update in which each neighbor
$j\in\omega\setminus\{i\}$ contributes through its \emph{aggregated variable-side cost}.

The proposed decentralized algorithm--Algorithm~\ref{alg:main_hyper}--is   obtained by iterating the fixed-point system~\eqref{eq:fixpoint2_hypergraph}-\eqref{eq:fac2var_fp_abs}  
in the variables $\{x_i^\star\}$ and intra-cluster messages $\{\mu_{\omega\to i}^\star\}_{(\omega,i)\in\cA_r}$,
using  the  same two design principles as in the
pairwise MP-Jacobi: \textbf{(i)}  the unknown out-of-cluster variables
$x_{\omega\setminus\{i\}}^\star$  are replaced by the most recently available values
$x_{\omega\setminus\{i\}}^{\nu}$ (Jacobi
correction); and   \textbf{(ii)}  the  unknown  fixed-point messages $\mu_{\omega'\to j}^{\star}$ in 
\eqref{eq:fac2var_fp_abs} are replaced by the current iterates   $\mu_{\omega'\to j}^{\nu}$, yielding one min-sum step   within each   intra-cluster factor
tree $\cF_r$. \vspace{-0.3cm}

\begin{algorithm}[th]
{
\scriptsize
\caption{ \underline{H}ypergraph-\underline{M}essage \underline{P}assing-\underline{J}acobi ({H-MP-Jacobi})}
\label{alg:main_hyper}
{\bf Initialization:} $x_i^{0}\in\R^d$ for all $i\in\cV$; initialize $\mu_{\omega\to i}^{0}(\cdot)$ and $\mu_{i\to\omega}^{0}(\cdot)$
(e.g., $\equiv 0$) for all $i\in\omega$ and $\omega\in\cE_r$.

\For{$\nu=0,1,2,\ldots$}{\BlankLine
\tcp*[l]{(1) Jacobi-style variable update (parallel over $i\in\cV$)}
\AgentFor{$i\in\cV$}{\label{line:var-loop-hyper}

\begin{subequations}\label{eq:hyper_jump_updates}
\begin{align}
\hat x_i^{\nu+1}
&\in \argmin_{x_i}\Big\{
\phi_i(x_i)
+ \sum_{\omega\in\cNin_i}\mu_{\omega\to i}^{\nu}(x_i)
+ \sum_{\omega\in\cNout_i}\psi_{\omega}\big(x_i,\,x_{\omega\setminus\{i\} }^{\nu}\big)
\Big\},
\label{eq:hyper_jump_updates:a}\\
x_i^{\nu+1}
&=x_i^{\nu}+\tau_r^\nu\big(\hat x_i^{\nu+1}-x_i^{\nu}\big).
\label{eq:hyper_jump_updates:b}
\end{align}
\end{subequations}

\texttt{Sends} $x_i^{\nu+1}$ to all inter-cluster factors $\omega\in\cNout_i$ (or to their implementation);
} 
\BlankLine
\tcp*[l]{(2) One intra-cluster min-sum (parallel) round on each factor tree $\cF_r$}
 \ForEach{$r\in[p]$}{
\ForEach{$\omega\in\cE_r$}{
\ForEach{$i\in\omega$}{
\begin{align}\label{eq:hyper_jump_factor2var}
 &\mu_{\omega\to i}^{\nu+1}(x_i)
=\min_{x_{\omega\setminus\{i\}}}
\Bigg\{
\psi_{\omega}\big(x_i,\,x_{\omega\setminus\{i\}}\big)
\nonumber \\&\quad + \sum_{j\in\omega\setminus\{i\}}
\underbrace{\Big(
\phi_j(x_j)
+\!\!\sum_{\omega'\in\cNin_j\setminus\{\omega\}}\!\!\mu_{\omega'\to j}^{\nu}(x_j)
+\!\!\sum_{\omega'\in\cNout_j}\!\!\psi_{\omega'}\big(x_j,\,x_{\omega'\setminus\{j\}}^{\nu}\big)
\Big)}_{:=\,\mu_{j\to\omega}^{\nu}(x_j)} 
\Bigg\}.
\end{align}
\texttt{Send} $\mu_{\omega\to i}^{\nu+1}(\cdot)$ to all $i\in\omega$.
}}} 
}}
\end{algorithm}

\paragraph{Practical implementation of Step (2).}  
In standard message passing, factor nodes are  generally not  physical computing entities. To implement Step~(2) of
Algorithm~\ref{alg:main_hyper} in a  single-hop manner, we  
 propose the following two implementations,  both realizing    the same recursion
\eqref{eq:hyper_jump_factor2var}, accommodating different practical settings.  

 $\bullet$ \textbf{option I (no factor processors: hosted factors).}
If factors are \emph{logical} and only agents can compute/communicate, then each intra-cluster factor
$\omega\in\cE_r$ is assigned to a host agent $h(\omega)\in\omega$ (e.g., the smallest index in $\omega$).
Each incident agent $j\in\omega$ locally forms the \emph{variable-side aggregate}
$\mu_{j\to\omega}^{\nu}(\cdot)$ defined   in \eqref{eq:hyper_jump_factor2var}, and sends it
(one hop within the hyperedge scope) to the host $h(\omega)$.
The host emulates the factor computation by evaluating \eqref{eq:hyper_jump_factor2var} to obtain
$\{\mu_{\omega\to i}^{\nu+1}(\cdot)\}_{i\in\omega}$, and then broadcasts these messages back to all
$i\in\omega$.
This avoids any cluster-level coordinator while preserving single-hop communication within each hyperedge.\smallskip 

 $\bullet$ \textbf{option II (factor processors / edge servers).}
Assume instead that each factor $\omega$ is a computational unit (e.g., an ``edge server'') that stores
$\psi_\omega$ and can communicate with all incident agents $j\in\omega$.
Then Step~(2) is executed directly at the factor:
each agent $j\in\omega$ sends to $\omega$ the same local quantity $\mu_{j\to\omega}^{\nu}(\cdot)$.
Upon collecting $\{\mu_{j\to\omega}^{\nu}(\cdot)\}_{j\in\omega}$, the factor evaluates
\eqref{eq:hyper_jump_factor2var} to compute $\{\mu_{\omega\to i}^{\nu+1}(\cdot)\}_{i\in\omega}$ and
broadcasts them back to the incident agents.
In this model, no hosted-factor mechanism is needed: the factor nodes perform the partial minimizations
natively.

When all   $|\omega|=2$ (pairwise case), identifying $\mu_{i\to j}(\cdot)\equiv \mu_{\{i,j\}\to j}(\cdot)$ reduce both implementation models above   to the edge-message updates in Algorithm~\ref{alg:main}.

 We remark that, as in Sec.~\ref{sec:surrogate_pairwise}, surrogate constructions can be incorporated to reduce \emph{both} computation and communication costs. Concretely, one may replace the exact intra-cluster min-sum updates in Step~(2)--i.e., the factor-to-variable recursion~\eqref{eq:hyper_jump_factor2var} (and, in implementations that materialize them, the corresponding variable-to-factor aggregates implicit in~\eqref{eq:hyper_jump_factor2var})--with surrogate message maps that admit \emph{finite-dimensional} parametrizations. In this way, each incidence $(i,\omega)$ exchanges only a lightweight summary (e.g., vectors or structured statistics such as sparse/low-rank matrices) instead of a functional message. When further savings are needed, the same surrogate principle can be applied to the Jacobi subproblem~\eqref{eq:hyper_jump_updates:a}, yielding an inexact/local-model update while preserving the overall decentralized architecture and the cost-rate trade-offs discussed in Sec.~\ref{sec:surrogate_pairwise}.
 
 \vspace{-0.3cm}

 \subsubsection{Convergence Analysis} \vspace{-0.2cm}

We establish convergence of Algorithm \ref{alg:main_hyper} under the following assumptions. 

 \begin{assumption}
 \label{asm:graph_hyper}
 The factor graph $\cF(\cG)$ associated with   $\cG=(\cV,\cE)$ is connected.
Moreover, communications are \emph{incidence-local}: for every   $(i,\omega)$ with $i\in\omega$,
agent $i$ can exchange information with the computing entity implementing factor $\omega$.
  
 \end{assumption}
  
\begin{assumption}
\label{asm:nonverlap_hyper} Given the partition $\{\cG_r=(\cC_r,\cE_r)\}_{r=1}^p$  satisfying Def.~\ref{def:condensed_factor_graph}, 
every cross-cluster factor intersects any cluster in at most one node, i.e.,\vspace{-0.1cm}\begin{equation}\label{eq:nonoverlap2_hyper}
|\omega\cap \cC_r|\le 1,\qquad \forall\,\omega\in\cE_{\rm out},\ \forall\,r\in[p].\vspace{-0.1cm}
\end{equation}
\end{assumption}
Assumption~\ref{asm:nonverlap_hyper} is not essential for convergence, but it simplifies the delayed
block-Jacobi reformulation, which is given below (the proof is omitted). 
\begin{proposition}
\label{prop:Jacobi_delay_refo_hyper}
Postulate  Assumptions~\ref{asm:on_the_partion_hyper} and~\ref{asm:graph_hyper}. Then,  Algorithm~\ref{alg:main_hyper} can be rewritten in the equivalent form: for any $i\in\cC_r$ and $r\in [p]$,  
\begin{subequations}
    \begin{align}
 &  x_i^{\nu+1} =x_i^{\nu}+\tau_r^\nu \big(\hat x_i^{\nu+1}-x_i^{\nu}\big),\label{eq:delay_reformulation-cvx-comb_hyper}\\
   & \hat x_i^{\nu+1}\in\argmin_{x_i}\min_{x_{\cC_r\setminus\{i\}}}\qty{\sum\limits_{i\in\cC_r}\phi_i(x_i)+\sum\limits_{\omega\in\cE_r}\psi_{\omega}(x_\omega)+\sum\limits_{j\in\cC_r,\omega\in\cNout_j}\psi_{\omega}\qty(x_j,x_{\omega\setminus \{j\}}^{\nu-d(i,j)})}.
\label{eq:delay_reformulation_hyper}
\end{align}\end{subequations} 
If, in addition, Assumption~\ref{asm:nonverlap_hyper} holds,  \eqref{eq:delay_reformulation_hyper}  reduces to the following  block-Jacobi update with delays:\vspace{-0.1cm}
\begin{equation}
\label{eq:coordinate_min_delay_hyper}
\hat x_i^{\nu+1}
\in\argmin_{x_i}\ \min_{x_{\cC_r\setminus\{i\}}}\,
\Phi\bigl(x_{\cC_r},\,x_{\overline{\cC}_r}^{\,\nu-\bd_i}\bigr),\vspace{-0.1cm}
\end{equation}
where $\bd_i\in\mathbb{N}^{|\overline{\cC}_r|}$ is defined as \vspace{-0.2cm}
\[
(\bd_i)_k :=
\begin{cases}
d(i,j_k), & k\in \cN_{\cC_r},\\
0, & k\in \overline{\cC}_r\setminus\cN_{\cC_r},
\end{cases}
\]
and   $j_k\in\cB_r$ is the unique node such that there exists a hyperedge
$\omega\in\cNout_{j_k}$ with $k\in\omega \cap   \cN_{\cC_r}$. It holds $\|\bd_i\|_\infty\le \mathrm{diam}(\cF_r)/2$, for all $i\in\cC_r$. 

\end{proposition}

Under Assumption~\ref{asm:nonverlap_hyper}, \eqref{eq:coordinate_min_delay_hyper} reduces to a block-Jacobi scheme with bounded delays; hence, Theorem~\ref{thm:convergence_scvx} applies. We omit the (essentially identical) details.  
\vspace{-0.3cm}

\subsection{Hyperedge (factor) splitting via surrogation}
\label{subsec:hyperedge_splitting}\vspace{-0.2cm}

In this section we address \emph{cycle-rich} hypergraphs, where many factors overlap on multiple
variable-nodes. If many pairs of distinct hyperedges $\omega,\omega'\in\cE$ satisfy $|\omega\cap\omega'|\ge 2$,
then $\cF(\cG)$ is cycle-rich (with short alternating cycles), which may preclude a nontrivial
partition into multi-factor hypertree clusters; see Fig.~\ref{fig:splitting_hypergph}. 

\begin{figure*}[t]
    \centering
    \setlength{\tabcolsep}{15pt} 
    \begin{tabular}{cc}
        \includegraphics[width=0.45\linewidth]{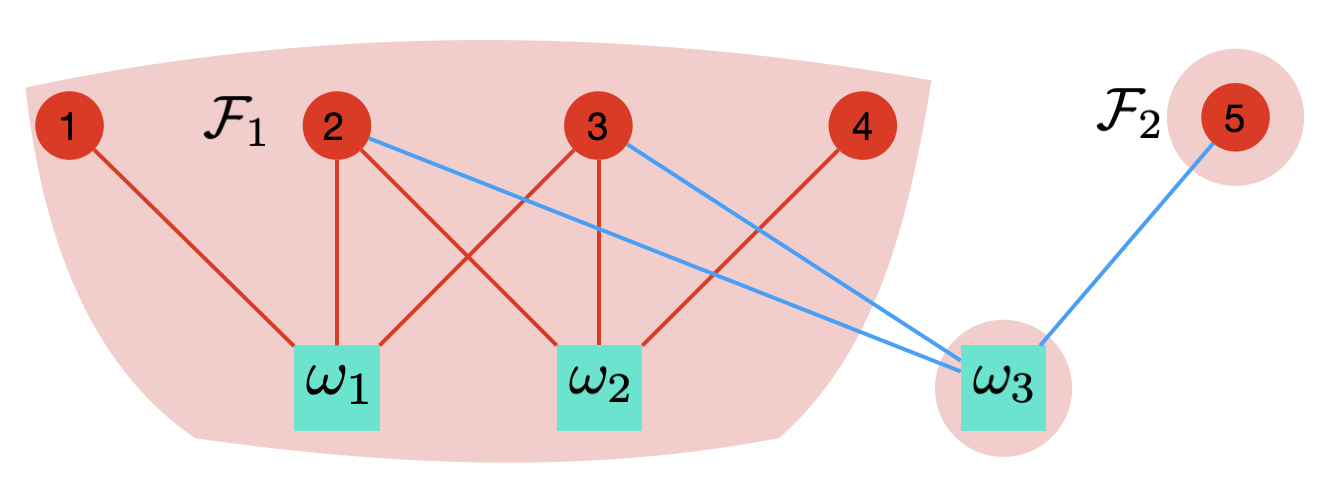} &
        \includegraphics[width=0.45\linewidth]{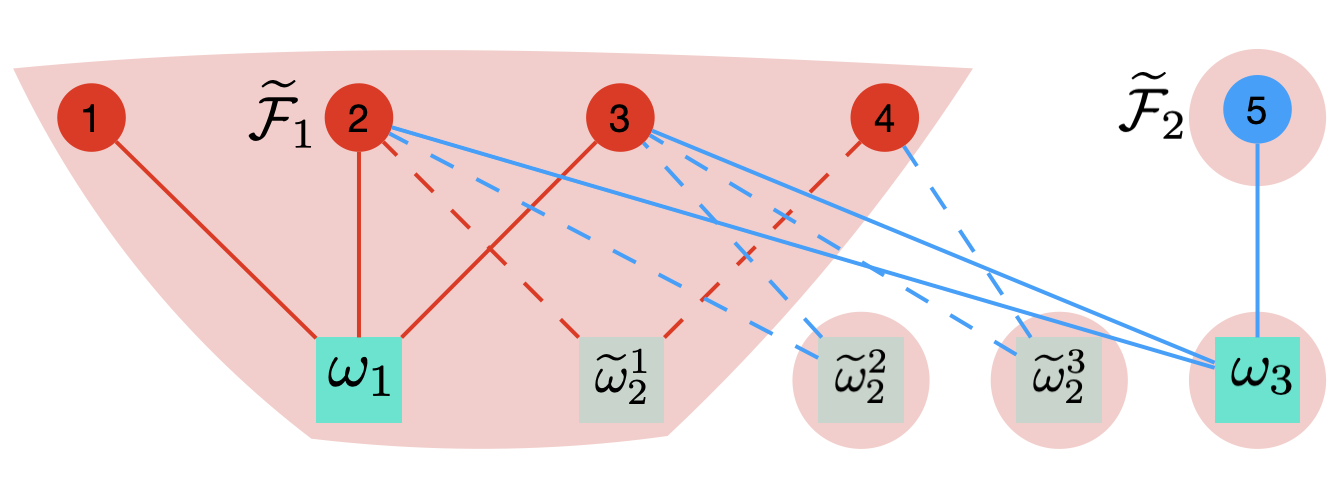} \\
        \footnotesize{(a) Cycle-rich hypergraph.} &
        \footnotesize{(b) Hyperedge splitting.}
    \end{tabular}   \caption{ 
    (a) The cluster $\cC_1=\{1,2,3,4\}$ contains $\omega_1=\{1,2,3\}$ and $\omega_2=\{2,3,4\}$ with overlap  $\omega_1\cap\omega_2=\{2,3\}$. Hence the induced factor graph $\cF_1$ contains a cycle.
    (b) Hyperedge splitting replaces $\omega_2$ with lower-arity component factors
    $\{\widetilde\omega_2^1,\widetilde\omega_2^2,\widetilde\omega_2^3\}$ (light square nodes). Keeping only $\widetilde\omega_2^1$ as an intra-cluster component yields $\widetilde\cE_1=\{\omega_1,\widetilde\omega_2^1\}$, for which the induced factor graph $\widetilde\cF_1$ becomes a tree.}
    \label{fig:splitting_hypergph}\vspace{-0.3cm}
\end{figure*}

We handle these  hypergraphs  via   \emph{hyperedge (factor) splitting} and {\it surrogation}. 
First, we perform a combinatorial splitting of hyperedges, replacing each $\omega\in\cE$ by component supports
$\cS(\omega)\subseteq 2^\omega$ and forming the expanded split set $\widetilde\cE$ (hence the split factor graph
$\cF(\widetilde\cG)$). The split is chosen so that, after selecting intra-cluster split hyperedges
$\widetilde\cE_r$, each induced split factor graph $\widetilde\cF_r=\cF(\cC_r,\widetilde\cE_r)$ is a tree.
Second, we construct \emph{surrogate split factors} consistent with this split incidence structure by replacing
each original heavily-overlapping factor $\psi_\omega(x_\omega)$ with a sum of lower-arity component factors indexed by
$\widetilde\omega\in\cS(\omega)$, where the removed variables $x_{\omega\setminus\widetilde\omega}$ are frozen at
reference values (taken as the latest iterates). This produces tree-structured intra-cluster couplings on which
min-sum terminates in finite time.
  
 More formally, the construction proceeds in the following three steps.  

\noindent $\bullet$ \textbf{Step 1: Combinatorial splitting rule.} Given the hypergraph $\cG=(\cV,\cE)$, a \emph{hyperedge-splitting rule} is  
specified by a set-valued map\vspace{-0.1cm}
\begin{equation}\label{eq:splitting_map}
\cS:\cE\rightrightarrows 2^{\cV},\qquad 
\omega\mapsto \cS(\omega)\subseteq 2^{\omega},\vspace{-0.1cm}
\end{equation}
where each $\widetilde\omega\in\cS(\omega)$ is a \emph{component hyperedge} supported on a subset of $\omega$.
Collecting all components yields the expanded factor set
\begin{equation}\label{eq:expanded_edge_set}
\widetilde\cE := \bigcup_{\omega\in\cE}\cS(\omega),\vspace{-0.1cm}
\end{equation}
and the corresponding \emph{split hypergraph} $\widetilde\cG:=(\cV,\widetilde\cE)$, with factor graph
$\cF(\widetilde\cG)$.

To avoid duplicated component supports, we make the following assumption.  
\begin{assumption}[unique-parent splitting]\label{asm:unique_parent_split}
The splitting map $\cS$ satisfies
\[
\cS(\omega)\cap \cS(\omega')=\emptyset,\qquad \forall\,\omega\neq \omega' \in \cE.
\]
\end{assumption}
Under Assumption~\ref{asm:unique_parent_split}, for every $\widetilde\omega\in\widetilde\cE$ there exists a
unique $\omega\in\cE$ such that $\widetilde\omega\in\cS(\omega)$; we denote this parent by
${\rm par}(\widetilde\omega):=\omega$. 

We remark that Assumption~\ref{asm:unique_parent_split} is not a  restriction: it  ensures solely that each component hyperedge admits a unique parent. The same effect can be obtained
without this assumption, at the cost of heavier notation, by defining the expanded split factor set as a
\emph{disjoint union} (so that each $\widetilde\omega\in\widetilde\cE$ carries its parent label).

\noindent $\bullet$ \textbf{Step 2: Hypertree clustering  of the split factor graph $\cF(\widetilde \cG)$. }
Given a node partition $\cV=\bigcup_{r=1}^p \cC_r$, we select, within each cluster, a subset of split
hyperedges to be treated intra-cluster; all remaining split hyperedges are treated as inter-cluster.
This is captured by the following analogue of Assumption~\ref{asm:on_the_partion_hyper}.

\begin{assumption}[Hypertree clusters  of $\cF(\widetilde \cG)$]\label{asm:on_the_partion_hyper_split}
Given a node partition $\cV=\bigcup_{r=1}^p \cC_r$, choose intra-cluster split hyperedges\vspace{-0.1cm}
\[
\widetilde\cE_r\subseteq\{\widetilde\omega\in\widetilde\cE:\ \widetilde\omega\subseteq \cC_r\},\qquad r\in[p],
\]
and define $\widetilde\cG_r:=(\cC_r,\widetilde\cE_r)$ with induced split factor graph
$\widetilde\cF_r:=\cF(\widetilde\cG_r)$. Each $\widetilde\cF_r$ is a tree (singleton clusters allowed).
\end{assumption}
Unlike Def.~\ref{def:condensed_factor_graph}, intra-cluster maximality is not required: within each cluster
$\cC_r$, one may keep only a subset $\widetilde\cE_r$ of the split components (possibly discarding some
$\widetilde\omega\subseteq\cC_r$) so that the induced split factor graph $\widetilde\cF_r$ is acyclic.
The discarded components are treated as inter-cluster and capture cross-cluster couplings through the
corresponding condensed factor graph (defined as in Def.~\ref{def:condensed_factor_graph}, replacing
$\cE$ by $\widetilde\cE$).

\noindent $\bullet$ \textbf{Step 3: Surrogate split factors (freezing removed coordinates).}
Splitting modifies only the \emph{supports} of factors. To obtain an optimization method for the original
problem~\eqref{prob_hyper}, we associate to each parent factor $\psi_\omega$ a family of \emph{surrogate
component factors} supported on $\{\widetilde\omega\in\cS(\omega)\}$.
The key idea is to freeze the coordinates removed by splitting at reference values.

Formally, for $\widetilde\omega\in\widetilde\cE$, let $\omega:={\rm par}(\widetilde\omega)$ and interpret
$\omega\setminus\widetilde\omega$ as the set of frozen coordinates. For a reference vector
$y_{\omega\setminus\widetilde\omega}$, introduce a component surrogate 
\[
\widetilde\psi_{\widetilde\omega}\big(x_{\widetilde\omega};\, y_{\omega\setminus\widetilde\omega}\big),\vspace{-0.1cm}
\]
which depends only on the active variables $x_{\widetilde\omega}$ and on the frozen reference values.
In the algorithm, $  y_{\omega\setminus\widetilde\omega}$ will be set to latest iterates (Jacobi-style), i.e.,
$  y_{\omega\setminus\widetilde\omega}=x_{\omega\setminus\widetilde\omega}^{\nu}$.

Collecting all components over all parents defines the split-surrogate objective on $\widetilde\cE$: for any given reference vector $\by$, let \vspace{-0.1cm}
\begin{equation}\label{eq:split_surrogate_objective_template}
\widetilde\Phi(\bx;\by)
:=\sum_{i\in\cV}\phi_i(x_i)+\sum_{\widetilde\omega\in\widetilde\cE}
\widetilde\psi_{\widetilde\omega}\Big(x_{\widetilde\omega};\,y_{{\rm par}(\widetilde\omega)\setminus\widetilde\omega}\Big).\vspace{-0.1cm}
\end{equation}



The constructions above decouples the \emph{graph design} (splitting/hypertree selection) from the \emph{analytic design} (choosing $\widetilde\psi_{\widetilde\omega}$). For consistency with~\eqref{prob_hyper}, $\widetilde\Phi(\cdot;\by)$ shall  satisfy the following regularity condition.  
\begin{assumption}
\label{assumption:splitting_regularity}
    [split-surrogate  regularity] For each parent $\omega\in\cE$ and each $i\in\omega$,  choosing references consistently,
$y_{\omega\setminus \widetilde\omega}=x_{\omega\setminus \widetilde\omega}$, yields \vspace{-0.1cm}\[
\sum_{\substack{\widetilde\omega\in\cS(\omega):\\ i\in\widetilde\omega}}
\nabla_{x_i}\widetilde\psi_{\widetilde\omega}\Big(x_{\widetilde\omega};\,x_{\omega\setminus\widetilde\omega}\Big)
=\nabla_{x_i}\psi_{\omega}(x_{\omega}),\quad \forall x_{\omega}.\vspace{-0.1cm}
\]
\end{assumption}
In fact, split-surrogate  regularity   preserves the critical points of $\Phi$ when the reference variables are chosen consistently, which underpins the algorithmic design.

\begin{proposition}\label{prop:splitting_func} Under Assumption~\ref{assumption:splitting_regularity}, any $\bar\bx$ satisfying
$\nabla_x \widetilde\Phi(\bar\bx;\bar\bx)=0$ is a critical point of $\Phi$, i.e.,
$\nabla \Phi(\bar\bx)=0$.
\end{proposition}

\noindent\textbf{Example:}
We illustrate the splitting/surrogate rule   (satisfying Assumption \ref{assumption:splitting_regularity}) on the following  concrete example. Consider
$\Phi(\bx)=\sum_{i=1}^5\phi_i(x_i)+\psi_{123}(x_1,x_2,x_3)+\psi_{234}(x_2,x_3,x_4)+\psi_{235}(x_2,x_3,x_5)$, whose corresponding factor graph is shown in Fig.~\ref{fig:splitting_hypergph}(a). We split   the factor $\omega_2=\{2,3,4\}$ and keep all other factors unchanged. Consider the following valid splitting choices for $\omega_2$:
\begin{enumerate}
\item[\textbf{(i)}] \emph{Pairwise split:} $\cS(\omega_2)= \{\{2,4\},\{2,3\},\{3,4\}\}$, and split-surrogate given by
\begin{equation}
\label{example:pair_split}
\frac{1}{2}\psi_{234}(x_3,x_4)+\frac{1}{2}\psi_{234}(x_2,y_3^{2},x_4)+\frac{1}{2}\psi_{234}(x_2,x_3,y_4^{3}).
\end{equation}

\item[\textbf{(ii)}] \emph{Two-component split:} $\cS(\omega_2)=\{\{2,3\},\{3,4\}\}$, and   split-surrogate given by
\begin{equation}
\label{example:two_split}
\Big(\psi_{234}(x_2,x_3,y_4^{1})-\frac{1}{2}\psi_{234}(y_2^{1},x_3,y_4^{1})\Big)+\Big(\psi_{234}(y_2^{2},x_3,x_4)-\frac{1}{2}\psi_{234}(y_2^{2},x_3,y_4^{2})\Big).
\end{equation}

\item[\textbf{(iii)}] \emph{Singleton split:} $\cS(\omega_2)=\{\{2\},\{3\},\{4\}\}$,  and split-surrogate given by
\begin{equation}
\label{example:sing_splt}
\psi_{234}(x_2,y_3^{1},y_4^{1})+\psi_{234}(y_2^{2},x_3,y_4^{2})+\psi_{234}(y_2^{3},y_3^{3},x_4).
\end{equation}
\end{enumerate}
Fig.~\ref{fig:splitting_hypergph}(b) depicts the pairwise split in \textbf{(i)}, where
$\cS(\omega_2)=\{\widetilde\omega_2^1,\widetilde\omega_2^2,\widetilde\omega_2^3\}$, with
$\widetilde\omega_2^1=\{2,4\}$, $\widetilde\omega_2^2=\{2,3\}$, and $\widetilde\omega_2^3=\{3,4\}$.


\vspace{-0.3cm}
\subsection{Algorithm design}\vspace{-0.1cm}

 The proposed decentralized algorithm is formally stated in  Algorithm~\ref{alg:main_hyper_split}. It  parallels Algorithm~\ref{alg:main_hyper}, with the original hypergraph replaced by the split hypergraph $\widetilde\cG=(\cV,\widetilde\cE)$ and message passing carried out only on the selected tree components $\{\widetilde\cE_r\}_{r=1}^p$. Compared with Algorithm~\ref{alg:main_hyper}, the additional arguments $x_{{\rm par}(\widetilde\omega)\setminus\widetilde\omega}^{\nu}$ in $\widetilde\psi_{{\rm par}(\widetilde\omega)}$ encode the reference coordinates introduced by splitting and are instantiated by available iterates. When $\cS$ is the identity map, i.e., $\cS(\omega)=\{\omega\}$ for all $\omega\in\cE$ (so $\widetilde\cE=\cE$ and $\widetilde\psi_{{\rm par}(\omega)}\equiv\psi_\omega$), Algorithm~\ref{alg:main_hyper_split} reduces to Algorithm~\ref{alg:main_hyper}.

\begin{algorithm}[th]
{
\scriptsize
\caption{
H-MP-Jacobi with Hyperedge Splitting}
\label{alg:main_hyper_split}
{\bf Initialization:} $x_i^{0}\in\R^d$ for all $i\in\cV$; initialize $\mu_{\widetilde\omega\to i}^{0}(\cdot)$
(e.g., $\equiv 0$) for all $r\in[p]$, $\widetilde\omega\in\widetilde\cE_r$, and $i\in\widetilde\omega$.

\For{$\nu=0,1,2,\ldots$}{\BlankLine
\tcp*[l]{(1) Jacobi-style variable update (parallel over $i\in\cV$)}
\AgentFor{$i\in\cV$}{\label{line:var-loop-hyper-split}

\begin{subequations}\label{eq:hyper_jump_split_updates}
\begin{align}
\hat x_i^{\nu+1}
&\in \argmin_{x_i}\Big\{
\phi_i(x_i)
+ \sum_{\widetilde\omega\in\tildecNin_i}\mu_{\widetilde\omega\to i}^{\nu}(x_i)
+ \sum_{\widetilde\omega\in\tildecNout_i}
\widetilde\psi_{{\rm par}(\widetilde\omega)}
\big(x_i,\,x_{\widetilde\omega\setminus \{i\}}^{\nu};\,x_{{\rm par}(\widetilde\omega)\setminus\widetilde\omega}^{\nu}\big)
\Big\},
\label{eq:hyper_jump_split_updates:a}\\
x_i^{\nu+1}
&=x_i^{\nu}+\tau_r^\nu\big(\hat x_i^{\nu+1}-x_i^{\nu}\big).
\label{eq:hyper_jump_split_updates:b}
\end{align}
\end{subequations}

\texttt{Sends} $x_i^{\nu+1}$ to all inter-cluster components $\widetilde\omega\in\tildecNout_i$
(or to their implementation);
} 

\BlankLine
\tcp*[l]{(2) One intra-cluster min-sum (parallel) round on each split factor tree $\widetilde\cF_r$}
\ForEach{$r\in[p]$}{
\ForEach{$\widetilde\omega\in\widetilde\cE_r$}{
\ForEach{$i\in\widetilde\omega$}{
\begin{align}\label{eq:hyper_jump_split_factor2var}
&\mu_{\widetilde\omega\to i}^{\nu+1}(x_i)
=\min_{x_{\widetilde\omega\setminus\{i\}}}
\left\{\widetilde\psi_{{\rm par}(\widetilde\omega)}
\big(x_i,\,x_{\widetilde\omega\setminus\{i\}};\,x_{{\rm par}(\widetilde\omega)\setminus\widetilde\omega}^{\nu}\big)
\nonumber\right.\\[-1mm]
&\left.\quad + \sum_{j\in\widetilde\omega\setminus\{i\}}
\qty(
\begin{aligned}
&\phi_j(x_j)
+\!\!\sum_{\widetilde\omega'\in\tildecNin_j\setminus\{\widetilde\omega\}}\!\!\mu_{\widetilde\omega'\to j}^{\nu}(x_j)\\
&+\!\!\sum_{\widetilde\omega'\in\tildecNout_j}\!\!
\widetilde\psi_{{\rm par}(\widetilde\omega')}
\big(x_j,\,x_{\widetilde\omega'\setminus\{j\}}^{\nu};\,x_{{\rm par}(\widetilde\omega')\setminus\widetilde\omega'}^{\nu}\big)
\end{aligned}
)
\right\}.
\end{align}
\texttt{Send} $\mu_{\widetilde\omega\to i}^{\nu+1}(\cdot)$ to all $i\in\widetilde\omega$.
}}}
}
}
\end{algorithm}

As   Algorithm~\ref{alg:main_hyper}, Algorithm~\ref{alg:main_hyper_split} 
admits a delayed block-Jacobi representation. 

\begin{proposition}    \label{prop:equiv_alg_hyper_splitting}
   Postulate  Assumptions~\ref{asm:graph_hyper} and ~\ref{asm:on_the_partion_hyper_split}. Then,  Algorithm~\ref{alg:main_hyper_split} can be rewritten in the equivalent form: for any $i\in\cC_r$ and $r\in [p]$,
    \begin{subequations}
    \label{eq:delay_Jacobi_hyper_split}
    \begin{align}
     &  x_i^{\nu+1} =x_i^{\nu}+\tau_r^\nu \big(\hat x_i^{\nu+1}-x_i^{\nu}\big),\label{eq:delay_reformulation-cvx-comb_hyper_splitting}\\
       & \hat x_i^{\nu+1}\in\argmin_{x_i}\min_{x_{\cC_r\setminus\{i\}}}\qty{\begin{aligned}     \;&\sum\limits_{i\in\cC_r}\phi_i(x_i)+\sum\limits_{\widetilde\omega\in\widetilde\cE_{r}}\widetilde\psi_{{\rm par}(\widetilde\omega)}\qty(x_{\widetilde\omega};x_{{\rm par}(\widetilde\omega)\setminus\widetilde\omega}^{\nu- d(i,\widetilde\omega)-1})\\
       &+\sum\limits_{j\in\cC_r,\widetilde\omega\in\tildecNout_j}\widetilde\psi_{{\rm par}(\widetilde\omega)}\qty(x_j,x_{{\rm par}(\widetilde\omega)\setminus\{j\}}^{\nu-d(i,j)};x_{{\rm par}(\widetilde\omega)\setminus \{j\}}^{\nu-d(i,j)})
       \end{aligned}},
    \label{eq:prop_equiv_alg_hyper_splitting}
    \end{align}\end{subequations}
where $d(\cdot,\cdot)$ denotes variable-to-variable or variable-to-factor distance in $\widetilde\cG_r$.
\end{proposition}


Proposition~\ref{prop:equiv_alg_hyper_splitting} enables a convergence analysis of Algorithm~\ref{alg:main_hyper_split} via essentially the same arguments as in Sec.~\ref{sec:surrogate_pairwise}. In particular, within the surrogate-regularity framework of Sec.~\ref{sec:surrogate-regularity}, if the \emph{aggregate objectives} in~\eqref{eq:prop_equiv_alg_hyper_splitting}---the cluster-relevant portion of the split-surrogate objective~\eqref{eq:split_surrogate_objective_template} (playing, for $\widetilde\Phi$, the same role that~\eqref{eq:Phi_r_def_simplified} plays for $\Phi$)---are \emph{uniformly strongly convex} in the cluster variables, uniformly over the frozen arguments, then Algorithm~\ref{alg:main_hyper_split} inherits the same rate guarantees as Algorithm~\ref{alg:main_hyper}: linear convergence for strongly convex $\Phi$, and sublinear rates for merely convex or nonconvex $\Phi$; see Theorem~\ref{thm:convergence_scvx_surrogate_uniform},~\ref{thm:convergence_cvx}, and~\ref{thm:convergence_ncvx}. When the objectives in~\eqref{eq:prop_equiv_alg_hyper_splitting} are nonconvex, the surrogate machinery can be used to ensure that the local minimizations in~\eqref{eq:hyper_jump_split_updates:a} and~\eqref{eq:hyper_jump_split_factor2var} are well posed, and the same convergence rate conclusions (linear for strongly convex $\Phi$, sublinear for merely convex or nonconvex $\Phi$) follow under the standard assumptions of Sec.~\ref{sec:surrogate-regularity}. We omit the technical details.

\vspace{-0.4cm}

\section{Numerical experiments}\label{sec:experiments}
\vspace{-0.1cm}
In this section, we evaluate our algorithms MP-Jacobi and its surrogate variants solving  convex quadratic programs (and decentralized optimization) defined over pairwise graphs (Sec.~\ref{sec:numerical-QP-pairwise} and Sec.~\ref{sec:QP-decentralized}) and hypergraphs (Sec.~\ref{sec:sim-hypergraps}). All simulations are conducted in MATLAB R2025a on a Windows desktop equipped with an 8-core AMD Ryzen 7 5700G (3.80~GHz) and 64~GB of RAM.
\vspace{-0.4cm}
 
\subsection{Convex quadratic programs (pairwise graph)}\label{sec:numerical-QP-pairwise}
\vspace{-0.1cm}

We consider a convex quadratic instance of~\eqref{P} with local and pairwise terms   
$\phi_i(x_i)=\frac{1}{2}\inner{H_{ii}x_i,x_i}+\inner{b_i,x_i}$ and $\psi_{ij}(x_i,x_j)=\inner{H_{ij}x_j,x_i}.$ The global matrix $\mathbf{H}$ is generated as follows: each block $H^{\rm tmp}_{ii}$ and $H^{\rm tmp}_{ij}$ is drawn with i.i.d. Gaussian entries $\mathcal{N}(0,1)$ and then symmetrized via $\bar{\mathbf{H}}^{\rm tmp}=\frac{1}{2}(H^{\rm tmp}+(H^{\rm tmp})^\top)$. We set $\mathbf{H}=\bar{\mathbf{H}}^{\rm tmp}+c\bI$, where the scalar $c$ is chosen so that $\kappa(\mathbf{H})=400$. The vector $\bb$ is generated with i.i.d. Gaussian entries $\mathcal{N}(0,1)$.

We compare MP-Jacobi and surrogate MP-Jacobi against centralized block Jacobi (BJac), standard Jacobi (Jac), and gradient descent (GD), over three different graphs, as showed in Fig.~\ref{fig:QP_compare_all}(c). Intra-cluster edges are shown in red in the same panels.   For the surrogate variant, surrogation is applied only to the message update: $\widetilde\phi_i$ is a first-order linearization plus a proximal term, and $\widetilde\psi_{ij}$ is chosen as in Sec.~\ref{subsec:surrogate_example}(ii), with $M_{ij}=\mathrm{Diag}(H_{ij})$ and $M_i=M_j=0$. Stepsizes are manually tuned and individually for each method, to achieve the best practical convergence.

  Fig.~\ref{fig:QP_compare_all}(a) (resp. Fig.~\ref{fig:QP_compare_all}(b)) plots the optimality gap $\norm{\bx^\nu-\bx^\star}$ versus the iterations $\nu$ (resp. communications).  The proposed MP-Jacobi and surrogate MP-Jacobi converge significantly faster than GD, and are competitive with Jacobi-type baselines. In particular, MP-Jacobi is close to centralized BJac and consistently faster than Jac. Since BJac is not directly implementable in a fully decentralized setting, this highlights that intra-cluster message passing effectively recovers much of the benefit of centralized block coordination. Surrogate MP-Jacobi achieves similar convergence behavior with substantially reduced communication, since incidences exchange compact surrogate summaries of full functional messages.

\begin{figure*}[t!]\vspace{-0.2cm}
    \begin{center}
        \setlength{\tabcolsep}{2pt}  
        \begin{tabular}{ccc}
            \includegraphics[width=0.34\linewidth]{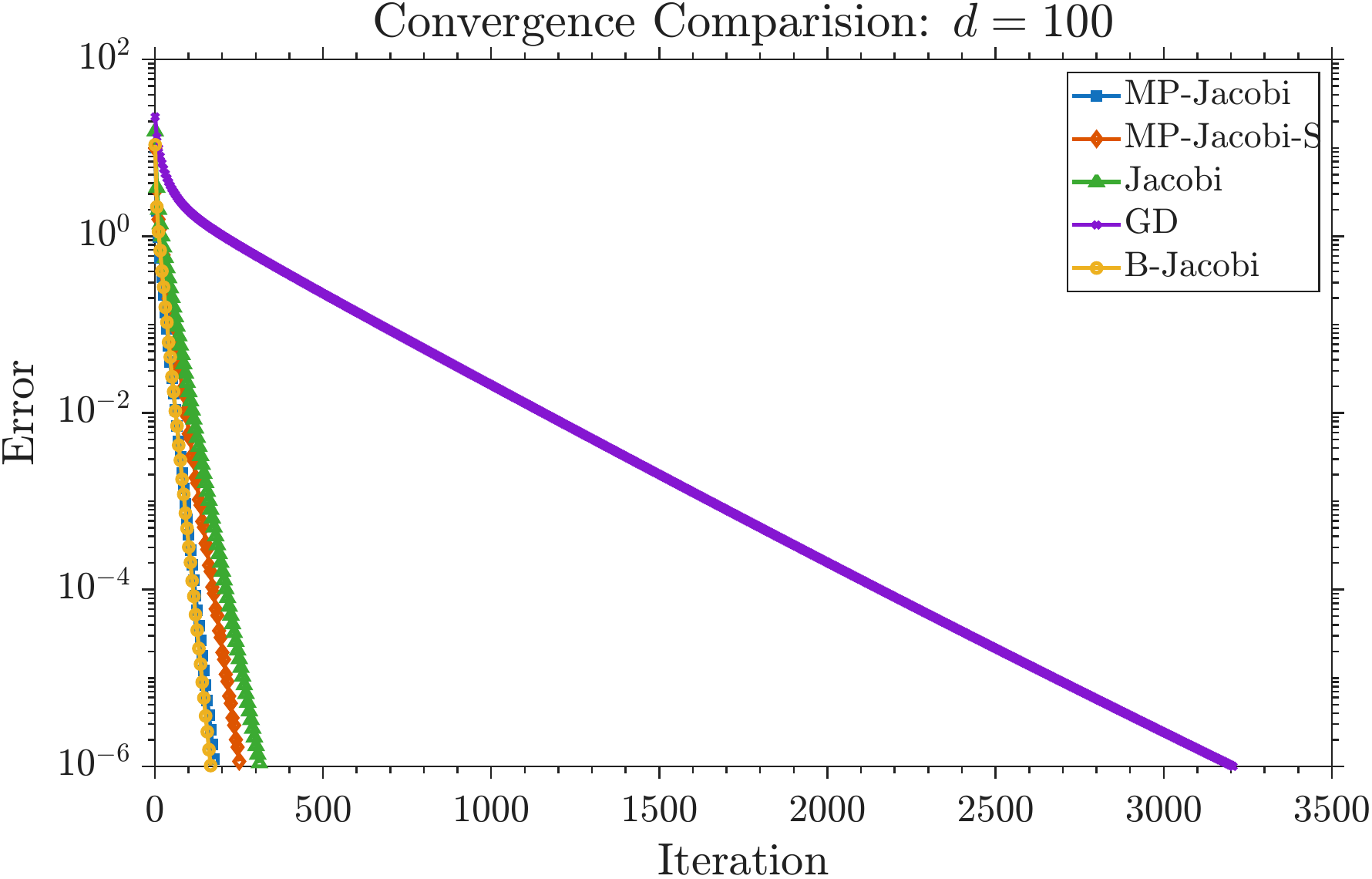}&
            \includegraphics[width=0.34\linewidth]{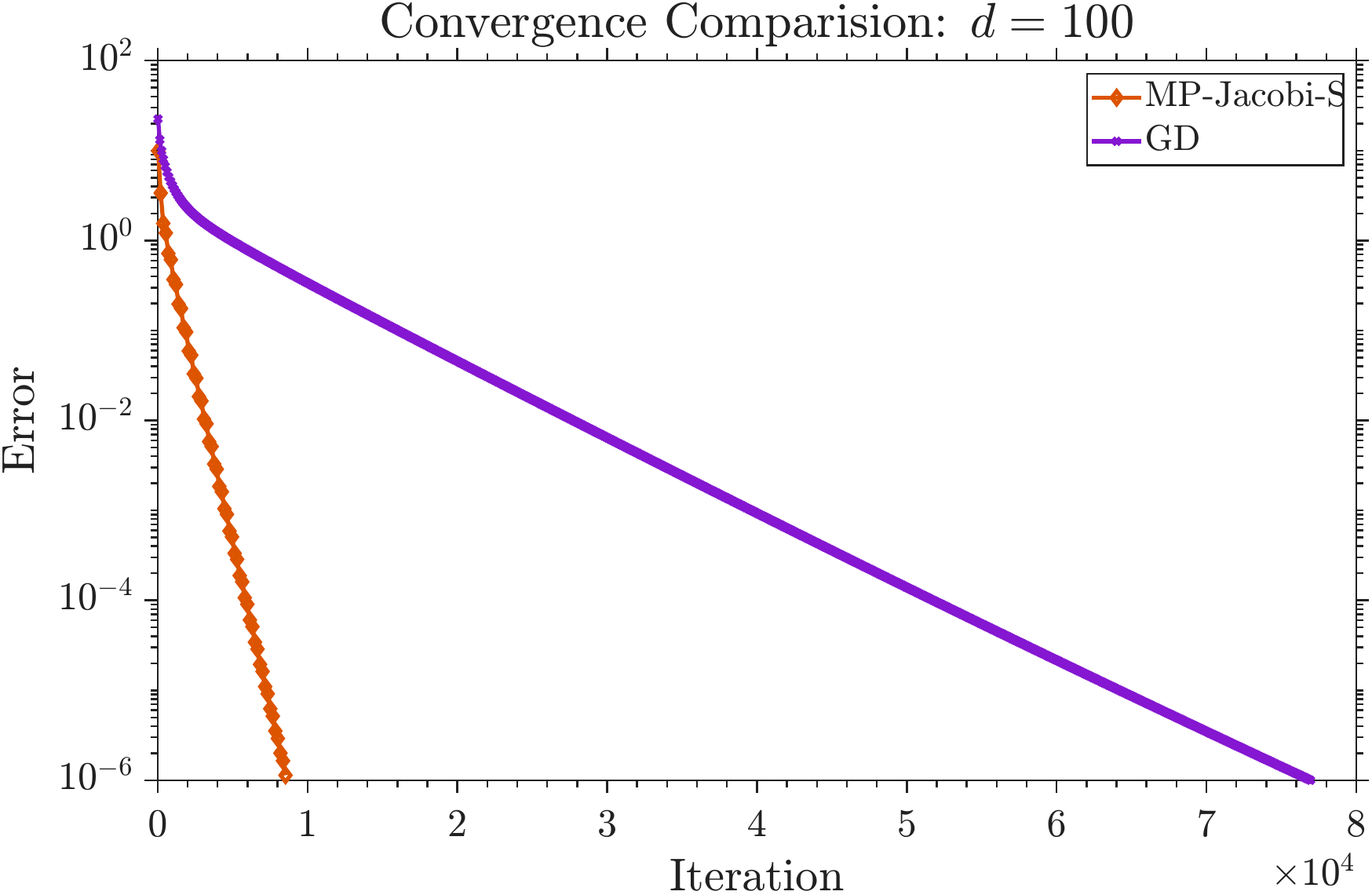}&
            \raisebox{.4cm}{\includegraphics[width=0.28\linewidth]{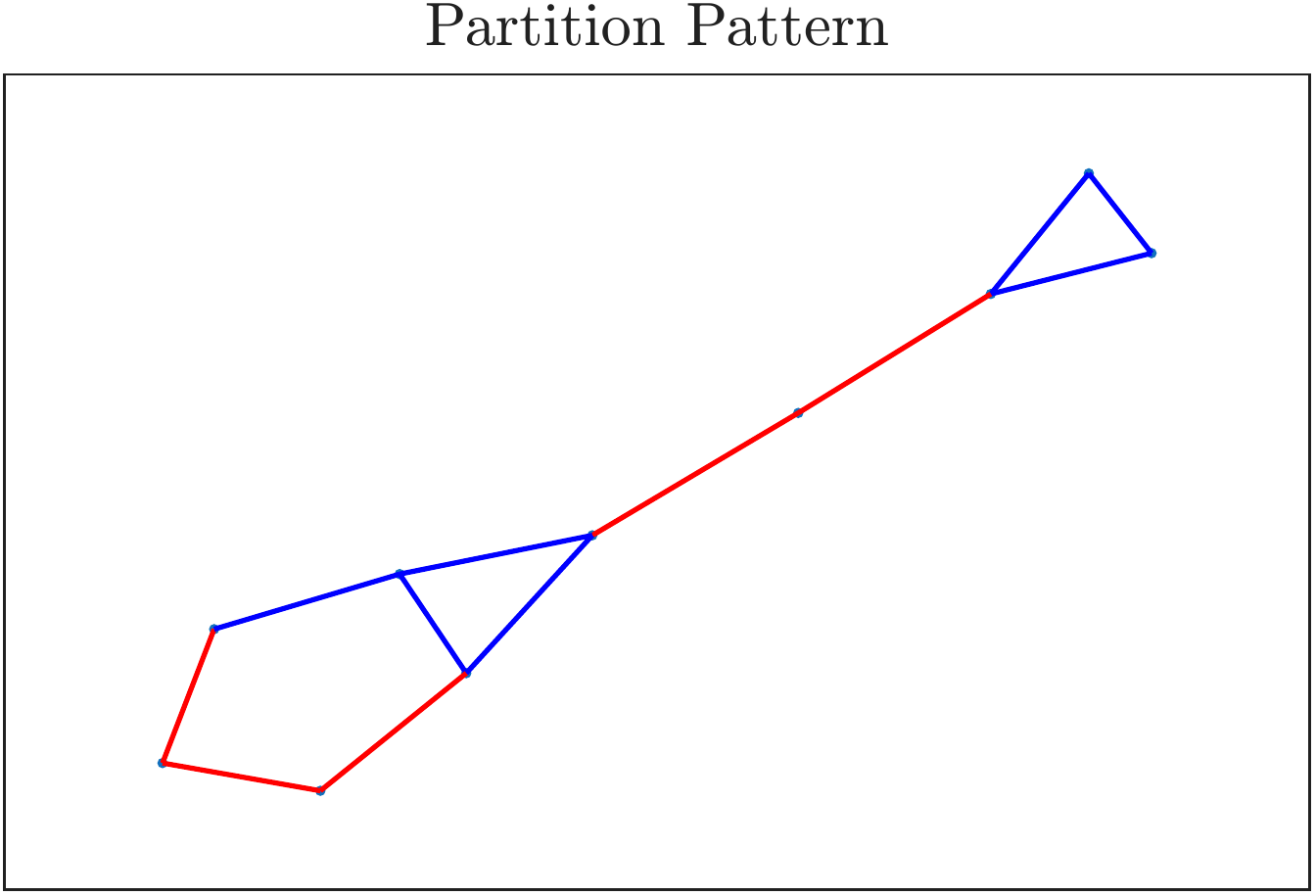}}\\
            \includegraphics[width=0.34\linewidth]{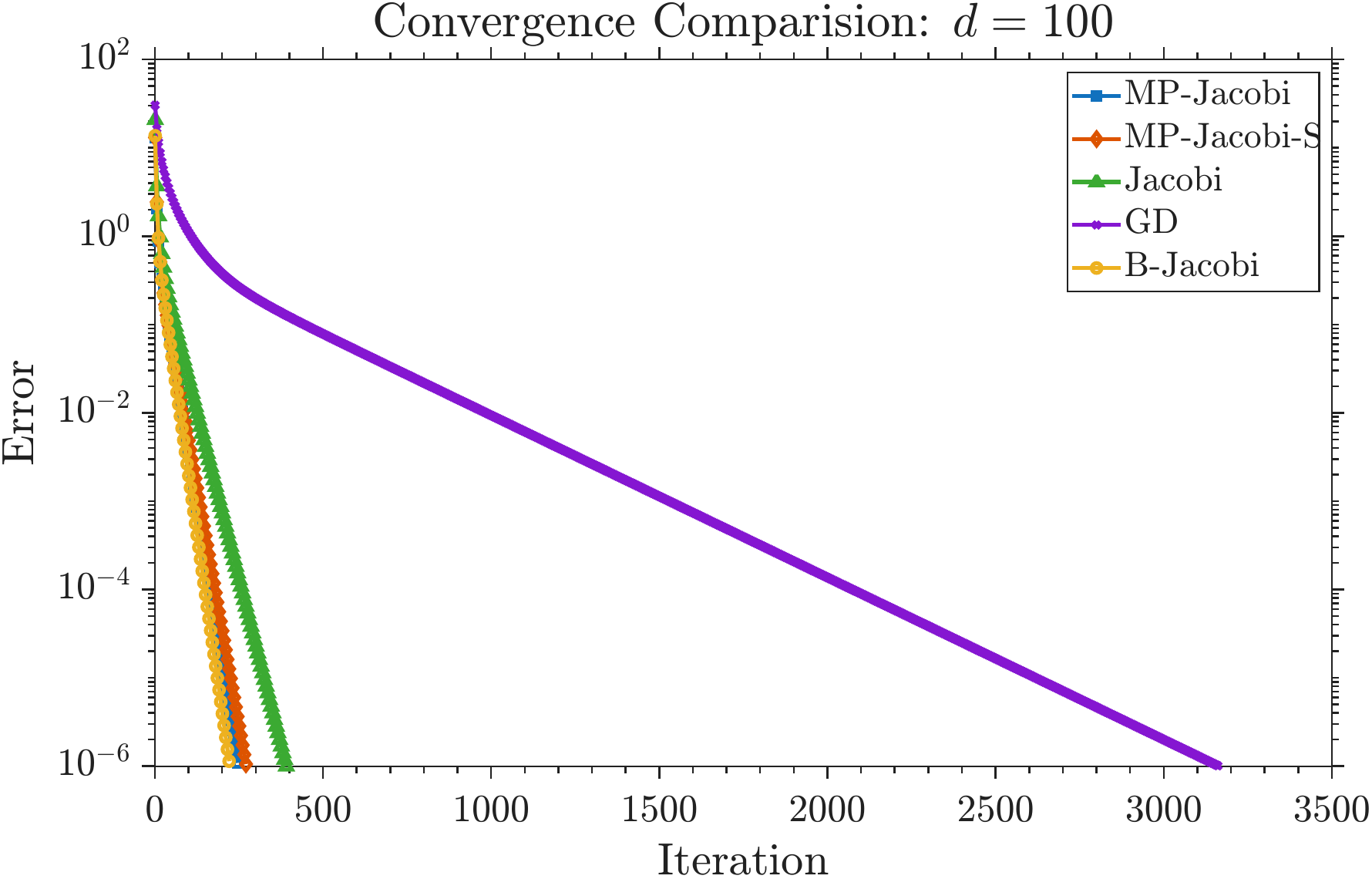}&
            \includegraphics[width=0.34\linewidth]{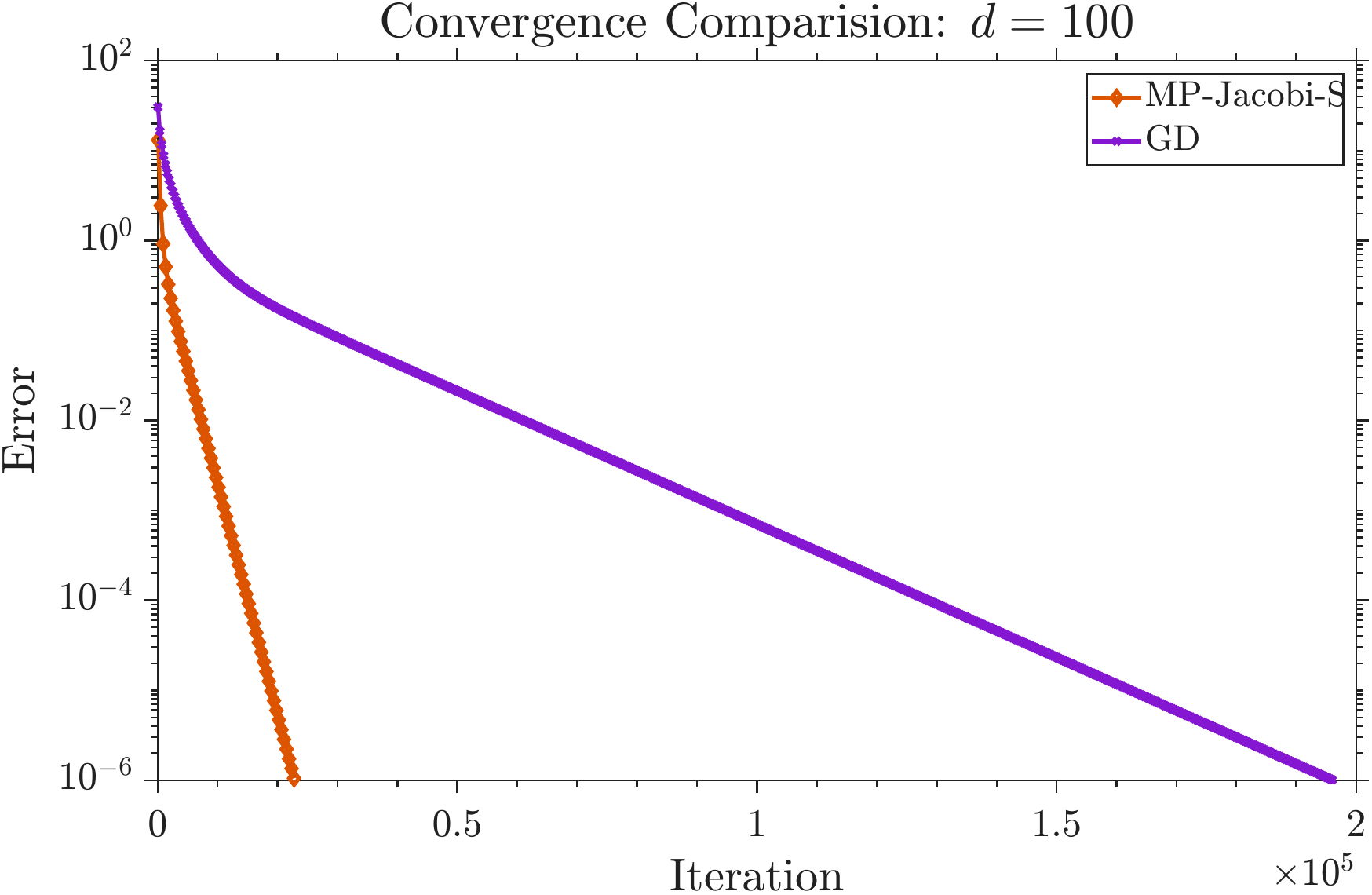}&
            \raisebox{.4cm}{\includegraphics[width=0.28\linewidth]{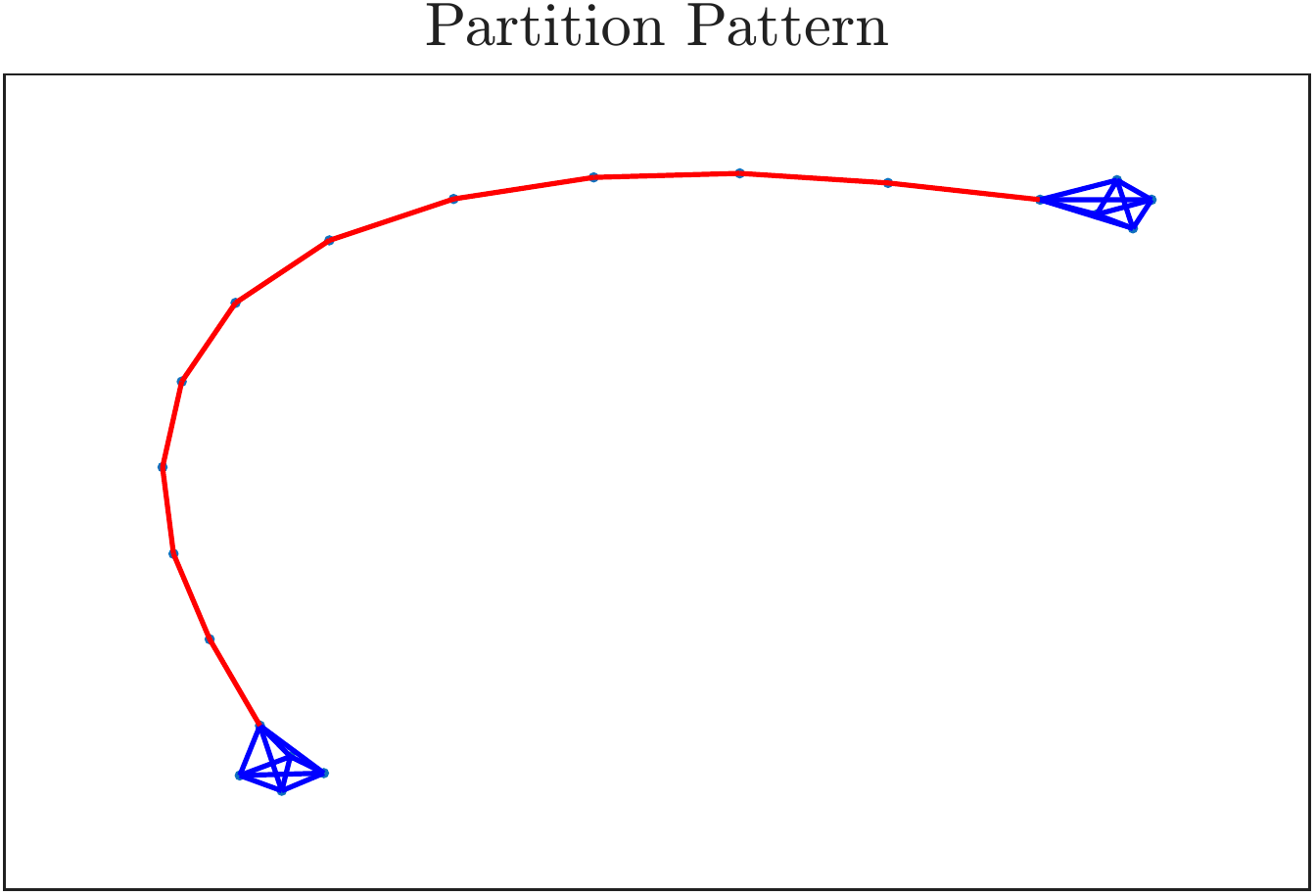}}\\
            \includegraphics[width=0.34\linewidth]{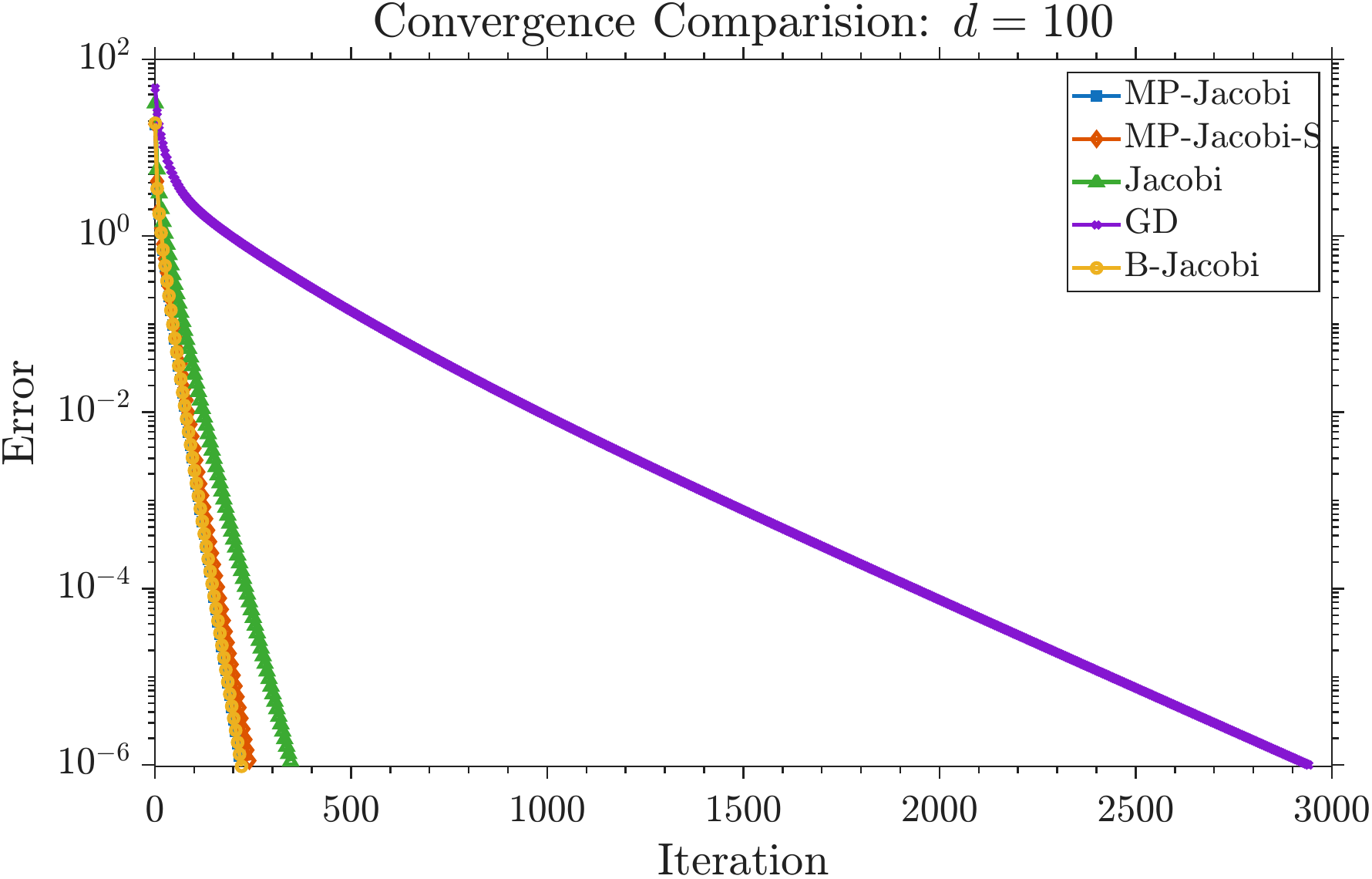}&
            \includegraphics[width=0.34\linewidth]{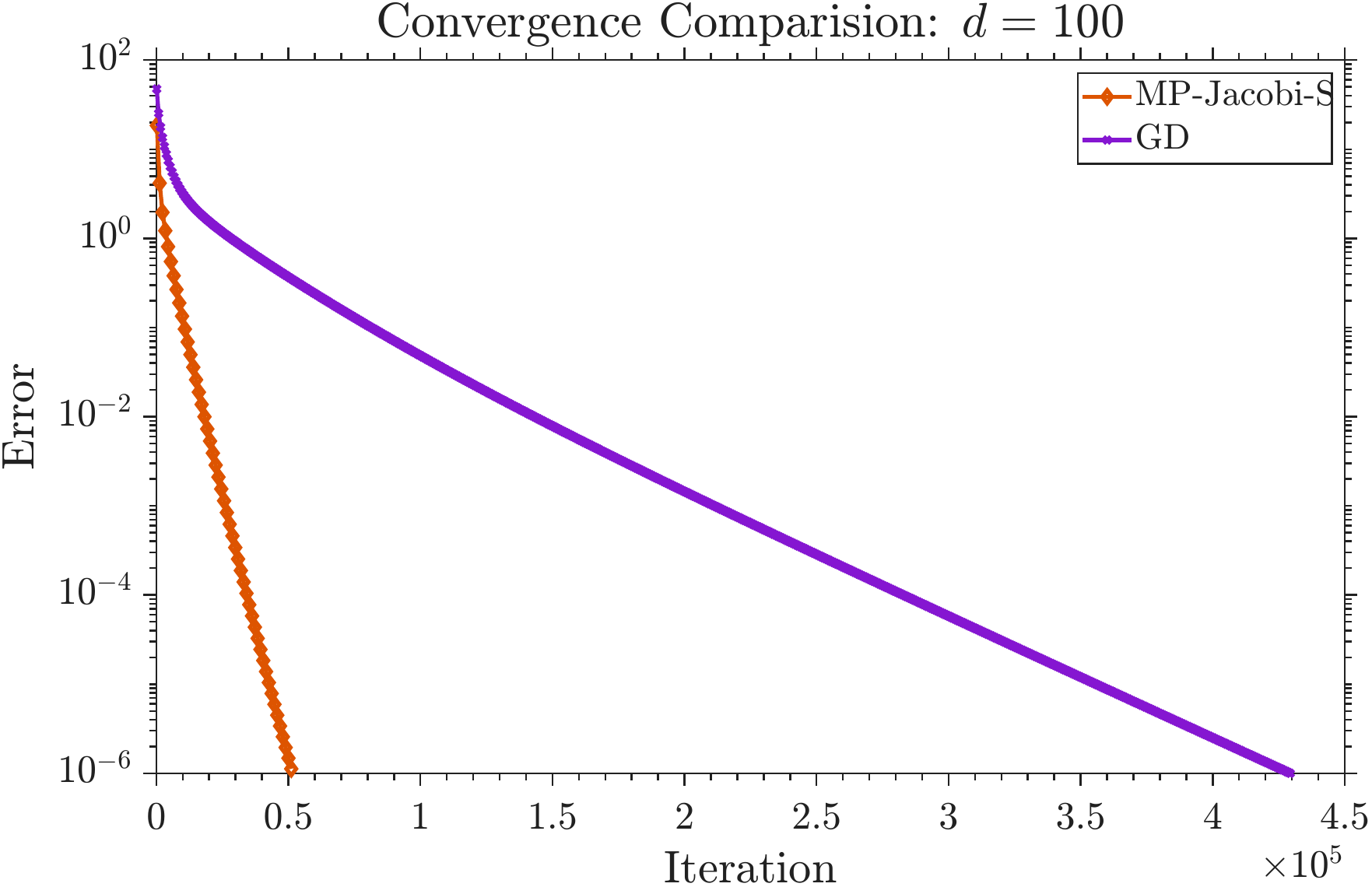}&
            \raisebox{.4cm}{\includegraphics[width=0.28\linewidth]{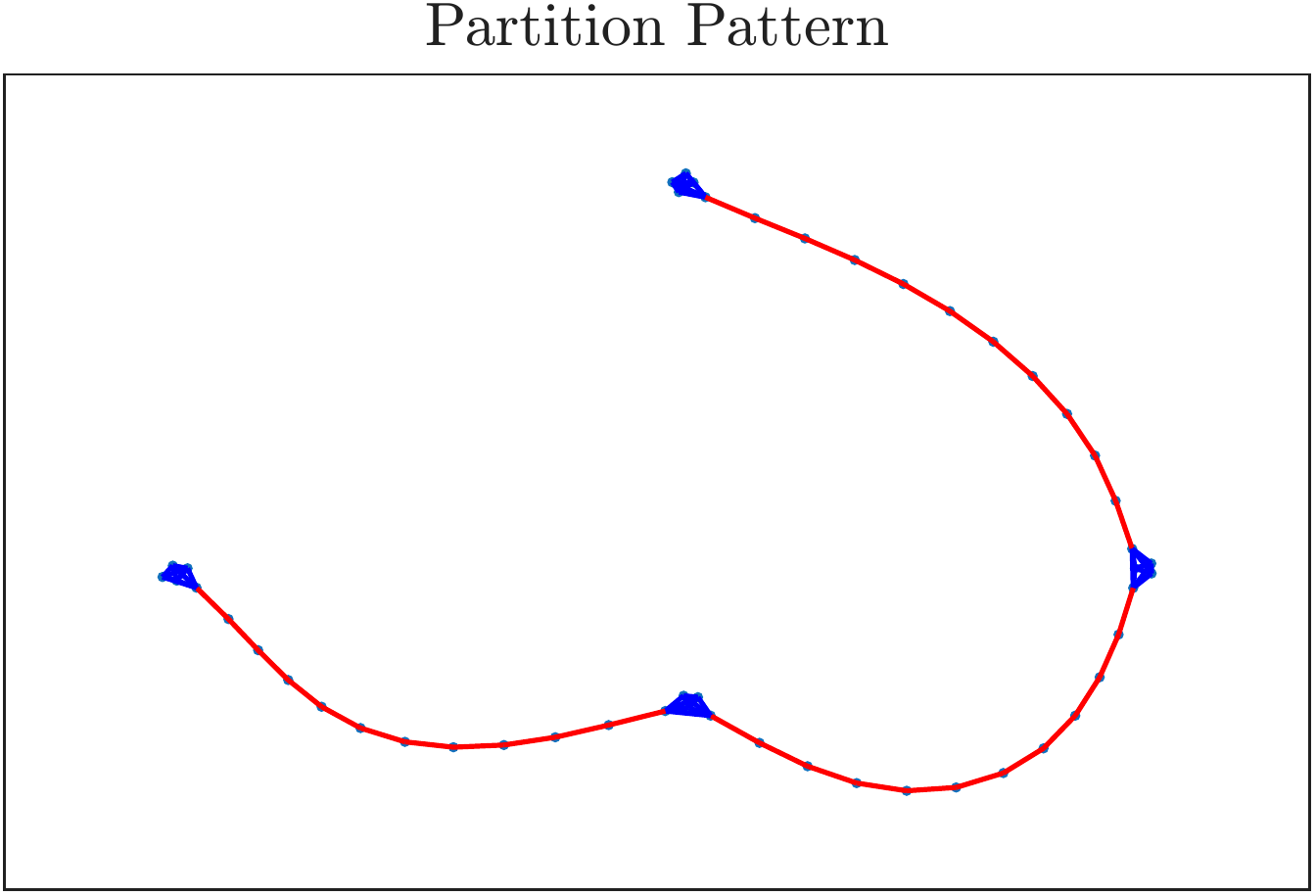}}\\
            \multicolumn{1}{c}{\footnotesize{(a) Iterations}} &  
            \multicolumn{1}{c}{\footnotesize{(b) Communication (\# vectors)}}&
            \multicolumn{1}{c}{\footnotesize{(c) Graph partitions}}                  
        \end{tabular}
    \end{center}\vspace{-0.2cm}
    \caption{Strongly convex quadratic problem:    MP-Jacobi, MP-Jacobi with surrogate (MP-Jacobi-S), Gradient Descent (GD), and Block centralized Jacobi (B-Jacobi).} \label{fig:QP_compare_all}\vspace{-0.4cm}
\end{figure*}

We evaluate scalability on a ring graph.  Let  $m=(D+1)\lceil D^{3/2}\rceil$ and increase   $D$ progressively. We compare two clustering strategies. {\it Partition 1} uses two clusters: one cluster is a long path (tree) of length  $\lceil m^{2/3}\rceil$, and  the other is a single node.   {\it Partition 2} uses $\lceil D^{3/2}\rceil$ clusters, each a path  tree of length $D=\Theta(m^{2/5})$; see the Fig.~\ref{fig:scale_m}(b). Our theory predicts that Partition 2 yields better scalability in 
$m$ than Partition 1, and Fig.~\ref{fig:scale_m} corroborates this prediction numerically.

\begin{figure*}[ht]\vspace{-0.2cm}
    \begin{center}
        \setlength{\tabcolsep}{0.0pt}  
        \scalebox{1}{\begin{tabular}{c@{\hspace{1cm}}c}
                \includegraphics[width=0.5\linewidth]{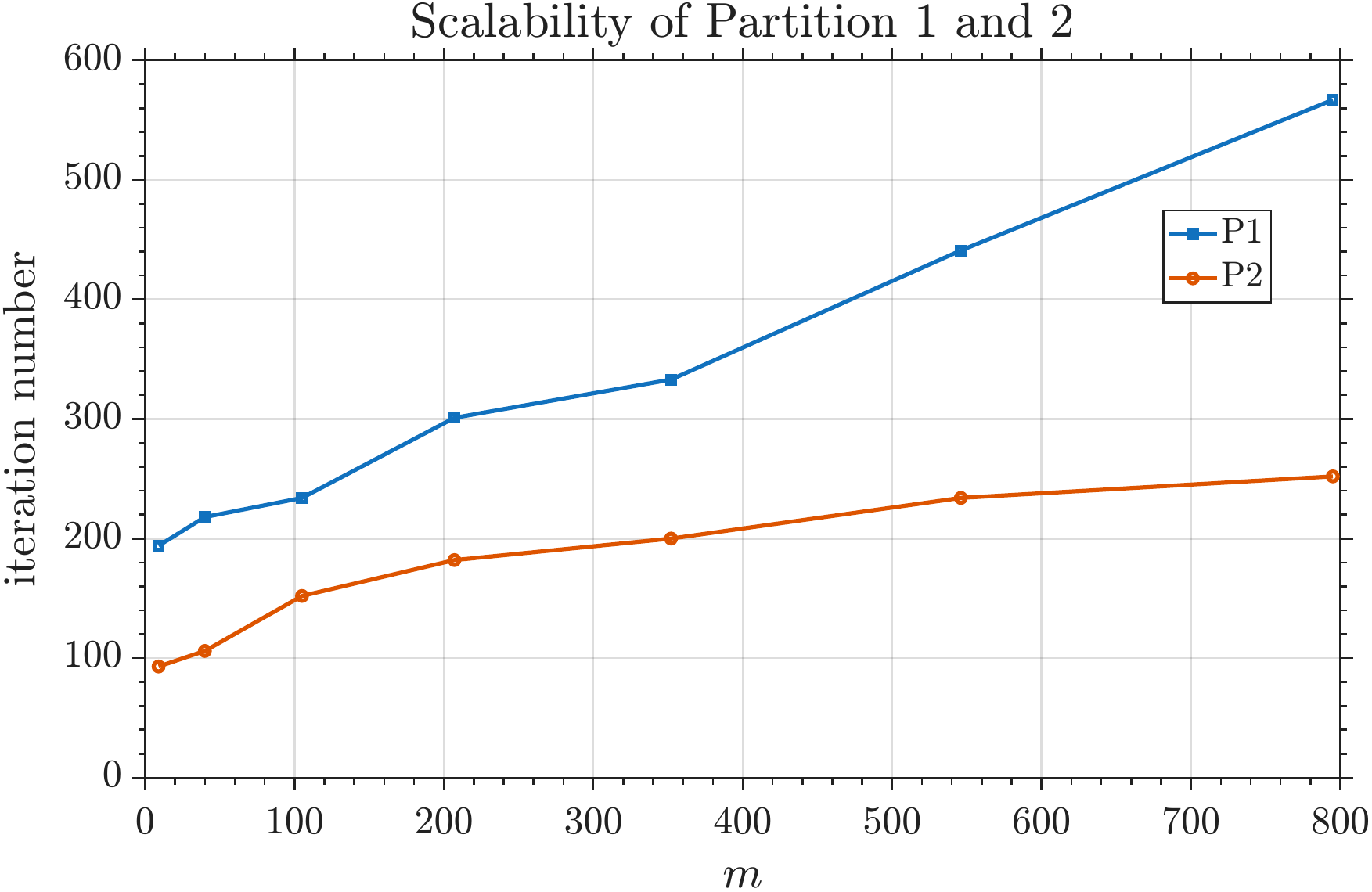} &  
                \raisebox{0.5cm}{\includegraphics[width=0.4\linewidth]{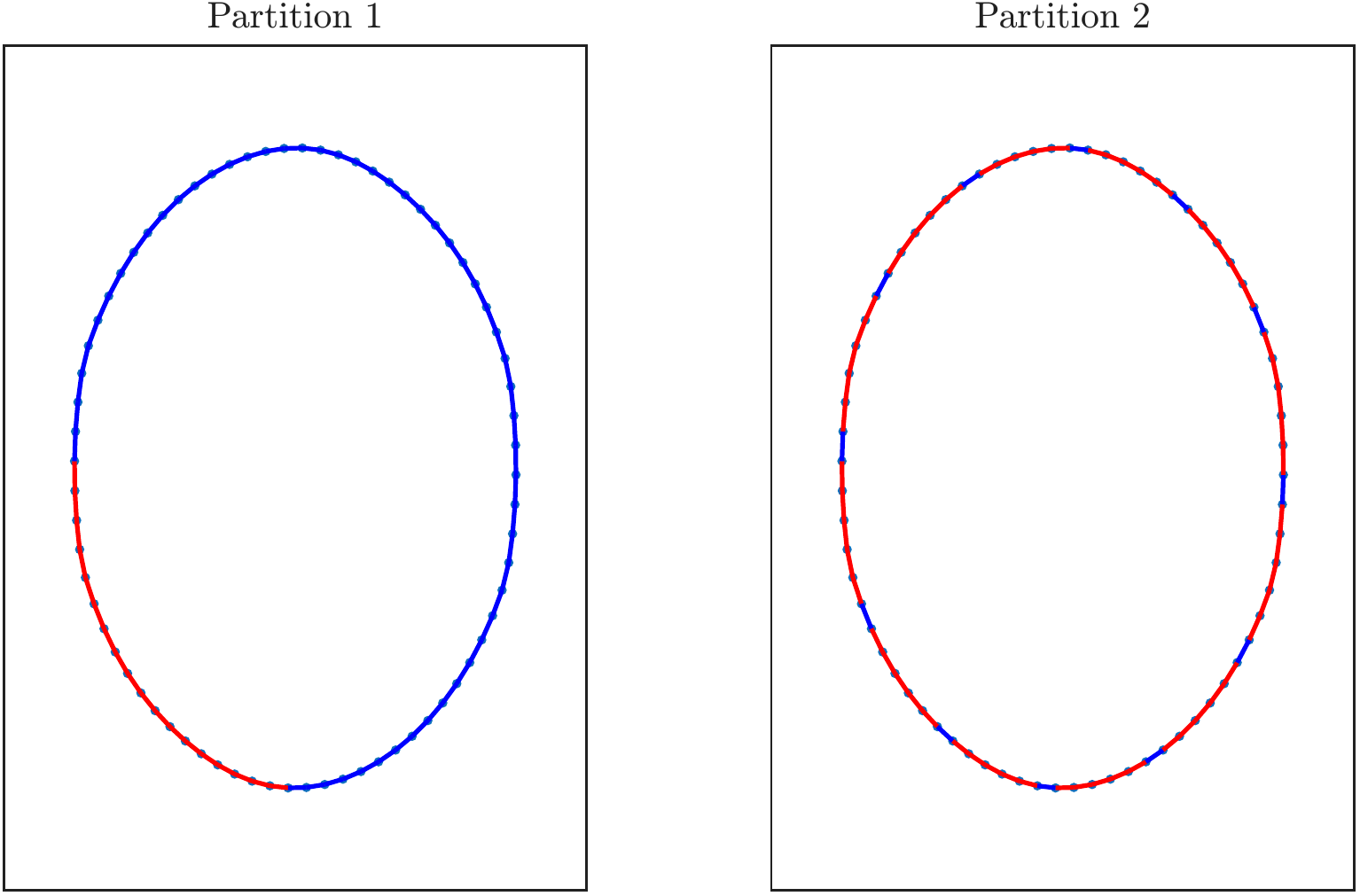}}\\
                \multicolumn{1}{c}{\footnotesize{(a) Scalability in $m$ for the  two partitions in (b)}} & \multicolumn{1}{c}{\footnotesize{(b) Ring graph and partitions (red lines)}} 
        \end{tabular}}
        \caption{Strongly convex quadratic problem: scalability (\# of iterations $\nu$ for $\norm{\bx^\nu-\bx^\star}\leq 10^{-3}$) in   $m$ under different tree-clustering strategies. \emph{Partition~1} uses two clusters: one long tree (a path) and one singleton. \emph{Partition~2} uses $m^{3/5}$ clusters, each a path of length $m^{2/5}-1$ (i.e., $m^{2/5}$ nodes per cluster). }
        \label{fig:scale_m}
    \end{center}\vspace{-0.6cm}
\end{figure*}

We finally remark that  our method is provably convergent on loopy graphs without requiring diagonal dominance.
In contrast, classical min-sum/message passing may fail in this regime;  Fig.~\ref{fig:compare_mp} shows a clear example.\vspace{-0.2cm}

\begin{figure*}[ht]
    \begin{center}
        \setlength{\tabcolsep}{0.0pt}  
        \scalebox{1}{\begin{tabular}{c@{\hspace{1cm}}c}
                \includegraphics[width=0.5\linewidth]{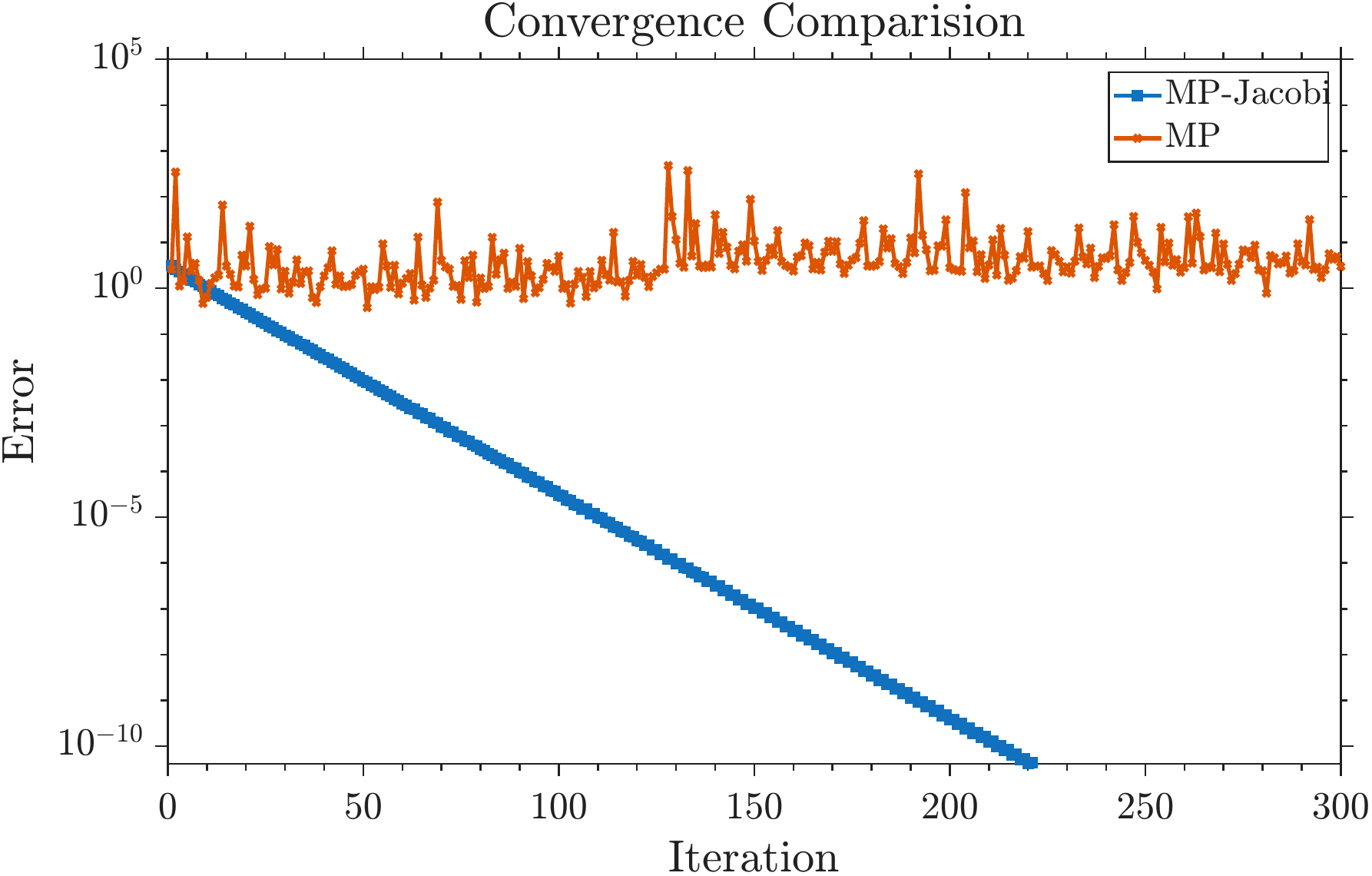} &  
                \raisebox{0.5cm}{\includegraphics[width=0.4\linewidth]{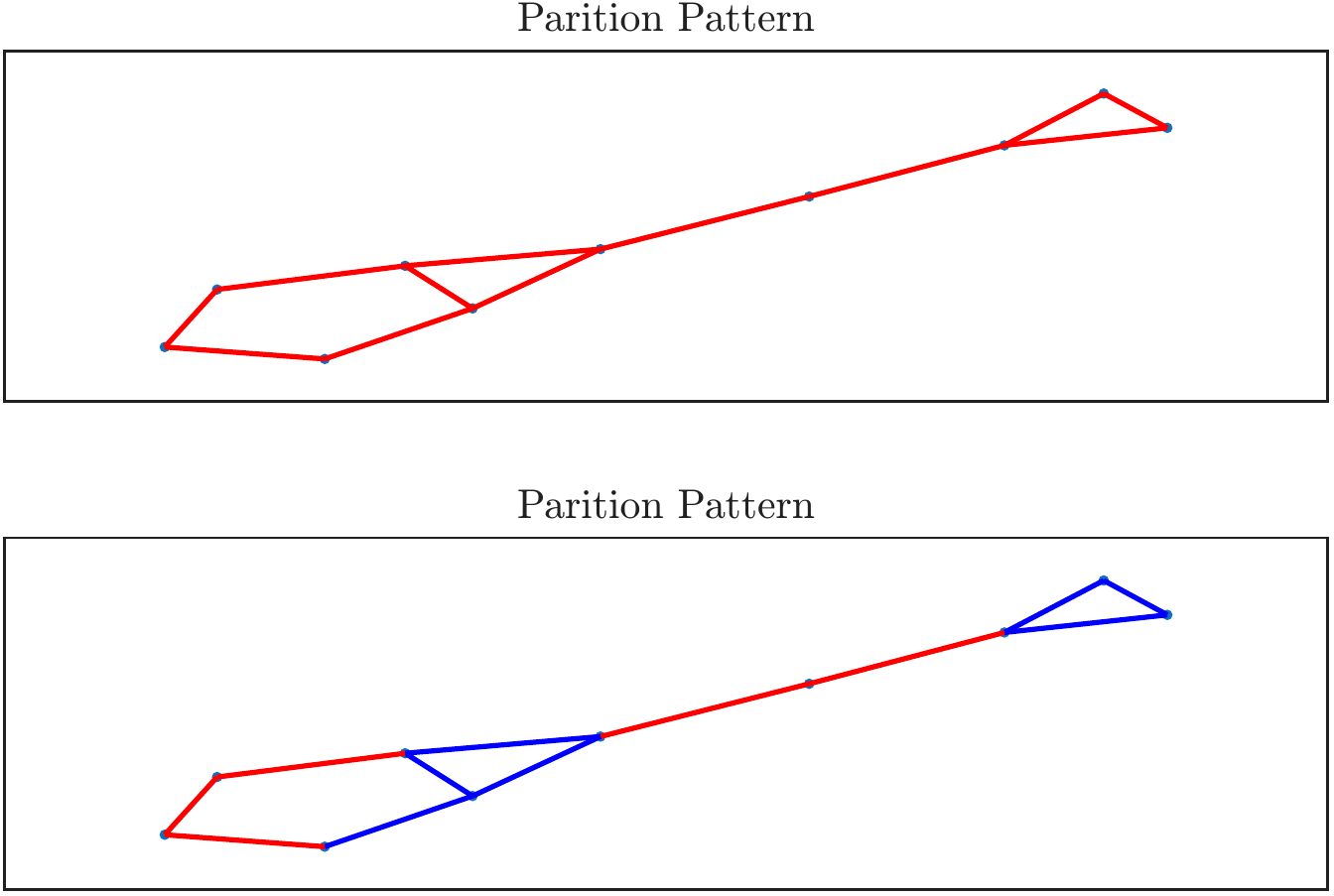}}\\
                \multicolumn{1}{c}{\footnotesize{(a) Scalability of $m$ for two partitions}} & \multicolumn{1}{c}{\footnotesize{(b) Graphs and partitions (in red).}} 
        \end{tabular}}
        \caption{Strongly convex quadratic problem on loopy graphs:  min-sum splitting algorithm (MP) vs. the proposed MP-Jacobi.  The min-sum splitting   fails to converge on loopy graph without diagonal dominance of the objective function.}
        \label{fig:compare_mp}
    \end{center}\vspace{-0.6cm}
\end{figure*}
\vspace{-0.3cm}

\subsection{Decentralized optimization   (pairwise graph)}\label{sec:QP-decentralized}\vspace{-0.2cm}
Consider the augmented formulation \eqref{eq:CTA} associated with DGD-CTA. The objective in~\eqref{eq:CTA} can be written in the pairwise form~\eqref{P} with local and coupling terms  
\begin{equation}
\begin{aligned}
        \phi_i(x_i)= f_i(x_i) + \frac{1-W_{ii}}{2\gamma}\|x_i\|^2,\;\text{ } \psi_{ij}(x_i, x_j)= -\frac{1}{\gamma}x_i^\top W_{ij}x_j. 
\end{aligned}
\label{eq:CTA_prob}
\end{equation}
We simulate quadratic functions   $f_i$'s   with     matrices $Q_{ii}$  generated by the {\tt MATLAB} command: 
$
{\tt B = randn(d,d); Qtmp = B'*B/m;Q\_ii=Qtmp}$ ${\tt +c*eye(d);}$
where  ${\tt c}$ is chosen such that  global cost $\sum_{i=1}^mf_i(x_i)$ has condition number 100. We set  $\gamma=10^{-3}$,   and $W$ is the Metropolis weight matrix. 
 \begin{figure*}[ht]
		\begin{center}
			\setlength{\tabcolsep}{0.0pt}  
			\scalebox{1}{\begin{tabular}{ccc}
					\includegraphics[width=0.33\linewidth]{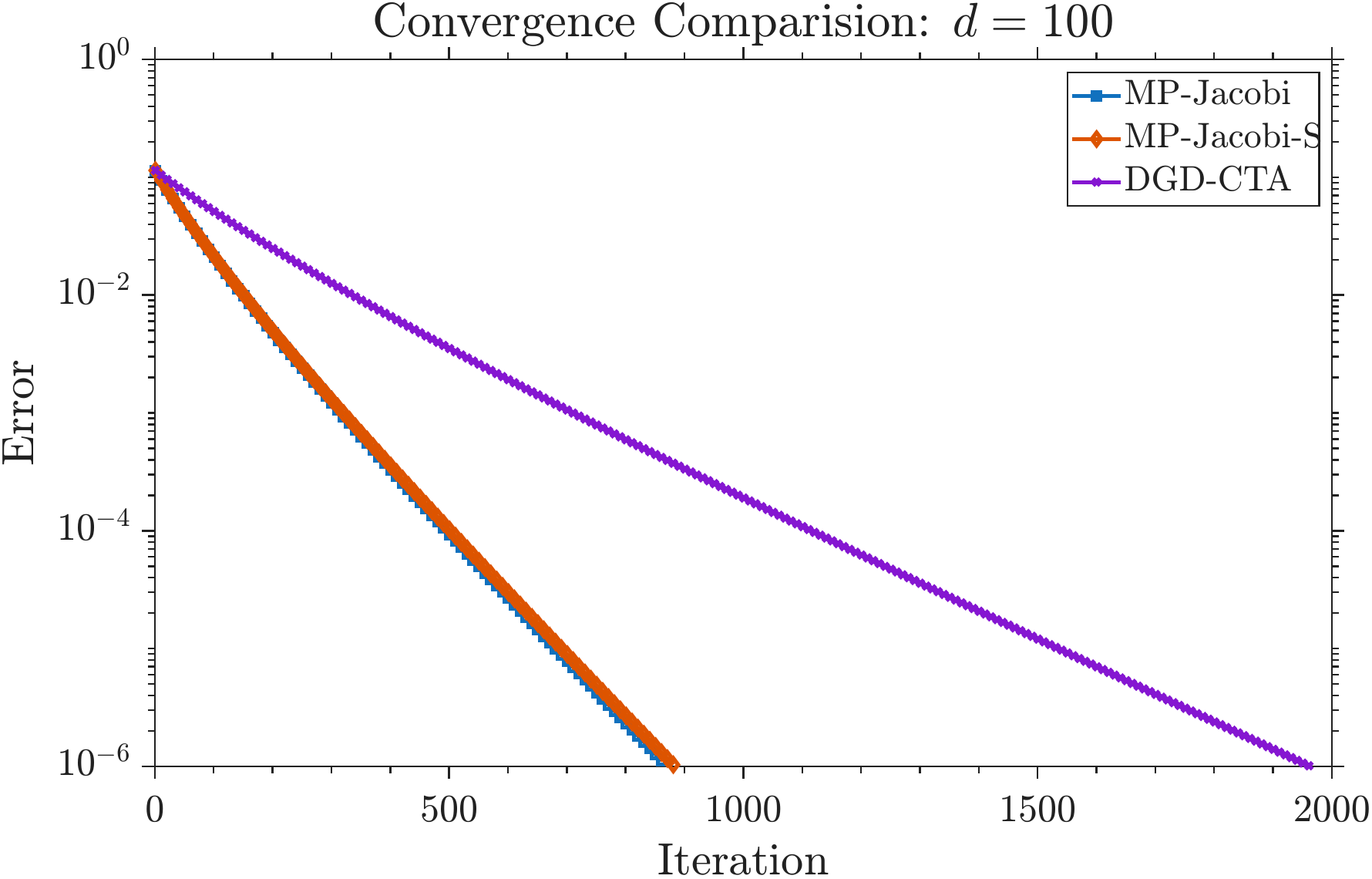}&
					\includegraphics[width=0.33\linewidth]{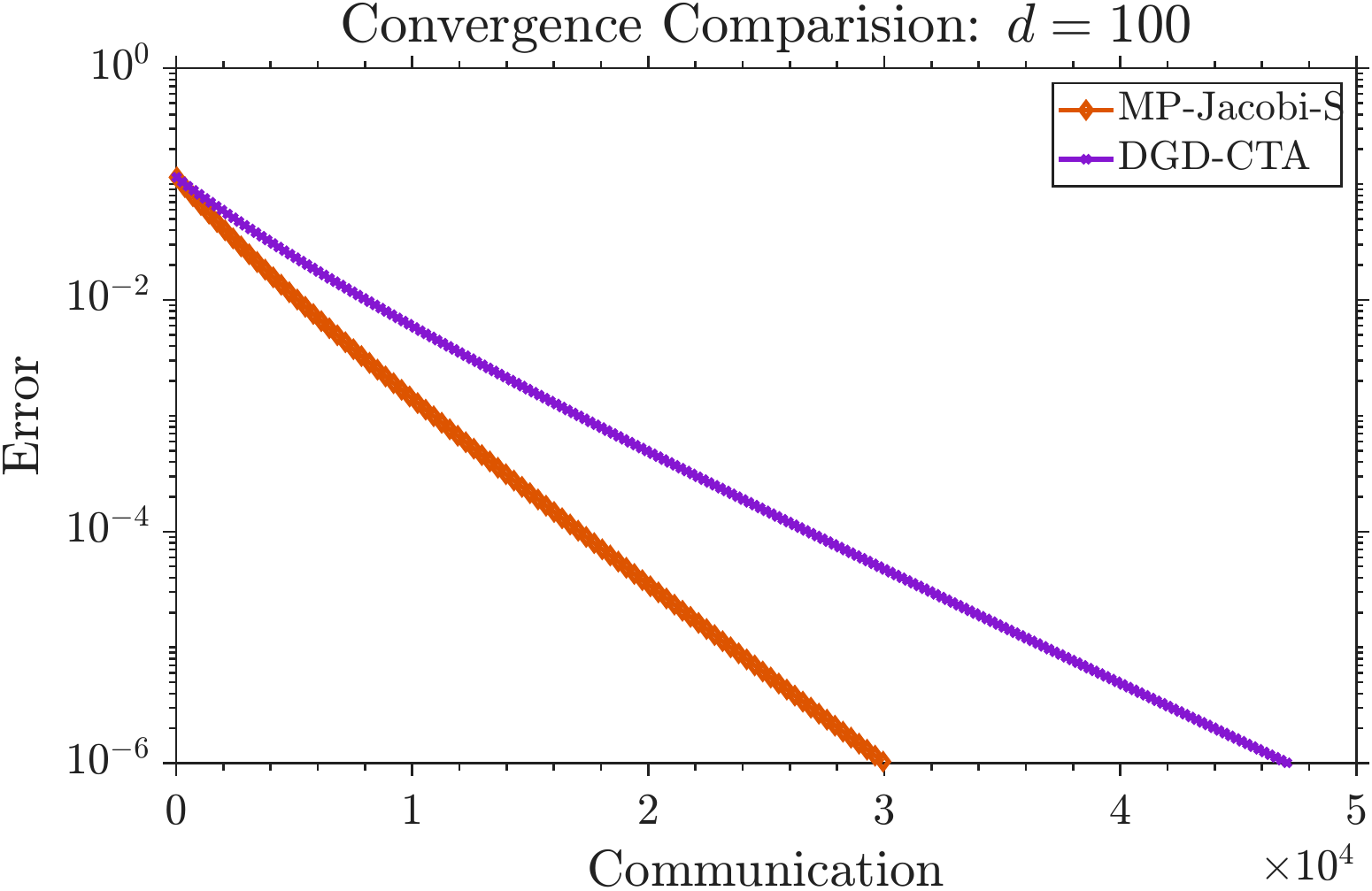}&
					\includegraphics[width=0.33\linewidth]{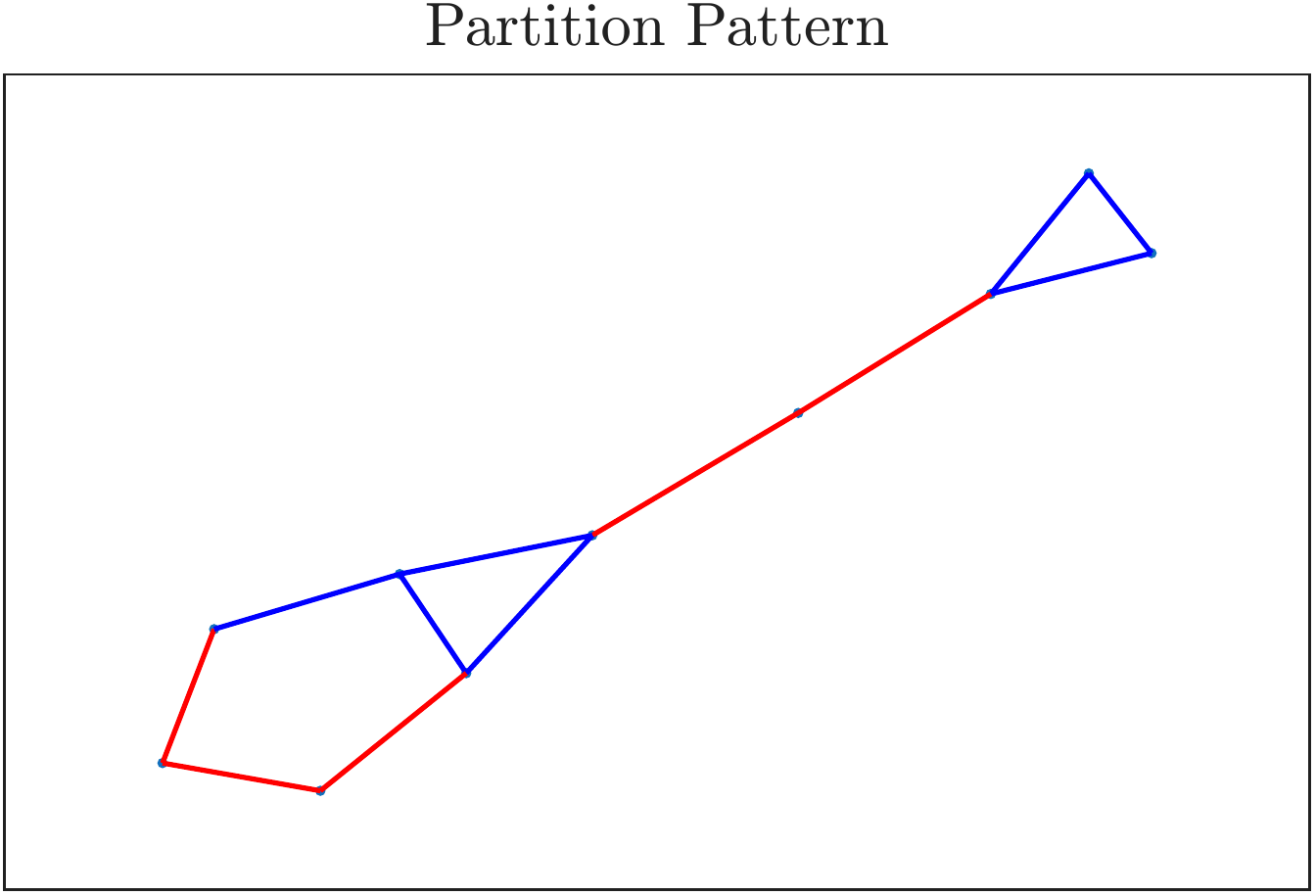}\\
                    \includegraphics[width=0.33\linewidth]{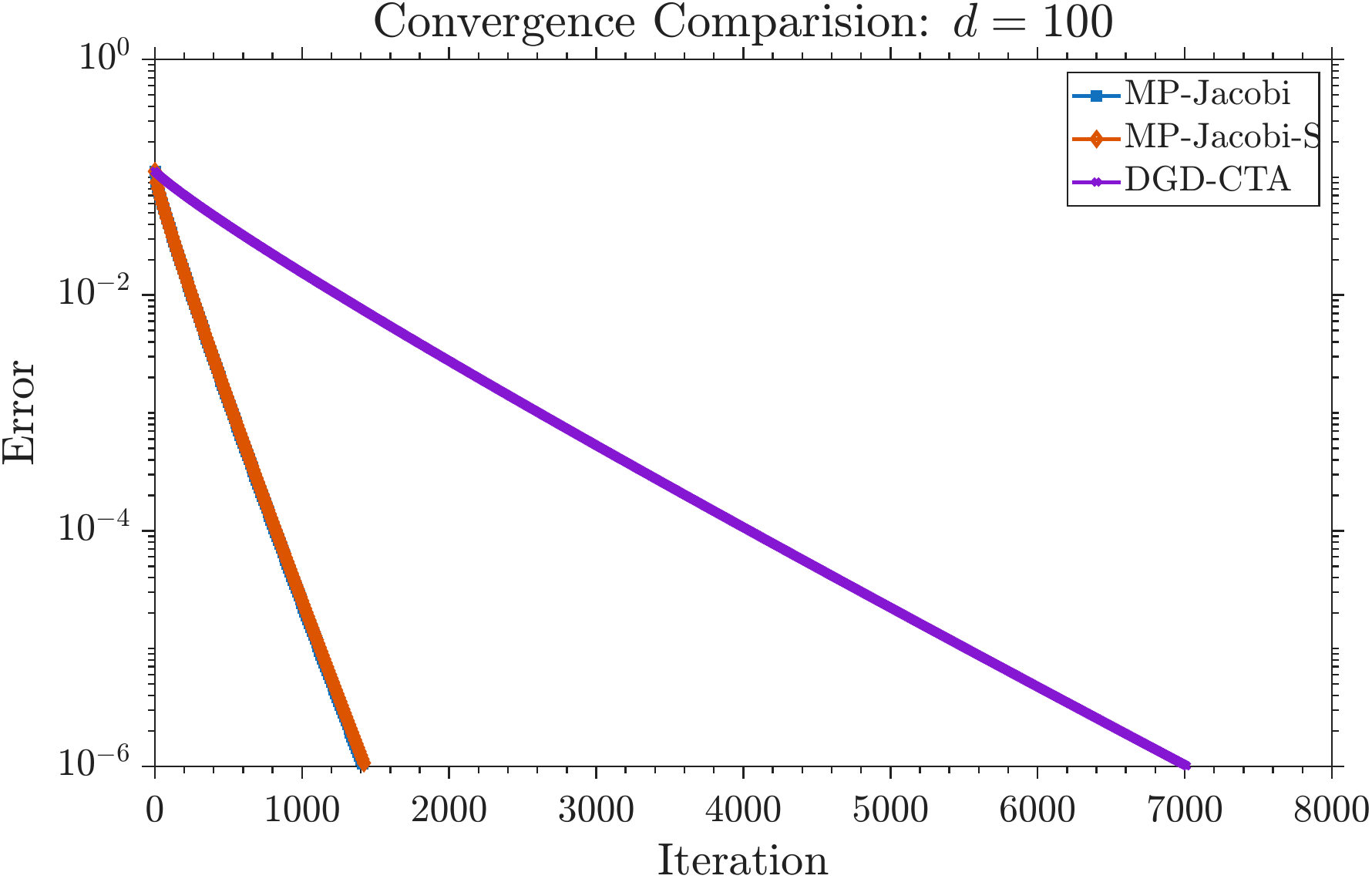}&
					\includegraphics[width=0.33\linewidth]{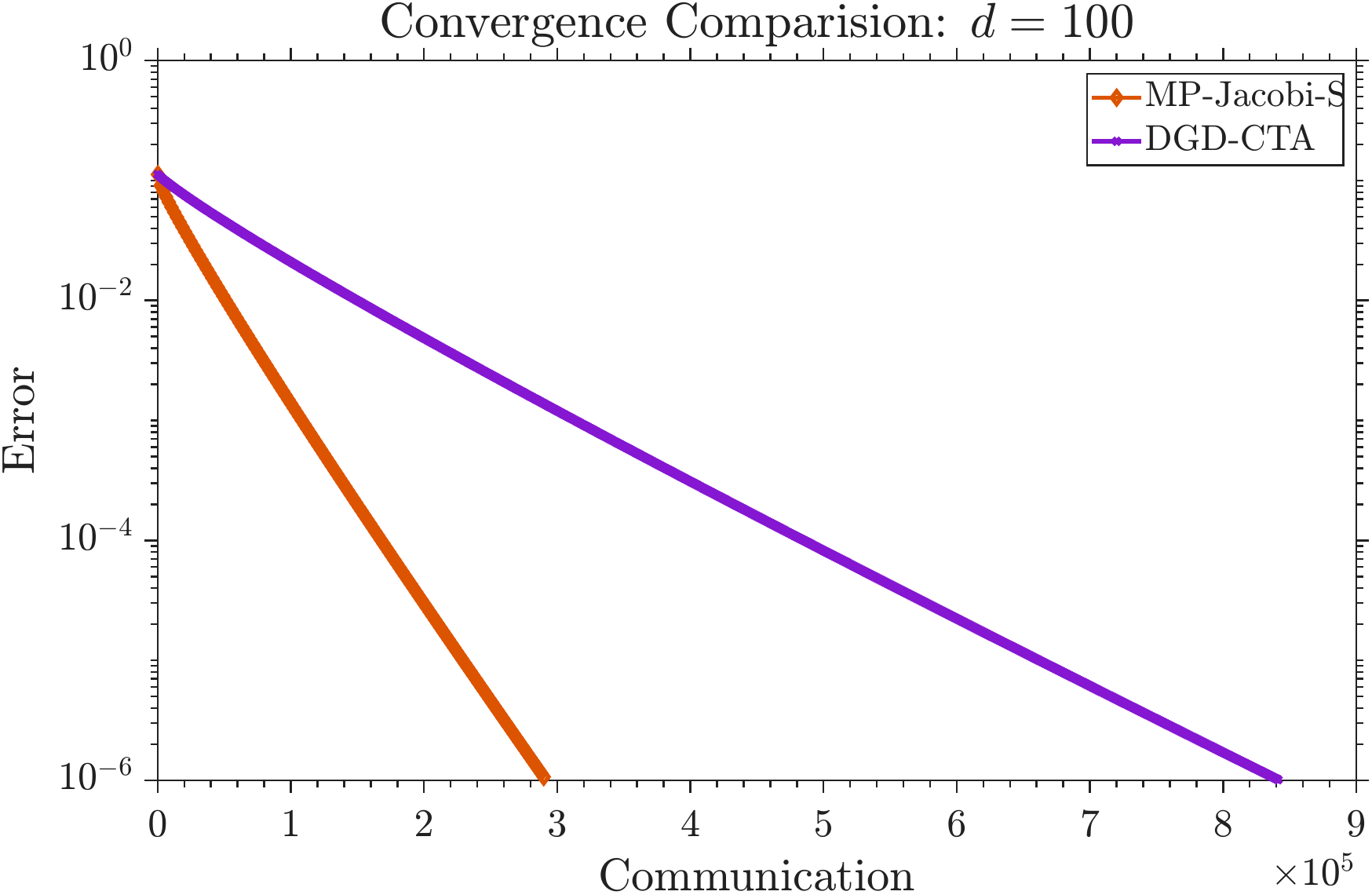}&
					\includegraphics[width=0.33\linewidth]{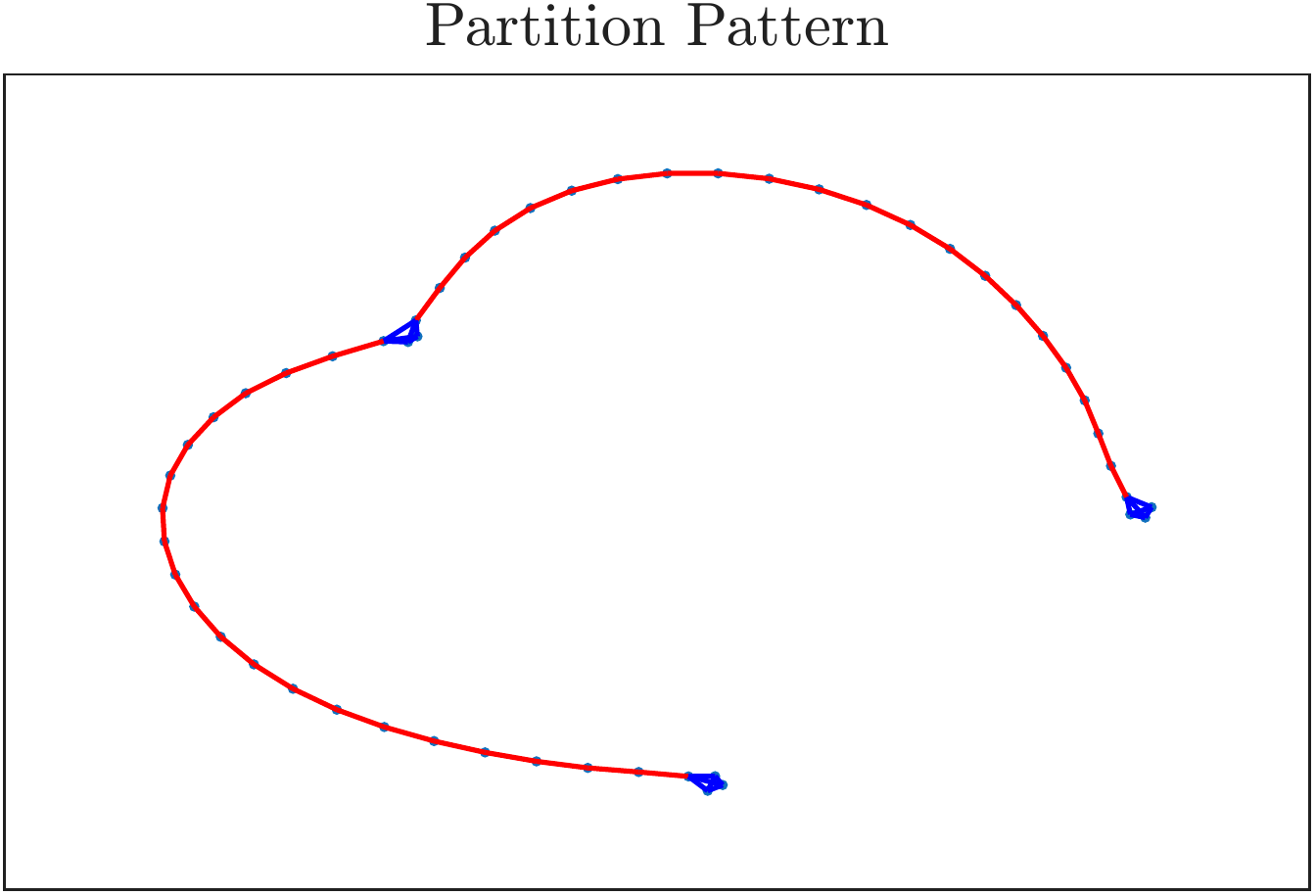}\\
                    \includegraphics[width=0.33\linewidth]{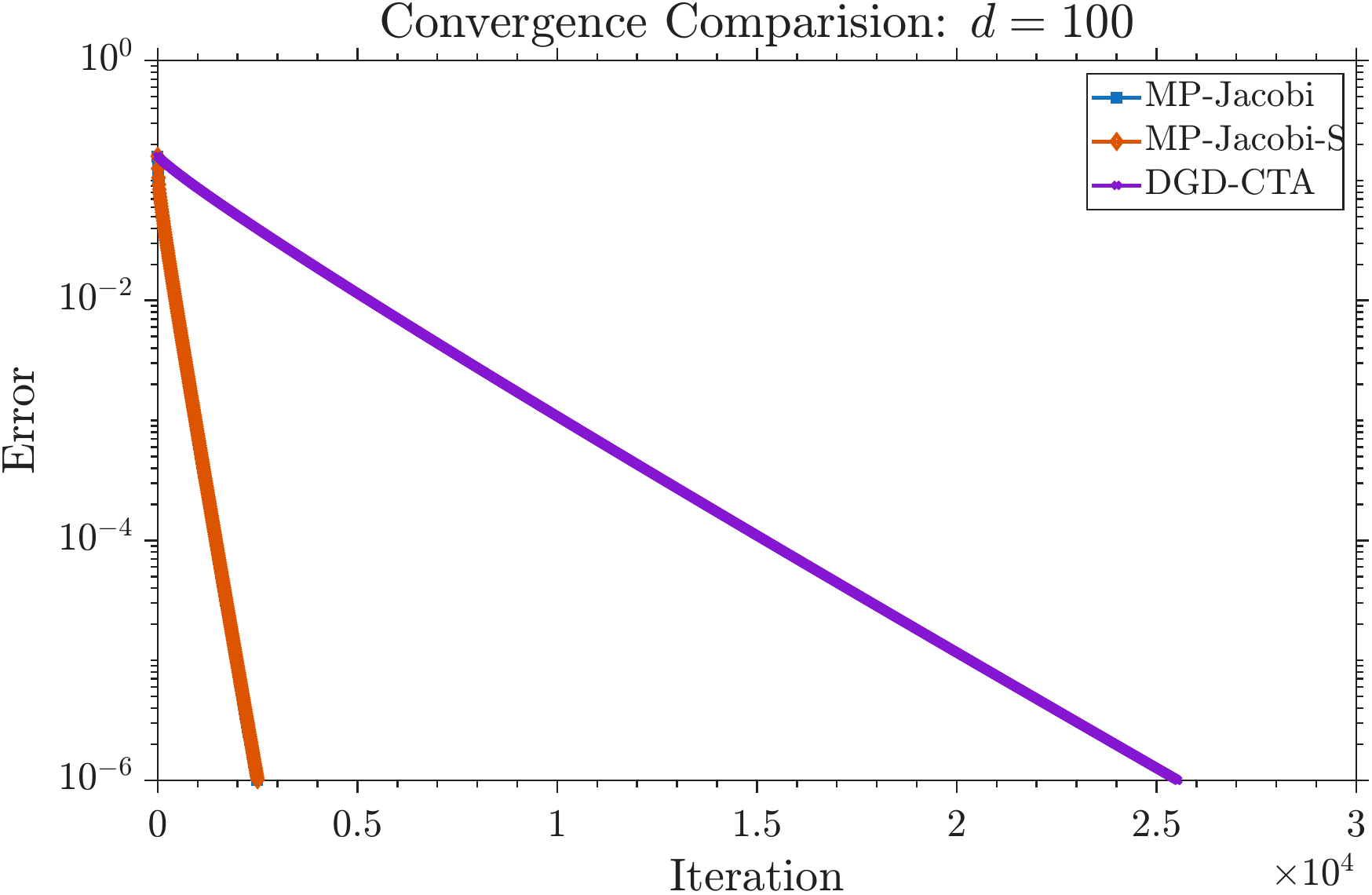}&
					\includegraphics[width=0.33\linewidth]{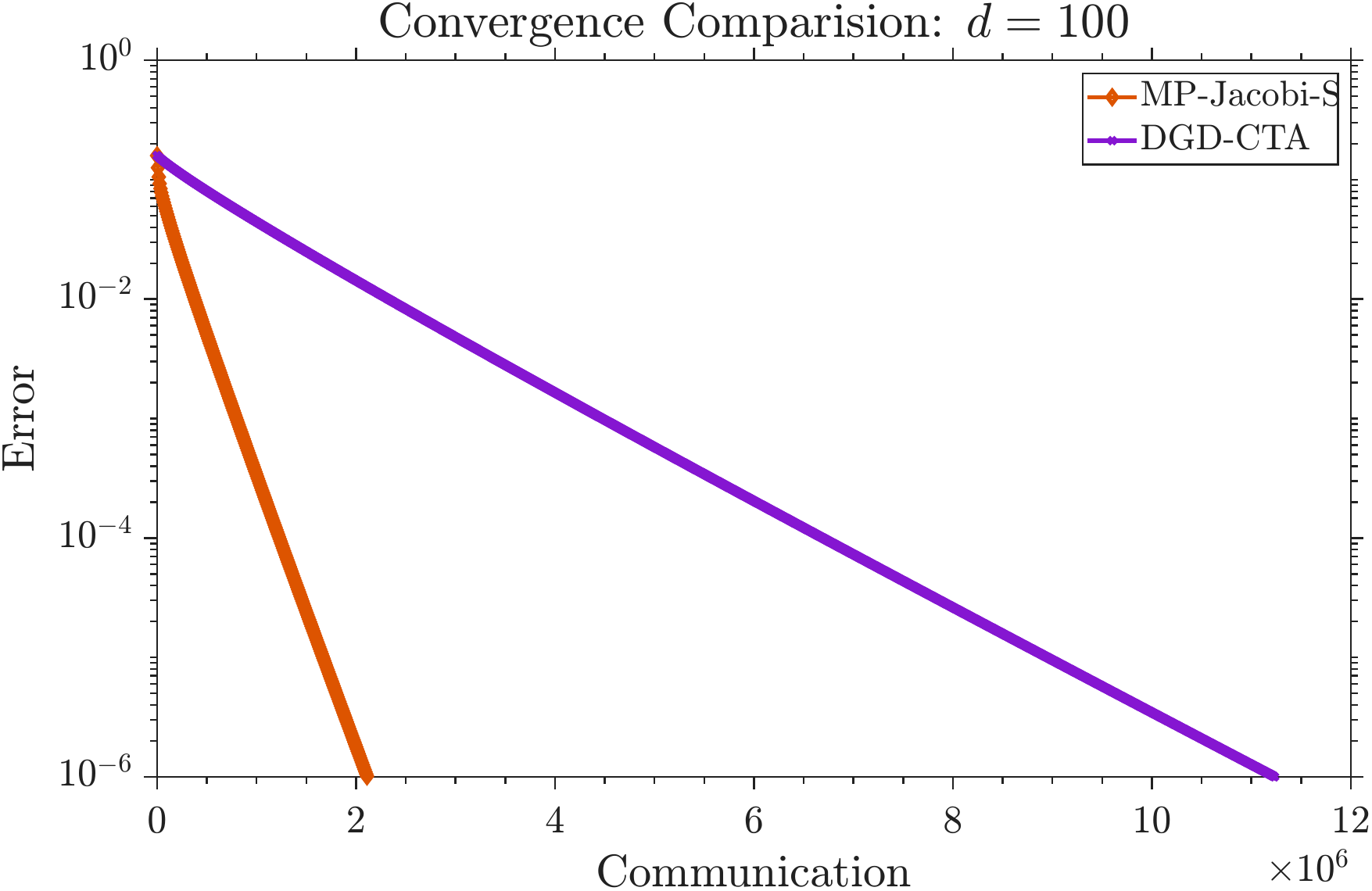}&
					\includegraphics[width=0.33\linewidth]{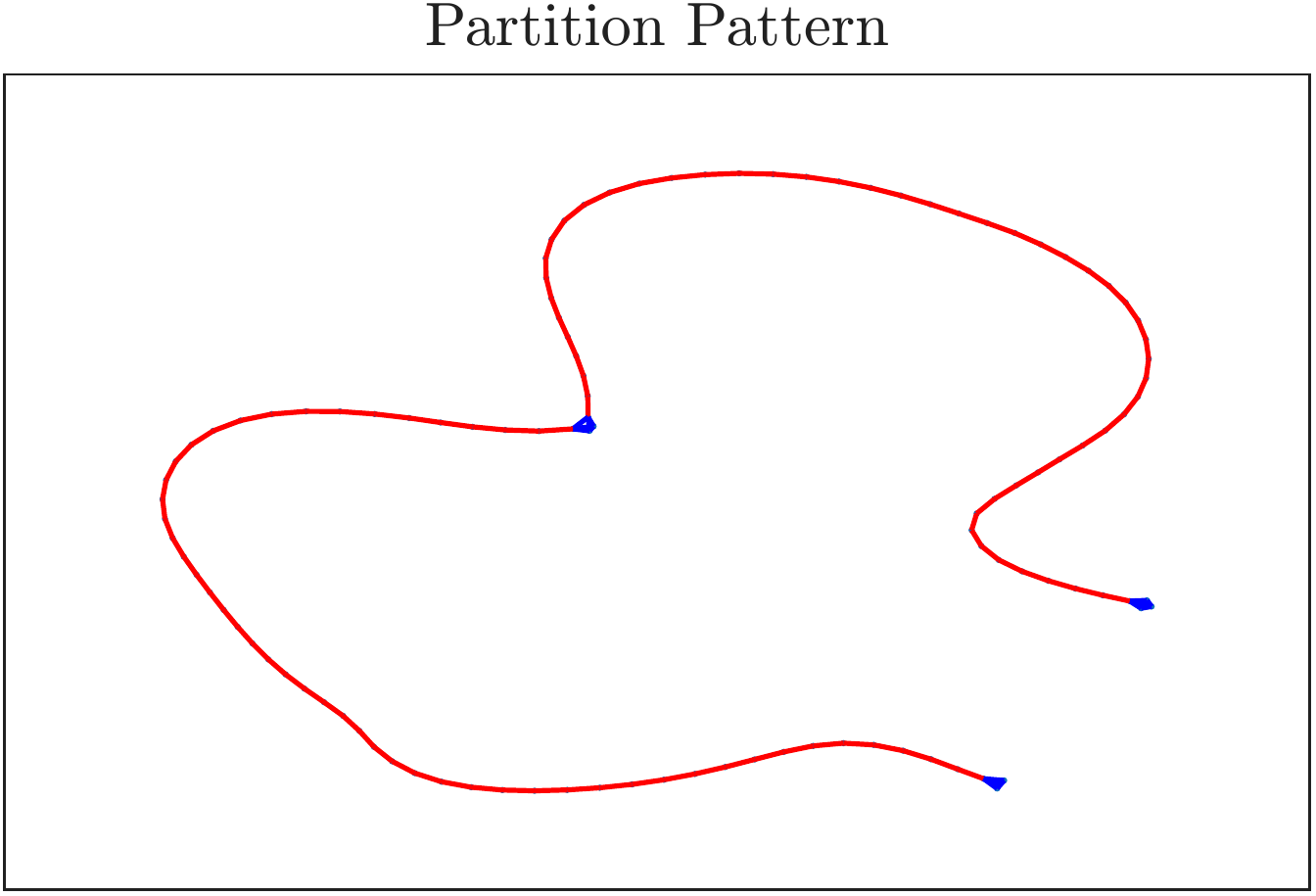}\\
					\multicolumn{1}{c}{\footnotesize{(a) Iterations}} &  \multicolumn{1}{c}{\footnotesize{(b) Communication (\# vectors)}}&
					\multicolumn{1}{c}{\footnotesize{(c) Graph decomposition}}                  
			\end{tabular}}
		\end{center} \vspace{-0.3cm}
        \caption{Decentralized optimization in the CTA form \eqref{eq:CTA}: MP-Jacobi and MP-Jacobi with surrogate vs. DGD-CTA.} \label{fig:CTA_compare_DGD}
		 \vspace{-0.4cm}
	\end{figure*}
    
We compare MP-Jacobi and its surrogate variant (based on the partial linearization in~\eqref{example:partial_linear}) against   DGD-CTA ~\cite{yuan2016convergence}, which is  gradient descent applied to  ~\eqref{eq:CTA} with stepsize $\gamma$.  The results are reported in Fig.~\ref{fig:CTA_compare_DGD}. Panel (a) (resp. (b)) plots  the optimality gap $\norm{\bx^\nu-\bx^\star}$ versus the iterations $\nu$ (resp. communications), while panel (c) reports the graphs and chosen partition patterns. Both MP-Jacobi and surrogate MP-Jacobi are markedly faster than DGD-CTA, with the advantage becoming pronounced on large-scale graphs (large $m$) and topologies with long ``thin'' structures (paths with large diameter). Moreover, the surrogate is essentially as effective as exact MP-Jacobi while significantly reducing per-iteration cost: incidences exchange only vectors (rather than functional messages), yielding substantial savings in communication. In terms of vector transmissions, surrogate MP-Jacobi is also consistently more communication-efficient than DGD-CTA.

To further assess scalability against existing decentralized  schemes, we also consider the original consensus-based decentralized  minimization 
$\min_{x\in\R^d}\;F(x):=\sum_{i=1}^m f_i(x)$, 
and solve it through its augmented CTA formulation~\eqref{eq:CTA}, over an instance of the dumbbell graph (Fig.~\ref{fig:scale_m_dopt}(b)). Here, $F$ is strongly convex, quadratic. We simulate quadratic functions   $f_i$'s   with     matrices $Q_{ii}$  generated by the {\tt MATLAB} command: 
$
{\tt B = randn(d,d); Qtmp = B'*B/m;Q\_ii=Qtmp}$ ${\tt +c*eye(d)}$
with   ${\tt c}$   chosen such that  global cost $\sum_{i=1}^mf_i(x_i)$ has condition number 100. We fix $W$ as the Metropolis matrix and tune   $\gamma$ in~\eqref{eq:CTA} so that the  MP-Jacobi and its surrogate variant, run on this augmented  problem, can achieve the same termination accuracy of  existing   decentralized  methods minimizing directly $F(\bx)$.  As representative of such methods,    we simulated  DGD--CTA, DGD--ATC, SONATA~\cite{sun2022distributed}, DIGing~\cite{nedic2017achieving}, and EXTRA~\cite{shi2015extra}. Let $x^\star$ be the minimizer of~$F(\bx)$. All the algorithms are terminated when    {$\norm{\bx^\nu-\mathbf{1}\otimes x^\star}/\sqrt{m}\leq10^{-3}$}, where $\bx^\nu$ is the stacked vectors of the agents' iterates $x_i^\nu$ produced by the different algorithms. 

 Fig.~\ref{fig:scale_m_dopt}(a) shows that   the iteration counts of DGD-CTA and DGD-ATC grow most rapidly with $m$, indicating poor scalability. In contrast, MP-Jacobi and its surrogate variant exhibit the mildest dependence on $m$ and remain consistently competitive, outperforming all baselines except EXTRA over the {\it full range} of values of $m$. While EXTRA is competitive for small networks,   beyond $m\approx 100$, our methods require fewer iterations, with the advantage increasing   with $m$.

\begin{figure*}[t!]
    \begin{center}
        \setlength{\tabcolsep}{0.0pt}  
        \scalebox{1}{\begin{tabular}{c@{\hspace{1cm}}c}
                \includegraphics[width=0.50\linewidth]{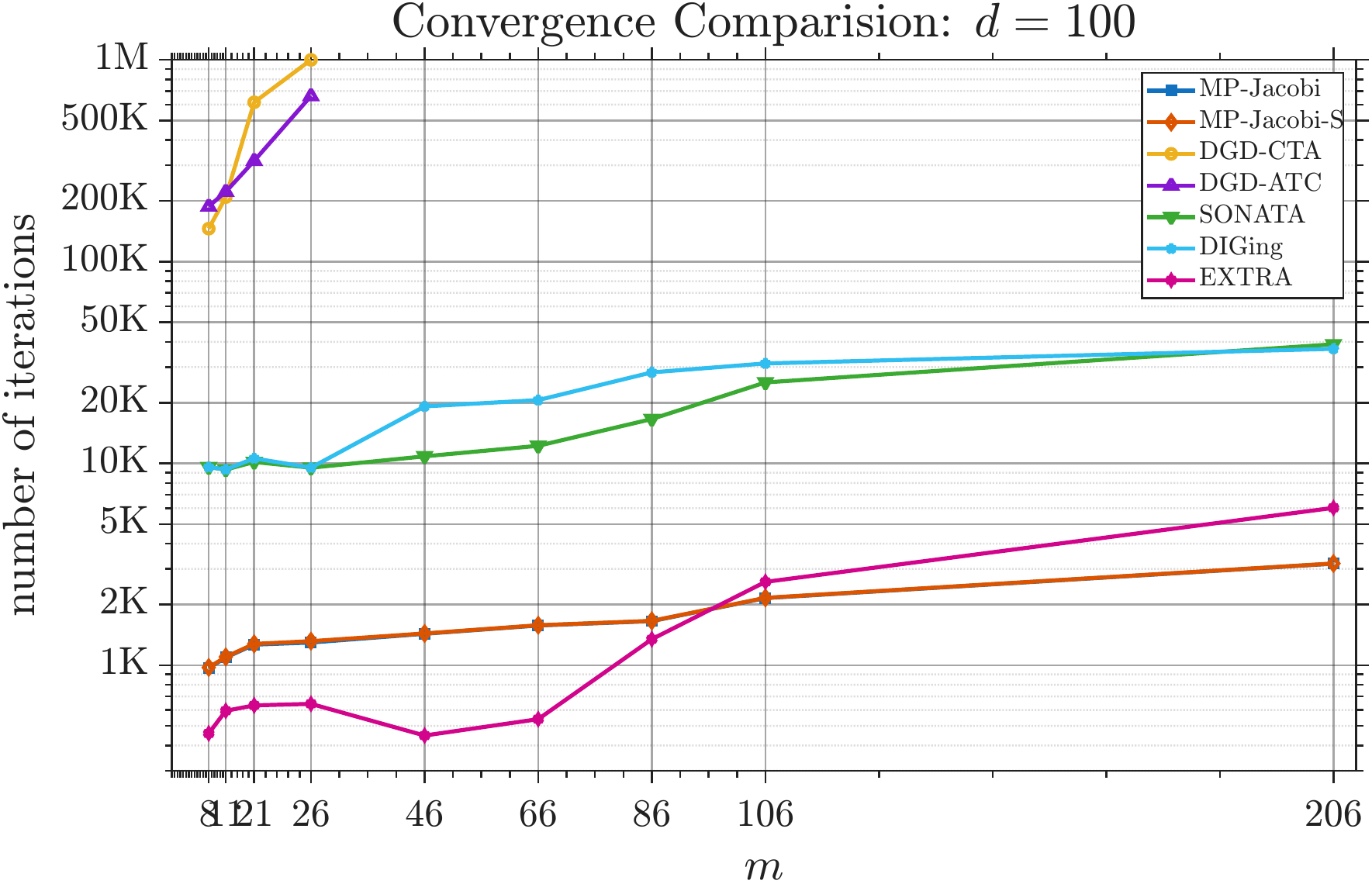} &  
                \raisebox{.45cm}{\includegraphics[width=0.42\linewidth]{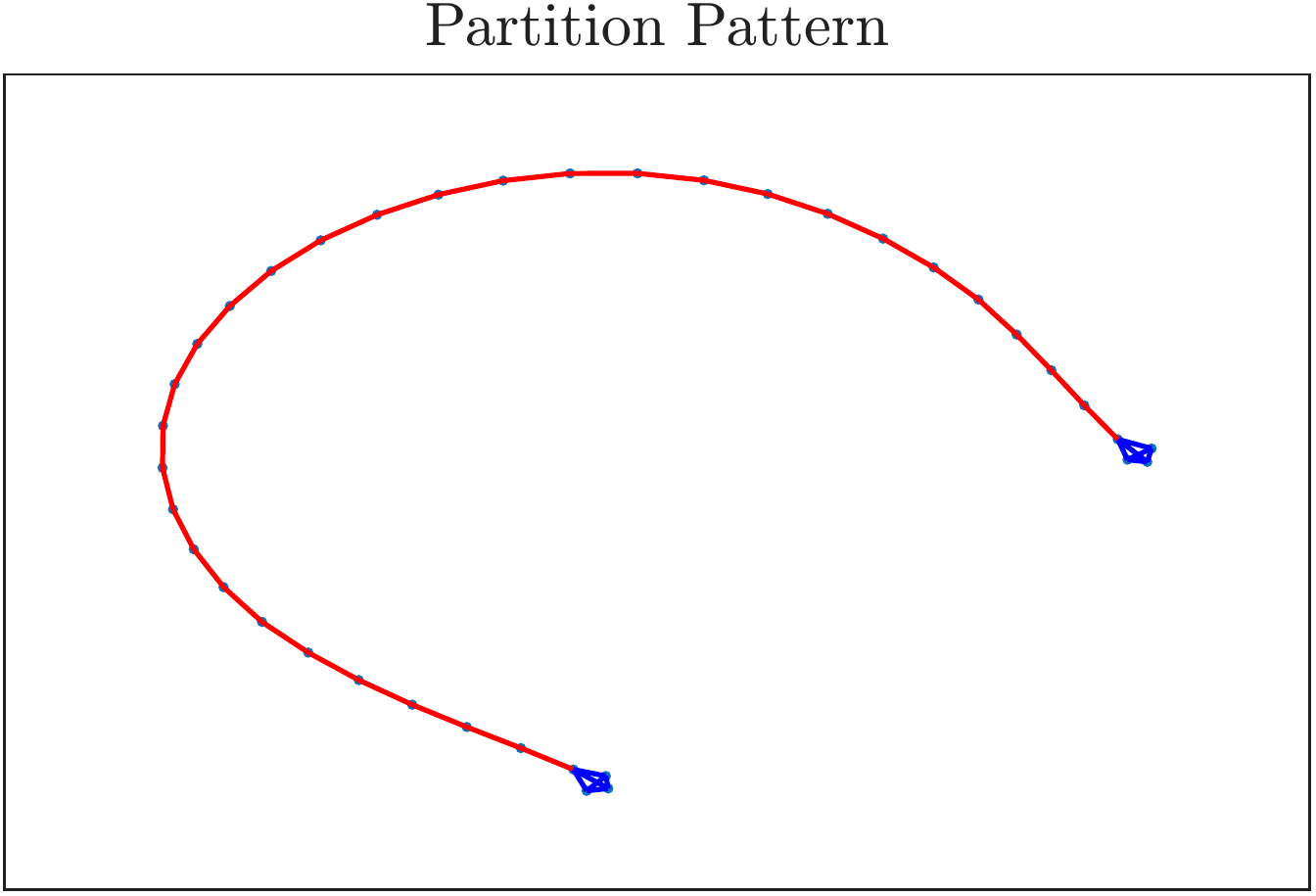}}\\
                \multicolumn{1}{c}{\footnotesize{(a) Scalability of $m$ for two partitions}} & \multicolumn{1}{c}{\footnotesize{(b) Graph partitions.}}  
        \end{tabular}}
        \caption{Decentralized optimization $\min_{x} \in \mathbb{R}^d\sum_{i}f_i(x)$: (a) \# iterations to $10^{-3}$ accuracy vs. $m$; MP-Jacobi vs. a variety of decentralized algorithms. (b) dumbbell graph   and graph partition (in red).}
        \label{fig:scale_m_dopt}\vspace{-0.6cm}
    \end{center}
\end{figure*}

\vspace{-0.4cm}
\subsection{Convex programs on hypergraphs} \label{sec:sim-hypergraps}
\vspace{-0.2cm}
We report three experiments that highlight the behavior of H-MP-Jacobi and its surrogate variant H-MP-Jacobi-S solving strongly convex quadratic programs over the following  hypergraphs:    \textbf{(i)}   hyper rings (no splitting); 
\textbf{(ii)}  loopy hypergraphs (two hyperedge splitting strategies); and 
\textbf{(iii)} ATC-induced hypergraph (splitting). 

 \noindent  \textbf{(i) Hyper rings (no splitting):}
Consider a strongly convex quadratic program on a \emph{hyper} {\it ring}, where adjacent hyperedges overlap on (almost) one nod--see Fig.~\ref{fig:hyperring}(c). We test Algorithm~\ref{alg:main_hyper} (H-MP-Jacobi) and its surrogate variant (H-MP-Jacobi-S). In H-MP-Jacobi-S, we apply a proximal-linear surrogate to the factor-to-variable message update~\eqref{eq:hyper_jump_factor2var}, using the diagonal of the Hessian of the exact message, while keeping the variable update~\eqref{eq:hyper_jump_updates:a} exact. As benchmark, we use the centralized  GD (albeit not implementable on the hypergraph), as   principled GD-type methods tailored to hypergraph couplings are not available. For clustering, we select one long path cluster containing $|\cE|-2$ hyperedges and all nodes incident to them, while assigning every remaining node to a singleton cluster--see Fig.~\ref{fig:hyperring}(c).    Fig.~\ref{fig:hyperring}(a)-(b) reports the results. Both H-MP-Jacobi and H-MP-Jacobi-S converge markedly faster than GD, with the advantage increasing as $|\cE|$ grows. Moreover, H-MP-Jacobi-S closely tracks H-MP-Jacobi, indicating that the surrogate preserves most of the performance while reducing per-iteration cost.

\begin{figure*}[ht]
		 \vspace{-0.3cm}
		\begin{center}
			\setlength{\tabcolsep}{0.0pt}  
			\scalebox{1}{\begin{tabular}{ccc}
					\includegraphics[width=0.33\linewidth]{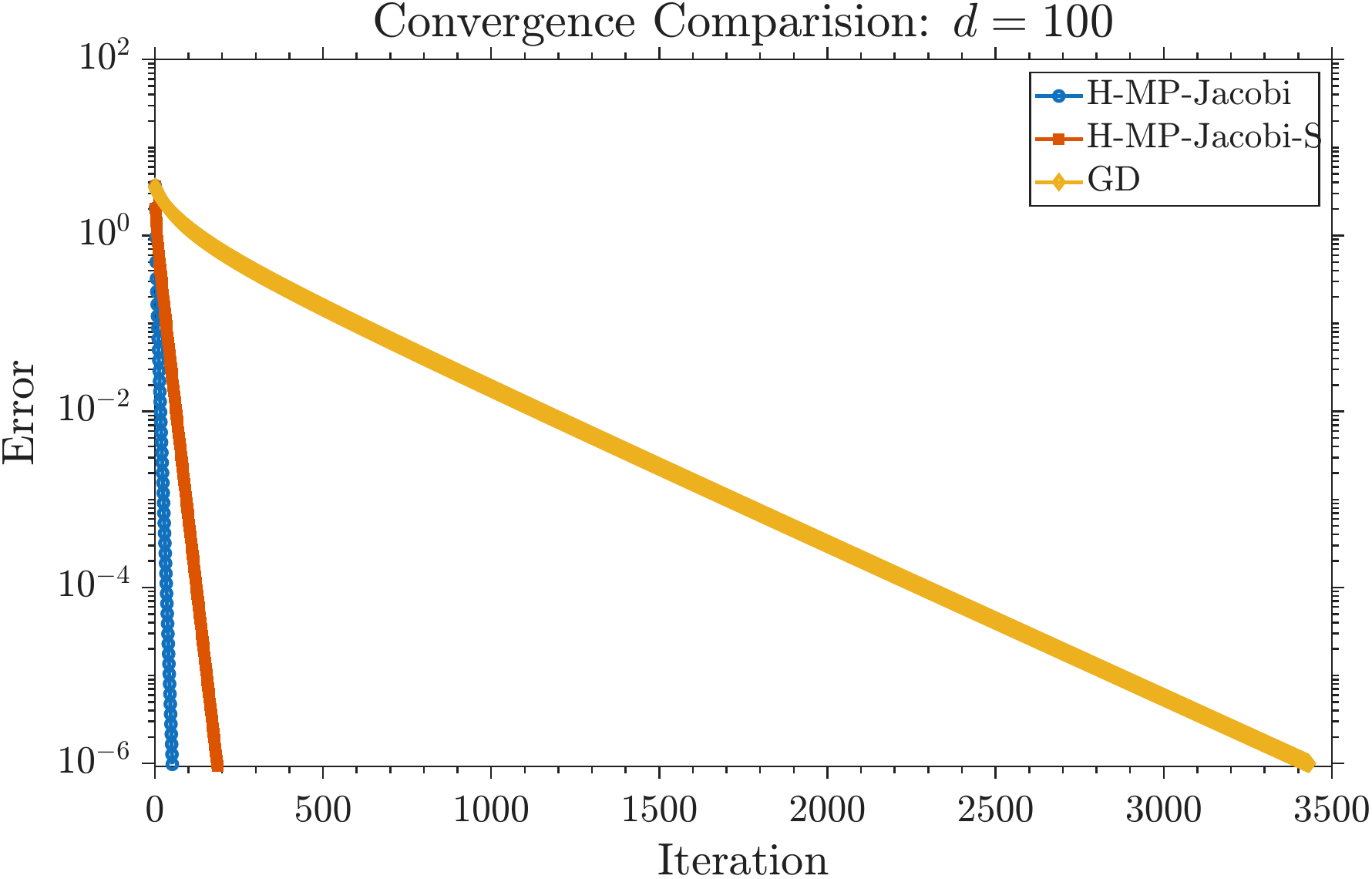}&
					\includegraphics[width=0.33\linewidth]{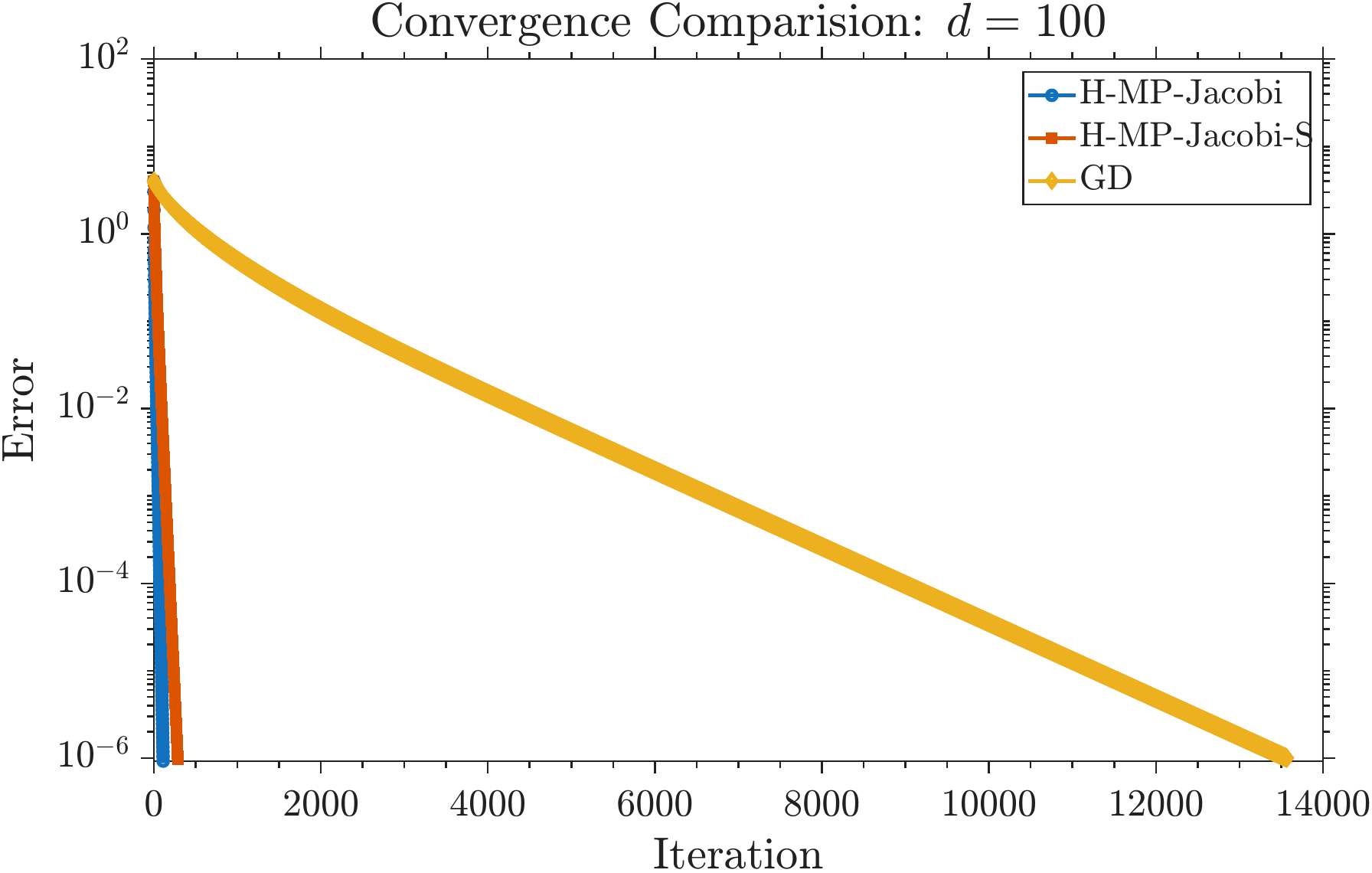}&
					\includegraphics[width=0.33\linewidth]{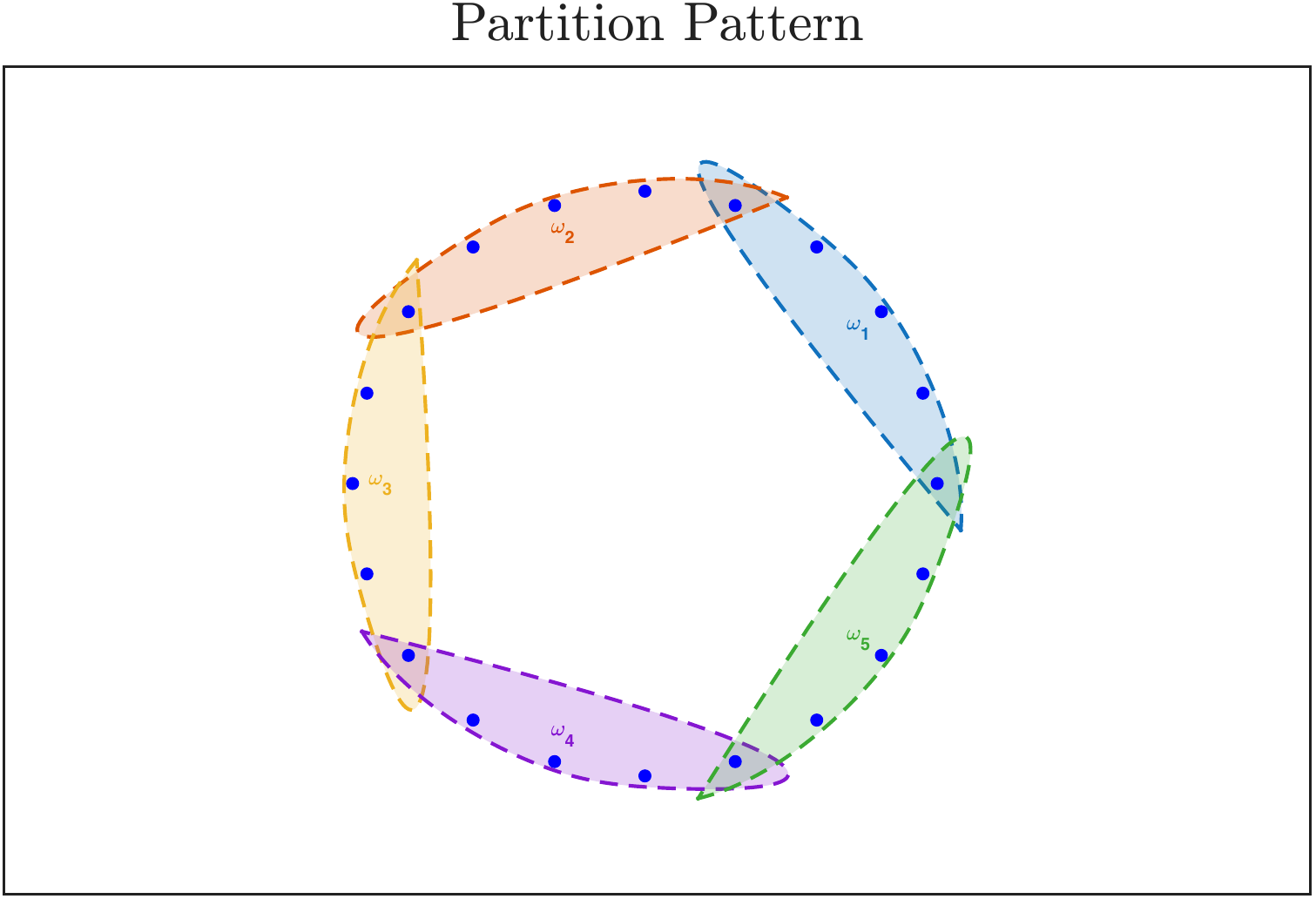}\\
					\multicolumn{1}{c}{\footnotesize{(a) $|\cE|=5|,|\omega|=5$}} &  \multicolumn{1}{c}{\footnotesize{(b) $|\cE|=20|,|\omega|=10$}}&
					\multicolumn{1}{c}{\footnotesize{(c) Hypergraph pattern}}                 
			\end{tabular}}
		\end{center}\vspace{-0.2cm}
        \caption{Strongly convex QP over a hyperring:   H-MP-Jacobi and H-MP-Jacobi-S.} 
        \label{fig:hyperring}
		\vspace{-0.4cm}
	\end{figure*}

\noindent \textbf{(ii)   Loopy hypergraphs (two hyperedge splitting strategies).} We next consider a toy quadratic program on a loopy hypergraph with $\cV=\{1,2,3,4\}$ and $\cE=\{\{1,2,3\},\{2,3,4\}\}$--see Fig.~\ref{fig:hyper_splitting_toy}. Since the factor graph contains a short cycle, we apply the surrogate hyperedge-splitting strategy to restore an acyclic intra-cluster factor structure.  We instantiate two splitting surrogates: the pairwise split~\eqref{example:pair_split} and the singleton split~\eqref{example:sing_splt}, denoted {\it H-MP-Jacobi-S1} and {\it H-MP-Jacobi-S2}, respectively (with stepsizes tuned for each method).  Fig.~\ref{fig:hyper_splitting_toy} shows that H-MP-Jacobi-S2 converges faster than H-MP-Jacobi-S1. This is consistent with the fact that the singleton split introduces fewer frozen coordinates (fewer reference points), thereby retaining more of the original coupling structure; the pairwise split is more aggressive and incurs a larger approximation error. The GD baseline is consistently, significantly slower.

\noindent\textbf{(iii)  ATC-induced hypergraph (splitting).} Finally, we test a QP induced by the DGD-ATC model~\eqref{eq:ATC}. Let $W$ be the gossip matrix over the graph with $\cV=[8]$ and edges $\{\{1,2\},\{1,3\},\{2,3\},\{3,4\},\{4,5\},\{5,6\},\{6,7\},\{6,8\},\{7,8\}\}$, and define hyperedges $\omega_i:=\{j:\,[W^2]_{ij}\neq 0\}$--see Fig.~\ref{fig:hyper_splitting_ATC}(c).  We compare two partitions/surrogates. For H-MP-Jacobi-S1, we use two clusters $\cC_1=\{1,2,3,4\}$ and $\cC_2=\{5,6,7,8\}$, with intra-cluster hyperedges $\omega_1=\{1,2,3,4\}$ and $\omega_2=\{5,6,7,8\}$. For H-MP-Jacobi-S2, we use $\cC_1=\{1,2,3,4,5\}$ and split the associated coupling into $\widetilde\omega_1^1=\{1,2,3\}$ and $\widetilde\omega_1^2=\{3,4,5\}$, while placing all remaining nodes in singleton clusters. 
The results in Fig.~\ref{fig:hyper_splitting_ATC} show that both surrogate designs perform similarly and significantly outperform DGD-ATC (GD). Moreover, H-MP-Jacobi-S1 is slightly faster than H-MP-Jacobi-S2, consistent with the fact that S1 preserves more coupling information.

\begin{figure*}[t!]
    \begin{center}
        \setlength{\tabcolsep}{2pt}  
        \begin{tabular}{@{}c@{\hspace{0.00\linewidth}}c@{\hspace{0.00\linewidth}}c@{}}
            \raisebox{-0.5\height}{\includegraphics[width=0.32\linewidth]{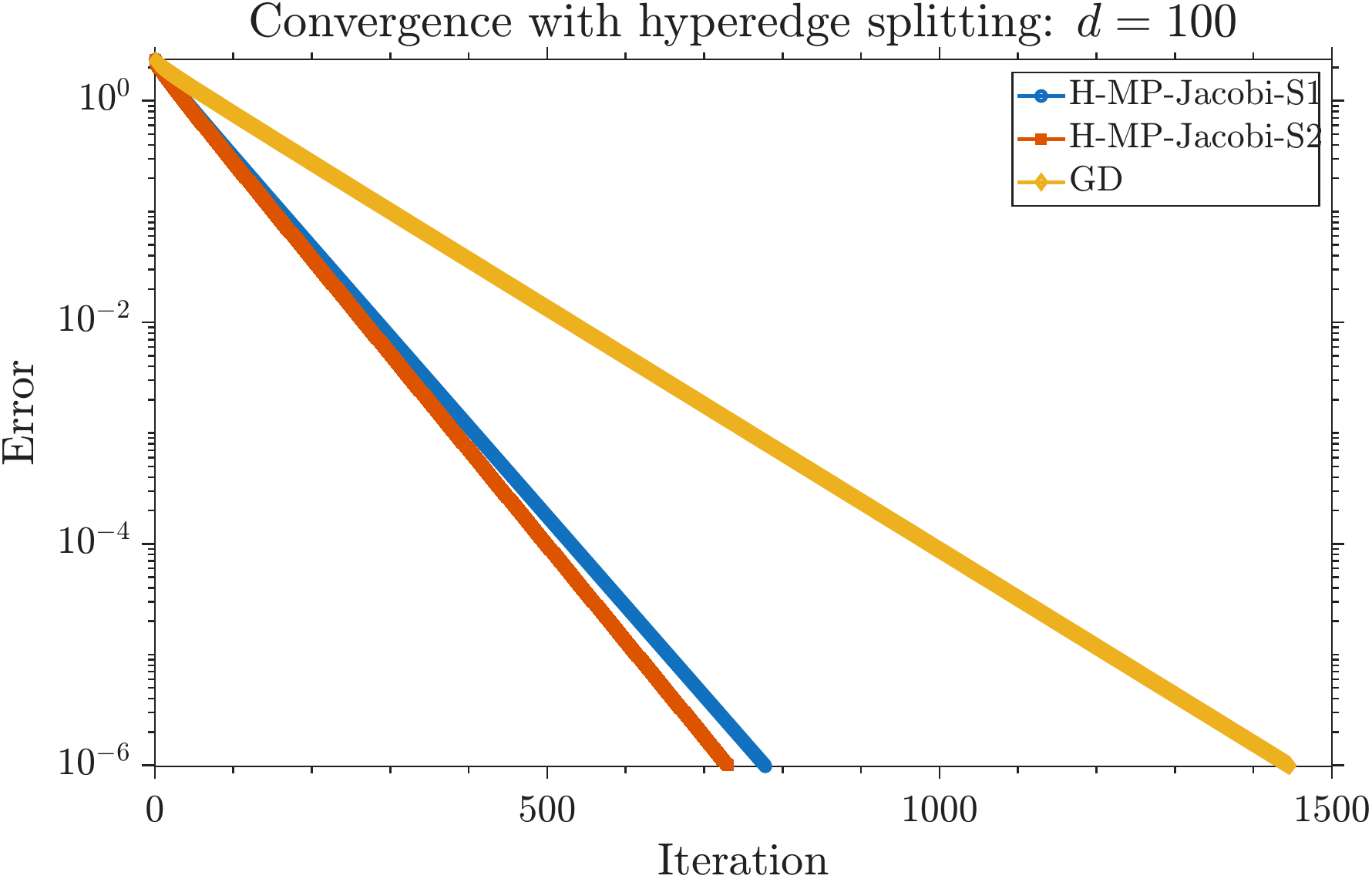}}&
            \raisebox{-0.5\height}{\includegraphics[width=0.32\linewidth]{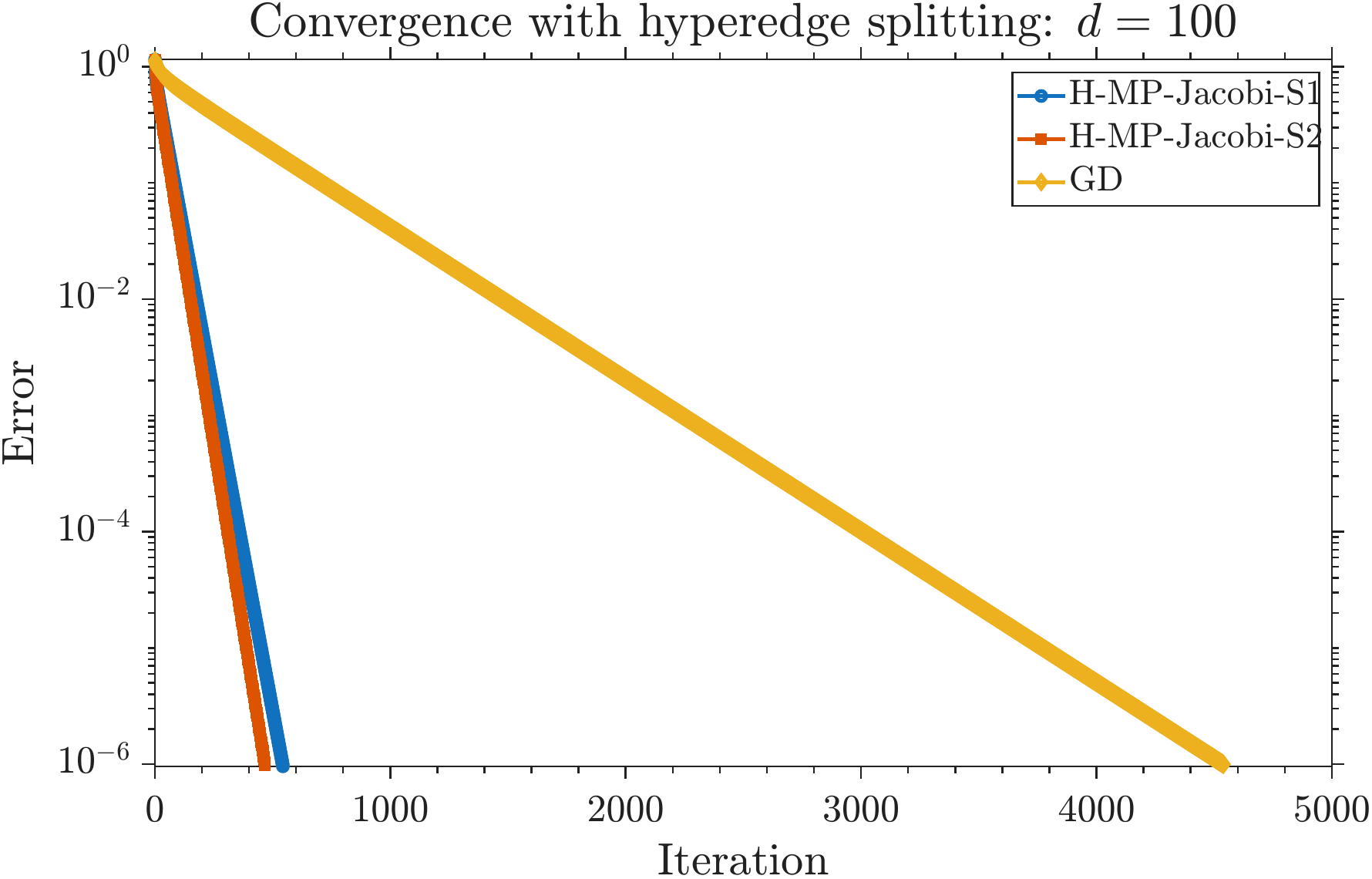}}&
            \raisebox{-0.45\height}{\includegraphics[width=0.35\linewidth]{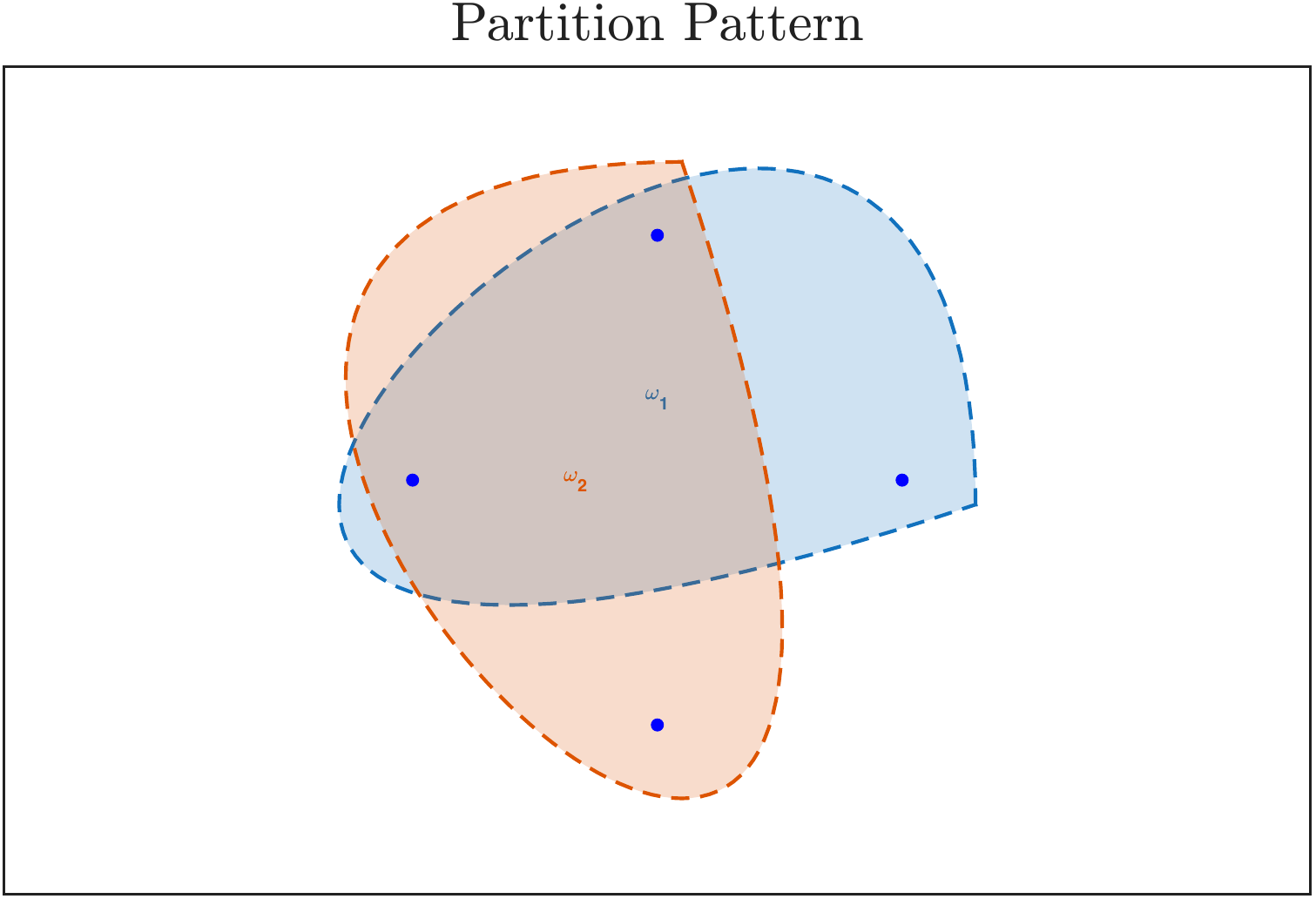}}\\[2pt]
            \footnotesize{(a) $\kappa=100$} &  
            \footnotesize{(b) $\kappa=500$}&
            \footnotesize{(c) Hypergraph pattern}                 
        \end{tabular}
    \end{center}\vspace{-0.2cm}
    \caption{Strongly convex QP: two examples of   surrogate hyperedge-splitting strategies to restore an acyclic intra-cluster factor structure.} 
\label{fig:hyper_splitting_toy}\vspace{-0.3cm}
\end{figure*}

    \begin{figure*}[ht]\vspace{-0.3cm}
		\begin{center}
			\setlength{\tabcolsep}{0.0pt}  
			\scalebox{1}{\begin{tabular}{ccc}       
            \raisebox{-0.5\height}
					{\includegraphics[width=0.33\linewidth]{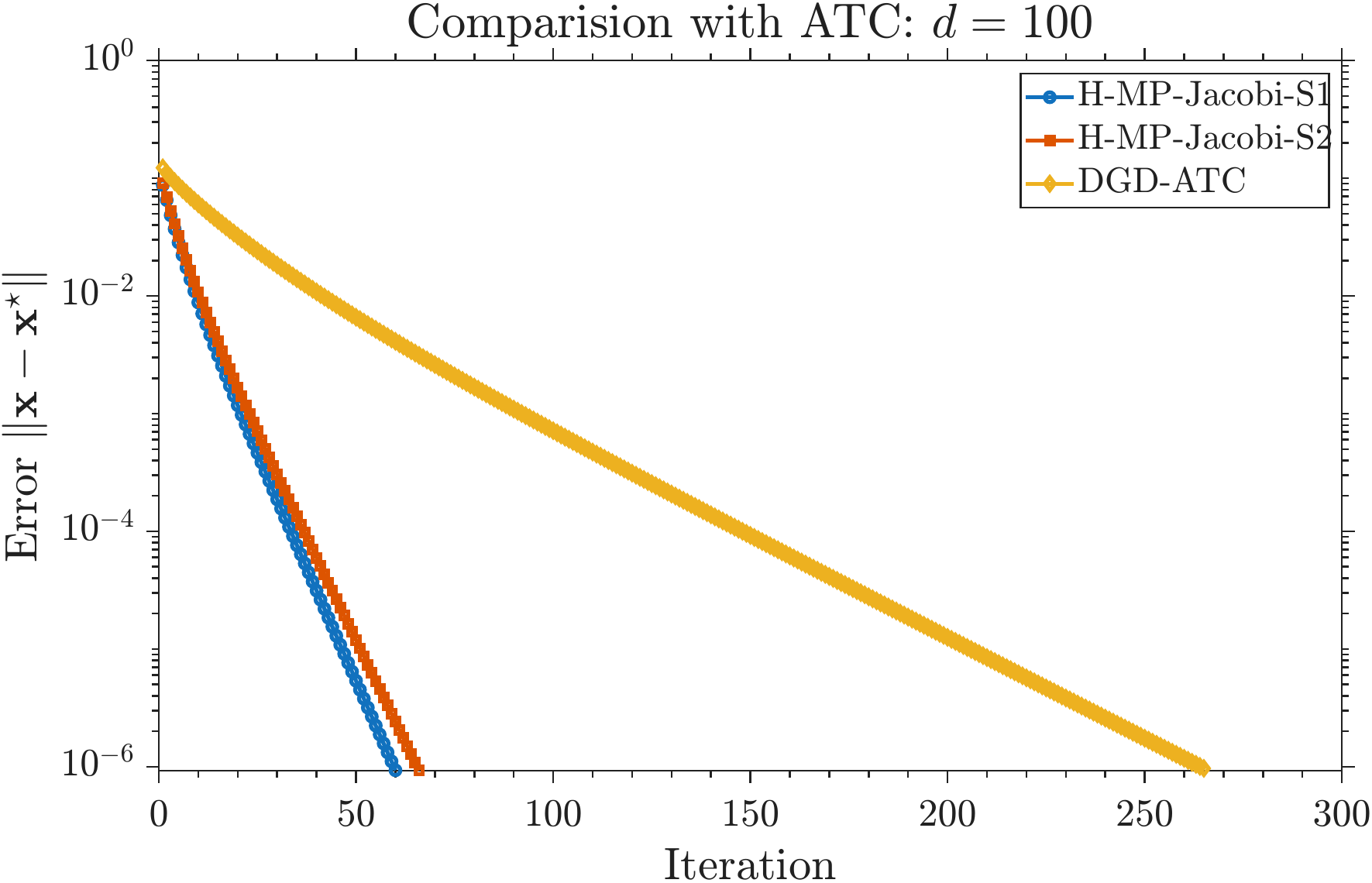}}&
                    \raisebox{-0.5\height}
					{\includegraphics[width=0.33\linewidth]{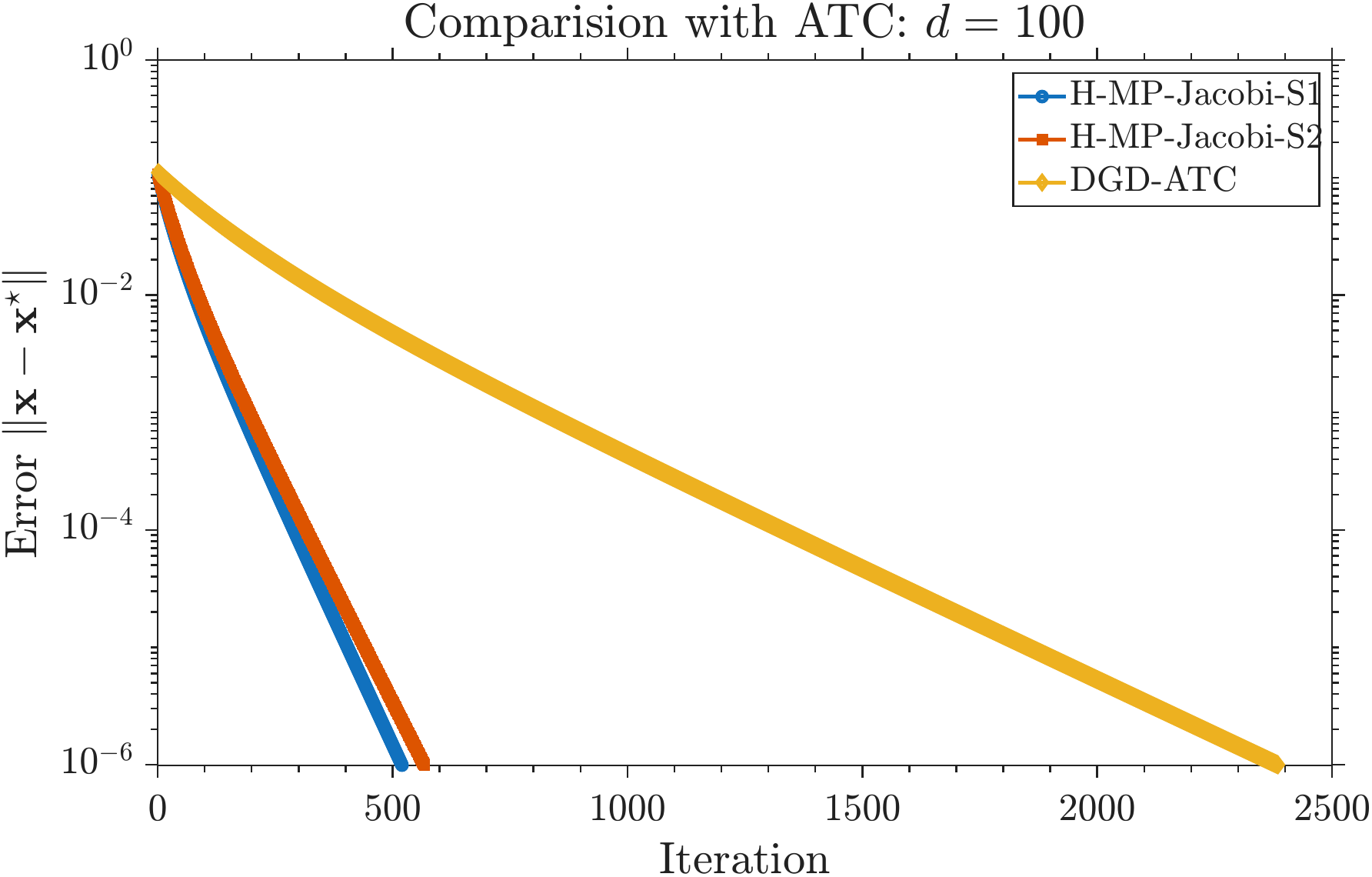}}&
                    \raisebox{-0.45\height}
					{\includegraphics[width=0.33\linewidth]{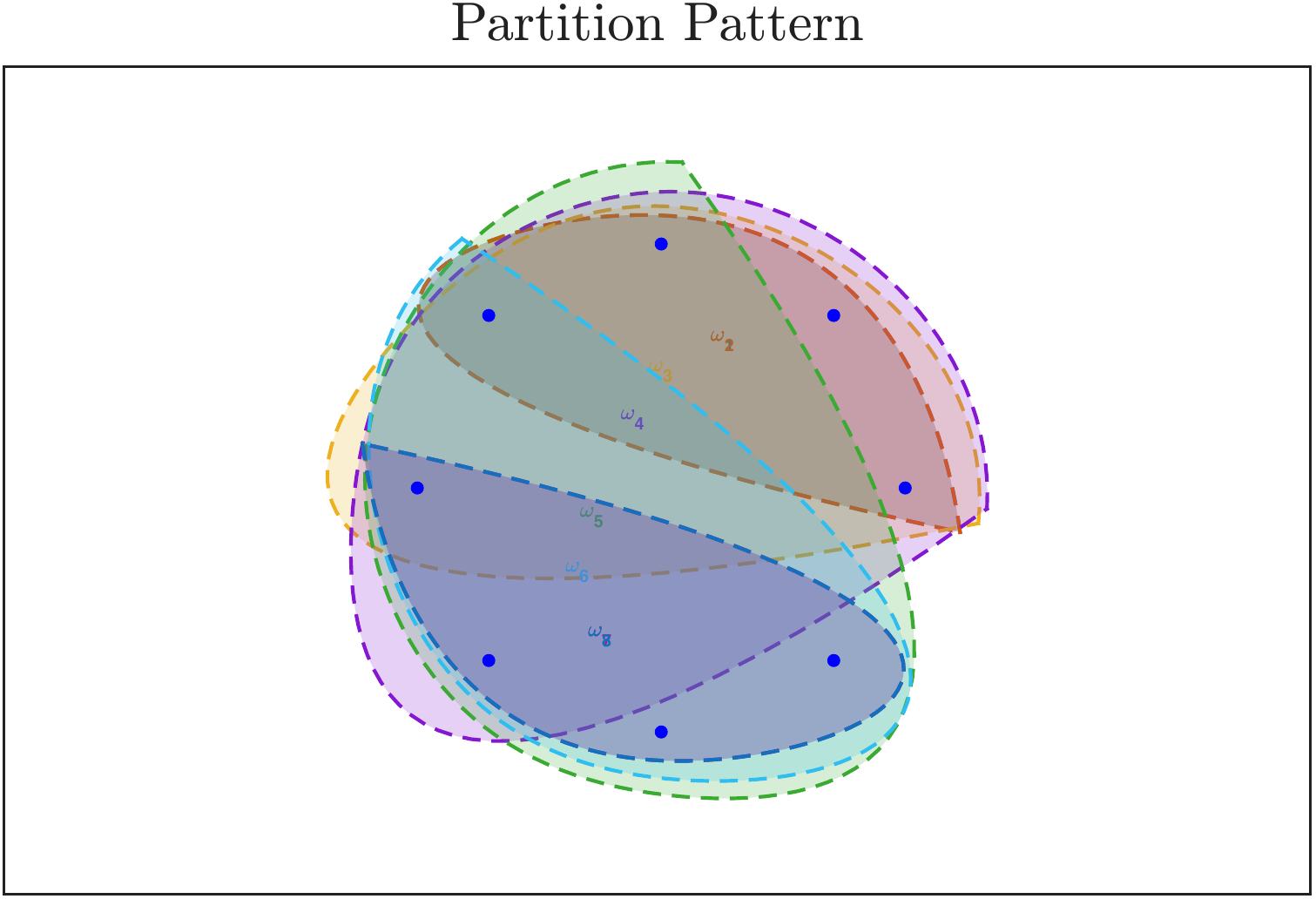}}\\
					\multicolumn{1}{c}{\footnotesize{(a) $\gamma=10^{-3}$}} &  \multicolumn{1}{c}{\footnotesize{(b) $\gamma=10^{-4}$}}&
					\multicolumn{1}{c}{\footnotesize{(c) Hypergraph pattern}}                 
			\end{tabular}}
		\end{center}\vspace{-0.2cm}
        \caption{DGD-ATC problem \eqref{eq:ATC}. Comparison of message passing-gossip algorithm and ATC over a hypergraph induced by ATC. $d=100$.} \label{fig:hyper_splitting_ATC}
		 \vspace{-0.4cm}
	\end{figure*}

\vspace{-0.7cm}

\section{Conclusions}\vspace{-0.2cm}
In this work, we introduced min-sum \emph{message passing} as an algorithmic primitive for decentralized optimization with localized couplings. Building on a fixed-point decomposition induced by a tree partition and its condensed graph, we proposed \emph{MP-Jacobi}: at each round, every agent performs a (damped) Jacobi update using the current intra-cluster messages and the current inter-cluster neighbors’ variables, and then updates one-hop pairwise messages along the intra-cluster tree edges. We establish convergence for (strongly) convex and nonconvex objectives,
with topology- and partition-explicit rates that quantify curvature/coupling
effects and guide clustering and scalability. To reduce computation and communication, we further developed structure-preserving surrogate algorithm and establish the analogous linear convergence for the strongly convex problem. Moreover, we extended the framework from graphs to hypergraphs; for dense  hypergraphs with heavy overlaps, we proposed a hyperedge-splitting strategy that enables the algorithm to remain convergent. Extensive numerical experiments on both graph and hypergraph instances corroborate the theory, demonstrate substantial efficiency gains over decentralized gradient-type baselines, and highlight the potential of MP-Jacobi as a scalable primitive for distributed optimization.

\vspace{-0.4cm}
\appendix 
\section*{Appendix} \vspace{-0.2cm}
\section{Proof of Proposition~\ref{prop:Jacobi_delay_refo}}\label{app:proof_prop_Jacobi-equivalence}
 
\textbf{(a) Proof of \eqref{eq:delay_reformulation}.} 
Fix a cluster $\cC_r$ and a root $i\in\cC_r$. View the tree $\cG_r=(\cC_r,\cE_r)$ as oriented towards $i$. For every $j\in\cC_r\setminus\{i\}$,
let $\mathtt{pa}(j)$ be the parent of $j$ on the path to $i$, let
$\mathtt{sub}(j)$ be the node set of the subtree rooted at $j$, and let
$d(i,j)$ be the hop distance from $i$ to $j$. Recall the reindexed recursion
\begin{equation}
\label{eq:msg-reindexed-proof}
\mu_{j\to i}^{\nu}(x_i)=\min_{x_j}\Big\{\phi_j(x_j)+\psi_{ij}(x_i,x_j)+\!\!\sum_{k\in\cNin_j\setminus\{i\}}\!\!\mu_{k\to j}^{\nu-1}(x_j)+\!\!\sum_{k\in\cNout_j}\!\!\psi_{jk}(x_j,x_k^{\nu-1})\Big\}.
\end{equation}\\
\noindent $\bullet$ \textbf{Step 1 (subtree representation)}: Given $j\in\cC_r\setminus\{i\}$,
we prove by induction on the depth $d(i,j)$, that the message from $j$ to
its parent at the appropriate iteration, $\mu_{j\to \mathtt{pa}(j)}^{\nu-d(i,j)+1}\bigl(x_{\mathtt{pa}(j)}\bigr)$,  can be written  as  the  optimal value of a  local subtree problem, that is   
\begin{equation}
\label{eq:subtree-form}
\mu_{j\to \mathtt{pa}(j)}^{\,\nu-d(i,j)+1}\bigl(x_{\mathtt{pa}(j)}\bigr)
=\min_{x_{\mathtt{sub}(j)}}\qty{
\begin{aligned}
&\sum_{\ell\in\mathtt{sub}(j)}\phi_\ell(x_\ell)
+\psi_{\mathtt{pa}(j),j}(x_{\mathtt{pa}(j)},x_j)
\\
&+\sum_{\substack{(\ell,m)\in\cE_r\\ \ell,m\in\mathtt{sub}(j)}}\psi_{\ell m}(x_\ell,x_m)+\sum_{\substack{\ell\in\mathtt{sub}(j)\\k\in\cNout_\ell}}\psi_{\ell k}\big(x_\ell,x_k^{\,\nu-d(i,\ell)}\big)
\end{aligned}
}.
\end{equation}

\emph{Base case.} If $j$ is a leaf, then $\mathtt{sub}(j)=\{j\}$ and $\cNin_j\setminus\{\mathtt{pa}(j)\}=\emptyset$. Evaluating \eqref{eq:msg-reindexed-proof} at iteration $\nu-d(i,j)+1$ reads
\[
\mu_{j\to \mathtt{pa}(j)}^{\,\nu-d(i,j)+1}(x_{\mathtt{pa}(j)})
=\min_{x_j}\Big\{\phi_j(x_j)+\psi_{\mathtt{pa}(j),j}(x_{\mathtt{pa}(j)},x_j)
+\sum_{k\in\cNout_j}\psi_{jk}\!\big(x_j,x_k^{\,\nu-d(i,j)}\big)\Big\},
\]
which is exactly \eqref{eq:subtree-form} when $\mathtt{sub}(j)=\{j\}$ (no intra-subtree edges).

 \emph{Inductive step.} Let $j$ be any internal node of $\mathcal G_r$.   Apply
\eqref{eq:msg-reindexed-proof} with $i=\mathtt{pa}(j)$ and
$\nu$ replaced by $\nu-d(i,j)+1$:
\begin{align}
&\mu_{j\to \mathtt{pa}(j)}^{\,\nu-d(i,j)+1}(x_{\mathtt{pa}(j)})=\nonumber\\
&=\min_{x_j}\Big\{\phi_j(x_j)+\psi_{\mathtt{pa}(j),j}(x_{\mathtt{pa}(j)},x_j)
+\!\!\!\sum_{h:\ \mathtt{pa}(h)=j}\mu_{h\to j}^{\,\nu-d(i,j)}(x_j)
+\!\!\!\sum_{k\in\cNout_j}\psi_{jk}\big(x_j,x_k^{\nu-d(i,j)}\big)\Big\}.\label{eq:internal-j-expansion}
\end{align}
Every child $h$ of $j$ satisfies $d(i,h)=d(i,j)+1$, hence
$
\mu_{h\to j}^{\,\nu-d(i,j)}(x_j)
=\mu_{h\to j}^{\,\nu-d(i,h)+1}(x_j).
$
Applying the induction  hypothesis \eqref{eq:subtree-form} 
to the child $h$ of $j$ (i.e., replace therein 
$j$   by $h$ and $\mathtt{pa}(h)=j$)  yields \begin{align}\label{eq:subtree-form-child}
&\mu_{h\to j}^{\,\nu-d(i,h)+1}(x_j)
=\nonumber\\&=\min_{x_{\mathtt{sub}(h)}}\qty{
\begin{aligned}
&\sum_{\ell\in\mathtt{sub}(h)}\phi_\ell(x_\ell)+\psi_{j,h}(x_j,x_h)+\sum_{\substack{(\ell,m)\in\cE_r,\\ \ell,m\in\mathtt{sub}(h)}}\psi_{\ell m}(x_\ell,x_m)\\
&+\sum_{\substack{\ell\in\mathtt{sub}(h),k\in\cNout_\ell}}\psi_{\ell k}\big(x_\ell,x_k^{\,\nu-d(i,\ell)}\big)
\end{aligned}
}.
\end{align}

Substituting \eqref{eq:subtree-form-child}  into \eqref{eq:internal-j-expansion} and merging the minimizations over the
pairwise-{\it disjoint} subtrees $\{\mathtt{sub}(h)\}_{h:\,\mathtt{pa}(h)=j}$ (together with $x_j$)
yields  
\begin{align*}
&\mu_{j\to \mathrm{pa}(j)}^{\,\nu-d(i,j)+1}(x_{\mathrm{pa}(j)})
=\nonumber\\
&\min_{x_j,\ \{x_{\mathrm{Sub}(h)}\}_{h:\,\mathrm{pa}(h)=j}}
\qty{
\begin{aligned}
&\phi_j(x_j)+\psi_{\mathtt{pa}(j),j}(x_{\mathtt{pa}(j)},x_j)
+\sum_{h:\,\mathtt{pa}(h)=j}\psi_{j,h}(x_j,x_h)\\
&+\sum_{h:\,\mathtt{pa}(h)=j}\ \sum_{\ell\in\mathtt{sub}(h)}\phi_\ell(x_\ell)
+\sum_{h:\,\mathtt{pa}(h)=j}\!\!\sum_{\substack{(\ell,m)\in\cE_r\\ \ell,m\in\mathtt{sub}(h)}}\!\!\psi_{\ell m}(x_\ell,x_m)\\
&+\sum_{h:\,\mathtt{pa}(h)=j} \sum_{\substack{\ell\in\mathtt{sub}(h)\\k\in\cNout_\ell}}\psi_{\ell k}\!\big(x_\ell,x_k^{\,\nu-d(i,\ell)}\big)
+\sum_{k\in\cNout_j}\psi_{jk}\!\big(x_j,x_k^{\,\nu-d(i,j)}\big)
\end{aligned}
}.
\end{align*}
Now note that $\mathtt{sub}(j)=\{j\}\cup\bigcup_{h:\,\mathtt{pa}(h)=j}\mathtt{sub}(h)$ is a disjoint union, and there are no intra-cluster edges across different child subtrees other than the edges   $(j,h)$. Regrouping  the sums over $\{j\}$ and $\mathtt{sub}(h)$ pieces, we we recover exactly the structure in \eqref{eq:subtree-form} for the subtree rooted at $j$. 
 This completes the inductive step and proves the subtree representation for all $j\neq i$.\\
\noindent $\bullet$ \textbf{Step 2 (sum of messages at the root):} 
The primal subproblem reads
\[
\hat x_i^{\nu+1}\in\argmin_{x_i}\Big\{\phi_i(x_i)+\sum_{j\in\cNin_i}\mu_{j\to i}^{\nu}(x_i)+\sum_{k\in\cNout_i}\psi_{ik}(x_i,x_k^{\nu})\Big\}.
\]
For each neighbor  $j\in\cNin_i$, we have $d(i,j)=1$ and $\mathtt{pa}(j)=i$. Therefore,   \eqref{eq:subtree-form} (with depth $d(i,j)=1$) gives   $\mu_{j\to i}^{\nu}(x_i)$ equal to the   optimal cost value of the subproblem  associated with the subtree $\mathtt{sub}(j)$ w.r.t.   $x_{\mathtt{sub}(j)}$.  
The subtrees $\{\mathtt{sub}(j):j\in\cNin_i\}$ are pairwise disjoint and their union is $\cC_r\setminus\{i\}$. Summing the objectives and merging the minimizations yields
\[
\sum_{j\in\cNin_i}\mu_{j\to i}^{\nu}(x_i)
=\min_{x_{\cC_r\setminus\{i\}}}\qty{
\begin{aligned}
&\;\sum_{j\in\cC_r\setminus\{i\}}\phi_j(x_j)
+\sum_{(j,k)\in\cE_r}\psi_{jk}(x_j,x_k)\\
&\;+\sum_{\substack{j\in\cC_r\setminus\{i\}, k\in\cNout_j}}\!\!\psi_{jk}\!\big(x_j,x_k^{\,\nu-d(i,j)}\big)
+\sum_{j\in\cNin_i}\psi_{ij}(x_i,x_j)
\end{aligned}
}.
\]
Adding the terms $\phi_i(x_i)$ and $\sum_{k\in\cNout_i}\psi_{ik}(x_i,x_k^{\nu})$  completes the joint intra-cluster cost and all boundary terms, now including the contribution $j=i$ (for which $d(i,i)=0$). Therefore
\[
\hat x_i^{\,\nu+1}\in\argmin_{x_i}\min_{x_{\cC_r\setminus\{i\}}}
\Bigg\{
\sum_{j\in\cC_r}\phi_j(x_j)
+\sum_{(j,k)\in\cE_r}\psi_{jk}(x_j,x_k)
+\sum_{\substack{j\in\cC_r,\,k\in\cNout_j}}\psi_{jk}\!\big(x_j,\,x_k^{\,\nu-d(i,j)}\big)
\Bigg\},
\]
which is exactly \eqref{eq:delay_reformulation}. 

\medskip
\noindent\textbf{(b) Proof of \eqref{eq:coordinate_min_delay}.} By maximality~\eqref{eq:nonoverlap}, every edge with both endpoints in
$\cC_r$ belongs to $\cE_r$, so for any fixed external block $z_{\overline{\cC}_r}$
the dependence of $\Phi(x_{\cC_r},z_{\overline{\cC}_r})$ on $x_{\cC_r}$ is
exactly through
\[
\sum_{j\in\cC_r}\phi_j(x_j)
+\sum_{(j,k)\in\cE_r}\psi_{jk}(x_j,x_k)
+\sum_{\substack{j\in\cC_r,\,k\in\overline{\cC}_r}}\psi_{jk}(x_j,z_k),
\]
all remaining terms being independent of $x_{\cC_r}$. Moreover, ~\eqref{eq:nonoverlap2} implies that each external node
$k\in\overline{\cC}_r$ is attached to a unique leaf $j_k\in\cB_r$, so the
boundary term can be rewritten as
\[
\sum_{\substack{j\in\cC_r,\,k\in\cNout_j}}
\psi_{jk}\bigl(x_j,x_k^{\,\nu-d(i,j)}\bigr)
=
\sum_{k\in\overline{\cC}_r}
\psi_{j_k k}\bigl(x_{j_k},x_k^{\,\nu-d(i,j_k)}\bigr),
\]
which depends on the delays only through the vector
$\bd_i := (d(i,j_k))_{k\in\overline{\cC}_r}$ and the block
$x_{\overline{\cC}_r}^{\,\nu-\bd_i}:=\bigl(x_k^{\,\nu-d(i,j_k)}\bigr)_{k\in\overline{\cC}_r}$. Combining the two observations above, we obtain, for some constant
$c_i^\nu$ independent of $x_{\cC_r}$,
\[
\sum_{j\in\cC_r}\phi_j(x_j)
+\sum_{(j,k)\in\cE_r}\psi_{jk}(x_j,x_k)
+\sum_{\substack{j\in\cC_r,\,k\in\cNout_j}}
\psi_{jk}\bigl(x_j,x_k^{\,\nu-d(i,j)}\bigr)
=
\Phi\bigl(x_{\cC_r},x_{\overline{\cC}_r}^{\,\nu-\bd_i}\bigr)-c_i^\nu.
\]
Thus, \eqref{eq:delay_reformulation} is equivalent to  \eqref{eq:coordinate_min_delay}, which completes the proof. \qed
\vspace{-0.3cm}

\section{Proof of Theorem \ref{thm:convergence_scvx_surrogate_uniform}}\label{proof_th_surrogate}
We prove  the theorem in two steps: (i) we cast Algorithm~\ref{alg:main1_surrogate} as a damped surrogate block-Jacobi method with bounded delay; and then  (ii)   establish its convergence.  
\paragraph{\it $\bullet$ Step 1: Algorithm \ref{alg:main1_surrogate} as a surrogate block-Jacobi method with   delays.}

Fix a cluster $\cC_r$ and a root $i\in\cC_r$. We view the tree $\cG_r=(\cC_r,\cE_r)$ as oriented towards $i$ (suppress the index $i$ in the notation), and write $(j,k)$ for the directed edge pointing to $i$ (i.e., $d(i,k)<d(i,j)$). 
Following similar steps as in  Section \ref{subsec:Step1}, we eliminate $\widetilde\mu_{k\to i}^\nu$ by reapplying \eqref{message_update_surrogate} and repeating the substitution recursively along the branches of the tree $\cG_r$. We obtain the following. 

\begin{proposition}
\label{prop:equiv_refo_surrogate}
Under Assumptions~\ref{asm:on_the_partion} and~\ref{asm:graph}, Algorithm~\ref{alg:main1_surrogate} can be rewritten in the equivalent form: for any $i\in\cC_r$ and $r\in [p]$,  
\begin{subequations}
    \begin{align}
 &  x_i^{\nu+1} =x_i^{\nu}+\tau_r^\nu \big(\hat x_i^{\nu+1}-x_i^{\nu}\big),\label{eq:delay_reformulation-cvx-comb_sur}\\
   & \hat x_i^{\nu+1}\in\argmin_{x_i}\min_{x_{\cC_r\setminus\{i\}}}\qty{
\begin{aligned}
&\sum_{j\in\cC_r}\widetilde\phi_j\qty(x_j;x_j^{\nu-d(i,j)})+\sum_{(j,k)\in\cE_r}\widetilde\psi_{jk}\qty(x_j,x_k;x_j^{\nu-d(i,j)},x_k^{\nu-d(i,j)})\\
&+\sum_{\substack{j\in\cC_r,k\in\cNout_j}}\widetilde\psi_{jk}\qty(x_j,x_k^{\nu-d(i,j)};x_j^{\nu-d(i,j)},x_k^{\nu-d(i,j)})
\end{aligned}\label{eq:reformulation_surrogate}
}.
\end{align}\end{subequations}
If, in addition, Assumption~\ref{asm:nonverlap} holds,  \eqref{eq:reformulation_surrogate}  reduces to the following  block-Jacobi update with delays:
\begin{equation}
\label{eq:surrogate_subprob_compact_fix}
\hat x_i^{\nu+1}\in
\argmin_{x_i}\;\min_{x_{\cC_r\setminus\{i\}}}
\ \widetilde\Phi_r\!\Big(\big(x_{\cC_r},\,x_{\overline{\cC}_r}^{\,\nu-\bdelta_i}\big)\,;\,
\big(x_{\cC_r}^{\,\nu-\bd_i},\,x_{\cE_r}^{\,\nu-\bD_i},\,x_{\overline{\cC}_r}^{\,\nu-\bdelta_i}\big)\Big),
\end{equation}
where $\bd_i := (d(i,j))_{j\in\cC_r}$ and 
$\bdelta_i := (\delta_{i,k})_{k\in\overline{\cC}_r}$ with
\[
\delta_{i,k} :=
\begin{cases}
d\big(i,\pi_r(k)\big), & k\in\cN_{\cC_r},\\
0, & k\notin\cN_{\cC_r},
\end{cases}
\]
and $\pi_r(k)\in\cC_r$ denotes the unique internal neighbor of $k$ guaranteed by~\eqref{eq:nonoverlap2}.
The edge-stacked delay vector is  $\bD_i := \big(d(i,j);\,d(i,k)\big)_{(j,k)\in\cE_r}$.  It holds $\max\big\{\|\bd_i\|_\infty,\ $ $\|\bdelta_i\|_\infty,\ \|\bD_i\|_\infty\big\}
\ \le\ D_r$, for all $i\in\cC_r$ and $r\in [p]$.
\end{proposition}

\vspace{-0.2cm}

\paragraph{\it $\bullet$ Step 2: Convergence   of \eqref{eq:surrogate_subprob_compact_fix}.}  
  For any $r\in[p]$ and $i\in\cC_r$, introduce the virtual   \emph{non-delayed} block update (used only in the analysis):
\[
\bar{x}_{\cC_r}^{\nu+1}\in\argmin_{x_{\cC_r}}
\ \widetilde\Phi_r\!\Big(\big(x_{\cC_r},\,x_{\overline{\cC}_r}^{\,\nu}\big)\,;\,
\big(x_{\cC_r}^{\,\nu},\,x_{\overline{\cC}_r}^{\,\nu},\,x_{\cE_r}^{\,\nu}\big)\Big),
\]

The next two lemmas establish a descent recursion for $\Phi$ along $\{\bx^\nu\}$. Their proofs are the analogous to Lemma \ref{le:scvx_descent_contraction}--\ref{le:scvx_descent} in Section~\ref{subsec:convergence_exact}, with one additional ingredient: we relate the aggregated surrogate to the aggregated objective \eqref{eq:Phi_r_def_simplified} via the upper bound condition in Assumption~\ref{assumption:surrogation}.(ii).

\begin{lemma}
\label{le:scvx_descent_contraction_surrogate}
Under Assumption \ref{assumption:surrogation} and any stepsize choice satisfying $\tau_r^\nu\geq0$ and  $\sum_{r=1}^p\tau_r^\nu\leq 1$, the following holds:  
\begin{equation}
\label{eq:scvx_descent_contraction_surrogate}
\begin{aligned}
\Phi(\bx^{\nu+1})\leq\Phi(\bx^\nu)+\sum_{r\in[p]}\tau_r^\nu\qty[-\frac{\norm{P_r\nabla\Phi(\bx^\nu)}^2}{2\tilde L_r}+\frac{\tilde L_r}{2}\norm{P_r(\hat \bx^{\nu+1}-\bar\bx^{\nu+1})}^2].
\end{aligned}
\end{equation} 
\end{lemma}
\begin{proof}
We first to bound of the nondelay iterate for the surrogate. Note that 
\begin{equation*}
\begin{aligned}
&\widetilde\Phi_r\!\Big(\big(\hat{x}_{\cC_r}^{\nu+1},\,x_{\overline{\cC}_r}^{\,\nu}\big)\,;\,
\big(x_{\cC_r}^{\,\nu},\,x_{\overline{\cC}_r}^{\,\nu},\,x_{\cE_r}^{\,\nu}\big)\Big)\\
\leq&      \widetilde\Phi_r\!\Big(\big(\bar{x}_{\cC_r}^{\nu+1},\,x_{\overline{\cC}_r}^{\,\nu}\big)\,;\,   \big(x_{\cC_r}^{\,\nu},\,x_{\overline{\cC}_r}^{\,\nu},\,x_{\cE_r}^{\,\nu}\big)\Big)+\underbrace{\inner{\widetilde\Phi_r\!\Big(\big(\bar{x}_{\cC_r}^{\nu+1},\,x_{\overline{\cC}_r}^{\,\nu}\big)\,;\,
\big(x_{\cC_r}^{\,\nu},\,x_{\overline{\cC}_r}^{\,\nu},\,x_{\cE_r}^{\,\nu}\big)\Big),\hat x_{\cC_r}^{\nu+1}-\bar{x}_{\cC_r}^{\nu+1}}}_{=0} \\
    &+\frac{\tilde L_r}{2}\norm{\hat x_{\cC_r}^{\nu+1}-\bar{x}_{\cC_r}^{\nu+1}}^2.
\end{aligned}
\end{equation*}
By definition of $\bar{x}_{\cC_r}^{\nu+1}$ and Assumption \ref{assumption:surrogation}.(i)(iii), we have
\[
\begin{aligned}
&\widetilde\Phi_r\!\Big(\big(\bar{x}_{\cC_r}^{\nu+1},\,x_{\overline{\cC}_r}^{\,\nu}\big)\,;\,   \big(x_{\cC_r}^{\,\nu},\,x_{\overline{\cC}_r}^{\,\nu},\,x_{\cE_r}^{\,\nu}\big)\Big)\leq \widetilde\Phi_r\!\Big(\big(x_{\cC_r}^{\nu},\,x_{\overline{\cC}_r}^{\,\nu}\big)\,;\,   \big(x_{\cC_r}^{\,\nu},\,x_{\overline{\cC}_r}^{\,\nu},\,x_{\cE_r}^{\,\nu}\big)\Big)-\frac{\norm{\nabla_{\cC_r}\Phi(\bx^\nu)}^2}{2\tilde L_r}\\
=&\Phi_r(\bx^\nu)-\frac{\norm{\nabla_{\cC_r}\Phi(\bx^\nu)}^2}{2\tilde L_r}.
\end{aligned}
\]
Combining the above inequalities, we obtain
\begin{equation}    \label{eq:basic1_surrogate}   \widetilde\Phi_r\!\Big(\big(\hat{x}_{\cC_r}^{\nu+1},\,x_{\overline{\cC}_r}^{\,\nu}\big)\,;\,    \big(x_{\cC_r}^{\,\nu},\,x_{\overline{\cC}_r}^{\,\nu},\,x_{\cE_r}^{\,\nu}\big)\Big)\leq\Phi_r(\bx^\nu)-\frac{\norm{\nabla_{\cC_r}\Phi(\bx^\nu)}^2}{2\tilde L_r}+\frac{\tilde L_r}{2}\norm{\hat x_{\cC_r}^{\nu+1}-\bar{x}_{\cC_r}^{\nu+1}}^2.
\end{equation}
By \eqref{eq:descent_sum_noz} and Assumption \ref{assumption:surrogation}.(ii), we have 
\[
\begin{aligned}
&\Phi(\bx^{\nu+1})-\Phi(\bx^\nu)
\leq
\sum_{r=1}^p\tau_r^\nu\Big(\Phi\big(\bx^\nu+P_r(\hat\bx^{\nu+1}-\bx^\nu)\big)-\Phi(\bx^\nu)\Big)\\
\leq&\sum_{r=1}^p\tau_r^\nu\qty(\widetilde\Phi_r\!\Big(\big(\hat{x}_{\cC_r}^{\nu+1},\,x_{\overline{\cC}_r}^{\,\nu}\big)\,;\,    \big(x_{\cC_r}^{\,\nu},\,x_{\overline{\cC}_r}^{\,\nu},\,x_{\cE_r}^{\,\nu}\big)\Big)-\Phi_r(\bx^\nu)).
\end{aligned}
\]
Combining this inequality with \eqref{eq:basic1_surrogate} completes the proof. \qed
\end{proof}

\begin{lemma}
\label{le:scvx_descent_surrogate}
Under Assumption \ref{assumption:surrogation} and any stepsize choice satisfying $\tau_r^\nu\geq0$ and  $\sum_{r=1}^p\tau_r^\nu\leq 1$, the following holds:
\begin{equation}
\label{eq:scvx_descent_surrogate}
\begin{aligned}
\Phi(\bx^{\nu+1})\leq\Phi(\bx^\nu)+\sum_{r\in[p]}\tau_r^\nu\qty[-\frac{\tilde\mu_r}{4}\norm{P_r(\hat \bx^{\nu+1}-\bx^\nu)}^2+\frac{\tilde L_r+\tilde\mu_r}{2}\norm{P_r(\bar{\bx}^{\nu+1}-\hat \bx^{\nu+1})}^2].
\end{aligned}
\end{equation}
\end{lemma}
\begin{proof}
The proof follows the same steps as Lemma~\ref{le:scvx_descent}, with the exact block objective replaced by the
aggregated surrogate and the final bound transferred to $\Phi$ via the upper bound condition in
Assumption~\ref{assumption:surrogation}(ii) (cf. the proof of Lemma~\ref{le:scvx_descent_contraction_surrogate}). We omit the details. \qed
\end{proof}

Next, we bound the discrepancy term $\norm{P_r(\bar\bx^{\nu+1}-\hat \bx^{\nu+1})}^2$.

\begin{lemma}
\label{le:iterates_gap_surrogate}
In the setting of Lemma \ref{le:scvx_descent_contraction_surrogate} and Lemma \ref{le:scvx_descent_surrogate}, the following holds: for any  $r\in [p]$, 
\[
\begin{aligned}
    \norm{P_r(\hat \bx^{\nu+1}-\bar{\bx}^{\nu+1})}^2\leq&\;\frac{2|\cC_r|D_r}{\tilde\mu_r^2}\sum_{\ell=\nu-D_r}^{\nu-1}\qty({\tilde L}_{\partial r}^2\norm{P_{\partial r}(\bx^{\ell+1}-\bx^{\ell})}^2+\sigma_r\tilde\ell_r^2\norm{P_r(\bx^{\ell+1}-\bx^{\ell})}^2).
\end{aligned}
\]
where $\sigma_r=\max_{i\in\cC_r}{\rm deg}_{\cG_r}(i)$.
\end{lemma}
\begin{proof}
Define the \emph{delayed} block update for the aggregated surrogate:
\begin{equation*}
\label{eq:C_ri_surrogate}
\hat x_{\cC_r,i}^{\nu+1}
:=\argmin_{x_{\cC_r}}
\ \widetilde\Phi_r\!\Big(\big(x_{\cC_r},\,x_{\overline{\cC}_r}^{\,\nu-\bdelta_i}\big)\,;\,
\big(x_{\cC_r}^{\,\nu-\bd_i},\,x_{\overline{\cC}_r}^{\,\nu-\bdelta_i},\,x_{\cE_r}^{\,\nu-\bD_i}\big)\Big), \text{ thus }
\hat x_i^{\nu+1}=\big[\hat x_{\cC_r,i}^{\nu+1}\big]_i.
\end{equation*}
By the optimality condition, 
\[
\begin{aligned}
0=&\nabla\widetilde\Phi_r\!\Big(\big(\bar{x}_{\cC_r}^{\nu+1},\,x_{\overline{\cC}_r}^{\,\nu}\big)\,;\,
\big(x_{\cC_r}^{\,\nu},\,x_{\overline{\cC}_r}^{\,\nu},\,x_{\cE_r}^{\,\nu}\big)\Big)-\nabla\widetilde\Phi_r\!\Big(\big(\hat{x}_{\cC_r,i}^{\nu+1},\,x_{\overline{\cC}_r}^{\,\nu-\bdelta_i}\big)\,;\,
\big(x_{\cC_r}^{\,\nu-\bd_i},\,x_{\overline{\cC}_r}^{\,\nu-\bdelta_i},\,x_{\cE_r}^{\,\nu-\bD_i}\big)\Big)\\
=&\nabla\widetilde\Phi_r\!\Big(\big(\bar{x}_{\cC_r}^{\nu+1},\,x_{\overline{\cC}_r}^{\,\nu}\big)\,;\,
\big(x_{\cC_r}^{\,\nu},\,x_{\overline{\cC}_r}^{\,\nu},\,x_{\cE_r}^{\,\nu}\big)\Big)-\nabla\widetilde\Phi_r\!\Big(\big(\bar{x}_{\cC_r}^{\nu+1},\,x_{\overline{\cC}_r}^{\,\nu-\bdelta_i}\big)\,;\,
\big(x_{\cC_r}^{\,\nu},\,x_{\overline{\cC}_r}^{\,\nu-\bdelta_i},\,x_{\cE_r}^{\,\nu}\big)\Big)\\
&+\nabla\widetilde\Phi_r\!\Big(\big(\bar{x}_{\cC_r}^{\nu+1},\,x_{\overline{\cC}_r}^{\,\nu-\bdelta_i}\big)\,;\,
\big(x_{\cC_r}^{\,\nu},\,x_{\overline{\cC}_r}^{\,\nu-\bdelta_i},\,x_{\cE_r}^{\,\nu}\big)\Big)-\nabla\widetilde\Phi_r\!\Big(\big(\bar{x}_{\cC_r}^{\nu+1},\,x_{\overline{\cC}_r}^{\,\nu-\bdelta_i}\big)\,;\,
\big(x_{\cC_r}^{\,\nu},\,x_{\overline{\cC}_r}^{\,\nu-\bdelta_i},\,x_{\cE_r}^{\,\nu-\bD_i}\big)\Big)\\
&+\nabla\widetilde\Phi_r\!\Big(\big(\bar{x}_{\cC_r}^{\nu+1},\,x_{\overline{\cC}_r}^{\,\nu-\bdelta_i}\big)\,;\,
\big(x_{\cC_r}^{\,\nu},\,x_{\overline{\cC}_r}^{\,\nu-\bdelta_i},\,x_{\cE_r}^{\,\nu-\bD_i}\big)\Big)-\nabla\widetilde\Phi_r\!\Big(\big(\hat{x}_{\cC_r,i}^{\nu+1},\,x_{\overline{\cC}_r}^{\,\nu-\bdelta_i}\big)\,;\,
\big(x_{\cC_r}^{\,\nu},\,x_{\overline{\cC}_r}^{\,\nu-\bdelta_i},\,x_{\cE_r}^{\,\nu-\bD_i}\big)\Big).
\end{aligned}
\]
Assumption~\ref{assumption:surrogation} yields
\[
\begin{aligned}
&\tilde\mu_r\norm{\hat x_{\cC_r,i}^{\nu+1}-\bar{x}_{\cC_r}^{\nu+1}}^2\\
\leq&-\inner{\nabla\widetilde\Phi_r\!\Big(\big(\bar{x}_{\cC_r}^{\nu+1},\,x_{\overline{\cC}_r}^{\,\nu}\big)\,;\,
\big(x_{\cC_r}^{\,\nu},\,x_{\overline{\cC}_r}^{\,\nu},\,x_{\cE_r}^{\,\nu}\big)\Big)-\nabla\widetilde\Phi_r\!\Big(\big(\bar{x}_{\cC_r}^{\nu+1},\,x_{\overline{\cC}_r}^{\,\nu-\bdelta_i}\big)\,;\,
\big(x_{\cC_r}^{\,\nu},\,x_{\overline{\cC}_r}^{\,\nu-\bdelta_i},\,x_{\cE_r}^{\,\nu}\big)\Big),\hat x_{\cC_r,i}^{\nu+1}-\bar{x}_{\cC_r}^{\nu+1}}\\
&-\inner{\nabla\widetilde\Phi_r\!\Big(\big(\bar{x}_{\cC_r}^{\nu+1},\,x_{\overline{\cC}_r}^{\,\nu-\bdelta_i}\big)\,;\,
\big(x_{\cC_r}^{\,\nu},\,x_{\overline{\cC}_r}^{\,\nu-\bdelta_i},\,x_{\cE_r}^{\,\nu}\big)\Big)-\nabla\widetilde\Phi_r\!\Big(\big(\bar{x}_{\cC_r}^{\nu+1},\,x_{\overline{\cC}_r}^{\,\nu-\bdelta_i}\big)\,;\,
\big(x_{\cC_r}^{\,\nu},\,x_{\overline{\cC}_r}^{\,\nu-\bdelta_i},\,x_{\cE_r}^{\,\nu-\bD_i}\big)\Big),\hat x_{\cC_r,i}^{\nu+1}-\bar{x}_{\cC_r}^{\nu+1}}.
\end{aligned}
\]
Then,
\begin{align*}
&\mu_r\norm{\hat x_{\cC_r,i}^{\nu+1}-\bar x_{\cC_r}^{\nu+1}}\leq{\tilde L}_{\partial r}\norm{x_{\cN_{\cC_r}}^{\nu-\bdelta_i}-x_{\cN_{\cC_r}}^{\nu}}+\tilde\ell_r\norm{x_{\cE_r}^{\nu-\bD_i}-x_{\cE_r}^{\nu}}.
\end{align*}
Note that, for any $w$, $\norm{w_{\cE_r}}^2\leq\sigma_r\norm{w_{\cC_r}}^2$, and $\norm{\hat x_{\cC_r}^{\nu+1}-\bar x_{\cC_r}^{\nu+1}}^2\leq\sum_{i\in\cC_r}\norm{\hat x_{\cC_r,i}^{\nu+1}-\bar x_{\cC_r}^{\nu+1}}^2$. Hence, 
\[
\begin{aligned}
    \norm{\hat x_{\cC_r}^{\nu+1}-\bar{x}_{\cC_r}^{\nu+1}}^2\leq&\frac{2|\cC_r|D_r}{\tilde\mu_r^2}\sum_{\ell=\nu-D_r}^{\nu-1}\qty({\tilde L}_{\partial r}^2\norm{x_{\cN_{\cC_r}}^{\ell+1}-x_{\cN_{\cC_r}}^{\ell}}^2+\sigma_r\tilde\ell_r^2\norm{x_{\cC_r}^{\ell+1}-x_{\cC_r}^{\ell}}^2).
\end{aligned}
\]
This completes the proof.\qed
\end{proof}


\medskip

 We can now proceed with the proof of Theorem~\ref{thm:convergence_scvx_surrogate_uniform}. 
 
Let $\Delta_{\cS}^\nu:=\sum_{\ell=\nu-D}^{\nu-1}\norm{x_{\cS}^{\ell+1}-x_{\cS}^\ell}^2$ for $\cS\subseteq[p]$, and write $\Delta^\nu:=\Delta^\nu_{[p]}$. By \eqref{eq:scvx_descent_contraction_surrogate}, \eqref{eq:scvx_descent_surrogate}, and Lemma~\ref{le:iterates_gap_surrogate}, we have
\begin{equation}
\label{eq:combined_one_step_bound}
\begin{aligned}
\Phi(\bx^{\nu+1})
\;\le\;&
\Phi(\bx^\nu)
-\sum_{r\in[p]}\frac{\tau}{4\tilde L_r}\,\|P_r\nabla\Phi(\bx^\nu)\|^2
-\sum_{r\in[p]}\frac{\tilde\mu_r}{8\tau}\,\|P_r(\bx^{\nu+1}-\bx^{\nu})\|^2\\
&+\sum_{r\in[p]}\tau A_r
\Delta_{\cN_{\cC_r}}^\nu+\sum_{r\in[p]}\tau \tilde{A}_r
\Delta_{\cC_r}^\nu.
\end{aligned}
\end{equation}
By strong convexity (implying $\|\nabla\Phi(\bx^\nu)\|^2\ge 2\mu\big(\Phi(\bx^\nu)-\Phi^\star\big)$), we have
\begin{equation*}
\begin{aligned}
\Phi(\bx^{\nu+1})-\Phi^\star
\;\le\;&
\Big(1-\frac{\tau}{2\kappa}\Big)\big(\Phi(\bx^\nu)-\Phi^\star\big)
-\sum_{r\in\cJ\cup\{s\in[p]:|\cC_s|>1\} }\frac{\tilde\mu_r}{8\tau}\,\|P_r(\bx^{\nu+1}-\bx^{\nu})\|^2\\
&+\tau A_\cJ\sum_{r\in\cJ}\Delta_{\cC_r}^\nu+\tau\qty(\max_{r:|\cC_r|>1}\tilde A_r)\sum_{r:|\cC_r|>1}\Delta_{\cC_r}^\nu\\
=\;&\Big(1-\frac{\tau}{2\kappa}\Big)\big(\Phi(\bx^\nu)-\Phi^\star\big)
-\sum_{r\in\cJ\cup\{s:|\cC_s|>1\} }\frac{\tilde\mu_r}{8\tau}\,\|P_r(\bx^{\nu+1}-\bx^{\nu})\|^2\\
&+\tau\qty(A_\cJ+\max_{r:|\cC_r|>1}\tilde A_r)\sum_{r\in\cJ\cup\{s:|\cC_s|>1\} }\Delta_{\cC_r}^\nu.
\end{aligned}
\end{equation*}
A standard delay-window inequality (e.g., ~\cite[Lemma~5]{feyzmahdavian2023asynchronous}) yields linear convergence under the following condition   
\begin{equation*}
2D+1\ \le\
\min\left\{
\frac{2\tilde L}{\tau\mu},
\frac{\frac{1}{8\tau}\min\limits_{r\in\cJ\cup\{s:|\cC_s|>1\}}\tilde\mu_r}
{\tau\qty(A_\cJ+\max_{r:|\cC_r|>1}\tilde A_r)}
\right\},
\end{equation*}
with the linear rate   given  by \eqref{eq:rate-constant-step-surrogate_clean4}.
Using uniform stepsize values, $\tau\in(0, {1}/{p}]$, yields \eqref{eq:delay_lemma_condition_2_surrogate_clean4}.  \qed
 \vspace{-0.3cm}

\section{Proof of Theorem \ref{thm:convergence_cvx}}
\label{proof_convergence_cvx}
Let $\Phi^\star:=\min\Phi(\bx)$ and choose any $\bx^\star\in\argmin\Phi(\bx)$. For any $r\in[p]$ and $i\in\cC_r$, by Assumption \ref{assumption:entire_surroagte}, we have
\[
\begin{aligned}
&\widetilde{\Phi}_r^{\rm all}\Big(
(\hat \bx_{\cC_r,i}^{\nu+1},\,\bx_{\overline{\cC}_r}^{\,\nu-\bdelta_i})
\;;\;
(\bx_{\cC_r}^{\,\nu-\bd_i},\,\bx_{\overline{\cC}_r}^{\,\nu-\bdelta_i},\,\bx_{\cE_r}^{\,\nu-\bD_i})
\Big)\\
\le& 
\widetilde{\Phi}_r^{\rm all}\Big(
(\bx_{\cC_r}^{\star},\,\bx_{\overline{\cC}_r}^{\,\nu-\bdelta_i})
\;;\;
(\bx_{\cC_r}^{\,\nu-\bd_i},\,\bx_{\overline{\cC}_r}^{\,\nu-\bdelta_i},\,\bx_{\cE_r}^{\,\nu-\bD_i})
\Big)
-\frac{\widetilde L_r}{2}\,\big\|\hat \bx_{\cC_r,i}^{\nu+1}-\bx_{\cC_r}^{\star}\big\|^2\\
\le&\Phi^*+\bar L_r\norm{\bx_{\overline{\cC}_r}^{\,\nu-\bdelta_i}-x_{\overline{\cC}_r}^\star}^2+\frac{\bar L_r}{2}\norm{\bx_{{\cC}_r}^{\,\nu-\bd_i}-x_{\cC_r}^\star}^2+\frac{\bar L_r}{2}\norm{\bx_{{\cE}_r}^{\,\nu-\bD_i}-x_{\cE_r}^\star}^2-\frac{\widetilde L_r}{2}\,\big\|\hat \bx_{\cC_r,i}^{\nu+1}-\bx_{\cC_r}^{\star}\big\|^2\\
\le&\Phi^\star
+\frac{\bar L_r(\sigma_r+1)(D_r+1)}{2}\,\qty(\big\|\bx^{\nu}-\bx^{\star}\big\|^2
+ \,\Delta^{\nu})
-\frac{\widetilde L_r}{2}\,\big\|\hat \bx_{\cC_r,i}^{\nu+1}-\bx_{\cC_r}^{\star}\big\|^2,
\end{aligned}
\]
where $\hat x_{\cC_r,i}^{\nu+1}$ is defined in \eqref{eq:C_ri_surrogate}. Similarly, we have
\[
\begin{aligned}
&\widetilde{\Phi}_r^{\rm all}\Big(
(\hat \bx_{\cC_r,i}^{\nu+1},\,\bx_{\overline{\cC}_r}^{\,\nu})
\;;\;
(\bx_{\cC_r}^{\,\nu},\,\bx_{\overline{\cC}_r}^{\,\nu},\,\bx_{\cE_r}^{\,\nu})
\Big)\\
\le&
\widetilde{\Phi}_r^{\rm all}\Big(
(\hat \bx_{\cC_r,i}^{\nu+1},\,\bx_{\overline{\cC}_r}^{\,\nu-\bdelta_i})
\;;\;
(\bx_{\cC_r}^{\,\nu-\bd_i},\,\bx_{\overline{\cC}_r}^{\,\nu-\bdelta_i},\,\bx_{\cE_r}^{\,\nu-\bD_i})
\Big)
+\frac{\bar L_r(\sigma_r+1)D_r}{2}\,\Delta^{\nu}.
\end{aligned}
\]
Next, we bound $\norm{\hat \bx_{\cC_r,i}^{\nu+1}-\hat \bx_{\cC_r}^{\nu+1}}^2$. Note that 
\[
\begin{aligned}
\norm{\hat \bx_{\cC_r,i}^{\nu+1}-\hat \bx_{\cC_r}^{\nu+1}}^2=\sum_{j\in\cC_r}\norm{(\hat \bx_{\cC_r,i}^{\nu+1})_j-(\hat \bx_{\cC_r,j}^{\nu+1})_j}^2\leq\sum_{j\in\cC_r}\norm{\hat \bx_{\cC_r,i}^{\nu+1}-\hat \bx_{\cC_r,j}^{\nu+1}}^2.
\end{aligned}
\]
Moreover, $\norm{\hat \bx_{\cC_r,i}^{\nu+1}-\hat \bx_{\cC_r,j}^{\nu+1}}^2$ can be bounded exactly as in Lemma \ref{le:iterates_gap_surrogate}. Therefore, 
\[
\norm{\hat \bx_{\cC_r,i}^{\nu+1}-\hat \bx_{\cC_r}^{\nu+1}}^2\leq\frac{2{\tilde L}_{\partial r}^2|\cC_r|^2D_r}{\tilde\mu_r^2}\Delta_{\cN_{\cC_r}^\nu}+\frac{2\sigma_r\tilde\ell_r^2|\cC_r|^2D_r}{\tilde\mu_r^2}\Delta_{\cC_r}^\nu.
\]
Thus, for any $\beta>0$,
\begin{align*}
&\widetilde{\Phi}_r^{\rm all}\Big(
(\hat \bx_{\cC_r}^{\nu+1},\,\bx_{\overline{\cC}_r}^{\,\nu})
\;;\;
(\bx_{\cC_r}^{\,\nu},\,\bx_{\overline{\cC}_r}^{\,\nu},\,\bx_{\cE_r}^{\,\nu})
\Big)
-\widetilde{\Phi}_r^{\rm all}\Big(
(\hat \bx_{\cC_r,i}^{\nu+1},\,\bx_{\overline{\cC}_r}^{\,\nu})
\;;\;
(\bx_{\cC_r}^{\,\nu},\,\bx_{\overline{\cC}_r}^{\,\nu},\,\bx_{\cE_r}^{\,\nu})
\Big)\\
&\le
\Big\langle \nabla_{\cC_r}{\Phi}(\bx^\nu),\ \hat \bx_{\cC_r}^{\nu+1}-\hat \bx_{\cC_r,i}^{\nu+1}\Big\rangle
+\frac{\widetilde L_r}{2}\,\big\|\hat \bx_{\cC_r}^{\nu+1}-\bx_{\cC_r}^{\nu}\big\|^2
-\frac{\widetilde\mu_r}{2}\,\big\|\hat \bx_{\cC_r,i}^{\nu+1}-\bx_{\cC_r}^{\nu}\big\|^2\\
&\le
\beta\,\big\|\nabla_{\cC_r}{\Phi}(\bx^\nu)\big\|^2
+\frac{K_{2,r}}{4\beta}\Delta^\nu+\frac{\widetilde L_r}{2\tau^2}\,\big\|\bx_{\cC_r}^{\nu+1}-\bx_{\cC_r}^{\nu}\big\|^2
-\frac{\widetilde\mu_r}{2\tau^2}\,\big\|\bx_i^{\nu+1}-\bx_i^{\nu}\big\|^2,
\end{align*}
where $K_{2,r}=\frac{2|\cC_r|^2D_r}{\widetilde\mu_r^2}(\tilde L_{\partial r}^2+\sigma_r\tilde\ell_r^2)$, and the first inequality comes from the fact that, for any $\mu$-strongly convex and $L$-smooth function $f$, 
\[
f(y)-f(z)\ \le\ \langle \nabla f(x),\,y-z\rangle
+\frac{L}{2}\|y-x\|^2-\frac{\mu}{2}\|z-x\|^2.
\]
Combining the above bounds all together yields
\begin{align*}
&\widetilde{\Phi}_r^{\rm all}\Big(
(\hat \bx_{\cC_r}^{\nu+1},\,\bx_{\overline{\cC}_r}^{\,\nu})
\;;\;
(\bx_{\cC_r}^{\,\nu},\,\bx_{\overline{\cC}_r}^{\,\nu},\,\bx_{\cE_r}^{\,\nu})
\Big)-\Phi^\star\\
&\le
\frac{\bar L_r(\sigma_r+1)(D_r+1)}{2}\,\big\|\bx^{\nu}-\bx^{\star}\big\|^2
-\frac{\widetilde L_r}{2\tau^2}\,\big\|\bx_i^{\nu+1}-\bx_i^\star\big\|^2
+ K_{3,r}\,\Delta^\nu
+\beta\,\big\|\nabla_{\cC_r}{\Phi}(\bx^\nu)\big\|^2\\
&\qquad
+\frac{\widetilde L_r}{2\tau^2}\,\big\|\bx_{\cC_r}^{\nu+1}-\bx_{\cC_r}^{\nu}\big\|^2
-\frac{\widetilde\mu_r}{2\tau^2}\,\big\|\bx_i^{\nu+1}-\bx_i^{\nu}\big\|^2,
\end{align*}
where $K_{3,r}=\frac{\bar L_r(\sigma_r+1)(2D_r+1)}{2}+\frac{K_{2,r}}{4\beta}$. Taking $\frac{1}{|\cC_r|}\sum_{i\in \cC_r}$ for both sides yields
\begin{align*}
&\widetilde{\Phi}_r^{\rm all}\Big(
(\hat \bx_{\cC_r}^{\nu+1},\,\bx_{\overline{\cC}_r}^{\,\nu})
\;;\;
(\bx_{\cC_r}^{\,\nu},\,\bx_{\overline{\cC}_r}^{\,\nu},\,\bx_{\cE_r}^{\,\nu})
\Big)-\Phi^\star\\
\le&\;
\frac{\bar L_r(\sigma_r+1)(D_r+1)}{2}\,\big\|\bx^{\nu}-\bx^{\star}\big\|^2
-\frac{\widetilde L_r}{2|\cC_r|\tau^2}\,\big\|\bx_{\cC_r}^{\nu+1}-\bx_{\cC_r}^\star\big\|^2\\
&\;+ K_{3,r}\,\Delta^\nu
+\beta\,\big\|\nabla_{\cC_r}{\Phi}(\bx^\nu)\big\|^2+\frac{1}{2\tau^2}\Big(\widetilde L_r-\frac{\widetilde\mu_r}{|\cC_r|}\Big)
\,\big\|\bx_{\cC_r}^{\nu+1}-\bx_{\cC_r}^{\nu}\big\|^2 .
\end{align*}
Now we relate $\Phi(\bx^{\nu+1})$ to the aggregated surrogates. Note that for $\tau\in[0,1/p]$,
\[
\Phi(\bx^{\nu+1})
\le (1-\tau p)\,\Phi(\bx^\nu)
+\tau\sum_{r\in[p]}\Phi(P_r(\hat\bx^{\nu+1}-\bx^\nu)),
\]
hence by Assumption \ref{assumption:surrogation}.(ii),
\begin{align*}
\Phi(\bx^{\nu+1})
&\le
(1-\tau p)\Big[\Phi^\star+\frac{L}{2}\|\bx^\nu-\bx^\star\|^2\Big]
+\tau\sum_{r\in[p]}
\widetilde{\Phi}_r^{\rm all}\Big(
(\hat \bx_{\cC_r}^{\nu+1},\,\bx_{\overline{\cC}_r}^{\,\nu})
\;;\;
(\bx_{\cC_r}^{\,\nu},\,\bx_{\overline{\cC}_r}^{\,\nu},\,\bx_{\cE_r}^{\,\nu})
\Big).
\end{align*}
Thus,
\begin{equation}\label{eq:star}
\begin{aligned}
\Phi(\bx^{\nu+1})-\Phi^\star
\le\;&\frac{L(1-\tau p)+\tau p(\sigma+1)(D+1)}{2}\,\|\bx^\nu-\bx^\star\|^2
-\frac{\widetilde L_{\min}}{2C\tau}\,\|\bx^{\nu+1}-\bx^\star\|^2\\
&\quad
+K_3\,\tau p\,\Delta^\nu
+\beta\tau\,\|\nabla{\Phi}(\bx^\nu)\|^2
+\frac{1}{2\tau}\Big(\widetilde L-\frac{\widetilde\mu}{C}\Big)\,\|\bx^{\nu+1}-\bx^\nu\|^2,
\end{aligned}
\end{equation}
where $K_3=\max_{r\in[p]}K_{3,r}$. Moreover, by \eqref{eq:combined_one_step_bound}, we have
\begin{equation}\label{eq:starstar}
\Phi(\bx^{\nu+1})-\Phi^\star
\le
\Phi(\bx^\nu)-\Phi^\star
-\frac{\tau}{4\widetilde L}\,\|\nabla{\Phi}(\bx^\nu)\|^2
-\frac{\widetilde\mu}{8\tau}\,\|\bx^{\nu+1}-\bx^\nu\|^2
+\tau K_1\,\Delta^\nu,
\end{equation}
where $K_1=A_\cJ+\max_{r:|\cC_r|>1}\tilde A_r$. Choose $\beta=1$ and $\tau$ such that $\frac{\widetilde L_{\min}}{2C\tau}\ > \frac{L(1-\tau p)+\tau p(\sigma+1)(D+1)}{2}$, and let $\nu\cdot\eqref{eq:starstar}+\eqref{eq:star}$, we have
\begin{align*}
V^{\nu+1}\le V^\nu
-\Big[\frac{\widetilde\mu\nu}{8\tau}-\frac{\widetilde L-\widetilde\mu/C}{2\tau}\Big]\|\bx^{\nu+1}-\bx^\nu\|^2+\big(\tau K_1 \nu+\tau p K_3\big)\Delta^\nu,
\end{align*}
where $V^{\nu}=\nu(\Phi(\bx^\nu)-\Phi^\star)+\frac{\widetilde L_{\min}}{2C\tau}\|\bx^\nu-\bx^\star\|^2$. For $\nu>4(\frac{\widetilde L}{\widetilde\mu}-\frac{1}{C})$, $\frac{\widetilde\mu\nu}{8\tau}-\frac{\widetilde L-\widetilde\mu/C}{2\tau}>0$, by the delay-window inequality in \cite[Lemma 5]{feyzmahdavian2023asynchronous}, if 
\[
(D+1)\big(\tau K_1\nu+\tau p K_3\big)
\le
\frac{\widetilde\mu\nu}{8\tau}-\frac{\widetilde L-\widetilde\mu/C}{2\tau},
\]
we have
\[
\Phi(\bx^\nu)-\Phi^\star
\le
\frac{\Phi(\bx^0)-\Phi^\star+\frac{\widetilde L_{\min}}{2C\tau}\,\|\bx^0-\bx^\star\|^2}{\nu}.
\]
Consequently, for sufficiently large $\nu$, it suffices to pick $\tau$ small enough so that the previous inequality holds; in particular,
\[
\tau
\le
\sqrt{
\frac{\frac{\widetilde\mu\nu}{8}-\frac{\widetilde L-\widetilde\mu/C}{2}}
{(D+1)(K_1\nu+p K_3)}
}
\ \asymp\
\sqrt{\frac{\widetilde\mu}{8(D+1)K_1}}.
\]
Therefore, for sufficiently large $\nu$, it is enough to choose $\tau$ satisfying all the above requirements, namely
\[
\tau \ \le\ \min\left\{\frac{1}{p},\ \sqrt{\frac{\widetilde\mu}{16(D+1)\,K_1}},\frac{\tilde L_{\min}}{CL+C(\sigma+1)(D+1)}\right\}.
\]
This completes the proof. \qed

\vspace{-0.3cm}
\section{Proof of Theorem \ref{thm:convergence_ncvx}}
\label{proof_convergence_ncvx}
By \eqref{eq:combined_one_step_bound}, we have 
\begin{align*}
\Phi(\bx^{\nu+1})\le&\Phi(\bx^\nu)-\frac{\tau}{4\tilde L}\,\|\nabla\Phi(\bx^\nu)\|^2-\frac{\tilde\mu}{8\tau}\,\|\bx^{\nu+1}-\bx^{\nu}\|^2+\tau K_1\Delta^\nu.
\end{align*}
where $K_1:=A_\cJ+\max_{r:|\cC_r|>1}\tilde A_r$. Observe that 
\[
\begin{aligned}
&\sum_{\ell=\nu-D}^{\nu-1}\qty[\ell-(\nu-1)+D]\norm{\bx^{\ell+1}-\bx^\ell}^2-\sum_{\ell=\nu+1-D}^{\nu}\qty[\ell-\nu+D]\norm{\bx^{\ell+1}-\bx^\ell}^2\\
=&\sum_{\ell=\nu-D}^{\nu-1}\norm{\bx^{\ell+1}-\bx^\ell}^2-D\norm{\bx^{\nu+1}-\bx^\nu}^2.
\end{aligned}
\]
Under the standard initialization $\bx^{-D}=\cdots=\bx^0$, define 
\[
V^\nu:=\Phi(\bx^\nu)+\tau K_1\sum_{\ell=\nu-D}^{\nu-1}\qty[\ell-(\nu-1)+D]\norm{\bx^{\ell+1}-\bx^\ell}^2.
\]
Then, we have
\[
V^{\nu+1}\leq V^\nu-\qty(\frac{\tilde\mu}{8\tau}-\tau D K_1)\norm{\bx^{\nu+1}-\bx^\nu}^2-\frac{\tau}{4\tilde L}\|\nabla\Phi(\bx^\nu)\|^2.
\]
Choose $\tau\leq\min\qty{\frac{1}{p},\frac{1}{2}\sqrt{\frac{\widetilde\mu}{2DK_1}}}$ so that $\frac{\tilde\mu}{8\tau}-\tau D K_1\ge 0$. Since $(V^\nu)$ is nonincreasing and $V^\nu\ge \Phi^\star$, summing the above inequality over $\nu$ yields \eqref{eq:sublinear_rate_cvx_ncvx}.   \qed

\bibliographystyle{abbrv}
\bibliography{ref}

\begin{thebibliography}{10}

\bibitem{BerTsi89}
D.~Bertsekas and J.~Tsitsiklis.
\newblock {\em Parallel and distributed computation: numerical methods}.
\newblock Athena Scientific, 2015.

\bibitem{bretto2013hypergraph}
A.~Bretto.
\newblock Hypergraph theory.
\newblock {\em An introduction. Mathematical Engineering. Cham: Springer},
  1:209--216, 2013.

\bibitem{cadena2017past}
C.~Cadena, L.~Carlone, H.~Carrillo, Y.~Latif, D.~Scaramuzza, J.~Neira, I.~Reid,
  and J.~J. Leonard.
\newblock Past, present, and future of simultaneous localization and mapping:
  Toward the robust-perception age.
\newblock {\em IEEE Transactions on robotics}, 32(6):1309--1332, 2017.

\bibitem{Cai_Keyes_02}
X.-C. Cai and D.~E. Keyes.
\newblock Nonlinearly preconditioned inexact newton algorithms.
\newblock {\em SIAM Journal on Scientific Computing}, 24(1):183--200, 2002.

\bibitem{Cannelli2021}
L.~Cannelli, F.~Facchinei, G.~Scutari, and V.~Kungurtsev.
\newblock Asynchronous optimization over graphs: Linear convergence under error
  bound conditions.
\newblock {\em IEEE Transactions on Automatic Control}, 66(10):4604–4619,
  Oct. 2021.

\bibitem{Cao_Chen_Scutari_2025}
T.~Cao, X.~Chen, and G.~Scutari.
\newblock Dcatalyst: A unified accelerated framework for decentralized
  optimization.
\newblock (arXiv:2501.18114), Jan. 2025.
\newblock arXiv:2501.18114.

\bibitem{chen2012diffusion}
J.~Chen and A.~H. Sayed.
\newblock Diffusion adaptation strategies for distributed optimization and
  learning over networks.
\newblock {\em IEEE Transactions on Signal Processing}, 60(8):4289--4305, 2012.

\bibitem{Chezhegov_et_al_22}
S.~Chezhegov, A.~Novitskii, A.~Rogozin, S.~Parsegov, P.~Dvurechensky, and
  A.~Gasnikov.
\newblock A general framework for distributed partitioned optimization.
\newblock {\em IFAC-PapersOnLine}, 55(13):139--144, 2022.
\newblock 9th IFAC Conference on Networked Systems NECSYS 2022.

\bibitem{Amir18}
A.~Daneshmand, G.~Scutari, and V.~Kungurtsev.
\newblock Second-order guarantees of distributed gradient algorithms.
\newblock {\em SIAM Journal on Optimization}, 30(4):3029--3068, 2020.

\bibitem{di2016next}
P.~Di~Lorenzo and G.~Scutari.
\newblock Next: In-network nonconvex optimization.
\newblock {\em IEEE Transactions on Signal and Information Processing over
  Networks}, 2(2):120--136, 2016.

\bibitem{dinh2022new}
C.~T. Dinh, T.~T. Vu, N.~H. Tran, M.~N. Dao, and H.~Zhang.
\newblock A new look and convergence rate of federated multitask learning with
  laplacian regularization.
\newblock {\em IEEE Transactions on Neural Networks and Learning Systems},
  35(6):8075--8085, 2022.

\bibitem{even2015analysis}
G.~Even and N.~Halabi.
\newblock Analysis of the min-sum algorithm for packing and covering problems
  via linear programming.
\newblock {\em IEEE Transactions on Information Theory}, 61(10):5295--5305,
  2015.

\bibitem{facchinei_feasible_2017}
F.~Facchinei, L.~Lampariello, and G.~Scutari.
\newblock Feasible methods for nonconvex nonsmooth problems with applications
  in green communications.
\newblock {\em Mathematical Programming}, 164(1-2):55--90, July 2017.

\bibitem{facchinei_vi-constrained_2014}
F.~Facchinei, J.-S. Pang, G.~Scutari, and L.~Lampariello.
\newblock {VI}-constrained hemivariational inequalities: distributed algorithms
  and power control in ad-hoc networks.
\newblock {\em Mathematical Programming}, 145(1-2):59--96, June 2014.

\bibitem{feyzmahdavian2023asynchronous}
H.~R. Feyzmahdavian and M.~Johansson.
\newblock Asynchronous iterations in optimization: New sequence results and
  sharper algorithmic guarantees.
\newblock {\em Journal of Machine Learning Research}, 24(158):1--75, 2023.

\bibitem{isufi2024graph}
E.~Isufi, F.~Gama, D.~I. Shuman, and S.~Segarra.
\newblock Graph filters for signal processing and machine learning on graphs.
\newblock {\em IEEE Transactions on Signal Processing}, 72:4745--4781, 2024.

\bibitem{KschischangFreyLoeliger2001}
F.~R. Kschischang, B.~J. Frey, and H.-A. Loeliger.
\newblock Factor graphs and the sum-product algorithm.
\newblock {\em IEEE Transactions on Information Theory}, 47(2):498--519, Feb.
  2001.

\bibitem{Latafat02112022}
P.~Latafat and P.~Patrinos.
\newblock Primal-dual algorithms for multi-agent structured optimization over
  message-passing architectures with bounded communication delays.
\newblock {\em Optimization Methods and Software}, 37(6):2038--2065, 2022.

\bibitem{malioutov2006walksums}
D.~M. Malioutov, J.~K. Johnson, and A.~S. Willsky.
\newblock Walk-sums and belief propagation in gaussian graphical models.
\newblock {\em Journal of Machine Learning Research}, 7(2):2031--2064, 2006.

\bibitem{Moallemi-VanRoy2009-quadratic}
C.~C. Moallemi and B.~Van~Roy.
\newblock Convergence of min-sum message passing for quadratic optimization.
\newblock {\em IEEE Transactions on Information Theory}, 55(5):2413--2423,
  2009.

\bibitem{moallemi2010convergence}
C.~C. Moallemi and B.~Van~Roy.
\newblock Convergence of min-sum message-passing for convex optimization.
\newblock {\em IEEE Transactions on Information Theory}, 56(4):2041--2050,
  2010.

\bibitem{nash1959random}
C.~S.~J. Nash-Williams.
\newblock Random walk and electric currents in networks.
\newblock In {\em Mathematical Proceedings of the Cambridge Philosophical
  Society}, volume~55, pages 181--194. Cambridge University Press, 1959.

\bibitem{Nedic_Olshevsky_Rabbat2018}
A.~Nedi\'{c}, A.~Olshevsky, and M.~Rabbat.
\newblock Network topology and communication-computation tradeoffs in
  decentralized optimization.
\newblock {\em Proceedings of the IEEE}, 106:953--976, 2018.

\bibitem{nedic2017achieving}
A.~Nedic, A.~Olshevsky, and W.~Shi.
\newblock Achieving geometric convergence for distributed optimization over
  time-varying graphs.
\newblock {\em SIAM Journal on Optimization}, 27(4):2597--2633, 2017.

\bibitem{AA09}
A.~Nedic and A.~Ozdaglar.
\newblock Distributed subgradient methods for multi-agent optimization.
\newblock {\em IEEE Transactions on Automatic Control}, 54(1):48--61, 2009.

\bibitem{AAP10}
A.~Nedic, A.~Ozdaglar, and P.~A. Parrilo.
\newblock Constrained consensus and optimization in multi-agent networks.
\newblock {\em IEEE Transactions on Automatic Control}, 55(4):922--938, 2010.

\bibitem{nk2012lx}
V.~NK.
\newblock Lx= b laplacian solvers and their algorithmic applications.
\newblock {\em Found. Tr. Theor. Comp. Sci}, 8(1-2):1--141, 2012.

\bibitem{romero2016kernel}
D.~Romero, M.~Ma, and G.~B. Giannakis.
\newblock Kernel-based reconstruction of graph signals.
\newblock {\em IEEE Transactions on Signal Processing}, 65(3):764--778, 2016.

\bibitem{roth2009fields}
S.~Roth and M.~J. Black.
\newblock Fields of experts.
\newblock {\em International Journal of Computer Vision}, 82(2):205--229, 2009.

\bibitem{ruder2017overview}
S.~Ruder.
\newblock An overview of multi-task learning in deep neural networks.
\newblock {\em arXiv preprint arXiv:1706.05098}, 2017.

\bibitem{ruozzi2013message}
N.~Ruozzi and S.~Tatikonda.
\newblock Message-passing algorithms for quadratic minimization.
\newblock {\em The Journal of Machine Learning Research}, 14(1):2287--2314,
  2013.

\bibitem{sayed2014adaptation}
A.~H. Sayed et~al.
\newblock Adaptation, learning, and optimization over networks.
\newblock {\em Foundations and Trends{\textregistered} in Machine Learning},
  7(4-5):311--801, 2014.

\bibitem{scaman2017optimal}
K.~Scaman, F.~Bach, S.~Bubeck, Y.~T. Lee, and L.~Massouli{\'e}.
\newblock Optimal algorithms for smooth and strongly convex distributed
  optimization in networks.
\newblock In {\em international conference on machine learning}, pages
  3027--3036. PMLR, 2017.

\bibitem{scutari_parallel_2017}
G.~Scutari, F.~Facchinei, and L.~Lampariello.
\newblock Parallel and {Distributed} {Methods} for {Constrained} {Nonconvex}
  {Optimization}—{Part} {I}: {Theory}.
\newblock {\em IEEE Transactions on Signal Processing}, 65(8):1929--1944, Apr.
  2017.

\bibitem{scutari2019distributed}
G.~Scutari and Y.~Sun.
\newblock Distributed nonconvex constrained optimization over time-varying
  digraphs.
\newblock {\em Mathematical Programming}, 176:497--544, 2019.

\bibitem{shi2015extra}
W.~Shi, Q.~Ling, G.~Wu, and W.~Yin.
\newblock Extra: An exact first-order algorithm for decentralized consensus
  optimization.
\newblock {\em SIAM Journal on Optimization}, 25(2):944--966, 2015.

\bibitem{WQKGW14}
W.~Shi, Q.~Ling, K.~Yuan, G.~Wu, and W.~Yin.
\newblock On the linear convergence of the admm in decentralized consensus
  optimization.
\newblock {\em IEEE Transactions on Signal Processing}, 62(7):1750--1761, 2014.

\bibitem{Shin_et_al_20}
S.~Shin, V.~M. Zavala, and M.~Anitescu.
\newblock Decentralized schemes with overlap for solving graph-structured
  optimization problems.
\newblock {\em IEEE Transactions on Control of Network Systems},
  7(3):1225--1236, 2020.

\bibitem{sun2022distributed}
Y.~Sun, G.~Scutari, and A.~Daneshmand.
\newblock Distributed optimization based on gradient tracking revisited:
  Enhancing convergence rate via surrogation.
\newblock {\em SIAM Journal on Optimization}, 32(2):354--385, 2022.

\bibitem{Toselli_Widlund_05}
A.~Toselli and O.~B. Widlund.
\newblock {\em Domain Decomposition Methods---Algorithms and Theory}, volume~34
  of {\em Springer Series in Computational Mathematics}.
\newblock Springer, 2005.

\bibitem{wainwright2005map}
M.~J. Wainwright, T.~S. Jaakkola, and A.~S. Willsky.
\newblock Map estimation via agreement on trees: message-passing and linear
  programming.
\newblock {\em IEEE transactions on information theory}, 51(11):3697--3717,
  2005.

\bibitem{WainwrightJordan2008}
M.~J. Wainwright and M.~I. Jordan.
\newblock Graphical models, exponential families, and variational inference.
\newblock {\em Foundations and Trends in Machine Learning}, 1(1--2):1--305,
  2008.

\bibitem{wan2025multitask}
Z.~Wan and S.~Vlaski.
\newblock Multitask learning with learned task relationships.
\newblock {\em arXiv preprint arXiv:2510.10570}, 2025.

\bibitem{wang2014efficient}
S.~Wang, A.~G. Schwing, and R.~Urtasun.
\newblock Efficient inference of continuous markov random fields with
  polynomial potentials.
\newblock {\em Advances in neural information processing systems}, 27, 2014.

\bibitem{wang2016trend}
Y.-X. Wang, J.~Sharpnack, A.~J. Smola, and R.~J. Tibshirani.
\newblock Trend filtering on graphs.
\newblock {\em Journal of Machine Learning Research}, 17(105):1--41, 2016.

\bibitem{weiss2001correctness}
Y.~Weiss and W.~T. Freeman.
\newblock Correctness of belief propagation in gaussian graphical models of
  arbitrary topology.
\newblock {\em Neural Computation}, 13(10):2173--2200, Oct. 2001.

\bibitem{Xu_Tian_Sun_Scutari_2020}
J.~Xu, Y.~Tian, Y.~Sun, and G.~Scutari.
\newblock Accelerated primal-dual algorithms for distributed smooth convex
  optimization over networks.
\newblock In {\em Proceedings of the Twenty Third International Conference on
  Artificial Intelligence and Statistics}, page 2381–2391. PMLR, June 2020.

\bibitem{KQW16}
K.~Yuan, Q.~Ling, and W.~Yin.
\newblock On the convergence of decentralized gradient descent.
\newblock {\em SIAM Journal on Optimization}, 26(3):1835--1854, 2016.

\bibitem{yuan2016convergence}
K.~Yuan, Q.~Ling, and W.~Yin.
\newblock On the convergence of decentralized gradient descent.
\newblock {\em SIAM Journal on Optimization}, 26(3):1835--1854, 2016.

\bibitem{yuan2018exact}
K.~Yuan, B.~Ying, X.~Zhao, and A.~H. Sayed.
\newblock Exact diffusion for distributed optimization and learning—part i:
  Algorithm development.
\newblock {\em IEEE Transactions on Signal Processing}, 67(3):708--723, 2018.

\bibitem{zhang2021convergence}
Z.~Zhang and M.~Fu.
\newblock Convergence properties of message-passing algorithm for distributed
  convex optimisation with scaled diagonal dominance.
\newblock {\em IEEE Transactions on Signal Processing}, 69:3868--3877, 2021.

\end{thebibliography}

\end{document}